\documentclass[12pt]{amsart}
\usepackage{amscd}
\usepackage{amssymb}
\usepackage{graphicx}
\usepackage{color}
\usepackage[centering,text={15.5cm,22cm},
		marginparwidth=20mm]{geometry}
\usepackage{hyperref}

\usepackage{mathrsfs}

\usepackage[all]{xy}

\newtheorem{theorem}{Theorem}[section]

\newtheorem{assumptions}[theorem]{Assumptions}

\newtheorem{lemma}[theorem]{Lemma}
\newtheorem{lem}[theorem]{Lemma}
\newtheorem{claim}[theorem]{Claim}
\newtheorem{proposition}[theorem]{Proposition}
\newtheorem{corollary}[theorem]{Corollary}
\newtheorem{conjecture}[theorem]{Conjecture}

\theoremstyle{definition}
\newtheorem{definition}[theorem]{Definition}
\newtheorem{definition-lemma}[theorem]{Definition-Lemma}

\newtheorem{construction}[theorem]{Construction}

\newtheorem{example}[theorem]{Example}

\theoremstyle{remark}
\newtheorem{remark}[theorem]{Remark}

\numberwithin{equation}{section}
\numberwithin{figure}{section}

\newcommand{\bbfamily}{\fontencoding{U}\fontfamily{bbold}\selectfont}
\newcommand{\textbb}[1]{{\bbfamily#1}}

\newcommand {\lfor} {\mbox{\textbb{[}}}
\newcommand {\rfor} {\mbox{\textbb{]}}}

\newcommand{\NN} {\mathbb{N}}
\newcommand{\ZZ} {\mathbb{Z}}

\newcommand{\RR} {\mathbb{R}}
\newcommand{\bR} {\RR}
\newcommand{\CC} {\mathbb{C}}
\newcommand{\FF} {\mathbb{F}}
\newcommand{\PP} {\mathbb{P}}
\newcommand{\VV} {\mathbb{V}}
\renewcommand{\AA} {\mathbb{A}}
\newcommand{\GG} {\mathbb{G}}
\newcommand{\invlim}{\varprojlim}

\newcommand {\shC}  {\mathcal{C}}

\newcommand {\shF}  {\mathcal{F}}

\newcommand {\shO}  {\mathcal{O}}

\newcommand {\shP}  {\mathcal{P}}

\newcommand {\foB}  {\mathfrak{B}}

\newcommand {\foD}  {\mathfrak{D}}

\newcommand {\oC} {\overline{C}}

\newcommand {\foI}  {\mathfrak{I}}

\newcommand {\foM}  {\mathfrak{M}}

\newcommand {\foS}  {\mathfrak{S}}
\newcommand {\foT}  {\mathfrak{T}}
\newcommand {\foU}  {\mathfrak{U}}

\newcommand {\foX}  {\mathfrak{X}}

\newcommand {\foZ}  {\mathfrak{Z}}

\newcommand {\fod}  {\mathfrak{d}}

\newcommand {\fof}  {\mathfrak{f}}

\newcommand {\fom}  {\mathfrak{m}}

\newcommand {\fop}  {\mathfrak{p}}

\newcommand {\fot}  {\mathfrak{t}}
\newcommand {\fou}  {\mathfrak{u}}

\newcommand {\Aff}  {\operatorname{Aff}}

\newcommand {\Aut}  {\operatorname{Aut}}

\newcommand {\can}  {\mathrm{can}}

\newcommand {\dlog} {\operatorname{dlog}}

\newcommand {\Gm} {\GG_m}
\newcommand {\gs}  {{\mathrm{gs}}}
\newcommand {\gp}  {{\operatorname{gp}}}

\newcommand {\hol}  {\mathrm{hol}}
\newcommand {\Hom}  {\operatorname{Hom}}

\newcommand {\Id}  {\operatorname{Id}}

\newcommand {\Int}  {\operatorname{Int}}

\newcommand {\kk} {\Bbbk}

\newcommand {\Lift}  {\operatorname{Lift}}

\newcommand {\Limits} {\operatorname{Limits}}

\renewcommand {\max} {{\operatorname{max}}}

\newcommand {\Mono} {\operatorname{Mono}}

\newcommand {\ord}  {\operatorname{ord}}
\newcommand {\out}  {\mathrm{out}}

\newcommand {\Pic}  {\operatorname{Pic}}

\newcommand {\Proj} {\operatorname{Proj}}

\newcommand {\rank} {\operatorname{rank}}

\newcommand {\red}  {{\operatorname{red}}}

\newcommand {\Scatter} {\operatorname{Scatter}}

\newcommand {\Sing} {\operatorname{Sing}}
\newcommand {\SL}  {\operatorname{SL}}
\newcommand {\Spec} {\operatorname{Spec}}

\newcommand {\Spf}  {\operatorname{Spf}}

\newcommand {\Supp} {\operatorname{Supp}}

\newcommand {\NE}   {\operatorname{NE}}

\newcommand {\Tr}  {\operatorname{Tr}}

\def\mapright#1{\smash{
  \mathop{\longrightarrow}\limits^{#1}}}

\def\oW{\overline{W}}

\def\bP{\Bbb P}

\def\bN{\Bbb N}

\def\SL{\operatorname{SL}}
\def\tSigma{\tilde{\Sigma}}

\def\dlog{\operatorname{dlog}}
\def\Aut{\operatorname{Aut}}

\def\Aff{\operatorname{Aff}}

\def\Supp{\operatorname{Supp}}

\def\tY{\tilde{Y}}

\def\tD{\tilde{D}}

\def\tvarphi{\tilde{\varphi}}

\def\trho{\tilde{\rho}}
\def\tq{\tilde{q}}

\def\ovarphi{\overline{\varphi}}

\def\tsigma{\tilde{\sigma}}

\def\tQ{\tilde{Q}}
\def\cO{\Cal O}

\def\rank{\operatorname{rank}}

\def\bZ{\Bbb Z}
\def\bC{\Bbb C}
\def\bQ{\Bbb Q}

\def\bG{\Bbb G}
\def\bR{\Bbb R}

\def\bA{\Bbb A}

\def\orho{\overline{\rho}}

\def\oX{\overline{X}}
\def\oP{\bar{P}}

\def\oI{\bar{I}}

\def\oU{\bar{U}}
\def\oX{\bar{X}}
\def\oY{\bar{Y}}
\def\oB{\bar{B}}
\def\oD{\bar{D}}
\def\ofoD{\bar{\foD}}
\def\oshP{\bar{\shP}}
\def\oSigma{\bar{\Sigma}}

\def\cA{\Cal A}
\def\cD{\Cal D}

\def\cF{\Cal F}
\def\tS{\tilde S}
\def\cI{\Cal I}
\def\cA{\Cal A}

\newcommand{\tcP}{\tilde{\mathcal P}}

\def\oS{\overline{S}}

\def\cP{\Cal P}

\def\tS{\tilde{S}}

\def\tB{\tilde{B}}

\def\oN{\overline{N}}

\def\cY{\Cal Y}

\def\cD{\Cal D}
\def\cX{\Cal X}

\def\cZ{\Cal Z}

\def\Spec{\operatorname{Spec}}
\def\Proj{\operatorname{Proj}}

\def\dlog{\operatorname{dlog}}

\def\ord{\operatorname{ord}}

\def\osigma{\overline{\sigma}}

\def\cK{\Cal K}

\def\cO{\Cal O}

\def\oR{\overline{R}}

\def\Sing{\operatorname{Sing}}

\def\Pic{\operatorname{Pic}}
\def\Spec{\operatorname{Spec}}
\def\Im{\operatorname{Im}}

\def\Hom{\operatorname{Hom}}
\def\Cal{\mathcal}
\def\oC{\overline C}

\def\tB{\tilde B}

\def\Efi#1#2#3#4#5{\displaystyle
#1\!\!-\!\!#2
\!\!-\!\!#3
\!\!-\!\!#4
\hskip-24.2pt\lower4.5pt\hbox{${\scriptstyle|}
\hskip-3.35pt\lower6pt\hbox{$#5$}$}}

\def\Evia#1#2#3#4#5{\displaystyle
#1\!\!-\!\!#2
\!\!-\!\!#3
\hskip-24.2pt\lower4.5pt\hbox{${\scriptstyle|}
\hskip-3.35pt\lower6pt\hbox{$#4\!\!-\!\!\!-\!\!\!-\!\!$}$\hskip2.3pt${\scriptstyle|}
\hskip-3.35pt\lower6pt\hbox{$#5$}$}}

\def\Ezia#1#2#3#4{\displaystyle
#1\!\!-\!\!#2
\hskip-14.8pt\lower4.5pt\hbox{${\scriptstyle|}
\hskip-3.35pt\lower6pt\hbox{$#3\!\!-\!\!$}$\hskip2.3pt${\scriptstyle|}
\hskip-3.35pt\lower6pt\hbox{$#4$}$}}

\def\Efia#1#2#3#4#5#6{\displaystyle
#1\!\!-\!\!#2
\!\!-\!\!#3
\!\!-\!\!#4
\hskip-24.2pt\lower4.5pt\hbox{${\scriptstyle|}
\hskip-3.35pt\lower6pt\hbox{$#5$}$\hskip5.7pt${\scriptstyle|}
\hskip-3.35pt\lower6pt\hbox{$#6$}$}}

\def\Esi#1#2#3#4#5#6{\displaystyle
#1\!\!-\!\!#2
\!\!-\!\!#3
\!\!-\!\!#4\!\!-\!\!#5
\hskip-24.2pt\lower4.5pt\hbox{${\scriptstyle|}
\hskip-3.35pt\lower6pt\hbox{$#6$
\lower3pt\hbox{\ }}$}}

\def\Esia#1#2#3#4#5#6#7{\displaystyle

#1\!\!-\!\!#2
\!\!-\!\!#3
\!\!-\!\!#4\!\!-\!\!#5
\hskip-24.2pt\lower4.5pt\hbox{${\scriptstyle|}
\hskip-3.35pt\lower6pt\hbox{$#6$\hskip-3.8pt\lower4.5pt\hbox{${\scriptstyle|}
\hskip-3.35pt\lower6pt\hbox{$#7$}$}}
\lower3pt\hbox{\ }$}}

\def\Ese#1#2#3#4#5#6#7{\displaystyle
#1\!\!-\!\!#2
\!\!-\!\!#3
\!\!-\!\!#4\!\!-\!\!#5\!\!-\!\!#6
\hskip-33.6pt\lower4.5pt\hbox{${\scriptstyle|}
\hskip-3.35pt\lower6pt\hbox{$#7$
\lower3pt\hbox{\ }
}$}}

\def\Esea#1#2#3#4#5#6#7#8{\displaystyle
#1\!\!-\!\!#2
\!\!-\!\!#3
\!\!-\!\!#4\!\!-\!\!#5\!\!-\!\!#6\!\!-\!\!#7
\hskip-33.6pt\lower4.5pt\hbox{${\scriptstyle|}
\hskip-3.35pt\lower6pt\hbox{$#8$
\lower3pt\hbox{\ }
}$}}

\def\Eei#1#2#3#4#5#6#7#8{\displaystyle
#1\!\!-\!\!#2
\!\!-\!\!#3
\!\!-\!\!#4\!\!-\!\!#5\!\!-\!\!#6\!\!-\!\!#7
\hskip-43.2pt\lower4.5pt\hbox{${\scriptstyle|}
\hskip-3.35pt\lower6pt\hbox{$#8$
\lower3pt\hbox{\ }
}$}}

\def\Eeia#1#2#3#4#5#6#7#8#9{{\displaystyle
#1\!\!-\!\!#2
\!\!-\!\!#3
\!\!-\!\!#4\!\!-\!\!#5\!\!-\!\!#6\!\!-\!\!#7\!\!-\!\!#8
\hskip-52.2pt\lower4.5pt\hbox{${\scriptstyle|}
\hskip-3.35pt\lower6pt\hbox{$#9$
\lower3pt\hbox{\ }
}$}}}

\def\os{\overline{S}}

\def\tS{\tilde{S}}
\def\ts{\tS}
\def\ts7{\tilde{S}_7}

\def\tq{\tilde{q}}

\def\cN{\Cal N}

\def\bP{\Bbb P}

\def\tD{\tilde{D}}

\def\cC{\Cal C}

\def\bV{\Bbb V}

\def\dlog{operatorname{dlog}}
\def\Aut{\operatorname{Aut}}

\def\Aff{\operatorname{Aff}}

\def\Supp{\operatorname{Supp}}

\def\Hom{\operatorname{Hom}}

\def\tY{\tilde{Y}}
\def\tsigma{\tilde{\sigma}}

\def\tQ{\tilde{Q}}
\def\cO{\Cal O}

\def\cD{\Cal D}

\def\Pic{\operatorname{Pic}}
\def\tB{\tilde{B}}

\def\rank{\operatorname{rank}}

\def\bZ{\Bbb Z}
\def\bC{\Bbb C}
\def\bQ{\Bbb Q}
\def\bG{\Bbb G}
\def\bR{\Bbb R}
\def\bA{\Bbb A}

\def\cA{\Cal A}
\def\cD{\Cal D}

\def\cF{\Cal F}
\def\tS{\tilde S}

\def\cI{\Cal I}
\def\cA{\Cal A}

\def\oS{\overline{S}}

\def\cP{\Cal P}

\def\tS{\tilde{S}}

\def\tB{\tilde{B}}

\def\oN{\overline{N}}

\def\cY{\Cal Y}
\def\cX{\Cal X}

\def\Spec{\operatorname{Spec}}
\def\Proj{\operatorname{Proj}}

\def\dlog{\operatorname{dlog}}

\def\Ker{\operatorname{Ker}}

\def\Im{\operatorname{Im}}

\def\ord{\operatorname{ord}}

\def\osigma{\overline{\sigma}}

\def\cK{\Cal K}

\def\oR{\overline{R}}

\def\Sing{\operatorname{Sing}}
\def\Cal{\mathcal}
\def\Pic{\operatorname{Pic}}

\def\oN{\overline{N}}
\def\os7p{\oS_7'}
\def\os7{\oS_7}
\def\on6{\oN_6}
\def\n6{\oN_6}
\def\bdy{\operatorname{bdy}}
\def\Mono{\operatorname{Mono}}
\def\Lift{\operatorname{Lift}}

\newcommand{\hR}{\hat{R}}

\DeclareMathOperator{\oNE}{\overline{NE}}
\DeclareMathOperator{\ocK}{\overline{\cK}}

\def\mydate{\ifcase\month \or January\or February\or March\or
April\or May\or June\or July\or August\or September\or October\or
November\or December\fi \space\number\day,\space\number\year}

\begin{document}

\title[Mirror Symmetry for log Calabi-Yau Surfaces I ]{Mirror Symmetry
for log Calabi-Yau Surfaces I}

\author{Mark Gross}
\address{DPMMS, Centre for Mathematical Sciences,
Wilberforce Road, Cambridge CB3 0WB, United Kingdom}
\email{mgross@dpmms.cam.ac.uk}
\author{Paul Hacking}
\address{Department of Mathematics and Statistics, Lederle Graduate
Research Tower, University of Massachusetts, Amherst, MA 01003-9305}
\email{hacking@math.umass.edu}

\author{Sean Keel}
\address{Department of Mathematics, 1 University Station C1200, Austin,
TX 78712-0257}
\email{keel@math.utexas.edu}

\begin{abstract} We give a canonical synthetic construction
of the mirror family to pairs $(Y,D)$ where $Y$
is a smooth projective surface and $D$ is an anti-canonical cycle of rational
curves. This mirror family is constructed as
the spectrum of an explicit algebra structure on a vector
space with canonical basis and multiplication rule defined in terms of
counts of rational curves on $Y$ meeting $D$ in a single point.
The elements of the canonical basis are called \emph{theta functions}.
Their construction depends crucially on the Gromov-Witten theory of
the pair $(Y,D)$.
\end{abstract}

\maketitle
\tableofcontents
\bigskip


\section*{Introduction}

\subsection{The main theorems}
Throughout  the paper $(Y,D)$ with $D = D_1 + \cdots +D_n$ will denote
a smooth rational projective surface over an algebraically closed
field $\kk$ of characteristic zero, with
$D \in |-K_Y|$ a singular nodal curve. The divisor $D$ is necessarily
either an
irreducible rational nodal curve, or a cycle of $n\geq 2$ smooth
rational curves.
We call $(Y,D)$ a \emph{Looijenga pair} for, as far as we know,
their rich geometry was first investigated in \cite{L81}.
We cyclically order the components of $D$
and take indices modulo $n$. By assumption there is a holomorphic symplectic
$2$-form $\Omega$, unique up to scaling, on $Y \setminus D$, with simple
poles along $D$, and thus $U:=Y \setminus D$ is a log Calabi-Yau
surface.

Our main result is a canonical synthetic construction of the mirror family
to such a pair. The construction gives an
embedded smoothing of the $n$-vertex $\bV_n \subset \bA^n$,
defined as,
for $n\ge 3$, the $n$-cycle of coordinate planes in $\AA^n$:
\[
\VV_n:=\bA^2_{x_1,x_2}\cup\bA^2_{x_2,x_3}\cup\cdots\cup\bA^2_{x_n,x_1}\subset
\bA^n_{x_1,\dots,x_n}.
\]
(See \eqref{VV2eq} and \eqref{VV1eq} for the definition of $\VV_1$ and
$\VV_2$.)
This family is in general
parameterized roughly by the formal completion of
the affine toric variety $\Spec \kk[\NE(Y)]$ along the union of toric
boundary strata corresponding to contractions $f \colon Y \rightarrow \oY$.
Here $\NE(Y)$ denotes the monoid $\NE(Y)_{\bR} \cap A_1(Y,\bZ)$ where $\NE(Y)_{\bR} \subset A_1(Y,\bR)$ is the cone generated by effective curve classes.
This is just an approximate statement of our result, as $\NE(Y)$ is not in
general finitely generated.

More precisely, fix $(Y,D)$, $D=D_1+\cdots+D_n$ as above.
Let $B_0(\bZ)$ be the set of pairs $(E,n)$ where $E$
is a prime divisor on some blowup of $Y$ along
which $\Omega$ has a pole and $n$ is a positive integer.
Set $B(\bZ) := B_0(\bZ) \cup \{0\}$.
Later we will describe this set as the set of
integer points in a natural integral affine manifold, the dual intersection
complex, or tropicalization, of the pair $(Y,D)$.
Let $v_i \in B(\bZ)$ be
the pair $(D_i,1)$. Choose $\sigma_P \subset A_1(Y,\bR)$ a strictly convex rational polyhedral cone containing $\NE(Y)_{\bR}$, let $P:=\sigma_P \cap A_1(Y,\bZ)$ be the associated monoid, and set $R:= \kk[P]$ to be the associated $\kk$-algebra.

For each monomial ideal
$I \subset R$, consider the free $R_I := R/I$-module
\begin{equation} \label{defA}
A_I := \bigoplus_{q \in B(\bZ)} R_I \cdot \vartheta_q.
\end{equation}
Let $\fom \subset R$ denote the maximal monomial ideal.
Let $T^D:= \bG_m^n$ be the torus with character group $\chi(T^D)$ having
basis $e_{D_i}$ indexed by the components $D_i \subset D$.
There is a homomorphism
$T^D \to \Spec \kk[P^{\gp}]$ induced by $C \mapsto \sum (C \cdot D_i) e_{D_i}$,
so $T^D$ acts on $\Spec R_I$.

\begin{theorem} \label{maintheoremlocalcase}
Let $I \subset R$ be a monomial ideal with $\sqrt{I} = \fom$.
In \S\S\ref{modifiedmumfordsection} and \ref{canscatdiag}, we
construct a finitely generated $R_I$-algebra structure on $A_I$,
determined by relative Gromov-Witten invariants of
$(Y,D)$ counting rational curves meeting $D$ in a single point. In
\S\ref{relativetorussection}, we construct a $T^D$ action on $\Spec A_I$.
This induces a flat $T^D$-equivariant map
\[
f: X_I:= \Spec A_I \to \Spec R_I
\]
with closed fibre $\VV_n$. By taking the limit over all such $I$, this
yields a formal flat family
\[
\fof:\foX_{\fom}\to \foS_{\fom}:=\Spf \widehat{R},
\]
where $\widehat{R}$ is the completion of $R$ with respect to the ideal
$\fom$. The generic fibre of $f$ is smooth in the sense of Definition
\ref{genericsmoothness}, so $\fof$ is a formal smoothing of $\VV_n$.
\end{theorem}

We use the notation $\vartheta_q$ for generators of our
algebra, as the construction fits into a more general family
of constructions which includes, as a special case, theta functions
on abelian varieties. The history of such functions is as follows.
Tyurin conjectured the existence of
canonical theta functions (i.e., a basis of global sections)
for polarized K3 surfaces, see \cite{Ty99}.
In 2007, discussions of the first
author with Abouzaid and Siebert involving a tropicalization of the
Fukaya category gave a stronger hint as to the existence
of theta functions on arbitrary degenerations of Calabi-Yau manifolds
in the context of the Gross-Siebert program. In particular, these
discussions led to what is now understood to be a variant of the notion
of \emph{broken line}.

The latter notion was introduced in \cite{G09}.
These were initially used to construct canonical perturbations of the
Landau-Ginzburg potential for $\PP^2$.
Broken lines were then used for constructing mirror Landau-Ginzburg potentials
for varieties with effective anti-canonical divisor in the
setting of the Gross-Siebert program by Carl, Pumperla and Siebert in
\cite{CPS}. The authors show that the mirrors to such varieties as constructed
in \cite{GS07} carry a canonical Landau-Ginzburg potential obtained by
using broken lines to lift monomial functions on the central fibre of
a toric degeneration to the toric degeneration.

Simultaneously, we used these same lifts to allow an extension of the
construction developed by Gross and Siebert to prove the above
main theorem. The main innovations we have introduced here are that
we use theta functions to provide partial compactifications of certain
canonically constructed deformations, and that these canonically
constructed deformations, along with the theta functions, can be
constructed relying only on the Gromov-Witten theory of $(Y,D)$.
The key point is that it is easy to build deformations
of the \emph{punctured} $n$-vertex $\VV_n^o:=\VV_n\setminus\{0\}$, but
it is difficult to extend these to deformations of $\VV_n$. This is
effectively done by using theta functions to embed a suitably chosen
deformation of $\VV_n^o$ in affine space, where the closure may then be taken.
This extension would be impossible without the existence of theta functions.

This result can be viewed as
log analogs of Tyurin's conjecture. In work in progress we apply similar
ideas to obtain Tyurin's conjecture in the K3 case as well,
and construct canonical
bases for cluster algebras, to cite two other generalizations. These are large
topics and will be expanded on elsewhere.
See also \cite{GSTheta} for more motivation from mirror symmetry, and upcoming
papers \cite{GHKS} and \cite{K3}.

Continuing with $(Y,D)$, $P$ and $R$ as above, our second main theorem is:

\begin{theorem} \label{extensiontheorem}
There is a unique smallest radical monomial ideal $J \subset R$ with the
following properties:
\begin{enumerate}
\item For every monomial ideal $I$ with $J \subset \sqrt{I}$
there is a finitely generated $R_I$-algebra structure on $A_I$ compatible with the $R_{I+\fom^N}$-algebra structure on $A_{I+\fom^N}$ of Theorem \ref{maintheoremlocalcase} for all $N>0$.
\item
If the intersection matrix $(D_i \cdot D_j)$ is not negative semi-definite then $J=0$.
In general, the zero locus $V(J) \subset \Spec R$ contains the union of the closed toric strata
corresponding to faces $F$ of $\sigma_P$ such that there exists an $i$
such that $[D_i]\not\in F$.
\item
Let $\widehat{R}$ denote the $J$-adic completion of $R$ and
$\foS_J:=\Spf \widehat{R}$ the associated formal scheme.
The algebras $A_I$ determine a canonical $T^D$-equivariant
formal flat family of affine surfaces
\[
\fof: \foX_J \to \foS_J
\]
with fibre $\bV_n$ over $0$.
The $\vartheta_q$ determine a canonical embedding $\foX_J
\subset \bA^{\max(n,3)} \times \foS_J$.
\end{enumerate}
\end{theorem}

\begin{remark} When $\NE(Y) \subset P' \subset P \subset A_1(Y)$,
then $J' \subset J$ and the formal family $\cX$ for $P$ comes from
the family for $P'$ by base-change. In this sense the family
is independent of the choice of $P$.
\end{remark}

\begin{remark} Note that in the case that the intersection matrix
$(D_i\cdot D_j)$ is not negative semi-definite (which includes the case
that $D$ supports an ample divisor), Theorem \ref{extensiontheorem}
tells us that our construction gives a family over $\Spec R$, so in
particular the construction is algebraic.
\end{remark}

In this paper, we will not address the question as to in what sense our
construction can be proved to be a mirror family. We expect, however,
that our families constructed by the above theorems are mirror
to $U = Y \setminus D$ in the sense of homological mirror symmetry in the
case $\kk=\bC$. Further justification for our construction yielding
the mirror family comes from the heuristic description of the construction
in terms of symplectic geometry as discussed below.

The third main result of this paper is an application of our general
construction, following from a more detailed analysis of the case where the
matrix $(D_i\cdot D_j)$ is negative definite:

\begin{theorem}[Looijenga's conjecture]  \label{loocor2}
A $2$-dimensional cusp singularity is smoothable
if and only if the exceptional cycle of the dual cusp occurs
as an anti-canonical cycle on a smooth projective rational
surface.
\end{theorem}

This was conjectured by Looijenga in \cite{L81},
where he also proved the forward implication. Partial results were obtained in
\cite{FM83} and \cite{FP84}.

\subsection{The symplectic heuristic}

Much of what we do in this paper, following the philosophy of the
Gross--Siebert program, is to tropicalize the SYZ picture \cite{SYZ96}. Thus
it is helpful to review informally this picture in the context of mirrors
to Looijenga pairs $(Y,D)$.
The SYZ picture will be a heuristic philosophical guide, and hence we make
no effort to be rigorous.
Here we follow the exposition from \cite{A07} concerning SYZ on the complement
of an anti-canonical divisor, itself a generalization of ideas of Cho and
Oh for interpreting the Landau-Ginzburg mirror of a toric variety in terms
of counting Maslov index two holomorphic disks \cite{CO06}.
For the most part we follow Auroux's notation, except that we use $Y$ instead
of his $X$, and our $X$ is his $M$.

We fix a K\"ahler form $\omega$ on $Y$, and a nowhere vanishing
holomorphic $2$-form $\Omega$ on $U:=Y\setminus D$.
Now suppose we have a fibration $f \colon U \rightarrow B$ by
special Lagrangian 2-tori (i.e., a fibre $L$ of $f$
satisfies $\Im \Omega|_L=\omega|_L=0$). Then
the SYZ mirror $X$ of $(U,\omega)$
is the dual torus fibration
$\check{f} \colon X \rightarrow B$. This can be
thought of as a moduli space of pairs $(L,\nabla)$
consisting of a special Lagrangian fibre $L$ of $f$
equipped with a unitary connection $\nabla$ modulo gauge
equivalence, or equivalently a holonomy
map $\hol_{\nabla}: H_1(L,\bZ) \to U(1) \subset \bC^*$.
The complex structure on $X$ is subtle, specified by so-called instanton
corrections.

In this picture we can define local holomorphic functions on $X$
associated to a basis of $H_2(Y,L,\bZ)$
(in a neighbourhood of a fibre of $\check{f}$ corresponding to a
non-singular fibre $L$ of $f$) as follows.
For $A \in H_2(Y,L,\bZ)$ define
\begin{equation}
\label{symplecticarea}
z^{A} := \exp\left(-2\pi \int_{A} \omega\right) \hol_{\nabla}(\partial A):
X \to \bC^*.
\end{equation}
By choosing a splitting
of $H_2(Y,L,\bZ) \twoheadrightarrow H_1(L,\bZ)$ we can pick out local
coordinates on $X$ which define a complex structure.
See \cite{A07}, Lemma~2.7. Note that as the fibre $L$ varies,
the relative homology group $H_2(Y,L,\bZ)$ forms a local system over
$B_0\subset B$, where $B_0$ is the subset of points with non-singular fibres. This
local system has monodromy, and as a consequence, the functions
$z^{A}$ are only well-defined locally.

However, there are also well-defined global functions $\vartheta_1,
\ldots,\vartheta_n$ on $X$.
These are defined locally in neighbourhoods
of fibres of $\check{f}$ corresponding to fibres of $f$ not bounding
holomorphic disks contained in $U$, via a (rough) expression
\begin{equation}
\label{varthetadiskexp}
\vartheta_i=\sum_{\beta\in H_2(Y,L,\ZZ)} n_{\beta}z^{\beta},
\end{equation}
where $n_{\beta}$ is a count of so-called Maslov index two disks with
boundary on $L$ representing the class $\beta$ and intersecting
$D$ transversally in one point lying in $D_i$.
(We note that in our setting the Maslov index $\mu$ of a holomorphic disk $f \colon \Delta \rightarrow Y$ with boundary lying on a special Lagrangian torus $L \subset Y$ is given by $\mu = 2\deg f^*D$. See \cite{A07}, Lemma~3.1.)
In the case that $D$ is ample, there are, for generic $L$, only finitely
many such disks; it is not known how to treat the general case in this
symplectic setting.

For $\vartheta_i$ to make sense the moduli space of Maslov
index $2$ disks with boundary on $L$
must deform smoothly with the Lagrangian $L$. This
fails for Lagrangians that bound holomorphic disks contained in $U$
(Maslov index zero disks).
This is a real codimension one condition on $L$, and thus defines canonical
{\it walls} in the affine manifold $B$. When we cross the wall
the $\vartheta_i$ are discontinuous. But the discontinuity
is corrected by
a holomorphic change of variable in the local coordinates $z^{\beta}$,
according to \cite{A07}, Proposition 3.9:
\begin{equation} \label{instantoncor}
z^{\beta} \to z^{\beta} \cdot h(z^{\alpha}) ^{[\partial\beta] \cdot
[\partial\alpha]}
\end{equation}
where here $\alpha \in H_2(Y,L_0,\bZ)$ represents the class
of the Maslov index zero disk with boundary on $L_0$ a Lagrangian fibre
over a point on the wall,
and $h(q)$ is a generating function
counting such holomorphic disks. Thus we can define a new complex
manifold, with the same local coordinates, by composing the
obvious gluing induced by identifications of fibres of the local system
on $B_0$ with fibres $H_2(Y,L,\ZZ)$
with the automorphism \eqref{instantoncor}. These regluings are
the instanton corrections, and the modified manifold
$X$ should be the mirror. By construction it
comes with canonical global holomorphic functions $\vartheta_i$.
In particular, the sum $W= \sum_i \vartheta_i$ is a well-defined
global function, the \emph{Landau--Ginzburg potential}.

\subsection{Outline of the proof}

We now outline how we realise the symplectic SYZ heuristic in terms of
algebraic geometry. There are three principal issues to consider:
\begin{itemize}
\item What information about a putative SYZ fibration can be seen inside
algebraic geometry?
\item What is the analogue of a Maslov index two disk in algebraic geometry?
\item How do we obtain the mirror by gluing together varieties?
\end{itemize}

The philosophy for dealing with the first and third issues was developed
by Gross and Siebert in \cite{GS07}.
For the first item, while we cannot build an SYZ fibration $f:U\rightarrow
B$ in general, we can roughly describe $B$ as a combinatorial object.
Given the Looijenga pair $(Y,D)$, we build a space $B$ homeomorphic to
$\RR^2$
along with a decomposition $\Sigma$ of $B$ into
cones. We construct $(B,\Sigma)$ as the dual intersection complex
of $(Y,D)$. For each double point of $D$, we take a copy of the first
quadrant in $\RR^2$, with the axes labelled by the two irreducible
components of $D$ (assuming $D$ is not irreducible)
passing through the double point. We then identify edges of these cones
if they are labelled with the same irreducible component of $D$. We
thus get a topological space abstractly homeomorphic to $\RR^2$ subdivided
into cones. This is $(B,\Sigma)$. In \S\ref{affinesection}, we show how
we can put an additional structure on $B$, namely the structure of
an \emph{affine manifold with singularities}. Indeed, we can give
$B_0:=B\setminus \{0\}$ a system of coordinate charts whose transition
maps are integral affine linear transformations. The affine structure
does not extend across the origin unless $(Y,D)$ is in fact a toric
pair, in which case we recover the fan $\Sigma$ defining $Y$.

The manifold $B$ can be viewed as the base of the SYZ fibration ``seen
from a great distance.'' In general the base of an SYZ fibration has
the structure of an affine manifold with singularities. Singular fibres
of the fibration occur over the singular points. One would expect $f:U
\rightarrow B$ to have a number of singular fibres in general, hence $B$
will have a number of singular points. So the above construction
moves all these singular points to the origin.

Next, let us consider the third item. Fixing $(Y,D)$ with $D=D_1+\cdots+D_n$,
let $P\subset A_1(Y,\ZZ)$ be a finitely generated monoid containing
the classes of all effective curves on $Y$, obtained by
choosing a strictly convex
rational polyhedral cone $\sigma_P\subset A_1(Y,\RR)$ containing
the Mori cone. Let $\fom$ be the maximal monomial ideal in
the ring $\kk[P]$, $I$ a monomial ideal with radical $\fom$,
and let $R_I=\kk[P]/I$.

We will describe the basic pieces we will glue together to describe a
scheme over $S_I:=\Spec R_I$ whose special fibre is $\VV_n^o:=\VV_n\setminus
\{0\}$.
Assume that the components $D_i$ are
numbered in cyclic order, with indices taken modulo $n$. We can define
an open cover of $\VV_n^o$ by taking sets, for
$1\le i\le n$,
\[
U_i = V(X_{i-1} X_{i+1})   \subset
\bA^2_{X_{i-1},X_{i+1}} \times ({\bG_m})_{X_i}.
\]
Note as subsets of $\VV_n$, they are disjoint except for
\[
U_{i,i+1} := U_i \cap U_{i+1} = (\bG_m)^2_{X_i,X_{i+1}}.
\]
In $\bV_n$ they are glued in the obvious way, i.e., via the canonical inclusions
\[
U_{i,i+1} = \{X_{i+1} \neq 0 \} \subset U_i, \quad
U_{i,i+1} = \{X_{i} \neq 0 \} \subset U_{i+1}.
\]

A deformation of $\VV_n^o$ over $S_I$ is obtained by gluing thickenings of the
$U_i$
\begin{equation} \label{charts}
U_{i,I} := V(X_{i-1} X_{i+1} - z^{[D_i]} X_i^{-D_i^2} )
\subset S_I \times \bA^2_{X_{i-1},X_{i+1}} \times (\bG_m)_{X_i}
\end{equation}
where $z^{[D_i]} \in \kk[P]$ is the corresponding
monomial.
The overlaps are relative tori, $U_{i,i+1,I} = S_I \times \bG_m^2$,
and the gluings are the obvious ones. The details
are given in \S \ref{pgroupbundle}.
This gluing gives a flat family
$X^o_I\rightarrow S_I$, which can be viewed as being
analogous to the naive complex structure on
the mirror described as the moduli of \emph{smooth} special Lagrangian
fibres with $U(1)$ connection.

There is no reason in general to believe that $X^o_I\rightarrow S_I$ can
be extended to a flat deformation $X_I\rightarrow S_I$ of
$\VV_n$. The reason is that such an $X_I$ should be an affine scheme,
and hence have many functions, while $X_I^o$ as constructed tends to have
few functions. The only case where $X_I^o$ extends to give a deformation
of $\VV_n$ is when $(Y,D)$ is a toric pair. In this case, we recover an
infinitesimal version of Givental's mirror family, which then easily
extends to Givental's mirror construction. We review this case in
\S\ref{Mumfordsection}.

To rectify this problem, we need to translate the instanton corrections
of the symplectic heuristic. We do this using the notion of scattering
diagram, here a variant of similar notions introduced in \cite{KS06}
and \cite{GS07}.

For us, a scattering diagram $\foD$ will be a collection of
pairs $(\fod,f_{\fod})$ where $\fod$ is a ray emanating from the origin
of $B$ with rational slope, and $f_{\fod}$ is a kind of function attached
to the ray. Any scattering diagram will dictate how to modify
both the definition of the open sets $U_{i,I}$ and the gluings of
$U_{i,I}$ with $U_{i+1,I}$. The precise details of this modification
are given in \S\ref{scatdiagsection}. Briefly, the rays define
automorphisms of the open sets $U_{i,i+1,I}$ analogous to
\eqref{instantoncor}, and are used to modify
the gluing.

While any scattering diagram can be used to obtain a modified
flat deformation $X^o_{I,\foD}$, we need to choose $\foD$ correctly
to have a chance of extending this deformation to $\VV_n$. The
symplectic heuristic can be used to motivate the choice of the
\emph{canonical scattering diagram}. The functions $f_{\fod}$
chosen are generating functions for certain Gromov-Witten invariants,
intuitively counting finite maps $\bA^1 \to U$. Heuristically,
each holomorphic disk contributing can be approximated by a proper rational
curve meeting $D$ in a single point.

Thus the canonical scattering diagram encodes the chamber structure
seen in the symplectic heuristic. But there still remains the question
of extending $X^o_{I,\foD}$ to a flat deformation of $\VV_n$. To do so,
we need to construct enough functions on $X^o_{I,\foD}$. This is where
the concept of \emph{theta function} comes in. The symplectic heuristic
suggests that there should be a canonical choice of holomorphic functions
on $X^o_{I,\foD}$ arising from a count of Maslov index two holomorphic disks.
Rather than trying to find an algebro-geometric analogue of a Maslov index
two holomorphic disk, one instead defines the counts using tropical geometry.
In particular, we use the notion of \emph{broken line}, introduced
in \cite{G09} and developed further by \cite{CPS} simultaneously with
this work, to provide the count. A broken line is essentially a tropical
analogue of a Maslov index two disk. They are piecewise linear paths
which only bend when they cross rays of the scattering diagram $\foD$,
in ways prescribed by the functions attached to the rays.

For any point $p\in B$ with integral coordinates, we can use a count of
broken lines to define a function on $U_{i,I}$ for any $i$. This procedure
is described in \S\ref{brokenlinesection}. Since this procedure is
dependent on the scattering diagram $\foD$, we can then ask whether
these functions on the various $U_{i,I}$ glue. We say $\foD$ is
\emph{consistent} if they always glue. If these functions do glue, then
we call the resulting global function on $X^o_{I,\foD}$ a \emph{theta
function}, writing it as $\vartheta_p$.

The bulk of the argument in this paper occurs in \S\ref{canscatdiag}, where
we prove that the canonical scattering diagram described above is in fact
consistent. This argument is rather involved, so we leave it to
\S\ref{proofoverview} to give an overview of the full argument for consistency.
Crucial to the argument is a reduction to methods of \cite{CPS} using
the main results of \cite{GPS09}.

Once consistency is proved, this gives global functions $\vartheta_p$
on $X^o_{I,\foD}$
for each $p\in B$ with integral coordinates. Let $v_i$ denote the first
integral point along the ray of $\Sigma$ corresponding
to the divisor $D_i$, and write $\vartheta_i:=\vartheta_{v_i}$. Then
we can use the functions $\vartheta_1,\ldots,\vartheta_n$ to embed (in the
case that $n\ge 3$) $X^o_{I,\foD}$ in $\AA^n\times S_I$. Taking the closure
of the image gives the desired deformation $X_I\rightarrow S_I$ of
$\VV_n$.

This construction essentially proves the first main theorem,
Theorem  \ref{maintheoremlocalcase}. The statement about the
scheme-theoretic singular locus
of $f$ is dealt with in \S\ref{smoothnesssection}. There we
again make a connection with the techniques of \cite{GS07}. The crucial
point is to show the singularity $0\in\VV^n$ is formally smoothed, and
for this, we need to work in a family where we have a local model for
the behaviour near $0$, much as Gross and Siebert have in \cite{GS07}.

More work is required for Theorem \ref{extensiontheorem}.
We need to show that the construction above, which really only produces a family
over the completion of $\Spec\kk[P]$ at the zero-dimensional torus orbit
of this scheme, extends across completions along larger strata.
Since the coordinate rings of the families constructed above are generated
by theta functions, we proceed by studying the products of theta functions.
In general, one expects the product of two theta functions to be a formal
series of theta functions. However, in many cases one can control the
terms sufficiently in these products to obtain the desired extensions.
This relies on a tropical interpretation of the product of theta
functions, given in \S\ref{algebrasection}, as well as the existence
of a torus action on our families, given in \S\ref{relativetorussection}.
This torus action only exists because of the canonical nature of our
scattering diagrams. Complete details for the arguments are given in the
last section, \S\ref{positivecase}.

Turning to Theorem \ref{loocor2}, the main point is that Looijenga's
conjecture is really a form of mirror symmetry. We start with a pair
$(Y,D)$ such that the intersection matrix $(D_i\cdot D_j)$ is negative
definite. Thus $D$ can be contracted analytically to give a cusp
singularity $p\in Y'$. (By definition, a cusp singularity is a surface
singularity whose minimal resolution is a cycle of rational curves.)
For the sake of exposition, assume this contraction is algebraic, so
that there is a divisor $L$ on $Y$ which is the pull-back of an ample
divisor on $Y'$. We choose the monoid $P$ so $L^{\perp}\cap P$ is a face
$P_{\bdy}$ of $P$, with $P_{\bdy}^{\gp}$ generated by the classes
$[D_1],\ldots,[D_n]$. The main goal is to extend our construction to
a formal neighbourhood of $\Spec \kk[P_{\bdy}]\subset \Spec \kk[P]$. The
problem is that Theorem \ref{extensiontheorem} explicitly does not apply
in this case. The main difficulty is that the charts \eqref{charts} overlap
too much when all the $z^{[D_i]}$ are invertible (in fact the fibres over
such points in $\Spec \kk[P_{\bdy}]$ coincide under the natural gluing
maps). There
is no way to glue these charts compatibly. However, this can be done
after shrinking these charts to analytic open subsets and working over
an analytic open neighborhood of the zero-dimensional stratum of $\Spec
\kk[P_{\bdy}]$. Here we work of course with $\kk=\CC$ only.

In doing so, we find over a general point of $\Spec\kk[P_{\bdy}]$
the dual cusp singularity to $p\in Y'$. Thus we see that our mirror
symmetry construction naturally produces the dual cusp. We then would like
to extend the family constructed over thickenings of $\Spec\kk[P_{\bdy}]$.
We use the same techniques as those used to prove Theorem
\ref{maintheoremlocalcase}. However, the construction of theta functions
is considerably more delicate. In general, theta functions are described
as a sum of monomials associated to broken lines. In the situation of
Theorem \ref{maintheoremlocalcase}, these sums are always finite. However,
in the current situation, they are always infinite. Thus there are serious
convergence issues, and this makes the proof rather technical. A delicate
analysis of the combinatorics of broken lines is necessary to prove
convergence.

Once convergence is shown, we then argue that the formal family produced
actually gives a smoothing of the cusp singularity. This follows from
the fact proved in Theorem \ref{maintheoremlocalcase} that we already have
a smoothing of the $n$-vertex
in a formal neighbourhood of the zero-dimensional stratum, but again the
argument is slightly delicate. All details are given in \S\ref{cusp}.

\subsection{Further directions}

Here we will briefly indicate the results of further study of our mirror
construction, to be given in sequel papers, as well as connections with
other recent work.

There are three broad classes of behaviour for our construction,
depending on the properties of the intersection matrix $(D_i\cdot D_j)$:
the matrix can be negative definite, negative semi-definite but not
negative definite, or not negative semi-definite. The first case is
analyzed here in detail to prove Theorem \ref{loocor2}. We will discuss
the third case in the sequel paper.

We call the case that the intersection matrix is not negative semi-definite
the \emph{positive} case. It holds if and only if
$U$ is the minimal
resolution of an affine surface, see Lemma~\ref{pospair}.
In this case, the cone $\NE(Y)_{\bR}$ is rational polyhedral, so we may
take $P=\NE(Y)$. Furthermore, the ideal $J$ of
Theorem \ref{extensiontheorem} equals $0$. Thus our construction defines
an algebraic family over $\Spec\kk[\NE(Y)]$, with smooth generic fibre.
We will show in
Part II that the restriction of this family to the structure torus
\[
\cX \to T_Y := \Pic(Y) \otimes \bG_m = \Spec \kk[A_1(Y)] \subset \Spec
\kk[\NE(Y)]
\]
is close to a universal family of deformations of $U = Y \setminus D$.

More precisely,
we will show independently of the positivity of the intersection
matrix that our formal family has a
simple and canonical (fibrewise) compactification
to a formal family $(\cZ,\cD)$ of Looijenga pairs
(with $\cX = \cZ \setminus \cD$), equivariant for the action of $T^D \subset T_Y$,
the subtorus generated by the components of $D$. The theta functions are $T^D$ eigenfunctions,
see \S \ref{relativetorussection}.

In the positive case this
extends naturally over all of $\Spec \kk[\NE(Y)]$, and its restriction
$(\cZ,\cD) \to T_Y$ comes
with a trivialization of the boundary $\cD = D_{*} \times T_Y$ realizing it as the universal
family of Looijenga pairs $(Z,D_Z)$ deformation equivalent to $(Y,D)$
together with an isomorphism $D_Z \stackrel{\sim}{\rightarrow} D_{*}$ constructed in \cite{GHK12}.
In particular, choosing such an isomorphism $D \stackrel{\sim}{\rightarrow} D_{*}$ for our
original pair $(Y,D)$ canonically
identifies it with a fibre of the family $(\cZ,\cD)/T_Y$.
More importantly, the restrictions of
the theta functions $\vartheta_q$ to $U \subset \cX$ endow
the affine surface $U = Y \setminus D$ with canonical functions.
We give a modular interpretation of the quotient of $\cZ \setminus \cD \to T_Y$
by $T^D$ as the universal deformation of $U$
(this shows in particular the quotient depends only on $U$, e.g.,
is independent of the choice of
compactification $U \subset Y$), and give a unique geometric characterisation of the theta function basis
of $H^0(U,\cO_U)$.

The fact that $(Y,D)$ appears as a fibre
is perhaps a bit surprising as, after all, we set out to construct the
mirror and have obtained the original surface back.
Note however that dual Lagrangian torus fibrations in dimension $2$ are topologically equivalent by Poincar\'e duality,
so this is consistent with the SYZ formulation of mirror symmetry.


To illustrate, in Example \ref{thetaexample1}
we explicitly compute the theta functions
in the case $(Y,D)$ is the del Pezzo of degree $5$ together
with a cycle of $5$ $(-1)$-curves. In Example \ref{thetaexample2},
we give the expression in
the case of a triangle of lines on a cubic surface,
deferring in this case the proof until
Part II. In each of these cases there is a characterisation of the
$\vartheta_{v_i}$ in terms of classical geometry.

In a different direction, in \cite{GHKK}, along with M.\ Kontsevich,
we extend
many of the methods introduced in this paper to prove a number of
significant conjectures about cluster varieties. In particular, the
technology of theta functions leads to a proof of positivity of the
Laurent phenomenon, and a proof of the Fock-Goncharov dual basis
conjecture for a broad class of cluster varieties. The latter can be
viewed as a generalization of the construction of theta functions
on $Y\setminus D$ in the positive case, described above. In fact,
in the case of cluster varieties associated to a skew-symmetric matrix
of rank two, the Fock-Goncharov $\cX$ variety fibres over a torus
with fibres being interiors of Looijenga pairs. This is described in
detail in \cite{GHK13}. In this case, the general construction of theta
functions in \cite{GHKK} coincides with the ones constructed here.

\medskip

Let us end with some mild speculation in all dimensions suggested by the
above discussion. By a Looijenga pair we mean
a dlt pair $(Y,D)$ (e.g., a simple normal crossings pair)
with $K_Y + D$ trivial and $(Y,D)$ having
a zero-dimensional log canonical center.
(In the simple normal crossing case, this means there is an intersection
point of $\dim Y$ different components of $D$.)
By a {\it log Calabi-Yau with maximal boundary}  we mean a
variety $U$ which can be realized as the interior $Y \setminus D$ of a Looijenga pair.
See \S 1 of \cite{GHK13} for background on these notions.
We expect that many of the results in this paper will extend to Looijenga pairs of
all dimensions.
This generalization will require
the further development of the technology of logarithmic Gromow-Witten
invariants, \cite{GS11},\cite{AC11}. We should obtain in complete
generality a mirror family $\cX\rightarrow \Spec \widehat{\kk[P]}$ for
suitable monoids $P$. Furthermore, one would expect in the case that $U = Y \setminus D$
is affine that
this family extends to $\cX\rightarrow
\Spec \kk[P]$. Using the two-dimensional case as a guide, the
general fibre of $\cX\rightarrow\Spec\kk[P]$ should itself be the interior of
a Looijenga pair $(\bar X,E)$, with
$\bar X\setminus E$ affine by construction. Thus we can then repeat the
process to obtain a family $\cX' \rightarrow \Spec\kk[P']$, and it would be expected, as taking mirror twice
should return to where we started, that $\cX'\rightarrow \Spec\kk[P']$
contains a fibre isomorphic to $U$. The family $\cX'$ carries
our canonically defined theta functions, indexed by tropical points of the
mirror. This leads us to propose:

\begin{conjecture} Let $U$ be an affine log Calabi-Yau variety with maximal boundary.
Then $H^0(U, \shO_U)$ has a canonical basis of theta
functions indexed by tropical points of the mirror. The structure
constants for multiplication of theta functions can be described combinatorially
in terms of broken lines.
\end{conjecture}

Versions of this conjecture have been proven for cluster varieties in many
cases in \cite{GHKK}.

\subsection{Acknowledgements.}
An initial (and ongoing) motivation for the project was to
find a geometric compactification of moduli of polarized K3
surfaces. We received a good deal of initial inspiration in this direction
from conversations with V. Alexeev. The project also owes a great deal to
the first author's collaboration with B.~Siebert.
We learned a great many things
from A. Neitzke, especially about the connections of
our work with cluster algebras and moduli of local systems. Our thinking about
Looijenga pairs was heavily influenced by conversations with R.~Friedman and
E.~Looijenga.
Many other people have helped
us with the project, and discussions
with D.~Allcock, D.~Ben-Zvi, V.~Fock, D.~Freed,
A.~Goncharov, R.~Heitmann, D.~Huybrechts, M.~
Kontsevich, A.~Oblomkov, T.~Perutz, M.~Reid, A.~Ritter, and
Y.~Soibelman were particularly helpful.
We would also like to thank IH\'ES for hospitality during the summer of 2009 when part of this research was done.
The first author was partially supported by NSF grants DMS-0805328
and DMS-0854987.
The second author was partially supported by NSF grant DMS-0968824 and DMS-1201439.
The third author was partially supported by NSF grant DMS-0854747.

\section{Basics}
\subsection{Looijenga pairs}

\begin{definition} A \emph{Looijenga pair} $(Y,D)$ is a smooth rational
projective surface $Y$ together with a reduced nodal curve $D\in |-K_Y|$ with
at least one singular point.
\end{definition}

Note that for a Looijenga pair, $p_a(D)=1$ by adjunction. Since
$H^1(Y,\shO_Y)=0$ by rationality of $Y$, $D$ is connected. Applying
adjunction to each irreducible component of $D$, one sees easily that
$D$
is either an irreducible genus one curve with a single node, or a cycle
of smooth rational curves. We will always write $D=D_1+\cdots+D_n$,
with a cyclic ordering of the irreducible components, and take the
indices modulo $n$.

We will need a few basic facts about Looijenga pairs, which we collect
here.

\begin{definition}
Let $(Y,D)$ be a Looijenga pair.
\begin{enumerate}
\item
A \emph{toric blow-up} of $(Y,D)$ is a birational morphism $\pi:\tilde
Y\rightarrow Y$ such that if $\tilde D$ is the reduced scheme
structure on $\pi^{-1}(D)$, then $(\tilde Y,\tilde D)$ is a Looijenga
pair. In particular, $\tilde Y$ is smooth.
\item A \emph{toric model} of
$(Y,D)$ is a birational morphism $(Y,D) \to (\oY,\oD)$
to a smooth toric surface $\oY$ with its toric boundary $\oD$ such
that $D \to \oD$ is an isomorphism.
\end{enumerate}
\end{definition}

Note that if $\pi:\tilde Y\rightarrow Y$ is the blow-up of a node of
$D$, then $\pi$ is a toric blow-up.

\begin{proposition}
\label{toricmodelexists}
Given $(Y,D)$ there exists a toric blowup
$(\tY,\tD)$ which has a toric model $(\tY,\tD) \to (\oY,\oD)$.
\end{proposition}

\begin{proof} First observe:
\begin{enumerate}
\item
Let $p:Y \to Y'$ be the blowdown of a $(-1)$-curve not contained in $D$, and
$D':= p_*(D) \subset Y'$. If the proposition holds for
$(Y',D')$ then it holds for $(Y,D)$.
\item
Let $Y'' \to Y$ be
the blowup at a node of $D$, and $D'' \subset Y''$ the reduced
inverse image of $D$. The proposition holds for $(Y'',D'')$ if and only if
it holds for $(Y,D)$.
\end{enumerate}

By using (1) and (2) repeatedly
we may assume $Y$ is minimal, and thus is either a ruled surface
or is $\bP^2$. In the latter case, by
blowing up a node of $D$ we reduce to the ruled case.
So we have $q: Y \to \bP^1$ a ruling. We next consider
the number of components of $D$ which are fibres
of $q$. There cannot be more than two such components, for
otherwise $D$ cannot be a cycle. If there are precisely
two such components, then $D$ necessarily has precisely four
components, and it is then easy to check that $D$ is the toric
boundary of $Y$, for a suitable choice of toric structure on $Y$.
In this case the proposition obviously holds.
Otherwise let $D' \subset D$ be the union of
components not contained in fibres.
If $D'$ has a node, then we can blowup the
node, blowdown the strict transform of the fibre through
the node, increasing the number of components of $D$
contained in fibres.

After carrying out this procedure for each node of $D'$, we are then in one
of two cases.

\emph{Case I}. $D$ has two components contained in fibres, and then we are
done.

\emph{Case II}. $D$ consists of a fibre $f$ and a non-singular
irreducible two-section $D'$ of $q$. Note that since $D'+f\sim -K_Y$
and $Y$ is isomorphic to the Hirzebruch surface $\FF_e$ for some $e$,
we can write $\Pic Y=\ZZ C_0\oplus\ZZ f$, with $C_0^2=-e$ and
$-K_Y=2C_0+(e+2)f$. Thus $D'\sim 2C_0+(e+1)f$ and $C_0\cdot D'
=-e+1$. Since $C_0$ is not contained in $D'$, $e=0$ or $1$.

If $e=0$, then there is a second ruling $q':Y\rightarrow\PP^1$,
with $D'$ and $f$ sections of this ruling. In this case, we follow
the same procedure as above of blowing up nodes
for this new ruling, arriving in Case I.

If $e=1$, then $C_0$ is disjoint from $D'$. Blowing down $C_0$,
we obtain $\PP^2$, and can then blowup one of the nodes of the image of
$D'\cup f$. Using this new ruled surface, we can again
blowup a node and find ourselves back in Case I.
\end{proof}

\subsection{Tropical Looijenga pairs}
\label{affinesection}

We explain how to tropicalize a Looijenga pair, first
recalling the following basic definition.
Fix a lattice $M\cong\ZZ^n$. In what follows, we will always use
the notation $M_{\RR}=M\otimes_{\ZZ}\RR$, $N=\Hom_{\ZZ}(M,\ZZ)$
and $N_{\RR}=N\otimes_{\ZZ}\RR$. We denote by $\Aff(M)$ the
group of affine linear transformations of the lattice $M$.
Recall the following definitions from \cite{GS06}.

\begin{definition} An \emph{integral affine manifold} $B$ is
a (real) manifold $B$ with an atlas of charts $\{\psi_i:U_i
\rightarrow M_{\RR}\}$ such that $\psi_i\circ\psi_j^{-1}
\in\Aff(M)$ for all $i,j$.

An \emph{integral affine manifold with singularities $B$} is
a (real) manifold $B$ with an open subset $B_0\subset B$
which carries the structure of an integral affine manifold,
and such that $\Delta:=B\setminus B_0$, the \emph{singular locus}
of $B$, is a locally finite union of locally closed submanifolds
of codimension at least two.

If $B$ is an integral affine manifold with singularities, there is
a local system $\Lambda_B$ on $B_0$ consisting of flat integral
vector fields: if $y_1,\ldots,y_n$ are local integral
affine coordinates, then $\Lambda_B$ is locally given by
linear combinations of the
vector fields $\partial/\partial y_1,\ldots,\partial/\partial y_n$.
If $B$ is clear from context, we drop the subscript $B$.

Similarly, $\check\Lambda_B$ is the dual local system, locally
generated by $dy_1,\ldots,dy_n$.
\end{definition}

We will be primarily interested in $\dim B=2$ in this paper, in which
case $\Delta$ will consist, in all our examples, of a finite
number of points. All integral affine manifolds we encounter will in fact
be \emph{linear}, in the sense that the coordinate transformations are
in fact linear rather than just affine linear.

We associate to a Looijenga pair $(Y,D)$ a pair $(B,\Sigma)$, where
$B$ is homeomorphic to $\RR^2$ and has the structure of
integral affine manifold with one singularity at the origin,
and $\Sigma$ is a decomposition of
$B$ into cones. We call $(B,\Sigma)$ the \emph{tropicalization} of
$(Y,D)$, and $\Sigma$ the \emph{fan} of $(Y,D)$.
The idea is that we pretend that
$(Y,D)$ is toric and we try to build the associated fan. More precisely,
the construction is as follows.

For each node
$p_{i,i+1} := D_i \cap D_{i+1}$ of $D$ we take a rank
two lattice $M_{i,i+1}$ with basis $v_i,v_{i+1}$, and
the cone $\sigma_{i,i+1} \subset M_{i,i+1}
\otimes_{\ZZ} \RR$ generated by $v_i$ and $v_{i+1}$.
We then glue $\sigma_{i,i+1}$ to $\sigma_{i-1,i}$ along the rays
$\rho_i := \RR_{\geq 0} v_i$ to obtain a piecewise-linear
manifold $B$ homeomorphic
to $\RR^2$ and a decomposition
\[
\Sigma=\{\sigma_{i,i+1}\,|\,1\le i\le n\}
\cup \{\rho_i\,|\,1\le i\le n\} \cup \{0\}.
\]
We define an integral affine
structure on $B \setminus \{0\}$ by defining charts
$\psi_i:U_i\rightarrow M_{\RR}$ (where $M=\ZZ^2$). Here
\begin{equation}
\label{Uidef}
U_i=\Int(\sigma_{i-1,i}\cup \sigma_{i,i+1})
\end{equation}
and $\psi_i$ is defined on the closure of $U_i$ by
\[
\psi_i(v_{i-1})=(1,0),\quad \psi_i(v_i)=(0,1),\quad
\psi_i(v_{i+1})=(-1,-D_i^2),
\]
with $\psi_i$ linear on $\sigma_{i-1,i}$ and $\sigma_{i,i+1}$.
The reason for choosing these particular vectors is that they
form the one-dimensional rays of a fan defining a toric variety
such that the divisor $D_i$ corresponding to the ray generated by $v_i$
has self-intersection $D_i^2$.

We note this construction makes sense even when
$n=1$, i.e., the anti-canonical divisor $D$ is an irreducible
nodal curve. In this case there is one cone $\sigma_{1,1}$, and
opposite sides of the cone are identified.
(Moreover, the integral affine charts are defined using the integer $D_1^2-2$ instead of $D_1^2$.
This is the degree of the normal bundle of the map from the normalization of $D_1$ to $Y$.)
However, this case
will often complicate arguments in this paper, so we will usually
replace $Y$ with a surface obtained by blowing up
the node of $D$, and replace $D$ with the reduced inverse image
of $D$ under the blowup. This does not change the underlying
integral affine manifold with singularities, but refines the
decomposition $\Sigma$, exactly as in the toric case:

\begin{definition}
Given $(B,\Sigma)$, a \emph{refinement} is a pair
$(B,\tilde\Sigma)$,
where $\tilde\Sigma$ is a
decomposition of $B$ into rational polyhedral
cones refining $\Sigma$, each cone of $\tilde\Sigma$ integral affine
isomorphic to the first quadrant in $\RR^2$.
\end{definition}

\begin{lemma}
\label{toricblowuplemma}
There is a one-to-one correspondence between
toric blow-ups of $(Y,D)$ and refinements of $(B,\Sigma)$. Furthermore,
if $(\tilde Y,\tilde D)$ is a non-singular toric blow-up of $(Y,D)$,
and $(\tilde B,\tilde\Sigma)$ is the affine manifold with singularities
constructed from $(\tilde Y,\tilde D)$, then $\tilde B$ and $B$
are isomorphic as integral affine manifolds with singularities in
such a way that $\tilde\Sigma$ is the corresponding refinement of $\Sigma$.
\end{lemma}

\begin{proof} Let $\pi:\tilde Y\rightarrow Y$ be a toric blow-up.
It follows from the condition that
$\pi^{-1}(D)_{\red}$ is an anti-canonical divisor that
$\pi:\tilde Y\setminus \pi^{-1}(\Sing(D))\rightarrow Y\setminus
\Sing(D)$ is an isomorphism. Indeed, if this restriction of $\pi$ has
an exceptional divisor, it must have discrepancy $a(E,Y,D)=-1$.
But by \cite{KM98},
Cor.\ 2.31, (3), the smallest discrepancy occuring is $0$.

Thus necessarily $\pi$ is a blow-up along a subscheme
supported on $\Sing(D)$. Let $x\in\Sing(D)$ be a double point of $D$,
corresponding to a cone $\sigma\in\Sigma$. Note $\sigma$ can be
viewed as a rational polyhedral cone defining a non-singular
toric variety $X_{\sigma}\cong\AA^2$. Then
\'etale locally near $x$,
the pair $(Y,D)$ is isomorphic to the pair
$(X_{\sigma},\partial X_{\sigma})$. One can then check that
in this local model, the only possible blow-ups satisfying
the definition of toric blow-ups come from subdivisions of the
cone $\sigma$, i.e., toric blow-ups of $X_{\sigma}$. Indeed,
the exceptional divisors of toric blowups are the only divisors
with discrepancy $-1$. This gives
the desired correspondence. The second statement is then easily
checked.
\end{proof}

\begin{example}
\label{basictoricexample}
It is easy to see that if $Y$ is a non-singular
toric surface and
$D=\partial Y$ is the toric boundary of $D$, then in fact
the affine structure on $B$ extends across the origin, identifying
$(B,\Sigma)$ with $(M_{\RR},\Sigma_Y)$, where $\Sigma_Y$ is the
fan for $Y$. Indeed, if $\rho_j\in\Sigma_Y$ corresponds to the
divisor $D_j$ and $\rho_j=\RR_{\ge 0} v_j$ with $v_j\in M$ primitive,
then it is a standard fact that
\[
v_{i-1}+(D_i)^2 v_i+v_{i+1}=0.
\]
Since $Y$ is non-singular, there is always a linear identification
$\varphi_i:M \stackrel{\sim}{\rightarrow} \ZZ^2$ taking $v_{i-1}$ to $(1,0)$, $v_i$ to $(0,1)$,
and thus $v_{i+1}$ must map to $(-1,-D_i^2)$. So on $U_i$, a chart
for the affine structure on $B$ is $\psi_i'=\varphi_i^{-1}\circ \psi_i:U_i
\rightarrow M_{\RR}$. The maps $\psi_i'$ glue to give an integral
affine isomorphism $B\rightarrow M_{\RR}$.
\end{example}

In fact, the converse is also true:

\begin{lemma} \label{toricsmooth} If the affine structure on $B_0
=B\setminus \{0\}$ extends across the origin, then $Y$ is toric
and $D=\partial Y$. \end{lemma}

\begin{proof}
We first note that by Lemma \ref{toricblowuplemma}, we can replace
$(Y,D)$ with a non-singular
toric blow-up without affecting the affine manifold
$B$. By Proposition \ref{toricmodelexists}, we can thus assume
the existence of a toric model $\pi:(Y,D)\rightarrow(\oY,\oD)$. If $\oD_i$
is the image of $D_i$ under this map, then $\oD_i^2\ge D_i^2$.

We first claim that $(Y,D)$ is isomorphic to $(\oY,\oD)$ if and only
if equality
holds for every $i$. Indeed, if equality holds for a given $i$,
then $\pi$ can't contract any curves which intersect $D_i$. On
the other hand, $\pi$ can't contract any curves contained in
$Y\setminus D$ since then $D$ would not be an anti-canonical cycle.

Now assume that $(Y,D)$ is \emph{not} toric, so that $\pi$ is
not an isomorphism. Let $(M_{\RR},\oSigma)$ be the fan for the toric
pair $(\oY,\oD)$, with rays $\bar\rho_1,\ldots,\bar\rho_n$ corresponding
to $\rho_1,\ldots,\rho_n$. In general, $B\setminus\rho_1$
has a coordinate chart $\psi:B\setminus\rho_1\rightarrow M_{\RR}$,
constructed by gluing together coordinate charts for $U_2,\ldots, U_n$.
This can be done so that $\sigma_{1,2}$
is mapped to the cone of $\oSigma$ generated by $\bar\rho_1$ and $\bar\rho_2$.
It is now enough to show the following:

\emph{Claim}.
For suitable
choice of $\rho_1$, in fact $\psi$ is injective and $\psi(B\setminus
\rho_1)$ is strictly contained in $M_{\RR}\setminus\bar\rho_1$.

To show this, first let us analyze the effects of one blow-up on these charts.
Let $(Y,D)\rightarrow (Y',D')$ be given by a blow-up of a single point
$p\in D_i'$ for some $i$, where $(Y',D')$ is obtained from $(\oY,\oD)$ by
a sequence of blow-ups with centers at smooth points of the boundary.
Let $(B',\Sigma')$ be
the tropicalization of $(Y',D')$.
Let us examine the difference between the charts
$\psi:B\setminus \rho_1\rightarrow M_{\RR}$ and $\psi':B'\setminus\rho_1'
\rightarrow M_{\RR}$ defined as above. If $i=1$, then $B\setminus\rho_1$ and
$B'\setminus\rho_1'$ are affine isomorphic and $\psi$, $\psi'$ agree. Otherwise,
let $\sigma_{1,i}=\bigcup_{j=2}^i \sigma_{j-1,j}\subset B$, with
$\sigma'_{1,i}\subset B'$ defined similarly. Then $\sigma_{1,i}$ and
$\sigma_{1,i}'$ are affine isomorphic and
$\psi|_{\sigma_{1,i}\setminus\rho_1}=\psi'|_{\sigma_{1,i}'\setminus\rho_1}$.
On the other hand,
$\psi|_{B\setminus \sigma_{1,i}}=T_i\circ \psi'|_{B'\setminus \sigma_{1,i}'}$
where $T_i:M_{\RR}\rightarrow M_{\RR}$ is the shear $T_i(m)=m+\langle m, n'_i
\rangle v_i'$, where $v_i'$ is a primitive generator of $\psi'(\rho_i')$
and $n'_i\in N$ is primitive, annihilates $v_i'$, and is positive on
$\psi'(\sigma_{i,i+1}')$.

Now note that $\oD_i^2>D_i^2$ for
at least one $i$, and by choosing $\rho_1$ appropriately, we can
assume that this is the case for some $i\not=1$. Furthermore, we can
also assume that if there is a $\bar\rho_j\in \oSigma$ with $\bar\rho_j
=-\bar\rho_1$, then there is an $i$ with $\oD_i^2>D_i^2$ with $i\not=1, j$.
Applying the above description of the change of the coordinate charts under
one blow-up repeatedly then shows the claim.

Now if the affine structure on $B_0$ extended across the origin,
then $\psi$ would extend to an isomorphism $\psi:B\rightarrow M_{\RR}$, which
contradicts the claim.
\end{proof}

\begin{example}
\label{M05example}

Let $Y$ be a del Pezzo surface of degree $5$. Thus $Y$ is isomorphic to the blowup of $\bP^2$ in $4$ points in general position.
The surface $Y$ contains exactly $10$ $(-1)$-curves. It is easy to find an anti-canonical cycle $D$ of length $5$ among these $10$ curves.

In this case,
consider $B\setminus\rho_1$. Each chart $\psi_i:U_i\rightarrow M_{\RR}$
can be composed with an integral linear function on $M_{\RR}$ in such
a way that the charts $\psi_2,\psi_3,\psi_4$ and $\psi_5$ glue
to give a chart $\psi:B\setminus\rho_1\rightarrow M_{\RR}$. This
can be done, for example, with
\[
\psi(v_1)=(1,0), \quad \psi(v_2)=(0,1),\quad
\psi(v_3)=(-1,1),\quad \psi(v_4)=(-1,0),\quad \psi(v_5)=(0,-1).
\]
We can then take a chart $\psi':U_5\cup U_1\rightarrow M_{\RR}$ which
agrees with $\psi$ on $\sigma_{5,1}$, and hence takes the values
\[
\psi'(v_5)=(0,-1),\quad \psi'(v_1)=(1,-1),\quad \psi'(v_2)=(1,0),
\]
see Figure \ref{M05affinemanifold}.

Thus $B$, as an affine manifold, can be constructed
by cutting $M_{\RR}$ along the positive real axis, and then identifying
the two copies of the cone $\sigma_{1,2}$ via an integral linear
transformation.
\end{example}

\begin{figure}
\input{M05affinemanifold.pstex_t}
\caption{}
\label{M05affinemanifold}
\end{figure}

\begin{example} \label{negdefexample}
Suppose given a Looijenga pair $(Y,D)$
with $D_i^2 \leq -2$ for all $i$
and $D$ is negative definite (which is equivalent to
$D_i^2 \leq -3$ for some $i$).
Then we have an analytic contraction $p:Y\rightarrow
\oY$ with $\oY$ having a single cusp singularity.
This case will lead to our proof of Looijenga's conjecture.
We can describe $(B,\Sigma)$ as follows.
Let $M\cong\ZZ^2$
and take $v_0,v_1$ to be a basis for $M$,
and define $v_i$ for $i\in\ZZ$ by the relation
\begin{equation}
\label{negdefrel}
v_{i-1}+(D_{i\bmod n}^2)v_i+v_{i+1}=0.
\end{equation}
We define an infinite fan $\tilde\Sigma$ in $M_{\RR}$ whose two-dimensional
cones are the cones generated by $v_i$ and $v_{i+1}$, $i\in\ZZ$.
It is easy to check that these cones do indeed form a fan and
that the support of the fan $|\tilde\Sigma|$ is a strictly convex cone.
If we define $T\in \SL(M)$ by
$T(v_0)=v_n$ and $T(v_1)=v_{n+1}$, then $T(v_i)=v_{i+n}$ for
each $i$. Necessarily $T$ takes $|\tilde\Sigma|$ to itself,
so the boundary rays of the closure of
$|\tilde\Sigma|$ are real eigenspaces for
$T$. Hence $T$ is hyperbolic, i.e., $\Tr T>2$.

We now obtain $(B,\Sigma)$ by dividing out $|\tilde\Sigma|$ by
the action of $T$.
\end{example}

\subsection{The Mumford degeneration and Givental's construction}
\label{Mumfordsection}

The toric case of Theorem \ref{maintheoremlocalcase}
yields Givental's construction for mirrors of toric varieties in the
surface case, and can
also be seen as a special case of a construction due to Mumford \cite{Mum}.
Mumford's construction in general produces degenerations of arbitrary
toric varieties; the construction as we review it here only gives
degenerations of the algebraic torus.
This should be regarded as a warmup for our general construction.

A toric monoid $P$ is a (commutative) monoid whose Grothendieck
group $P^{\gp}$ is a finitely generated free abelian group and
$P=P^{\gp}\cap \sigma_P$, where $\sigma_P\subseteq P^{\gp}\otimes_{\ZZ}\RR$
is a convex rational polyhedral cone. Let $M=\ZZ^n$ be a lattice, for some
arbitrary rank $n$.
Fix a fan $\Sigma$ in $M_{\RR}=M\otimes_{\ZZ}\RR$,
whose support, $|\Sigma|$, is convex. In what follows, we
view $B=|\Sigma|$ as an affine manifold with boundary.
We denote by $\Sigma_{\max}$ the set of maximal cones in $\Sigma$.

We now generalize the usual notion of a convex piecewise linear function
on a fan. If one is interested in $\RR$-valued convex functions, then
one can take $P=\NN$, $\sigma_P=\RR_{\ge 0}$. Then the following definition
yields the notion of a piecewise linear $\RR$-valued function with
integral slopes, and convexity here means upper convexity, i.e., the function
is the supremum of a collection of linear functions.

\begin{definition}
\label{convexPLdef}
A \emph{$\Sigma$-piecewise linear
function}
$\varphi:|\Sigma|\rightarrow P^{\gp}_\RR$
is a continuous function such that for each
$\sigma\in \Sigma_{\max}$, $\varphi|_{\sigma}$ is given by an element
$\varphi_{\sigma}\in \Hom_{\ZZ}(M,P^{\gp})=N\otimes_{\ZZ} P^{\gp}$.

For each codimension one cone $\rho\in\Sigma$
contained in two maximal cones $\sigma_+,\sigma_-\in\Sigma_{\max}$,
we can write
\[
\varphi_{\sigma_+}-\varphi_{\sigma_-}=n_{\rho}\otimes \kappa_{\rho,\varphi}
\]
where $n_{\rho}\in N$ is the unique primitive element annihilating
$\rho$ and positive on $\sigma_+$, and $\kappa_{\rho,\varphi}\in P^{\gp}$.
We call $\kappa_{\rho,\varphi}$ the {\it bending parameter}. Note
(as the notation suggests)
it depends only on the codimension one cone $\rho$ (not on the ordering of
$\sigma_+,\sigma_-$).

We say a $\Sigma$-piecewise linear
function $\varphi:|\Sigma|\rightarrow P^{\gp}$
is $P$-\emph{convex} (\emph{strictly $P$-convex})
if for every codimension one cone $\rho\in\Sigma$,
$\kappa_{\rho,\varphi}\in P$ ($\kappa_{\rho,\varphi}\in P\setminus P^{\times}$,
where $P^{\times}$ is the group of invertible elements of $P$).
\end{definition}

\begin{example}
\label{fundexample}
Take a complete fan $\Sigma$ in $M_{\RR}$. This
defines a toric variety $Y= Y_{\Sigma}$, which we assume is non-singular.
We let $P \subset P^{\gp}$ be given by the cone of effective
curves,
\[
\NE(Y) \subset A_1(Y,\bZ).
\]
Each codimension one cone
$\rho \in \Sigma$ corresponds to a one-dimensional toric stratum
$D_{\rho}\subset \partial Y$, hence a class $[D_{\rho}]\in\NE(Y)=P$.
If $\omega\in\Sigma(1)$, the set of rays of $\Sigma$, we also write
$D_{\omega}$ for the corresponding toric divisor.

\begin{lemma} \label{varphitoric}
Define $s: T_{\Sigma}:=\ZZ^{\Sigma(1)} \to M$ to send the basis element $t_{\omega}$, 
$\omega \in \Sigma(1)$
to the first lattice point $m_{\omega}$ on $\omega$. Then
\[
A_1(Y,\ZZ)\ni \beta\mapsto \sum_{\omega \in \Sigma(1)} 
(D_{\omega}\cdot\beta) t_{\omega}
\]
identifies $A_1(Y,\ZZ)$ with $\Ker(s)$, giving rise to an exact sequence
\begin{equation} \label{seseq}
0 \to A_1(Y,\ZZ) \to T_{\Sigma} \overset s \to  M \to 0.
\end{equation}
Then there is a unique $\Sigma$-piecewise linear section 
$\tvarphi: M \to T_{\Sigma}$ satisfying $\tvarphi(m_{\omega})=t_{\omega}$.
Let $\pi: T_{\Sigma} \to A_1(Y,\ZZ)$ be any splitting, and set
$\varphi:= \pi \circ \tvarphi$. Then $\varphi: M \to A_1(Y,\ZZ) = P^{\gp}$ is
$\Sigma$-piecewise linear and
strictly $P$-convex, with
\begin{equation} \label{bendeq}
\kappa_{\rho,\varphi}=[D_{\rho}]
\end{equation}
for each codimension one cone $\rho\in\Sigma$.
Up to a linear function, $\varphi$ is the unique $\Sigma$-piecewise linear map with
these bending parameters.
\end{lemma}

\begin{proof}
The exact sequence is standard. Since $\Sigma$ is a complete non-singular
fan, it is clear that there exists such a unique $\tvarphi$. 
To calculate the kink along a codimension one $\rho\in\Sigma$, suppose
$\rho$ is generated by basis vectors $e_1,\ldots,e_{n-1}$ and $\rho$
is contained in two maximal cones, generated by $e_1,\ldots,e_n$ and
$e_1,\ldots,e_{n-1},e_n':=-e_n+\sum_{i=1}^{n-1} a_ie_i$. Let $t_1,\ldots,t_n,t_n'$
be the generators of $T_{\Sigma}$ mapping to $e_1,\ldots,e_n,e_n'$ respectively.
Then the kink is $\kappa_{\rho,\tvarphi}=t_n+t_n'-\sum_{i=1}^{n-1} a_it_i$.
On the other hand, if $D_1,\ldots,D_n,D_n'$ are the divisors corresponding
to the rays generated by $e_1,\ldots,e_n,e_n'$ respectively, then
$D_n\cdot D_{\rho}=D_n'\cdot D_{\rho}=1$ and using the rational function
$z^{e_i^*}$, $D_i$ is linearly equivalent to
$-a_i D_n'$ plus a sum of toric divisors disjoint from $D_{\rho}$.
Thus $D_i\cdot D_{\rho}=-a_i$ and we see that $\kappa_{\rho,\tvarphi}$
is the image of $[D_{\rho}]$ under the inclusion 
$A_1(Y,\ZZ)\rightarrow T_{\Sigma}$. Thus
$\kappa_{\pi\circ\tvarphi,\rho} = \pi(\kappa_{\tvarphi,\rho})=[D_{\rho}]$
as required.
\end{proof}
\end{example}

Given a $\Sigma$-piecewise linear and $P$-convex function
$\varphi:|\Sigma|\to P^{\gp}$ we can define a monoid
$P_{\varphi}\subset M \times P^{\gp}$ by
\begin{equation}\label{Mumford cone}
P_{\varphi}:=\{(m,\varphi(m)+p)\,|\, m\in |\Sigma|, p\in P\}.
\end{equation}
This is the set of integral points lying above the graph of $\varphi$,
in the sense given by the partial order on $P^{\gp}$ defined by
$p_1\ge p_2$ if $p_1-p_2\in P$.
The convexity of $\varphi$ is equivalent to
$P_{\varphi}$ being closed under addition. Furthermore, we have
a natural inclusion $P\hookrightarrow P_{\varphi}$ given by $p\mapsto (0,p)$,
which gives us a morphism
\[
f:\Spec \kk[P_{\varphi}]\rightarrow \Spec \kk[P].
\]
This morphism is flat as $\kk[P_{\varphi}]$ is freely generated
as a $\kk[P]$-module by all elements of the form $z^{(m,\varphi(m))}$,
$m\in |\Sigma|$.
It is easy to see that a general fibre of $f$
is isomorphic to the algebraic torus $\Spec\kk[M]$: in fact,
if we consider the big torus orbit $U=\Spec \kk[P^{\gp}]
\subset \Spec\kk[P]$, $f^{-1}(U)=U\times \Spec\kk[M]$.

We now describe the fibres over other toric strata of $\Spec\kk[P]$.
Let $x\in\Spec\kk[P]$ be a point
in the torus orbit corresponding to a face $Q\subset P$.
Then by replacing $P$ with the localized monoid $P-Q$ obtained
by inverting all elements of $Q$, we may assume that
$x$ is contained in the smallest toric stratum of $\Spec\kk[P]$.
Consider the composed map
\[
\bar\varphi:|\Sigma|\mapright{\varphi} P^{\gp}\rightarrow P^{\gp}/
P^{\times}.
\]
Note $\bar\varphi$ is also piecewise linear. Let
$\bar\Sigma$ be the fan (of convex but not necessarily strictly convex
cones) whose maximal cones are the  maximal domains of linearity
of $\bar\varphi$.
Then $f^{-1}(x)$
can be written as
\[
f^{-1}(x)=\Spec \kk[\bar\Sigma].
\]
Here,
\[
\kk[\bar\Sigma]=\bigoplus_{m\in M\cap |\Sigma|}\kk z^m
\]
with multiplication given by
\begin{equation}
\label{fanmult}
z^m\cdot z^{m'}=\begin{cases}
z^{m+m'}&\hbox{if $m,m'$ lie in a common cone of $\bar\Sigma$},\\
0&\hbox{otherwise}.
\end{cases}
\end{equation}
In particular, the irreducible components of
$f^{-1}(x)$ are the toric varieties
$\Spec\kk[\sigma\cap M]$ for $\sigma\in\bar\Sigma_{\max}$.

In the particular case that $\rank M=2$ and $\Sigma$ defines a non-singular
complete surface with $n$ toric divisors, suppose $\varphi$
is strictly convex. If $x$ is a point of the smallest toric
stratum of $\Spec\kk[P]$, then $f^{-1}(x)$ is just
$\VV_n \subset \bA^n$, the reduced cyclic union of coordinate
$\bA^2$'s:
\[
\VV_n = \bA^2_{x_1,x_2} \cup \bA^2_{x_2,x_3}\cup
\cdots\cup \bA^2_{x_n,x_1}\subset
\bA^n_{x_1,\dots,x_n}.
\]
We call $\VV_n$
the \emph{vertex}, or more specifically, the \emph{$n$-vertex}.

We will need in the sequel the degenerate case of the $n$-vertex for
$n=2$. This is a union of two affine planes and can be described
as the double cover
\begin{equation}
\label{VV2eq}
\VV_2=\Spec\kk[x_1,x_2,y]/(y^2-x_1^2x_2^2)=\AA^2_{x_1,x_2}\cup
\AA^2_{x_2,x_1}.
\end{equation}
Of course, this does not appear as a central fibre of a Mumford degeneration.
Analogously, one can define
\begin{equation}
\label{VV1eq}
\VV_1=\Spec\kk[x,y,z]/(xyz-x^2-z^3),
\end{equation}
the affine cone over a nodal curve embedded in weighted projective
space $W\PP^3(3,1,2)$.

\begin{example}
In Example \ref{fundexample}, with the choice of $\varphi$
given by Lemma \ref{varphitoric}, the family
\[
\Spec\kk[P_{\varphi}]\rightarrow\Spec\kk[\NE(Y)]
\]
in fact gives the family of mirror manifolds to the toric variety
$Y$, as constructed by Givental \cite{Giv}.

In fact, the mirror of a toric variety also includes the data of
a \emph{Landau-Ginzburg potential}, which is a regular function.
If $Y$ is Fano, the potential is
\[
W=\sum_{\rho} z^{(m_{\rho},\varphi(m_{\rho}))}
\]
where we sum over all rays $\rho\in\Sigma$, and $m_{\rho}
\in M$ denotes the primitive generator of $\rho$.
If $Y$ is not Fano, the potential receives
corrections which can be viewed as coming from
degenerate holomorphic disks on $Y$ with irreducible
components mapping into $D$.
\end{example}

\section{Modified Mumford degenerations}\label{modifiedmumfordsection}

In this section, we fix $(Y,D)$ a Looijenga pair,
and let $(B,\Sigma)$ be the tropicalisation of $(Y,D)$
defined in \S \ref{affinesection}. The fan $\Sigma$ contains rays $\rho_1,
\ldots,\rho_n$ corresponding to divisors $D_1,\ldots,D_n$, ordered
cyclically. As usual, we write the two-dimensional cones of $\Sigma$ as
$\sigma_{i,i+1}$ being the cone with edges $\rho_i$ and $\rho_{i+1}$,
with indices taken modulo $n$.

We explain how to generalize Mumford's degeneration, to
give a canonical formal deformation of $\bV^o_n
=\VV_n\setminus\{0\}$ associated to $(Y,D)$ if $n\ge 3$.
Locally on $B_0$ the picture is toric and we have Mumford's
degenerations described in \S \ref{Mumfordsection}. As Mumford's construction
is functorial, the deformations built locally patch together canonically:
this is a minor variation on the ideas of \cite{GS07}. In particular,
\S\S\ref{pgroupbundle} and \ref{scatdiagsection} are variations of ideas
in \cite{GS07}. However, it differs crucially in several respects which
prevent us from just referring to \cite{GS07}. First, we work with piecewise
linear functions with values in a vector space $P^{\gp}_{\RR}$ rather
than just $\RR$. This allows us to construct higher dimensional
formal families, namely
over the completion of $\Spec\kk[P]$ at the zero-dimensional
stratum. Second, by avoiding a description of local models in codimension
at least two, we avoid some of the technical complexities of
\cite{GS07}.

Here are the details.

\subsection{The uncorrected degeneration}
\label{pgroupbundle}
We fix some notation. For any locally constant sheaf $\shF$
on $B_0$, and any simply connected subset $\tau \subset B_0$
we write $\shF_{\tau}$ for the stalk of this local system at any point of
$\tau$ (as any two such stalks are canonically identified by parallel
transport). In particular, we apply this for the sheaf $\Lambda$
of integral constant vector fields, as well
as for the sheaf $\Lambda_{\RR}:=\Lambda\otimes_{\ZZ}
\RR$.

For each cone $\tau \in \Sigma$ with $\dim\tau=1$ or $2$, we write
$\tau^{-1}\Sigma$ for the \emph{localized fan} of
convex (but not strictly convex) cones in $\Lambda_{\tau,\RR}$ described
as follows. If $\dim\tau=2$, then $\tau^{-1}\Sigma$ just consists of the
single cone $\Lambda_{\tau,\RR}$. If $\dim\tau=1$, then $\tau^{-1}\Sigma$
consists of three cones: the tangent line to $\tau$ and the two half-spaces
bounded by the tangent line to $\tau$.

Let $P\subseteq P^{\gp}$ be a toric monoid as in \S\ref{Mumfordsection}.

\begin{definition}
A \emph{($P^{\gp}_\RR$-valued) $\Sigma$-piecewise linear multivalued
function} on $B$ is a collection $\varphi=\{\varphi_i\}$ with
$\varphi_i$ a
$\Sigma$-piecewise linear function on $U_i$ with values in $P^{\gp}_\RR$.

Note this is equivalent to giving a $\rho_i^{-1}\Sigma$-piecewise
linear function $\varphi_i:\Lambda_{\RR,\rho_i} \to P^{\gp}_{\bR}$
for each ray $\rho_i \in \Sigma$.
Two such functions $\varphi$, $\varphi'$ are said to be \emph{equivalent}
if $\varphi_i-\varphi'_i$ is linear for each $i$. Note the equivalence
class of $\varphi$ is determined by the collection of bending parameters
$\kappa_{\rho,\varphi} \in P^{\gp}$. We say the function is \emph{convex}
(\emph{strictly convex}) if
$\kappa_{\rho,\varphi}\in P$ ($\kappa_{\rho,\varphi}\in P\setminus P^{\times}$)
for each $\rho$.

We drop the modifiers $\Sigma$ and $P$ when they are clear
from context.
\end{definition}

\begin{construction}
\label{torsorremark}
The collection $\{\varphi_i\}$ determines a local system $\shP$
on $B_0$ as follows.
First, we can construct an affine manifold $\PP_0$ which comes along
with the structure of
$P^{\gp}_{\bR}$-principal bundle
$\pi: \bP_0 \to B_0$ and a piecewise linear section
$\varphi: B_0 \to \bP_0$ as follows: we glue
$U_i \times P^{\gp}_{\bR}$ to $U_{i+1} \times P^{\gp}_{\bR}$
along $(U_i\cap U_{i+1})\times P^{\gp}_{\bR}$
by
\[
(x,p) \to (x,p+\varphi_{i+1}(x) - \varphi_i(x)).
\]
By construction we have local sections $x\mapsto (x,\varphi_i(x))$
which patch to give a piecewise linear section $\varphi$.
One checks immediately the isomorphism class (of the
$P^{\gp}_\RR$-principal
bundle together with the section) depends only on the equivalence
class of $\{\varphi_i\}$. The bundle $\PP_0\rightarrow B_0$ can be
viewed as a tropical analogue of a sum of line bundles, and $\{\varphi_i\}$
yield a section of this vector bundle. Convexity is analogous to
holomorphicity of the section.

We then define
\[
\shP := \pi_*\Lambda_{\bP_0}
\cong \varphi^{-1}\Lambda_{\PP_0}
\]
on $B_0$. We have an exact sequence
\begin{equation}
\label{standardexactseq}
0\rightarrow \underline{P}^{\gp}\rightarrow \shP\mapright{r}
\Lambda \rightarrow 0
\end{equation}
of local systems on $B_0$, where $r$ is the derivative of $\pi$.

Note over $U_i$, the description of $\PP_0$ as $U_i\times P^{\gp}_{\RR}$
gives a splitting $\shP|_{U_i}\cong \Lambda|_{U_i} \times P^{\gp}$.
\end{construction}

\begin{example}
\label{standardphi}
Our standard example, fundamental to this paper, will be as follows.
Suppose $P$ is a monoid which comes with a homomorphism
$\eta:\NE(Y)\rightarrow P$ of monoids.

Choose $\varphi$ by specifying $\varphi_i$ on $U_i$ by
the formula
\[
\kappa_{\rho_i,\varphi_i}=\eta([D_i]).
\]
Such a $\varphi$ is well-defined up to linear functions,
and always exists.
This is always convex, and is strictly convex provided
$\eta([D_i])$ is not invertible for any $i$.
\end{example}

Now suppose given a piecewise linear multivalued $P$-convex function
$\varphi$ on $B$.  We explain how
Mumford's construction determines a canonical formal deformation of $\bV_n^o$,
restricting to the case $n\ge 3$ for ease of exposition.

For each $\tau\in\Sigma$ with $\dim\tau>0$,
$\varphi$ determines a canonically defined
$\tau^{-1}\Sigma$-piecewise linear section
$\varphi_{\tau}: \Lambda_{\RR,\tau}
\to \shP_{\RR,\tau}$ of the projection $\shP_{\RR,\tau}\rightarrow
\Lambda_{\RR,\tau}$. If $U_i\cap \tau\not=\emptyset$, we use the representative
$\varphi_i$ on $U_i$ and extend it linearly on each cone in the fan
$\tau^{-1}\Sigma$ to obtain a $P$-convex piecewise linear function on
$\tau^{-1}\Sigma$, which we also write as $\varphi_i$. Then the section
$\varphi_{\tau}$ is defined
as in Construction \ref{torsorremark} by $x\mapsto (x,\varphi_i(x))$, using
the splitting $\shP_{\RR,\tau}=\Lambda_{\RR,\tau}\times P^{\gp}_{\RR}$. We note
a different choice of representative of $\varphi_i$ leads to a different
choice of splitting and the same section $\varphi_{\tau}$, so this section
is well-defined.

Now define the toric monoid $P_{\varphi_{\tau}} \subset \shP_{\tau}$ by
\begin{equation}
\label{Pvarphitaudef}
P_{\varphi_{\tau}} := \{q \in \shP_{\tau}\,|\,\hbox{$q = p +
\varphi_{\tau}(m)$ for some $p\in P$, $m\in \Lambda_{\tau}$}\}.
\end{equation}

By the definition of convexity of $\varphi$,
we have canonical inclusions
\begin{equation}
\label{coneinclusions}
P_{\varphi_{\rho}} \subset P_{\varphi_{\sigma}} \subset \shP_\rho
\end{equation}
whenever $\rho \subset \sigma \in \Sigma$.
If $\rho\in\Sigma$ is a ray with $\rho \subset \sigma_{\pm} \in \Sigma_{\max}$
we have the equality
\begin{equation}
\label{coneintersections}
P_{\varphi_{\sigma_+}} \cap P_{\varphi_{\sigma_-}} =
P_{\varphi_{\rho}}.
\end{equation}

\begin{definition}[Monomial ideals] \label{idealdef}
A \emph{(monoid) ideal} of a monoid $P$ is a subset
$I\subset P$ such that $p\in I, q\in P$ implies $p+q\in I$.
An ideal determines a \emph{monomial ideal} in the monoid ring $\kk[P]$,
generated by monomials $z^p$ for $p \in I$. We also denote this ideal by $I$, hopefully with no confusion. As a consequence, we shall sometimes write
certain ideal operations either additively or multiplicatively, i.e.,
for $J\subset P$,
\[
kJ=\{p_1+\cdots+p_k\,|\, p_i\in J, 1\le i \le k\},
\]
and the corresponding monomial ideal is $J^k$.

Let $\fom=P\setminus P^{\times}$. This
is the unique maximal ideal of $P$, defining a monomial
ideal $\fom\subset\kk[P]$. Note $\kk[P]/\fom\cong\kk[P^{\times}]$.

We say an ideal $I\subset P$ is \emph{$\fom$-primary}
if
\[
\fom=\sqrt{I}:=\{p\in P\,|\,\hbox{there exists a positive integer $k$
such that $kp\in I$}\},
\]
in which case the same holds for the associated
monomial ideal $I\subset \kk[P]$.
\end{definition}

Recall from \S\ref{Mumfordsection} that we are only considering
toric monoids $P$, i.e., monoids which are the intersection of rational
polyhedral cones $\sigma_P$ with lattices.
Such monoids are always finitely generated, so that $\kk[P]$ is
Noetherian. If $\sigma_P$ is strictly convex, then
$\fom$ is the maximal ideal corresponding to the unique torus fixed point of $\Spec \kk[P]$.

Fix an ideal $I\subset P$, and recalling that we write $R=\kk[P]$, set
\[
R_I:=\kk[P]/I.
\]
We define for
$\tau\in\Sigma$, $\dim\tau>0$, the ring
\[
R_{\tau,I}:=\kk[P_{\varphi_\tau}]\otimes_R R_I,
\]
noting that $P$ acts naturally on $P_{\varphi_{\tau}}$ by addition.
So $\Spec R_{\tau,I}$ is a base-change of the Mumford degeneration induced
by $\varphi_{\tau}$ on the localized fan $\tau^{-1}\Sigma$.

One observes

\begin{proposition}
\label{ringfirstversion}
Let $v_i$ denote the primitive generator of the tangent ray to $\rho_i$, for
each $i$.
Then viewing $z^{\kappa_{\rho,\varphi}}\in \kk[P]$ as determining an
element in $R_I$, we have
\begin{equation}
\label{ringfirstequation}
{R_I[X_{i-1},X_i^{\pm},X_{i+1}]\over (X_{i-1}X_{i+1}-
z^{\kappa_{\rho_i,\varphi}}X_i^{-D_{\rho_i}^2})}
\cong
R_{\rho_i,I}
\end{equation}
via the map
\[
X_j\mapsto z^{\varphi_{\rho_i}(v_j)}, \quad j\in \{i-1,i,i+1\}.
\]
Furthermore, there are natural maps
\[
\psi_{\rho_i,-}:R_{\rho_i,I}\rightarrow R_{\sigma_{i-1,i},I},\quad
\psi_{\rho_i,+}:R_{\rho_i,I}\rightarrow R_{\sigma_{i,i+1},I}
\]
induced by the inclusions $P_{\varphi_{\rho_i}}\subseteq
P_{\varphi_{\sigma_{\pm}}}$ which induce isomorphisms
\[
(R_{\rho_i,I})_{X_{i-1}}\cong R_{\sigma_{i-1,i},I},\quad
(R_{\rho_i,I})_{X_{i+1}}\cong R_{\sigma_{i,i+1},I}.
\]
\end{proposition}

\begin{proof}
We need to check that the ideal on the left-hand side is mapped to zero,
as the rest is obvious. Note by construction of $B$,
$v_{i-1}+D_{\rho_i}^2v_i+v_{i+1}=0$ as elements of $\Lambda_{\rho_i}$,
so one sees in fact that $\varphi_{\rho_i}(v_{i-1})+\varphi_{\rho_i}(v_{i+1})
=\kappa_{\rho_i,\varphi}-D_{\rho_i}^2\varphi_{\rho_i}(v_i)$. The result
then follows easily.
\end{proof}

\begin{remark}
\label{localcentralfibre}
Since $\Spec R_{\rho_i,I}\rightarrow\Spec R_I$ is
a base-change of the Mumford  degeneration, we can in fact say
what a fibre of this morphism is over a closed point $x$ in the
smallest toric stratum of $\Spec\kk[P]$, i.e., a point in
$\Spec R_{\fom}$. This depends on whether $\kappa_{\rho_i,
\varphi}\in P$ is invertible or not. If it is not invertible, then
the fibre is $\Spec\kk[\rho_i^{-1}\Sigma]\cong\Spec\kk[X_{i-1},X_{i+1},X_i^{\pm 1}]
/(X_{i-1}X_{i+1})$. If $\kappa_{\rho_i,\varphi}$ is invertible,
then the fibre is $\Spec\kk[\ZZ^2]$. In this latter case, if
$\rho_i\subset\sigma$, in fact
the map $R_{\rho_i,I}\rightarrow R_{\sigma,I}$ induced by the
inclusion $P_{\varphi_{\rho_i}}\subseteq P_{\varphi_{\sigma}}$ is an
isomorphism.

Somewhat more generally,
if $J\subset P$ is a radical ideal with $\kappa_{\rho_i,\varphi}\in J$, then
in fact $R_{\rho_i,J}\cong \Spec R_J[\rho_i^{-1}\Sigma]$.
\end{remark}

For $\tau\in\Sigma$, $\dim\tau\ge 1$, set
\[
U_{\tau,I}:=\Spec R_{\tau,I}.
\]
The maps $\psi_{\rho_i,\pm}$ induce open immersions $U_{\sigma_{i-1,i},I}
\hookrightarrow U_{\rho_i,I}$ and
$U_{\sigma_{i,i+1}}\hookrightarrow U_{\rho_i,I}$.
Denoting the image of each of these
immersions as $U_{\rho_i,\sigma_{i-1,i},I}$ and $U_{\rho_i,\sigma_{i,i+1},I}$
respectively, we note that
\begin{equation}
\label{chartintersection}
U_{\rho_i,\sigma_{i-1,i},I}\cap U_{\rho_i,\sigma_{i,i+1},I}\cong
\Spec (R_{\rho_i,I})_{X_{i-1}X_{i+1}}
\cong (\Gm)^2\times
\Spec(R_I)_{z^{\kappa_{\rho_i,\varphi}}}
\end{equation}
Note that if $\kappa_{\rho_i,\varphi}\in\sqrt{I}$ then the localization
$(\kk[P]/I)_{z^{\kappa_{\rho_i,\varphi}}}$ is zero, and the intersection
is empty.

We can now define our analogue of the Mumford degeneration.

\begin{construction}
\label{XoIdef}
Suppose that the number of irreducible components $n$ of $D$ satisfies
$n\ge 3$, that $\varphi$ is a PL multivalued function, and $I\subset P$
an ideal such that
$\kappa_{\rho,\varphi}\in\sqrt{I}$ for all rays $\rho\in\Sigma$.
Then there are canonical identifications of open subsets
\[
U_{\rho_i,I}\supset
U_{\rho_i,\sigma_{i,i+1},I}\cong U_{\sigma_{i,i+1},I}\cong
U_{\rho_{i+1},\sigma_{i,i+1},I}\subset U_{\rho_{i+1},I}
\]
which generate an equivalence relation on $\coprod_iU_{\rho_i,I}$,
and the quotient by this equivalence relation defines
a scheme $X^o_I$ over $\Spec R_I$.

One checks easily that the canonical
isomorphisms of
\[
\hbox{$U_{\rho_i,\sigma_{i,i+1},I}\subseteq U_{\rho_i,I}$
and $U_{\rho_{i+1},\sigma_{i,i+1},I}\subseteq U_{\rho_{i+1},I}$}
\]
satisfy the requirements
for gluing data for schemes along open subsets, see e.g., \cite{H77}, Ex.\
II 2.12.
\end{construction}


\begin{remark} \label{glueingremarks}
$X^o_I$ only depends on the equivalence class of
$\varphi$, since the monoids $P_{\varphi_{\tau}}$
are canonically defined, independently of the choice of representative
for $\varphi$.
\end{remark}

We first analyze this construction in the purely toric case:

\begin{lemma}
\label{opensubschemeGKZ} For $(Y,D)$ toric and $\varphi$ a single-valued
convex function on $B=M_{\RR}$,
$X^o_I$ is an open subscheme of
the Mumford  degeneration $\Spec\kk[P_{\varphi}]/I\kk[P_{\varphi}]$,
and
\[
H^0(X^o_I,\cO_{X^o_I}) = \kk[P_{\varphi}]/I \kk[P_{\varphi}].
\]
\end{lemma}

\begin{proof}
Note that for $\tau\in\Sigma$, $\tau\not=\{0\}$, the monoid $P_{\varphi_{\tau}}$
is isomorphic to the localization of $P_{\varphi}$ along the face
$\{(m,\varphi(m))\,|\, m\in\tau\cap M\}$. Thus $\Spec \kk[P_{\varphi_{\tau}}]$
is an open subset of $\Spec \kk[P_{\varphi}]$ and $\Spec
\kk[P_{\varphi_{\tau}}]\otimes_{\kk[P]} \kk[P]/I$ is an open subset
of $\Spec \kk[P_{\varphi}]/I\kk[P_{\varphi}]$. Furthermore, the gluing
procedure constructing $X^o_I$ is clearly compatible with these inclusions,
so $X^o_I$ is an open subscheme of $\Spec \kk[P_{\varphi}]/I\kk[P_{\varphi}]$.
Next, looking at the fibre over a closed point, one sees
easily that the underlying topological space of these fibres is
obtained just by removing the zero-dimensional torus orbit from the
corresponding fibre of the Mumford degeneration.
The closed fibres of the Mumford degeneration are $S_2$ by \cite{A02}, 2.3.19.
Thus by Lemma \ref{relS2}, the result follows.
\end{proof}

\begin{lemma}\label{relS2}
Let $\pi:\cX\rightarrow S$ be a flat family of surfaces
such that the fibre $\cX_s$
satisfies Serre's condition $S_2$ for each $s\in S$.
Let $i \colon \cX^o \subset \cX$ be the inclusion of an
open subset such that the complement has finite fibres.
Then $i_*\cO_{\cX^o} =\cO_{\cX}$.
Similarly, if $\cF$ is a coherent sheaf on $S$ then
$i_*(\cO_{\cX^0} \otimes \pi^*\cF)=\cO_{\cX} \otimes\pi^*\cF$.
\end{lemma}

\begin{proof}
For the first statement see, e.g., \cite{H04}, Lemma~A.3,
(the assumption that the fibres are semi log canonical is not used).
The second statement follows from the first by d\'evissage.
\end{proof}

\begin{definition}
Let $B_0(\bZ)$ denote the set of points of $B_0$ with integral coordinates in an integral affine chart.
We also write $B(\ZZ)=B_0(\ZZ)\cup \{0\}$.
\end{definition}

Given the description of Remark \ref{localcentralfibre}, the
following lemma is obvious.

\begin{lemma}
\label{invsystemindep} Suppose $n\ge 3$ and
we are given a convex multivalued piecewise linear function
$\varphi$ and a radical monomial ideal $J\subset P$ such that
$\kappa_{\rho,\varphi} \in J$ for all rays $\rho\in\Sigma$.
Then if $x\in\Spec R_J$ is a closed point,
the fibre of $X^o_J\rightarrow\Spec R_J$ over $x$ is
$(\Spec\kk[\Sigma])\setminus \{0\}$. Here, $\kk[\Sigma]$
denotes the $\kk$-algebra with a $\kk$-basis $\{z^m\,|\,
m\in B(\ZZ)\}$ with multiplication given exactly as in
\eqref{fanmult}, and $0$ is the closed point whose ideal
is generated by $\{z^m\,|\,m\not=0\}$. In particular, the fibre
is isomorphic to $\VV^o_n$. Furthermore, with $R_J[\Sigma]:=
R_J\otimes_{\kk}\kk[\Sigma]$,
\[
X^o_J\cong (\Spec R_J[\Sigma])\setminus (\Spec R_J)\times \{0\}.
\]
\end{lemma}

\subsection{Scattering diagrams on $B$}
\label{scatdiagsection}

Next we translate into algebraic geometry the instanton corrections.
To construct our mirror family we will use the canonical scattering
diagram $\foD^{\can}$ defined in \S\ref{Definitionsection},
(which is the translation of the instanton corrections associated
to Maslov index zero disks), but as the regluing process works for
any scattering diagram (and we will make use of this
greater generality in \cite{K3}), we carry it out for an
arbitrary scattering diagram.

We continue with the notation of the previous sections, with
$(Y,D)$, $(B,\Sigma)$, $P$ an arbitrary toric monoid, and $\varphi$ given.
We also fix a monomial ideal $J\subset P$ such that
$J=\sqrt{J}$. Denote by $\widehat{R}$ the completion of
$\kk[P]$ with respect to the ideal $J$, and for any
$\tau\in\Sigma$, $\tau\not=0$, denote by
\[
\widehat{\kk[P_{\varphi_{\tau}}]}
\]
the completion of the ring $\kk[P_{\varphi_{\tau}}]$ with respect to
the ideal $J\kk[P_{\varphi_{\tau}}]$.

We will now define
a \emph{scattering diagram}, which encodes a modification of
the construction of $X_I^o$. Unlike the previous subsection, where
we assumed $n\ge 3$ for ease of exposition throughout, in this subsection
we can allow any number of irreducible components of $D$
except where noted.

\begin{definition}
\label{scatdiagdef}
A \emph{scattering diagram} for the data $(B,\Sigma)$, $P,\varphi,$ and $J$
is a set
\[
\foD=\{(\fod,f_{\fod})\}
\]
where
\begin{enumerate}
\item $\fod\subset B$ is a ray in $B$ with endpoint the origin with
rational slope.
$\fod$ may coincide with a ray of $\Sigma$, or lie in the interior
of a two-dimensional cone of $\Sigma$.
\item Let $\tau_{\fod}\in\Sigma$ be the smallest cone containing
$\fod$. Then $f_{\fod}$ is a formal sum
\[
f_{\fod}=1+\sum_{p} c_pz^p\in \widehat{\kk[P_{\varphi_{\tau_{\fod}}}]}
\]
for $c_p\in \kk$ and $p$ running over elements of
$P_{\varphi_{\tau_{\fod}}}$ such that
$r(p)\not=0$ and $r(p)$ is tangent to $\fod$. Here $r$ is defined by
\eqref{standardexactseq}.
We further require that $\fod$ satisfy one of the following two properties:
\begin{enumerate}
\item
For those $p$ with $c_p\not=0$,
$r(p)$, viewed as a tangent vector at an interior point of $\fod$, points
towards the origin, in which case we say that
$\fod$ is an \emph{outgoing ray}.
\item
For those $p$ with $c_p\not=0$, $r(p)$
points away from the origin, in which case we say that $\fod$ is
an \emph{incoming ray}.
\end{enumerate}
\item If $\dim\tau_{\fod}=2$ or if $\dim\tau_{\fod}=1$ and
$\kappa_{\tau_{\fod},\varphi}\not\in J$,
then $f_{\fod}\equiv 1\mod J$.
\item  For any ideal $I\subset P$ with $\sqrt{I}=J$,
there are only a finite number of $(\fod,f_{\fod})\in\foD$
such that $f_{\fod}\not\equiv 1\mod I\kk[P_{\varphi_{\tau_{\fod}}}]$.
\end{enumerate}
\end{definition}

\begin{construction}
\label{gluingconstruction}
We now explain how a scattering diagram
$\foD$ is used to modify the construction
of $X^o_I$, as given in Construction \ref{XoIdef}.
Suppose we are given a scattering diagram $\foD$ for
the data $(B,\Sigma)$, $P$, $\varphi$ and $J$, and an ideal $I$
with $\sqrt{I}=J$. We assume that $\kappa_{\rho,\varphi}\in J$ for all
rays $\rho\in\Sigma$ and that $n\ge 3$ as in Construction \ref{XoIdef}.

We will use the scattering diagram $\foD$ to modify both the definition
of the rings $R_{\rho_i,I}$
as well as the gluings of the schemes defined
by these rings. First, we modify the definition of $R_{\rho_i,I}$,
setting
\begin{equation}
\label{localequations}
R_{\rho_i,I} :=
{R_I[X_{i-1},X_i^{\pm 1},X_{i+1}]
\over (X_{i-1} X_{i+1} -
z^{\kappa_{\rho_i,\varphi}} X_i^{-D_{\rho_i}^2} f_{\rho_i})}
\end{equation}
Here $f_{\rho_i}$ is an element of $R_I[X_i^{\pm 1}]$ defined by
\[
f_{\rho_i}=\prod_{(\fod,f_{\fod})\in\foD\atop \fod=\rho_i} f_{\fod}
\mod I\kk[P_{\varphi_{\rho_i}}],
\]
identifying $X_i$ with $z^{\varphi_{\rho_i}(v_i)}$ as in Proposition
\ref{ringfirstversion}. Note this is a generalization of the old definition
of $R_{\rho_i,I}$, which we obtain if $f_{\rho_i}=1$. Thus we continue
to use the same notation.

Retaining the definition $R_{\sigma,I}=\kk[P_{\varphi_{\sigma}}]
\otimes_R R_I$ for $\dim\sigma=2$ from the previous subsection, we note that
there are maps
\[
\psi_{\rho_i,-}:R_{\rho_i,I}\rightarrow R_{\sigma_{i-1,i},I},
\quad \psi_{\rho_i,+}:R_{\rho_i,I}\rightarrow R_{\sigma_{i,i+1},I},
\]
given by
\begin{align}
\label{psirhoi}
\begin{split}
\psi_{\rho_i,-}(X_i)=&z^{\varphi_{\rho_i}(v_i)},
\quad \psi_{\rho_i,-}(X_{i-1})=z^{\varphi_{\rho_i}(v_{i-1})},
\quad \psi_{\rho_i,-}(X_{i+1})=f_{\rho_i}z^{\varphi_{\rho_i}(v_{i+1})},\\
\psi_{\rho_i,+}(X_i)=&z^{\varphi_{\rho_i}(v_i)},
\quad \psi_{\rho_i,+}(X_{i-1})=f_{\rho_i}z^{\varphi_{\rho_i}(v_{i-1})},
\quad \psi_{\rho_i,+}(X_{i+1})=z^{\varphi_{\rho_i}(v_{i+1})}.
\end{split}
\end{align}
Furthermore, $\psi_{\rho_i,\pm}$ induce isomorphisms
\[
\psi_{\rho_i,+}:(R_{\rho_i,I})_{X_{i+1}}\rightarrow R_{\sigma_{i,i+1},I},
\quad
\psi_{\rho_i,-}:(R_{\rho_i,I})_{X_{i-1}}\rightarrow R_{\sigma_{i-1,i},I}.
\]
Set for $\tau\in \Sigma\setminus \{0\}$
\[
U_{\tau,I}:=\Spec R_{\tau,I}.
\]
One checks easily that the natural map $U_{\rho,I}\rightarrow \Spec R_I$
is flat. The maps $\psi_{\rho_i,\pm}$ induces canonical embeddings
$U_{\sigma_{i-1,i},I},U_{\sigma_{i,i+1},I}\hookrightarrow U_{\rho_i,I}$,
and we denote their image
by $U_{\rho_i,\sigma_{i-1,i},I}$ and $U_{\rho_i,\sigma_{i,i+1},I}$
respectively. Note that \eqref{chartintersection}
continues to hold.

Next, consider
$(\fod,f_{\fod})\in\foD$ with $\tau_\fod=\sigma\in\Sigma_{\max}$.
Let $\gamma$ be a path in $B_0$ which crosses $\fod$
transversally at time $t_0$. Then define
\[
\theta_{\gamma,\fod}:R_{\sigma,I}\rightarrow R_{\sigma,I}
\]
by
\[
\theta_{\gamma,\fod}(z^p)=z^pf^{\langle n_{\fod}, r(p)\rangle}_{\fod}
\]
where $n_{\fod}\in \Lambda_{\sigma}^*$ is primitive and satisfies,
with $m$ a non-zero tangent vector of $\fod$,
\[
\langle n_{\fod},m\rangle=0,\quad \langle n_{\fod},\gamma'(t_0)\rangle<0.
\]
If $\gamma$ is not differentiable at $t_0$, which might occur for
broken lines, see Definition \ref{brokenlinedef}, this inequality is
interpreted to mean that $n_{\fod}$ is positive at $\gamma(t_0-\epsilon)$
and negative at $\gamma(t_0+\epsilon)$
for $\epsilon>0$ small.
Note that $f_{\fod}$ is invertible in $R_{\sigma,I}$ since
$f_{\fod}\equiv 1\mod J\kk[P_{\varphi_{\sigma}}]$, so $f_{\fod}-1$
is nilpotent in $R_{\sigma,I}$.

Let $\foD_I\subset\foD$ be the finite set of rays $(\fod,f_{\fod})$
with $f_{\fod}\not\equiv 1\mod I\kk[P_{\varphi_{\tau_{\fod}}}]$.
For a path $\gamma$ wholly contained in the interior of
$\sigma\in\Sigma_{\max}$ and crossing elements of $\foD_I$ transversally,
we define
\[
\theta_{\gamma,\foD}:=\theta_{\gamma,\fod_n}\circ\cdots
\circ\theta_{\gamma,\fod_1},
\]
where $\gamma$ crosses precisely the elements $(\fod_1,f_{\fod_1}),
\ldots,(\fod_n,f_{\fod_n})$ of $\foD_I$,
in the given order. Note that if two rays
$\fod_i,\fod_{i+1}$ in fact coincide as subsets of $B$, then $\theta_{\gamma,
\fod_i}$ and $\theta_{\gamma,\fod_{i+1}}$ commute, so the ordering is
not important for overlapping rays.

To construct $X^o_{I,\foD}$, we modify the gluings of the sets $U_{\rho,I}$
along the open subsets $U_{\rho,\sigma,I}$. For each $i$, we have
canonical identifications of open subsets
\[
U_{\rho_i,I}\supset
U_{\rho_i,\sigma_{i,i+1},I}\cong U_{\sigma_{i,i+1},I}\cong
U_{\rho_{i+1},\sigma_{i,i+1},I}\subset U_{\rho_{i+1},I}
\]
We can modify this identification via any automorphism of $U_{\sigma_{i,i+1}}$.
We do this by
choosing a path $\gamma_i:[0,1]\rightarrow B$ whose image is contained
in the interior of $\sigma_{i,i+1}$,
with $\gamma_i(0)$ a point in $\sigma_{i,i+1}$ close to $\rho_i$
and $\gamma_i(1)\in\sigma_{i,i+1}$ close to $\rho_{i+1}$, chosen so that
$\gamma_i$
crosses every ray $(\fod,f_{\fod})$ of $\foD_I$ with
$\tau_{\fod}=\sigma_{i,i+1}$ exactly once, see Figure \ref{basicgluingpath}.

\begin{figure}
\input{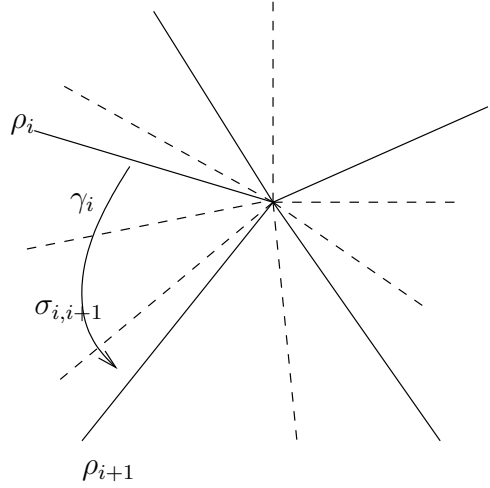}
\caption{The path $\gamma_i$. The solid lines indicate the fan, the dotted
lines are additional rays in $\foD$. The solid lines may also support
rays in $\foD$}
\label{basicgluingpath}
\end{figure}

We then obtain an automorphism
\[
\theta_{\gamma_i,\foD}:R_{\sigma_{i,i+1},I}\rightarrow R_{\sigma_{i,i+1},I},
\]
hence, after taking $\Spec$, an isomorphism
\[
\theta_{\gamma_i,\foD}:U_{\rho_{i+1},\sigma_{i,i+1},I}\rightarrow
U_{\rho_i,\sigma_{i,i+1},I}.
\]
We now define $X^o_{I,\foD}$ by dividing out $\coprod_i U_{\rho_{i},I}$
by the equivalence relation given by identifying
$x\in U_{\rho_{i+1},\sigma_{i,i+1},I}\subseteq U_{\rho_{i+1},I}$ with
$\theta_{\gamma_i,\foD}(x)\in U_{\rho_i,\sigma_{i,i+1},I}\subseteq
U_{\rho_i,I}$.
\qed
\end{construction}

\subsection{Broken lines}
\label{brokenlinesection}

We continue to fix a rational surface with anti-canonical
cycle $(Y,D)$ as usual, $D$ having an arbitrary number of irreducible
components, giving $(B,\Sigma)$, as well as a monoid
$P$, a multivalued $P$-convex function $\varphi$ on $B$, $J\subset P$ an
ideal with $\sqrt{J}=J$, and a scattering diagram $\foD$ for this data.
Broken lines were introduced in \cite{G09}
and their theory was further developed in \cite{CPS}.

\begin{definition} Let $B$ be an integral affine manifold.
An integral affine map $\gamma:(t_1,t_2)\rightarrow B$ from an
open interval $(t_1,t_2)$ is a continuous map such that for
any integral affine coordinate chart $\psi:U\rightarrow \RR^n$ of $B$,
$\psi\circ \gamma:\gamma^{-1}(U)\rightarrow \RR^n$ is integral
affine, i.e., is given by $t\mapsto t v+b$ for some $v\in \ZZ^n$
and $b\in \RR^n$.

Note that for an integral affine map, $\gamma'(t)\in\Lambda_{B,\gamma(t)}$.
\end{definition}

\begin{definition}
\label{brokenlinedef}
A \emph{broken line} $\gamma$
in $(B,\Sigma)$ for $q\in B_0(\ZZ)$ with endpoint $Q\in B_0$ is
a proper continuous piecewise
integral affine map $\gamma:(-\infty,0]\rightarrow B_0$ with only a
finite number of domains of linearity, together
with, for each $L\subset (-\infty,0]$ a maximal connected domain of linearity
of $\gamma$, a choice of monomial $m_L=c_Lz^{q_L}$ where $c_L \in \kk^{\times}$
and $q_L \in \Gamma(L,\gamma^{-1}(\shP)|_{L})$,
satisfying the following properties.
\begin{enumerate}
\item For the unique unbounded domain of linearity $L$,
$\gamma|_L$ goes off to infinity in a cone $\sigma\in\Sigma_{\max}$ as
$t\to -\infty$, and $q\in\sigma$.
Furthermore, using the identification of the stalk
$\shP_x$ for $x\in\sigma$ with $\shP_{\sigma}$,
$m_L=z^{\varphi_{\sigma}(q)}$.
\item For each $L$ and $t\in L$,
$-r(q_L)=\gamma'(t)$, where $r$ is defined in \eqref{standardexactseq}.
Also $\gamma(0)=Q\in B_0$.
\item Let $t\in (-\infty,0)$ be a point at which $\gamma$ is not linear,
passing from domain of linearity $L$ to $L'$. If $\gamma(t)\in\tau\in\Sigma$,
then $\shP_{\gamma(t)}=\shP_{\tau}$, so
that we can view $q_L\in\shP_{\tau}$ and $r(q_L) \in \Lambda_{\tau}$.
Let $\fod_1,\ldots,\fod_p\in \foD$ be the rays of
$\foD$ that contain $\gamma(t)$, with attached functions
$f_{\fod_j}$. Then we require that $\gamma$ passes from one side of these
rays to the other at time $t$, so that $\theta_{\gamma,\fod_j}$ is defined.
Let $n=n_{\fod_j}$ be the primitive element of $\Lambda_{\tau}^*$
used to define $\theta_{\gamma,\fod_j}$. Expand
\begin{equation}
\label{bendingterms}
\prod_{j=1}^p f_{\fod_j}^{\langle n,r(q_L)\rangle}
\end{equation}
as a formal power series in $\widehat{\kk[P_{\varphi_{\tau}}]}$.
Then there is a term $cz^{s}$ in this sum with
\[
m_{L'}=m_L\cdot (c z^{s}).
\]
\end{enumerate}
\end{definition}

\begin{remark}
\label{positiveexpremark}
Using the notation of item (3) above, by item (2) of the definition,
\begin{equation}
\label{raycrossinginequality}
\langle n, r(q_L) \rangle > 0.
\end{equation}
This is vital to interpret \eqref{bendingterms}.
Indeed, if $\tau$ is a ray, $f_{\fod_i}$ need not be invertible
in $\widehat{\kk[P_{\varphi_{\tau}}]}$, so
\eqref{raycrossinginequality} tells us that
\eqref{bendingterms} makes sense in this ring.
\end{remark}

\begin{example}
\label{firstbrokenexample}
We give a first example of broken lines, in the case where $B$ is as given in
Example \ref{negdefexample}
and $\foD=\emptyset$, so that there is no possibility
of bending. Nevertheless, there is quite non-trivial behaviour. For an example
including bending, see Example \ref{M05more} after the introduction
of the canonical scattering diagram.

Given $q\in B_0(\ZZ)$, $Q\in B_0$ general, we can choose lifts
$\tilde q, \tilde Q$ to the universal cover $\tilde B_0=|\tilde\Sigma|
\setminus \{0\}$ of $B_0$. Let $\pi:\tilde B_0\rightarrow B_0$ be the
covering map. Fixing the lift $\tilde Q$, for any lift $\tilde q$
we obtain a broken line $\gamma:(-\infty,0]\rightarrow B_0$ given by
$\gamma(t)=\pi(\tilde Q-t\tilde q)$. As this has one domain of linearity $L$,
we decorate $L$ with the monomial $z^{\varphi_{\sigma}(q)}$, where
$q\in \sigma\in \Sigma$. Note there are an infinite number of such broken
lines, one for each lift of $q$. Dealing with this non-finiteness is a key
part of the proof of Looijenga's conjecture in \S\ref{cusp}.
\end{example}

The next lemma and corollary are crucial for interpreting the
monomials $m_L$:

\begin{lemma} \label{transportlemma}
Let $\sigma_-,\sigma_+ \in \Sigma_{\max}$ be the two
maximal cones containing the ray $\rho \in \Sigma$.
If $q \in P_{\varphi_{\sigma_-}}$ with
$-r(q) \in \Int(\rho^{-1}\sigma_+)\subset \Lambda_{\rho}\otimes_{\ZZ}\RR$,
then
\[
q \in P_{\varphi_\rho} = P_{\varphi_{\sigma_-}}\cap P_{\varphi_{\sigma_+}}.
\]
\end{lemma}

\begin{proof}
By the definitions there exist $p, \kappa_{\rho,\varphi} \in P$
and $n_{\rho}\in \Lambda_{\rho}^*$ annihilating the tangent space to
$\rho$ and positive on $\sigma_+$ such that
\begin{align*}
q = {} &\varphi_{\sigma_-}(r(q)) + p \\
\varphi_{\sigma_+}(-r(q))= {} &
\varphi_{\sigma_-}(-r(q)) + \langle n_{\rho},-r(q)\rangle \kappa_{\rho,
\varphi}.
\end{align*}
Since $\langle n_{\rho},-r(q)\rangle>0$,
\[
q= \varphi_{\sigma_+}(r(q)) + p +
\langle n_{\rho},-r(q)\rangle \kappa_{\rho,\varphi}\in P_{\varphi_{\sigma_+}}.
\]
\end{proof}

An immediate consequence of this lemma is

\begin{corollary}
\label{stayinPtaucor}
\begin{enumerate}
\item
Let $\gamma:[t_1,t_2]\rightarrow B_0$ be integral affine.
Suppose that $\gamma(t_1)\in\tau_1$, $\gamma(t_2)\in\tau_2$.
Suppose also we are given a section $q\in \Gamma(\gamma^{-1}
\shP)$ such that $-r(q)=\gamma'(t)$ for each $t$.
If
\[
q(t_1) \in P_{\varphi_{\tau_1}}\subset \shP_{\tau_1}=
\shP_{\gamma(t_1)},
\]
then
\[
q(t_2)\in P_{\varphi_{\tau_2}}\subset \shP_{\tau_2}=
\shP_{\gamma(t_2)}.
\]
\item If $\gamma$ is a broken line, $t\in L$ a maximal domain
of linearity with $\gamma(t)\in \tau$, then
\[
q_L\in P_{\varphi_{\tau}}\subset \shP_{\tau}=\shP_{\gamma(t)}.
\]
\end{enumerate}
\end{corollary}

\begin{proof}
The first item follows immediately from the lemma. The second item
follows from the fact that if $t\ll 0$ lies in the unbounded domain
of linearity with $\gamma(t)\in\sigma$,
then $m_L=z^{\varphi_{\tau}(q)}
\in P_{\varphi_\sigma}$ by construction. Then this holds for
all $t$ by item (1) and Definition \ref{brokenlinedef}, (3).
\end{proof}

The convexity of $\varphi$ puts further restrictions on the monomial
decorations of a broken line.

\begin{definition}
\label{orderdef}
Let $J \subset P$ be a proper monoid ideal.
For $p \in J$ there exists a maximal $k \geq 1$ such that
$p = p_1 + \cdots+p_k$ with $p_i \in J$. We define $\ord_J(p)$
to be this maximum, and define $\ord_J(p) =0$ if $p \in P \setminus J$.

For $x\in \tau$, $q \in P_{\varphi_{\tau}}$,
define $\ord_{J,x}(q) := \ord_J(q - \varphi_{\tau}(r(q)))$. This measures
how high $q$ is above the graph of $\varphi_{\tau}$.
If $\gamma$ is a broken line and $t\in L$ a maximal domain of
linearity, define
\[
\ord_{J,\gamma}(t)=\ord_{J,\gamma(t)}(q_L),
\]
using $\gamma(t)\in\tau$ and $q_L\in P_{\varphi_\tau}\subset\shP_{\gamma(t)}$.
\end{definition}

\begin{lemma}
\label{orderincreasinglemma}
Let $\gamma$ be a broken line.
Then if $t<t'$,
\[
\ord_{J,\gamma}(t) \leq \ord_{J,\gamma}(t'),
\]
with strict inequality if either $t$ and $t'$
lie in different domains of linearity or for some $t''$ with
$t<t''<t'$,  $\gamma(t'')$ lies in a ray $\rho\in\Sigma$ with
bending parameter $\kappa_{\rho,\varphi} \in J$.
\end{lemma}

\begin{proof} This is immediate from the definitions and
the proof of Lemma \ref{transportlemma}.
\end{proof}

\begin{definition}
For $I$ an ideal in $P$ with $\sqrt{I}=J$, let
\[
\Supp_I(\foD):=\bigcup_{\fod}\fod
\]
where the union is over all $(\fod,f_{\fod})\in\foD$ such
that $f_{\fod}\not\equiv 1\mod I\kk[P_{\varphi_{\tau_{\fod}}}]$.
By Definition \ref{scatdiagdef}, (4), this is a finite union.
\end{definition}

\begin{definition}
\label{deflift}
Let $I$ be an ideal of $P$ with $\sqrt{I}=J$, and let
$Q\in B\setminus\Supp_I(\foD)$, $Q\in\tau\in\Sigma$. For
$q\in B_0(\ZZ)$, define
\begin{equation}\label{Liftdef}
\Lift_Q(q):=\sum_{\gamma} \Mono(\gamma)\in \kk[P_{\varphi_{\tau}}]/
I\cdot \kk[P_{\varphi_{\tau}}],
\end{equation}
where the sum is over all broken lines $\gamma$ for $q$
with endpoint $Q$, and $\Mono(\gamma)$ denotes the monomial attached
to the last domain of linearity of $\gamma$. The word ``Lift'' is used
to indicate that this is, as we shall show, a lifting of a monomial $z^q$
on $X^o_{J,\foD}$ to $X^o_{I,\foD}$. The fact that
$\Lift_Q(q)$ lies in the
stated ring follows from:
\end{definition}

\begin{lemma}
\label{finitenesslemma}
Let $Q \in \sigma \in \Sigma_{\max}$, $q\in B_0(\ZZ)$.
Let $I$ be an ideal with $\sqrt{I}=J$. Assume that
$\kappa_{\rho,\varphi}\in J$ for at least one ray $\rho\in\Sigma$.
Then the following hold:
\begin{enumerate}
\item The collection of $\gamma$
in Definition \ref{deflift}
with
\[
\Mono(\gamma) \not \in I \cdot \kk[P_{\varphi_{\sigma}}]
\]
is finite.
\item
If one boundary ray of the connected component
of $B \setminus \Supp_I(\foD)$ containing $Q$
is a ray $\rho\in\Sigma$, then $\Mono(\gamma) \in
\kk[P_{\varphi_{\rho}}]$,
and the collection of $\gamma$ with
\[
\Mono(\gamma) \not \in I \cdot \kk[P_{\varphi_{\rho}}]
\]
is finite.
\end{enumerate}
\end{lemma}

\begin{proof} Note there is some $k$ such that
$J^k\subset I$ because $\kk[P]$ is Noetherian. If $\gamma$ is a broken line with
$\Mono(\gamma)\not\in I\cdot \kk[P_{\varphi_{\sigma}}]$, then
$\gamma$ crosses the rays of $\Sigma$
in a cyclic order. Indeed, this follows from condition (3) of Definition
\ref{brokenlinedef}, as a broken line must cross from one side to the other
of each ray of $\foD$ it intersects. From this, the hypotheses on
the $\kappa_{\rho,\varphi}$ imply that in any set of at least $n$
consecutive rays of $\Sigma$ that it crosses, there is
at least one ray $\rho$ with $\kappa_{\rho,\varphi} \in J$.
By Lemma \ref{orderincreasinglemma}, $\ord_{J,\gamma}$
increases every time $\gamma$ crosses such a ray, and also
every time $\gamma$ bends at a ray $\fod$ not contained in a ray
of $\Sigma$. Once $\ord_{J,\gamma}\ge k$,
$\Mono(\gamma)\in I\cdot \kk[P_{\varphi_{\sigma}}]$.
Hence there is an absolute bound
on the number of rays of $\Sigma$ that $\gamma$ can cross,
and the number of times $\gamma$ can bend. When $\gamma$ crosses
a ray, there are a finite number of terms in \eqref{bendingterms} modulo
$I\cdot\kk[P_{\varphi_{\sigma}}]$, as the exponent is always positive,
see Remark \ref{positiveexpremark}. Thus there are a finite number of
possible choices of bend, and hence only a finite number of possible
choices for the exponent of
$\Mono(\gamma)$ modulo $I\cdot \kk[P_{\varphi_{\sigma}}]$ once
the initial monomial of $\gamma$ is fixed. Given any prescribed sequence
of bends and initial direction, one sees that there is only one possibility
for the underlying map $\gamma$ with endpoint a fixed point $Q$ by tracing
the broken line back from $Q$. Each such underlying map $\gamma$ supports
only a finite number of broken lines modulo $I\cdot \kk[P_{\varphi_{\sigma}}]$.
This yields the finiteness of (1).

The argument for the finiteness statement in (2), once the
first part of (2) is established, is the same. For the
first part of (2), consider a broken line $\gamma$ contributing
to $\Lift_Q(q)$. We take $Q \in \sigma_+$, in the notation
of Lemma \ref{transportlemma}. Write $\Mono(\gamma)=c_Lz^{q_L}$.
If $r(q_L) \in \rho^{-1}\sigma_+$  then the
statement follows from Lemma \ref{transportlemma}. Otherwise
(by the definition of broken line) $\gamma$ crosses $\rho$, which is
the last ray of $\Sigma$ and the last ray of
$\Supp_I(\foD)$ it crosses
before reaching $Q$. Now the result follows from Lemma \ref{transportlemma}
and the definition of broken line.
\end{proof}

\begin{definition} \label{consistentdef}
Assume that $\kappa_{\rho,\varphi}\in J$ for at least one ray
$\rho\in\Sigma$. We say a scattering diagram $\foD$ is
\emph{consistent} if for all ideals $I\subset P$ with
$\sqrt{I}=J$ and for all $q\in B_0(\ZZ)$, the following holds.
Let $Q\in B_0$ be chosen so that the line joining
the origin and $Q$ has irrational slope, and $Q'\in B_0$ similarly.
Then:
\begin{enumerate}
\item If $Q,Q'\in\sigma\in\Sigma_{\max}$, then we can view $\Lift_{Q}(q)$
and $\Lift_{Q'}(q)$ as elements of $R_{\sigma,I}$, and as such,
we have
\[
\Lift_{Q'}(q)=\theta_{\gamma,\foD}(\Lift_Q(q))
\]
for $\gamma$ a path contained in the interior of $\sigma$ connecting $Q$ to
$Q'$.
\item
If $Q_-\in \sigma_-$ and $Q_+\in\sigma_+$ with $\sigma_{\pm}\in\Sigma_{\max}$
and $\rho=\sigma_+\cap\sigma_-$ a ray, and furthermore
$Q_-$ and $Q_+$ are contained in connected components of
$B \setminus\Supp_I(\foD)$ whose closures contain $\rho$,
then $\Lift_{Q_{\pm}}(q)\in R_{\sigma_{\pm},I}$ are both images
under $\psi_{\rho,\pm}$ of a single element
\[
\Lift_{\rho}(q)\in R_{\rho,I}.
\]
\end{enumerate}
\end{definition}

Of course the definition is introduced so that the following construction
works:

\begin{construction}[Construction of $\vartheta_q$]
\label{thetaconstruction}
Suppose $D\subset Y$ has $n\ge 3$ irreducible components, and that
$\foD$ is a consistent scattering diagram for data $(B,\Sigma)$, $P$, $\varphi$
and $J$. Assume further that $\kappa_{\rho,\varphi}\in J$ for all
$\rho\in\Sigma$, so that we may apply Construction \ref{gluingconstruction}.
We now construct for any $I$ with $\sqrt{I}=J$ a function
$\vartheta_q\in\Gamma(X^o_{I,\foD},
\shO_{X^o_{I,\foD}})$ for $q\in B(\ZZ)=B_0(\ZZ)\cup\{0\}$.

We define $\vartheta_0=1$. Next, let $q\in B_0(\ZZ)$.
For each ray $\rho\in\Sigma$ contained
in $\sigma_{\pm}\in\Sigma_{\max}$,
choose two points $Q_{\rho}^{\pm}\in B$, one each in the two
connected components of
$B\setminus (\Supp_I(\foD)\cup\rho)$ which are adjacent to $\rho$, with
$Q^+_{\rho}\in\sigma_+$ and $Q^-_{\rho}\in\sigma_-$.

We first note that $\Lift_{Q_{\rho}^{\pm}}(q)$ is a well-defined
element of $R_{\sigma_{\pm},I}$, independent of the particular
choice of $Q_{\rho}^{\pm}$: given a choice say of $Q=Q_{\rho}^{+}$ and
another choice $Q'$, we take a path $\gamma$ connecting $Q$ and $Q'$
wholly contained in the connected component of $B\setminus
(\Supp_I(\foD)\cup\rho)$
containing $Q$ and $Q'$. By Definition \ref{consistentdef}, (1), it then follows
that $\Lift_Q(q)=\Lift_{Q'}(q)$.

By Definition \ref{consistentdef}, (2),
we have an element $\Lift_{\rho}(q)\in R_{\rho,I}$ whose
image under $\psi_{\rho,\pm}$ is $\Lift_{Q_{\rho}^{\pm}}(q)$.
It then follows via
another application of Definition \ref{consistentdef}, (1), applied to the
path of Figure \ref{basicgluingpath}, that if $\rho$, $\rho'$ are adjacent
rays in $\Sigma$, then $\Lift_{\rho}(q)$ and $\Lift_{\rho'}(q)$ glue
under the identification of open subsets of $U_{\rho,I}$ and $U_{\rho',I}$
given by $\theta_{\gamma,\foD}$.
Thus all these elements of the rings $R_{\rho,I}$ for $\rho\in\Sigma$
glue to give a regular
function on $X^o_{I,\foD}$, by construction of this latter space.
This regular function is what we call $\vartheta_q$.
\end{construction}

\begin{theorem} \label{univfamth}
Suppose $D$ has $n\ge 3$ irreducible components, and
let $\varphi$ be a multivalued piecewise linear function on $B$ such that
$\kappa_{\rho,\varphi}\in J$ for all rays $\rho\in\Sigma$.
Let $\foD$ be a consistent scattering diagram and $I\subset P$
an ideal with $\sqrt{I}=J$. Set
\[
X_I:=\Spec \Gamma(X^o_{I,\foD},\shO_{X^o_{I,\foD}}).
\]
Since $X^o_{I,\foD}$ has the structure of a scheme
over $\Spec R_I$, so does $X_I$, which we write as
\[
f_I:X_I\rightarrow\Spec R_I.
\]
Then
\begin{enumerate}
\item $X_I$ contains $X^o_{I,\foD}$ as an open subset and
$f_I$ is flat with fibre over a closed
point $x$ of $\Spec R_I$ isomorphic to
the $n$-vertex $\VV_{n}$.
\item
For each $q\in B(\ZZ)$, there is a section $\vartheta_q\in \Gamma(X_I,
\shO_{X_I})$, and the set
\[
\{\vartheta_q\,|\,q\in B(\ZZ)\}
\]
is a free $R_I$-module basis for $\Gamma(X_I,\cO_{X_I})$.
\end{enumerate}
\end{theorem}

\begin{proof}
Construction \ref{thetaconstruction} constructs regular
functions $\vartheta_q$ on $X^o_{I,\foD}$, hence by definition of $X_I$,
we obtain $\vartheta_q\in \Gamma(X_I,\shO_{X_I})$.

Now note that $X^o_{J,\foD}=X^o_{J}$ as defined in
\S\ref{pgroupbundle}. Indeed, for any $(\fod,f_{\fod})\in\foD$ with
$\dim\tau_\fod=2$ we have $f_{\fod}\equiv 1\mod J$, so the open
sets $U_{\rho_i,J}$, $U_{\rho_{i+1},J}$ are glued trivially. Similarly,
if $\dim\tau_{\fod}=1$
then since $\kappa_{\rho,\varphi}\in J$, the rings $R_{\rho,I}$ as given
in \eqref{localequations} and \eqref{ringfirstequation} coincide
and are glued trivially. Thus with $I=J$, we see the gluing constructions
Constructions \ref{XoIdef} and \ref{gluingconstruction} coincide.

Note that with the assumption that $\kappa_{\rho,\varphi}\in J$ for all rays
$\rho$,
\[
X^o_J\cong (\Spec R_J[\Sigma])\setminus (\Spec R_J)\times\{0\}
\]
by Lemma \ref{invsystemindep}. We see that the canonical map
\[
\bigoplus_{q \in B(\bZ)} R_J\cdot \vartheta_q \rightarrow \Gamma(X^o_{J},\cO_{X^o_J})
\]
is an isomorphism. Indeed, by Lemma \ref{relS2},
$\Gamma(X^o_J,\cO_{X^o_J})\cong R_J[\Sigma]$. Furthermore, under this
isomorphism, $\vartheta_q$ is clearly taken to $z^q\in R_J[\Sigma]$. This
is because the only broken lines contributing to $\Lift_Q(q)$ modulo $J$
for any $Q$ is the straight line with endpoint $Q$, and this provides a
contribution only if $Q$ lies in the same maximal cone as $q$.

It also follows that $X_J:=\Spec \Gamma(X^o_J,\cO_{X^o_J})=\Spec R_J[\Sigma]$
is flat
over $\Spec R_J$ and the fiber over a closed point $x$ is given by
$\Spec\kk[\Sigma]$.


Now let $I$ be an ideal with $\sqrt{I}=J$.
Let $i \colon X^o_{J} \subset X_J$ be the inclusion.
Define a ringed space $X'_I$ with underlying topological space $X_J$ by $\cO_{X'_I}:=i_*\cO_{X^o_{I,\foD}}$.
Then the natural map $\cO_{X'_I} \rightarrow \cO_{X_J}$ is surjective by the existence of the lifts $\vartheta_q$.
Thus $X'_I / \Spec R_I$ is a flat deformation of $X_J / \Spec R_J$ by Lemma~\ref{flatnesslemma} below.
Now since $X_J$ is affine it follows that $X'_I$ is also affine, so $X'_I=X_I:=\Spec \Gamma(X^o_{I,\foD},\cO_{X^o_{I,\foD}})$.

We showed above that the $\vartheta_q$ form an
$R_J$-module basis of $\Gamma(X_J,\cO_{X_J})$.
Now since $X_I/\Spec R_I$ is a flat infinitesimal deformation of $X_J/\Spec
R_J$ it follows that the $\vartheta_q$ form a $R_I$-module basis of $\Gamma(X_I,\cO_{X_I})$, see Lemma~\ref{basis} below.
\end{proof}

\begin{lemma}
\label{flatnesslemma}
Let $X_0/S_0$ be a flat family of surfaces such that the fibres satisfy
Serre's condition $S_2$.
Let $i \colon X_0^o \subset X_0$ be the inclusion of an open subset
such that the complement has finite fibres. Note that $i_*\cO_{X^o_0}=\cO_{X_0}$ by Lemma~\ref{relS2}.

Let $S_0 \subset S$ be an infinitesimal thickening of $S_0$ and let $X^o \rightarrow S$ be a flat deformation of $X^o_0/S_0$ over $S$.
Define a family of ringed spaces $X \rightarrow S$
by $\cO_{X}:=i_*\cO_{X^o}$.

Then $X/S$ is a flat deformation of $X_0/S_0$ (that is, $X/S$ is flat and $X_0=X \times_S S_0$) if and only if the map
\begin{eqnarray} \label{infthicksurj}
\cO_{X}:=i_*\cO_{X^o} \rightarrow i_*\cO_{X^o_0}=\cO_{X_0}
\end{eqnarray}
is surjective.
\end{lemma}

\begin{proof}
The condition is clearly necessary.

Conversely, suppose (\ref{infthicksurj}) is surjective.
Let $\cI \subset \cO_S$ be the nilpotent ideal defining $S_0 \subset S$.
Let $X_n^o/S_n$ denote the $n$th order infinitesimal thickening of $X_0/S_0$ determined by $X^o/S$, that is,
$\cO_{X^o_n}=\cO_{X^o}/\cI^{n+1}\cdot\cO_{X^o}$ and $\cO_{S_n}=\cO_{S}/\cI^{n+1}$. Define $X_n/S_n$ by $\cO_{X_n}:=i_*\cO_{X^o_n}$.
Note that $\cO_{X_n} \rightarrow \cO_{X_0}$ is surjective because $\cO_{X} \rightarrow \cO_{X_0}$ is surjective by assumption.
We show by induction on $n$ that $X_n/S_n$ is a flat deformation of $X_0/S_0$. For $n=0$ there is nothing to prove. Suppose the induction hypothesis is true for $n$.
Since $X^o_{n+1}/S_{n+1}$ is flat (being the restriction of the flat family $X^o/S$ to $S_{n+1}$) we have a short exact sequence
$$0 \rightarrow \cI^{n+1}/\cI^{n+2} \otimes \cO_{X_0^o} \rightarrow \cO_{X^o_{n+1}} \rightarrow \cO_{X_n^o} \rightarrow 0.$$
Applying $i_*$ we obtain an exact sequence
\[
0 \rightarrow i_*(\cI^{n+1}/\cI^{n+2} \otimes \cO_{X^o_0}) \rightarrow \cO_{X_{n+1}} \rightarrow \cO_{X_n}.
\]
By Lemma~\ref{relS2} the first term is equal to $\cI^{n+1}/\cI^{n+2} \otimes \cO_{X_0}$. Moreover, the last arrow is surjective because $\cO_{X_{n+1}} \rightarrow \cO_{X_0}$ is surjective, $\cO_{X_n}/\cI \cdot \cO_{X_n} = \cO_{X_0}$ by the induction hypothesis, and $\cI$ is nilpotent, as in Theorem 8.4
of \cite{Matsumura89}
(where the module $M$ need not be finitely generated for the
argument given there to work). So we have an exact sequence
\begin{eqnarray} \label{flatnessinduction}
0 \rightarrow \cI^{n+1}/\cI^{n+2} \otimes \cO_{X_0} \rightarrow \cO_{X_{n+1}} \rightarrow \cO_{X_n} \rightarrow 0.
\end{eqnarray}
It follows that $\cO_{X_{n+1}}/\cI^{n+1} \cdot \cO_{X_{n+1}}=\cO_{X_n}$.
(Indeed, consider the map $$\alpha \colon \cI^{n+1} \otimes \cO_{X_{n+1}} \rightarrow \cO_{X_{n+1}}, \quad \alpha(f \otimes g)=fg.$$ We claim that $\alpha$ is equal to the composition of the map
$$\beta \colon \cI^{n+1} \otimes \cO_{X_{n+1}} \rightarrow \cI^{n+1}/\cI^{n+2} \otimes \cO_{X_0}$$ given by the natural maps on the factors and the first map $\gamma$ of the exact sequence (\ref{flatnessinduction}).
Since $\cO_{X_{n+1}}=i_*\cO_{X^o_{n+1}}$ by definition, it suffices to check the equality after restriction to $X^o$, where it is obvious. The map $\beta$ is surjective because $\cO_{X_{n+1}} \rightarrow \cO_{X_0}$ is surjective.
So the image of $\gamma$ is equal to the image of $\alpha$, namely $\cI^{n+1} \cdot \cO_{X_{n+1}}$.)
Now by \cite{Matsumura89}, Theorem~22.3, p.\ 174, the exact sequence (\ref{flatnessinduction}) shows that $X_{n+1}/S_{n+1}$ is a flat deformation of $X_0/S_0$.
\end{proof}

\begin{lemma} \label{basis}
Let $A \rightarrow B$ be a flat homomorphism of Noetherian rings and
$I \subset A$ a nilpotent ideal. Suppose given a set $S$ of elements of $B$ such that the reductions of the elements of $S$ form an $A/I$-module basis of $B/IB$. Then $S$ is an $A$-module basis of $B$.
\end{lemma}

\begin{proof}
Since $I$ is nilpotent and $S$ generates $B/IB$ it is clear that $S$ spans $B$
by Theorem 8.4 of \cite{Matsumura89}.
So we have an exact sequence
\[
0 \rightarrow K \rightarrow A^{S} \rightarrow B \rightarrow 0.
\]
Tensoring with $A/I$ we obtain an exact sequence
$$0 \rightarrow K/IK \rightarrow (A/I)^{S} \rightarrow B/IB \rightarrow 0$$
using flatness of $B$ over $A$. We deduce that $K/IK=0$ by our assumption, hence $K=0$ because $I$ is nilpotent.
\end{proof}

\begin{proposition} \label{volform}
Let $X_I/S_I:=\Spec R_I$ be the family of Theorem~\ref{univfamth}.
Then the relative dualizing sheaf $\omega_{X_I/S_I}$ is trivial. It is generated by the global section $\Omega$ given on local patches $U_{\rho_i,I}$ by
$\dlog X_{i-1} \wedge \dlog X_i = \dlog X_i \wedge \dlog X_{i+1}$. Here we
take the rays $\rho_j$ in counter-clockwise order, after
choosing an orientation on $B$, to obtain a consistent choice of signs.
\end{proposition}

\begin{proof}
By the adjunction formula for the closed embedding
$$U_{\rho_i} \subset \bA^2_{X_{i-1},X_{i+1}} \times \bG_{m,X_i} \times S_I,$$
the dualizing sheaf $\omega_{X_I/S_I}$ is freely generated over $U_{\rho_i}$
by the local section in the statement.
These sections patch to give a generator $\Omega$ of
$\omega_{X^o_{I,\foD}/S_I}$ because the scattering automorphisms preserve the
torus invariant two-forms. Both $\omega_{X_I/S_I}$ and $\cO_{X_I}$ satisfy the relative $S_2$ property $i_*i^*\cF=\cF$ where $i \colon X^o_I \subset X_I$ is the inclusion (\cite{H04}, Appendix, where the hypothesis of slc is not needed),
hence $\omega_{X_I/S_I}$ is freely generated by $\Omega$.
\end{proof}

\subsection{The algebra structure}
\label{algebrasection}
In the previous section, we saw that the $R_I$-algebra
\[
A_I=\Gamma(X^o_{I,\foD},\shO_{X^o_{I,\foD}})
\]
defining the flat deformation $X_I$ has an $R_I$-module basis
of theta functions $\{\vartheta_m\,|\, m\in B(\ZZ)\}$.
Here we derive a description of the multiplication rule on
$R_I$ using the geometry of the integral affine manifold $B$.
Besides being an attractive combinatorial description of the multiplication
rule, we will
use this (in the case of the canonical scattering diagram $\foD^{\can}$) in \S \ref{positivecase} to prove that our deformation extends over completions
of larger strata of $\Spec\kk[P]$, as well as for the case that $D$
has $1$ or $2$ irreducible components.

\begin{definition}
For a broken line $\gamma$ with endpoint $Q\in\tau\in\Sigma$,
define $s(\gamma) \in \Lambda_{\tau}$, $c(\gamma) \in \kk[P]$
by demanding that
\[
\Mono(\gamma) = c(\gamma) \cdot z^{\varphi_{\tau}(s(\gamma))}.
\]
Write $\Limits(\gamma) = (q,Q)$ if $\gamma$ is a broken line for
$q$ and has endpoint $Q\in B$.
\end{definition}

\begin{remark}
\label{canonicalidremark}
Recall that $B_0$ in fact has the structure of an integral \emph{linear}
manifold. One feature of such manifolds is that for any
simply connected set $U\subset B_0$, there is is a canonical linear immersion
$U\rightarrow\Lambda_{\RR,U}$, compatible with parallel transport inside
$U$.

In particular, if $q$ is a point of $B_0$ with $q\in\sigma\in\Sigma$,
and $\tau\subset\sigma$, then the canonical embedding of a neighbourhood
of $\tau\setminus\{0\}$ in $\Lambda_{\tau,\RR}$ identifies $q$ with a point of
$\Lambda_{\tau,\RR}$.
\end{remark}

\begin{theorem}
\label{multrule}
Let $q_1,q_2 \in B(\bZ)$. In the
canonical expansion
\[
\vartheta_{q_1} \cdot \vartheta_{q_2} =
\sum_{q \in B(\bZ)} \alpha_q \vartheta_q,
\]
where $\alpha_q\in R_I$ for each $q$, we have
\[
\alpha_q = \sum_{\substack{(\gamma_1,\gamma_2) \\
                 \Limits(\gamma_i) = (q_i,z) \\
                 s(\gamma_1) + s(\gamma_2) = q}}
                           c(\gamma_1)c(\gamma_2)
\]
Here $z\in B_0$ is a point very close to $q$ contained in a cell
$\tau$, and we identify $q$ with a point of $\Lambda_{\tau}$ using
Remark \ref{canonicalidremark}.
\end{theorem}

\begin{proof}
To identify the coefficient of $\vartheta_q$,
choose a point $z \in B$ very close to $q$,
and describe the product using the lifts of $z^{q_1}$, $z^{q_2}$
at $z$:
\[
(\Lift_z(q_1))(\Lift_z(q_2)) = \sum_{q'} \alpha_{q'}\Lift_z(q').
\]
Now observe first that there is only one broken line $\gamma$
with endpoint $z$ and $s(\gamma)=q\in \Lambda_{\tau}$:
this is the broken line whose image is $z+\RR_{\ge 0} q$.
Indeed, the final segment of such a $\gamma$ is on this ray, and this ray
meets no scattering rays, so the broken line cannot bend.
Thus the coefficient of $\Lift_z(q)$ on the right-hand side of
the above equation can be read off by looking at the coefficient
(in $R_I$) of $z^{\varphi_{\tau}(q)}$.
This gives the desired description.
\end{proof}

\section{The canonical scattering diagram}\label{canscatdiag}

\subsection{Definition}
\label{Definitionsection}
Here we give the precise definition of $\foD^{\can}$.
As explained in the introduction, it is, roughly speaking,
defined in terms of maps $\AA^1\rightarrow Y\setminus D$, which
are algebro-geometric analogues of the holomorphic disks used for instanton
corrections in the symplectic heuristic.
We begin by recalling necessary facts about
relative Gromov-Witten invariants used to count these curves.

\begin{definition} \label{subsectionGW}
\label{basicrelinvs}
Let $(\tilde Y,\tilde D)$ be a non-singular rational surface with $\tilde D$
an anti-canonical
cycle of rational curves, and let $C$ be an irreducible component
of $\tilde D$. Consider a class $\beta\in A_1(\tilde Y,\ZZ)$
such that
\begin{equation}
\label{betadef}
\beta\cdot \tilde D_i=\begin{cases} k_{\beta} & \tilde D_i=C\\
0 & \tilde D_i\not= C\end{cases}
\end{equation}
for some $k_{\beta}>0$. Let $F$ be the closure of $\tilde D\setminus C$,
and let
\[
\tilde Y^o:=\tilde Y\setminus F,\quad C^o:=C\setminus F.
\]
Let $\overline{\foM}(\tilde Y^o/C^o,\beta)$ be the
moduli space of stable relative maps of genus zero curves representing
the class $\beta$ with tangency of order $k_{\beta}$ at an unspecified
point of $C^o$. (See \cite{Li00}, \cite{Li02} for the algebraic definition
for these relative Gromov-Witten invariants, and \cite{LR01}, \cite{IP03}
for the original symplectic definitions.)
We refer to $\beta$ informally as an {\it $\bA^1$-class}.
The virtual dimension of this moduli space is
\[
-K_{\tilde Y}\cdot\beta+(\dim \tilde Y-3)-(k_{\beta}-1)=0.
\]
Here the first two terms give the standard dimension formula
for the moduli space of stable rational curves in $\tilde Y$ representing
the class $\beta$, and the term $k_{\beta}-1$ is the change in dimension
given by imposing the $k_{\beta}$-fold tangency condition.
The moduli space carries a virtual fundamental class. Furthermore,
we have

\begin{lemma}
\label{GPSinvariancelemma}
$\overline{\foM}(\tilde Y^o/C^o,\beta)$ is proper
over $\kk$.
\end{lemma}

\begin{proof} This follows as in the proof of \cite{GPS09}, Theorem 4.2.
In brief, let $R$ be a valuation ring with quotient field $K$,
with $S=\Spec R$, $T=\Spec K$. We would like to extend a morphism
$T\rightarrow \overline{\foM}(\tilde Y^o/C^o,\beta)$ to $S$.
We know that the moduli space $\overline{\foM}(\tilde Y/C,\beta)$
is proper, so we obtain a family of relative stable maps $\shC
\rightarrow S$ to $\tilde Y$. We just need to show that in fact
the image of the closed fibre $\shC_0$ lies in $\tilde Y^o$.
However, the argument in the proof of \cite{GPS09}, Theorem 4.2
shows that if the image of $\shC_0$ intersects $F$, then
$\shC_0$ must be of genus at least $1$, which is not the case.
\end{proof}

Given this, we define
\[
N_{\beta}:=\int_{[\overline{\foM}(\tilde Y^o/C^o,\beta)]^{vir}} 1.
\]
Morally, one should view $N_{\beta}$ as counting maps from
affine lines to $\tilde Y\setminus \tilde D$ whose closures represent
the class $\beta$.
\end{definition}

In what follows, we fix as usual the pair $(Y,D)$, with
tropicalisation $(B,\Sigma)$,
and $\varphi$ the function given by Example \ref{standardphi} for
some choice of $\eta:\NE(Y)\rightarrow P$. We can assume here
that $D$ has an arbitrary number of irreducible components.

\begin{definition}
\label{canscatdiagdef}
Fix a ray $\fod\subset B$ with endpoint the origin, with rational
slope. If $\fod$ coincides with a ray of $\Sigma$, set $\Sigma':=
\Sigma$; otherwise, let $\Sigma'$ be a refinement of $\Sigma$
obtained by adding the ray $\fod$ and a number of other rays chosen
so that each cone of $\Sigma'$ is integral affine isomorphic to the first
quadrant of $\RR^2$. This gives a toric blow-up
$\pi:\tilde Y\rightarrow Y$ (the identity in the first case) by
Lemma \ref{toricblowuplemma}.
Let $C\subset\pi^{-1}(D)$ be the irreducible component
corresponding to $\fod$.

Let $\tau_{\fod}\in\Sigma$ be the smallest cone containing
$\fod$.
Let $m_{\fod}\in \Lambda_{\tau_\fod}$ be a primitive generator of the
tangent space to $\fod$, pointing away from the origin.
Define
\[
f_{\fod} := \exp \left[ \sum_{\beta} k_{\beta} N_{\beta}
z^{\eta(\pi_*(\beta))-\varphi_{\tau_{\fod}}(k_{\beta} m_{\fod})}\right].
\]
Here the sum is over all classes $\beta\in A_1(\tilde Y,\ZZ)$
satisfying \eqref{betadef}. Note that if $N_{\beta}\not=0$, then
necessarily $\overline{\foM}(\tilde Y^o/C^o,\beta)$ is non-empty, and
thus $\beta\in \NE(\tilde Y)$, so $\pi_*(\beta)\in \NE(Y)$.
We note the numbers $N_{\beta}$ do not depend on the particular
choice of refinement $\Sigma'$. Indeed, further
refining $\Sigma'$ does not change
the pair $\tilde Y^o/C^o$, and hence does not change
the numbers $N_{\beta}$.

We define
\[
\foD^{\can}:=\{(\fod, f_{\fod}) \,|\, \hbox{$\fod\subset B$
a ray of rational slope}\}.
\]

We call a class $\beta\in A_1(\tilde Y,\ZZ)$ an \emph{$\AA^1$-class}
if $N_{\beta}\not=0$.
\end{definition}

Note that all rays of the canonical scattering diagram are outgoing.

\begin{remark}
In theory, one should be able to use logarithmic
Gromov-Witten invariants (\cite{GS11} or \cite{AC11}) to define $N_{\beta}$
without the technical trick of blowing up and working on
an open variety. This would be done by working relative
to $D$, and counting rational curves of class $\beta$
with one point mapping to the boundary with specified
orders of tangency with each boundary divisor, with non-zero
order of tangency with either one divisor $D_i$ or two adjacent
divisors $D_i$, $D_{i+1}$. However, some additional arguments are required
to compare logarithmic invariants with the invariants described above
as developed in \cite{GPS09}, and we do not wish to do this here.
\end{remark}

\begin{lemma}
\label{Scanlemma} Let $J\subset P$ be an ideal with
$\sqrt{J}=J$.
Suppose the map $\eta \colon \NE(Y)\rightarrow P$ satisfies
the following conditions:
\begin{enumerate}
\item For any ray $\fod\subset B$ of rational slope, let $\pi:\tilde Y
\rightarrow Y$ be the corresponding blow-up. We require that
if $\dim\tau_{\fod}=2$ or $\dim\tau_{\fod}=1$ and $\kappa_{\tau_{\fod},\varphi}
\not\in J$ then
for any $\AA^1$-class $\beta$ contributing to $f_{\fod}$,
we have $\eta(\pi_*(\beta)) \in J$.
\item For any ideal $I$ with $\sqrt{I}=J$, there are
only a finite number of $\fod$ and $\AA^1$-classes $\beta$
such that $\eta(\pi_*(\beta))\not\in I$.
\end{enumerate}
Then $\foD^{\can}$ is a scattering diagram for the data $(B,\Sigma), P,
\varphi,$ and $J$.
\end{lemma}

\begin{proof}
Note that
\[
z^{\eta(\pi_*(\beta))-\varphi_{\tau_{\fod}}(k_{\beta}
m_{\fod})}\in I\kk[P_{\varphi_{\tau_{\fod}}}]
\]
if and only if
$\eta(\pi_*(\beta))\in I$. So the hypotheses of the lemma imply conditions
(2)-(4) of Definition \ref{scatdiagdef}.
\end{proof}

\begin{example}
\label{goodPexample}
Let $\sigma \subset A_1(Y)\otimes_{\ZZ}\RR$ be a strictly convex rational
polyhedral cone containing $\NE(Y)$.
(This can be obtained as the dual of a strictly
convex rational polyhedral cone in $\Pic(Y)\otimes_{\ZZ}\RR$ which spans
this latter space and is contained in the nef cone.) Let $P=\sigma\cap A_1(Y)$.
Since $\sigma$ is strictly convex, $P^{\times}=0$. For any
$\fom$-primary ideal $I$, $P\setminus I$ is a finite set. Let
$\eta:\NE(Y)\rightarrow P$ be the inclusion. Then the finiteness hypotheses
of the above Lemma hold for $J=\fom$ (note that
the conditions (\ref{betadef}) determine $\beta
\in A_1(\tilde Y)$ given $\pi_*(\beta)$).
\end{example}

\begin{example}
\label{M05more}
We return to the example $(Y,D)$ of a del Pezzo surface together with a cycle of $5$ $(-1)$-curves studied
in Example \ref{M05example}.
Let $P=\NE(Y)$ and $\eta$ be
the identity. Let $J=\fom\subset P$ and $I\subset P$ be an ideal with
$\sqrt{I}=J$. Then $\foD^{\can}$ consists of five rays:
\[
\foD^{\can}=\{(\rho_i,1+z^{[E_i]-\varphi_{\rho_i}(v_i)})\,|\,1\le i\le 5\}.
\]
Here $E_i$ is the unique $(-1)$-curve in $Y$ which is not contained in $D$ and meets $D_i$
transversally, and $v_i$ is the primitive generator of the ray
$\rho_i$ corresponding to $D_i$.
To derive this formula from the above definition of the canonical
scattering diagram one needs to show that the only possible
stable relative maps contributing to $\foD^{\can}$ are multiple covers
of the $E_i$'s, and that a $k$-fold multiple cover contributes
a Gromov-Witten invariant of $(-1)^{k-1}/k^2$. It is easier
to compute this using the main result of \cite{GPS09}, which is done
by way of Theorem \ref{relatingcanandscatter}. See Example \ref{M05stillmore}.

If we accept this description of $\foD^{\can}$, then we can describe
all broken lines and the multiplication law given by this diagram.

We first note that no broken line can wrap around $0 \in B$, i.e., if
a broken line leaves a cone $\sigma \in \Sigma_{\max}$, it will never return
to that cone.
It is enough to check this for a straight
line (as the bending in any broken line for $\foD_{\can}$
is always \emph{away} from
the origin), and this is easily verified, using e.g.,
Figure \ref{M05affinemanifold}.

Next, since the only scattering rays are the rays $\rho\in\Sigma$,
if $q,Q \in \sigma \in \Sigma_{\max}$,
then the obvious straight line is the unique broken line for $q$
with endpoint $Q$.
Thus if we describe $\vartheta_q$ in the open subset of $X_{I,\foD^{\can}}^o$
corresponding to $\sigma$, $\vartheta_q$
is just the monomial $z^{\varphi_{\sigma}(q)}$.
It follows that
\[
\vartheta_{v_i}^{a} \vartheta_{v_{i+1}}^b = \vartheta_{av_i + b v_{i+1}}
\]
for $a,b \geq 0$.
In particular, the $\vartheta_{v_i}$'s generate the $\kk[P]/I$-algebra
$\Gamma(X_I,\shO_{X_I})$, and the
algebra structure is determined once we compute
$\vartheta_{v_i} \cdot \vartheta_{v_{i+2}}$.

\begin{figure}
\input{M05newfigure.pstex_t}
\caption{The different types of broken lines in Example \ref{M05more}.}
\label{M05newfigure}
\end{figure}

We consider a broken line for $v_i$.
One checks the following, using Figure \ref{M05affinemanifold} and
the above description of $\foD^{\can}$:
The broken line
can cross at most two rays of $\Sigma$, and it bends at most
once, at the last ray of $\Sigma$ that it crosses.
See Figure \ref{M05newfigure}.
>From this one deduces using Theorem \ref{multrule}:
\begin{equation}
\label{eqmo5}
\vartheta_{v_{i-1}} \vartheta_{v_{i+1}} = z^{[D_i]}(\vartheta_{v_i}+z^{[E_i]}).
\end{equation}
The term $z^{[D_i]} \cdot \vartheta_{v_i}$ corresponds to two straight
broken lines for $v_{i-1}, v_{i+1}$, with endpoint the point $v_i$ of
$\rho_i$. The term $z^{[D_i]} \cdot z^{[E_i]}$ is the coefficient
of $1 = \vartheta_0$. To compute this we use the invariance of broken
lines, and so choose a generic point $Q$ near $0$ and compute
the coefficient
$\alpha_0$ of $\vartheta_0$ using pairs $\gamma_i$ as in
Theorem \ref{multrule} whose
final directions are opposite, i.e., $s(\gamma_1)+s(\gamma_2) = 0$.
If we take $Q \in \sigma_{i,i+1}$,
then there is exactly one term contributing to $\alpha_0$:
$\gamma_1$ will bend once where it crosses $\rho_i$, and
$\gamma_2$ is straight. Alternatively, one can use the explicit expressions
for $\Lift_Q(v_j)$, $j=i-1, i$ and $i+1$, and see they satisfy the
relation \eqref{eqmo5}.

One can check that the five equations \eqref{eqmo5}
define $X_I$. These equations are algebraic, and in fact define
a flat family over $\Spec\kk[\NE(Y)]$. (This is always the case
in the non-negative semi-definite case, see Corollary
\ref{polymultpositive}).
\end{example}

Our goal now is to prove consistency of $\foD^{\can}$, as stated
in the following (the final step in the construction of our mirror family):

\begin{theorem}
\label{canonicalconsistency}
Suppose that we are given a map $\eta:\NE(Y)\rightarrow P$ such that
$\varphi$ is defined as in Example \ref{standardphi} by
$\kappa_{\rho,\varphi}=\eta([D_{\rho}])$. Suppose furthermore
the following conditions hold:
\begin{itemize}
\item[(I)] For any $\AA^1$-class $\beta$, $\eta(\pi_*(\beta))\in  J$;
\item[(II)] For any ideal $I$ with $\sqrt{I}=J$, there are only a finite
number of $\AA^1$-classes $\beta$ such that $\eta(\pi_*(\beta))\not
\in I$.
\item[(III)] $\eta([D_i])\in J$ for at least one boundary component
$D_i\subset D$.
\end{itemize}
Then $\foD^{\can}$ is a consistent scattering diagram.
\end{theorem}

We include here an observation we will need later showing that
the canonical scattering diagram only depends on the deformation
class of $(Y,D)$.

\begin{lemma}
\label{definvlemma}
Let $(\cY,\cD)\rightarrow S$ be a flat family of pairs over a connected base $S$, with
each fibre $(\cY_s,\cD_s)$ being a non-singular
rational surface with anti-canonical
cycle. Suppose further that
there is a trivialization $\cD\cong D\times S$ and the restriction map
$\Pic(\cY)\rightarrow \Pic(\cY_s)$ is an isomorphism for any $s\in S$. This
in particular gives a canonical identification $A_1(\cY_s,\ZZ)$ with
$A_1(\cY_{s'},\ZZ)$ for any $s,s'\in S$.
Then for any
$s,s'\in S$, $(\cY_s,\cD_s)$ and $(\cY_{s'},\cD_{s'})$ induce
the same canonical scattering diagram.
\end{lemma}

\begin{proof}
It is enough to show that the numbers $N_{\beta}$ are deformation
invariants in the above sense, i.e., if we are given a family
$\pi:(\tilde\cY, \tilde\cD)\rightarrow S$ with each fibre as in
Definition \ref{basicrelinvs}, with an irreducible component
$\cC\subset \tilde \cD$, then the number
\[
N_{\beta,s}:=\int_{[\overline\foM(\tilde\cY^o_s/\cC^o_s,\beta)]^{vir}} 1
\]
is independent of $s$. Indeed, once this is shown, then if
$N_{\beta,s}\not=0$, necessarily $\beta$ defines a class
in $\NE(\cY_s)$, as well as in $\NE(\cY_{s'})$, under the chosen
identification. This invariance follows from the standard argument that
(relative) Gromov-Witten invariants are deformation invariants, with
a little care because our target spaces are open.
For this, one considers the moduli space
$\overline\foM(\tilde\cY^o/\cC^o,\beta)$ of stable
maps to $\tilde\cY^o$ relative to $\cC^o$ and whose composition
with $\pi$ is constant. Then one has a map
$\psi:\overline\foM(\tilde\cY^o/\cC^o,\beta)\rightarrow S$
whose fibre over $s$ is $\overline\foM(\tilde\cY^o_s/\cC^o_s,\beta)$.
Letting $\xi$ be the inclusion of this fibre in
$\overline\foM(\tilde\cY^o/\cC^o,\beta)$, deformation invariance
will follow if we know that
\[
\xi^![\overline\foM(\tilde\cY^o/\cC^o,\beta)]^{vir}=
[\overline\foM(\tilde\cY^o_s/\cC^o_s,\beta)]^{vir}
\]
and $\psi$ is proper. The first statement is standard in Gromov-Witten
theory, see e.g.\ the argument of Theorem 4.2 of \cite{LT98} in
the non-relative case, which carries over to the relative case.
The second point,
the properness of $\psi$, follows exactly as in the proof of Lemma
\ref{GPSinvariancelemma}.
\end{proof}

\subsection{Consistency: Overview of the proof}
\label{proofoverview}
We will describe in detail the intuition behind each step of the proof
of consistency. In the next subsection, we will work somewhat more
generally with a more general scattering diagram $\foD$ for certain
steps, as this will be needed in \cite{K3} for the K3 case. However, for
the discussion here let us assume we are only studying the consistency
of the canonical scattering diagram.

\emph{Step I}. \emph{We can replace $(Y,D)$ with a toric blow-up of
$(Y,D)$.} This is straightforward --- toric blowups just correspond to
refinements of $\Sigma$, but do not change broken lines or scattering
diagrams.

\emph{Step II}. \emph{We can assume that $(Y,D)$ has a toric model
and $P$ is a finitely generated
submonoid of $A_1(Y,\ZZ)$ containing $\NE(Y)$, with
$\eta$ the inclusion}. By Step I and Proposition \ref{toricmodelexists},
we can assume $(Y,D)$ has a toric model. We can then always factor
$\eta$ as $\NE(Y)\mapright{\bar\eta}\oP\mapright{\psi} P$
where $\oP$ is a finitely generated submonoid of $A_1(Y,\ZZ)$ containing
$\NE(Y)$ with $\bar\eta$ the inclusion. In this case there are
two canonical scattering diagrams, $\ofoD$ and $\foD$ defined using
$\bar\eta:\NE(Y)\rightarrow\oP$ and $\eta:\NE(Y)\rightarrow P$ respectively.
Then $\foD$ can be obtained from $\ofoD$ essentially just by applying $\psi$
to each exponent appearing in each function $f_{\fod}$.

In this case we show that if consistency holds for $\ofoD$ then it holds
for $\foD$. The idea is that given a broken line $\gamma$ for $\ofoD$,
we can get something like a broken line for $\foD$ by applying $\psi$
to the exponents of monomials attached to $\gamma$. However, this isn't
necessarily a broken line for $\foD$. Indeed, there might be two different
broken lines for $\ofoD$, say $\gamma_1$ and $\gamma_2$, which after
we apply $\psi$ give broken lines with the same sequence of attached
exponents. These should not arise as distinct broken lines for
$\foD$, and we have to combine the monomials attached to
these broken lines. This requires a certain amount of book-keeping.

\emph{Step III. Reduction to the Gross--Siebert locus.}
By Step II we can assume we have a
toric model $p:Y\rightarrow\oY$. Let $H$ be an ample divisor on
$\oY$. Shrinking $P$ if necessary, we can assume that $P$ has a face
of the form $P\cap (p^*H)^{\perp}$.
Let $G$ be the monomial ideal which is the complement of this
face, $E$ the subgroup of $P^{\gp}$ generated by $P\setminus G$.
The main work in this step is to show that we can replace $P$ by $P+E$.
This requires a bit of analysis of the rays $(\rho_i,f_{\rho_i})$
of $\foD^{\can}$. In particular, we need to understand the contribution
to $f_{\rho_i}$ coming from the exceptional curves of $p$ meeting $D_i$.

After doing this, we have $P^{\times}=E$, so now $X^o_{I,\foD}$ lives over
the thickening of a torus $T^{\gs}$ we call the Gross-Siebert locus.

\emph{Step IV}. \emph{Pushing the singularities to infinity}. This is the
crucial step, and we explain carefully the intuition here. In \cite{GS07},
Gross and Siebert considered a smoothing construction associated to an
integral affine manifold with singularities where (in the two-dimensional
case) the singularities occurred only in the interior of edges of a
polyhedral decomposition of $B$, rather than at the vertices. The case
at hand, with one singularity at the origin, does not fit into that
framework.
In particular, in the Gross-Siebert world, the singularities must have
monodromy of the
form $\begin{pmatrix} 1&k\\0&1\end{pmatrix}$ for some $k>0$, with the tangent
line to the edge containing the singularity being the invariant direction;
in analogy with
the Kodaira classification, we call this an $I_k$ singularity. Indeed, one
expects a cycle of $k$ two-spheres as fibre over such a point in the SYZ
picture.

Here, we can view such a surface as being obtained by \emph{factoring}
the complicated singularity $0 \in B$ into $I_k$ singularities
along the edges of $\Sigma$. We should have an $I_{k_i}$ singularity
on the ray $\rho_i$ where $k_i$ is the number of exceptional
divisors of $p: Y \to \oY$ intersecting $D_i$. This process can be
described as follows.
Let $(\oB,\oSigma)$ be the fan associated to $(\oY,\oD)$.
There is a piecewise linear isomorphism
\[
\nu: B \to \oB
\]
which identifies each cone in $\Sigma$ with the corresponding cone
in $\oSigma$. This is an isomorphism of integral affine manifolds
outside of $\rho_i$, but it is not
affine along $\rho_i$. There is a natural one-parameter family of
integral affine manifolds
interpolating between the two structures by a process Kontsevich and Soibelman
\cite{KS06} call
{\it moving worms}.  Precisely, choose points $y_i \in \rho_i\setminus\{0\}$.
Let $\Delta:=\{y_i\,|\,1\le i\le n\}$, $B_0':=B\setminus\Delta$.
Put a new affine structure on $B_0'$ compatible with the
affine structures on the interior of each maximal cell by defining a
$\Sigma$-piecewise linear function to be linear if its
restriction to a small neighbourhood of $(y_i,+\infty) \subset \rho_i$ in
$B_0'$ is $B$-linear, and its restriction to a small neighbourhood of
$[0,y_i)\subset\rho_i$ in $B_0'$ is $\oB$-linear.
Call the resulting integral affine manifold with singularities $B'$.
The map $\nu:B' \to \oB$ is a linear isomorphism near $0$
This new manifold can be seen to have
an $I_{k_i}$ singularity at $y_i$, with invariant direction $\rho_i$.

Now if we were to apply the algorithm of Gross and Siebert \cite{GS07}
to $B'$, one would find roughly that one obtains a scattering diagram
which initially has two rays emanating from each singularity. The rays
emanating from $\rho_i$ are initially contained in $\rho_i$; one of these
goes out to infinity and the other passes through the origin and then
to infinity. Where all these rays meet at the origin, one must
follow a procedure of Kontsevich and Soibelman \cite{KS06} and add
some additional rays to ensure that the composition of automorphisms
associated to the rays about a loop centered at the origin is the identity.
We then obtain a scattering diagram which can be shown is very close to
the canonical scattering diagram, the only difference being the segments of
the rays between the $y_i$ and the origin.

We do not actually work with this affine manifold with singularities.
Rather, we instead push the singularities $y_i$ to infinity. In doing so, we
replace $B'$ with $\oB$. We transfer
the canonical scattering diagram $\foD$ to a scattering diagram $\ofoD$
on $\oB$, differing from $\foD$ essentially only by changing the rays
supported on the $\rho_i$'s in a simple way motivated by the above description.
Once this is done, we show consistency of $\foD$ is equivalent to
consistency of $\ofoD$. Now we no longer have to deal with any singularities.

It is much easier to determine consistency when there are no singularities.
In particular, we appeal to a result in \cite{CPS}, which shows that
$\ofoD$ is consistent
provided that the composition of automorphisms
associated to the rays about a loop centered at the origin is the identity.
We say such a scattering diagram is \emph{compatible}.
The important point is that we can now make sense of such a statement: when
we had a singularity at the origin, there was no common ring which the
automorphisms associated to rays could act on. However, without a singularity
at the origin, there are such rings, as appeared in \cite{GS07}.

\emph{Step V}. \emph{$\ofoD$ satisfies the required compatibility
condition}.
This step is really the punch-line, explaining why the particular choice
of the canonical scattering diagram $\foD$ gives a diagram $\ofoD$ which is compatible. We make use of \cite{GPS09} to link the enumerative
definition of $\foD$ to the notion of compatibility. Indeed, the definition
of the canonical scattering diagram was originally obtained by working
backwards from the enumerative description of \cite{GPS09}. This connection
is worked out in \S \ref{connectionwithGPS}.

\subsection{Consistency: Reduction to the Gross--Siebert locus}
\label{GSreductionsection}

We now begin the proof of Theorem \ref{canonicalconsistency}, following
the outline given in \S\ref{proofoverview}.
We will, however, prove a number of lemmas in a slightly more general
context, as we will need some more general consistency results
in \cite{K3}. We assume we are given $(Y,D)$, $\eta:\NE(Y)\rightarrow P$
and $\varphi$ defined as in Example \ref{standardphi}, and a radical
ideal $J\subseteq P$. Suppose
we are given a scattering diagram $\foD$ for this data; the application
in this paper will be $\foD=\foD^{\can}$. In particular, the hypotheses
of Theorem \ref{canonicalconsistency} imply $\foD^{\can}$ is a scattering
diagram for this data.

\emph{Step I}. \emph{Replacing $(Y,D)$ with a toric blowup.}

\begin{proposition} \label{reducetotoricmodel}
Let $p: (\tY,\tD) \to (Y,D)$ be a toric
blowup. Then if we take $\tilde\eta:=
\eta\circ p_*:\NE(\tY)\rightarrow P$, then $\foD$ can also
be viewed as a scattering diagram for $B_{(\tY,\tD)},P$.
Furthermore, if $\foD$ is consistent for this latter data, it
is consistent for the data $B_{(Y,D)},P$.
\end{proposition}

\begin{proof}
Decorate notation,
writing for example $\tB, \tSigma$
for the singular affine manifold with subdivision into cones associated to $(\tY,\tD)$.
By Lemma \ref{toricblowuplemma}, we have a
canonical identification of the underlying
singular affine manifolds $B = \tB$, and $\tSigma$ is the refinement
of $\Sigma$ obtained by
adding one ray for each $p$-exceptional
divisor. We have multivalued piecewise linear functions $\varphi$
on $B$ and $\tvarphi$ on $\tilde B$. We can in fact choose representatives
so that $\tvarphi=\varphi$. Indeed,
$\kappa_{\rho,\tvarphi}=\eta(p_*([\tD_{\rho}]))$
where $\tD_{\rho}$ is the irreducible component of $\tD$ corresponding to
$\rho$. But $p_*([\tD_{\rho}])=0$ if $\rho\not\in\Sigma$, and $p_*([\tD_{\rho}])
=[D_{\rho}]$ if $\rho\in\Sigma$. Thus $\tvarphi$ in fact has the same
domains of linearity as $\varphi$, and the same bending parameters,
so we can choose representatives which agree.

As a consequence, we note that the sheaves $\shP$ and $\tilde\shP$ on $B_0$
defined using $\varphi$ and $\tilde\varphi$ coincide.
Furthermore,
if $\tilde\tau\subset\tilde\sigma$ are cones in $\tilde\Sigma$,
with $\tau\in\Sigma$ the smallest cone containing $\tilde\tau$ and
$\sigma\in\Sigma$ the smallest cone containing $\tilde\sigma$, there
is a canonical identification of $P_{\varphi_\tau}$ with $P_{\tvarphi_{\tilde
\tau}}$ and a canonical isomorphism
\begin{equation}
\label{tBBcaniso}
\tilde R_{\tilde\tau,I}\cong R_{\tau,I};
\end{equation}
note the slightly
non-trivial case when $\dim\tilde\tau=1$ but $\dim\tau=2$,
in which case we use the fact that $\kappa_{\tilde\tau,\tilde\varphi}=0$.

Using these identifications, we can view $\foD$ as living on $\tilde B$,
and as such, one sees from the definition that $\foD$ is a scattering
diagram for the data $\tilde B,P,\tilde\varphi$.

Now suppose $\sqrt{I}=J$.
One observes that the set of broken lines contributing
to $\Lift_Q(q)$ are the same whether we are working in $B$ or $\tilde B$.
Thus if $Q\in\tilde\sigma\in\tilde\Sigma_{\max}$,
$\Lift_Q(q)\in\tilde R_{\tilde\sigma,I}$, defined using
$\tilde B$, coincides under the isomorphism \eqref{tBBcaniso} with
$\Lift_Q(q)\in R_{\sigma,I}$. From this one sees easily
that if $\foD$ is consistent for $\tY$, it is consistent for $Y$.
\end{proof}

\begin{corollary}
\label{reducetotoricmodelcor}
Given $Y, P, \eta, J$ satisfying the hypotheses of Theorem
\ref{canonicalconsistency}, then Theorem \ref{canonicalconsistency}
holds for this data if it holds for the data $\tY,P,\tilde\eta,J$.
\end{corollary}

\begin{proof}
By the proposition, one just needs to check that the
canonical scattering diagrams defined using $Y$ or $\tY$ are identical.
Indeed, given a
ray $\fod\subset B$, we can choose a refinement $\Sigma'$ of $\Sigma$
which is also a refinement of $\tilde\Sigma$, giving maps $\tilde\pi:
Y'\rightarrow\tY$ and $\pi:Y'\rightarrow Y$. Then for an $\AA^1$-class
$\beta\in A_1(Y',\ZZ)$, $\eta(\pi_*(\beta))=\tilde\eta(\tilde\pi_*(\beta))$,
and so $f_{\fod}$ is the same for $\tY$ and $Y$.
\end{proof}

\emph{Step II}. \emph{Changing the monoid $P$}.

We would like to change the monoid $P$, which was fairly arbitrary,
to one with better properties. For this
step, assume we are given monoid homomorphisms
\[
\NE(Y)\mapright{\bar\eta} \oP \mapright{\psi} P
\]
with $\eta=\psi\circ\bar\eta$. Then $\eta$ and $\bar\eta$
induce multivalued piecewise linear functions $\varphi$ and $\bar\varphi$
respectively, via Example \ref{standardphi}, with $\varphi=\psi\circ\bar
\varphi$.
The monoids $P$,$\oP$ and functions $\varphi,\bar\varphi$ yield sheaves $
\shP$ and $\oshP$ over $B_0$.
The map $\psi:\oP\rightarrow P$
induces a map of sheaves $\psi:\oshP\rightarrow\shP$ using
$\varphi=\psi\circ\bar{\varphi}$, and hence it also induces monoid
homomorphisms $\psi:\oP_{\bar\varphi_{\tau}}\rightarrow P_{\varphi_{\tau}}$
for any $\tau\in\Sigma\setminus\{0\}$.

Suppose $\ofoD$ is a scattering
diagram for the data $B,\oP,\bar\fom=\oP\setminus\oP^{\times}$.
For each ray $(\fod,f_{\fod})$, $f_{\fod}\in
\widehat{\kk[\oP_{\bar\varphi_{\tau_{\fod}}}]}$. Now we can try to define
$\psi(f_{\fod})$ by applying $\psi$ to each exponent of $f_{\fod}$, but
in general, this need not make sense even formally since $\psi$ may take
an infinite number of exponents occuring in $f_{\fod}$ to a single
element of $P$. However, we shall write $\psi(f_{\fod})$ for such
an expression if it does make sense as an element of
$\widehat{\kk[P_{\varphi_{\tau_{\fod}}}]}$. If $\psi(f_{\fod})$ makes
sense for each $(\fod,f_{\fod})\in \ofoD$, we write
\[
\psi(\ofoD)=\{(\fod,\psi(f_{\fod}))\,| \, (\fod,f_{\fod})\in \ofoD\}.
\]

\begin{proposition}
\label{barPtoPprop}
In the above situation,
suppose $\ofoD$ is a scattering
diagram for the data $B,\oP,\bar\fom=\oP\setminus\oP^{\times}$,
such that $\foD=\psi(\ofoD)$ makes sense and is a scattering
diagram for the data $B,P,J$, where $J$ is a radical ideal in $P$.
Assume that $\kappa_{\rho,\varphi}\in J$ for at least one ray $\rho\in\Sigma$.
If $\ofoD$ is consistent for $\oP,\bar\eta,\bar\fom$, then
$\foD$ is consistent for $P,\eta,J$.
\end{proposition}

\begin{proof}
Let $q\in B_0(\ZZ)$. Then if $\bar{\gamma}$ is a broken line for $q$
with endpoint $Q$ with respect to the barred data, i.e., $\oP$, $\oshP$ etc.,
we can construct what we shall call $\psi(\bar\gamma)$. This will
be the data required for defining a broken line for the unbarred
data. The underlying map of $\psi(\bar\gamma)$ coincides with that
of $\bar\gamma$. For the attached monomials,
we simply apply $\psi$ to the monomial $m_L(\bar\gamma)$ attached to
a domain of linearity $L$ of $\bar\gamma$ to get the attached monomial
for $\psi(\bar\gamma)$. This is not a broken line for the unbarred
data, as condition (3) of Definition \ref{brokenlinedef} need not
hold. Indeed, when a broken line bends at a ray, the attached monomial
will be replaced by a term in \eqref{bendingterms}. However, there might
be several different terms $c_iz^{\bar s_i}$ appearing in \eqref{bendingterms}
with $\bar s_i\in \bar P_{\varphi_{\tau}}$ such that $\psi(\bar s_i)$
all coincide with some $s\in P_{\varphi_{\tau}}$. Each choice $c_iz^{\bar s_i}$
leads
to a different broken line $\bar\gamma_i$, but $\psi(\bar\gamma_i)$
is not a broken line because $c_iz^{\psi(\bar s_i)}=c_iz^s$ is not a term
in the formula \eqref{bendingterms} for the monoid $P$. Rather, one needs
to replace the collection of broken lines $\bar\gamma_i$ with a single one
which has monomial $\sum_i c_iz^s$ attached after the bend. To deal with
this, we need to do a certain amount of book-keeping.

Fix an ideal $I\subset P$ with $\sqrt{I}=J$, $Q\in\sigma\in\Sigma$,
and let $\bar\foB$ be the set of broken lines $\bar\gamma$ for the barred data
with endpoint $Q$ such that $\psi(\Mono(\bar\gamma))\not\in I\cdot
\kk[P_{\varphi_\sigma}]$. The same finiteness argument of
Lemma \ref{finitenesslemma} shows that $\bar\foB$ is a finite set. Note this
uses the facts (1) at least one $\kappa_{\rho,\varphi}\in J$ and (2) all but
a finite number of monomials appearing in $\foD$ lie in $I$.

We define an equivalence relation on $\bar\foB$ by saying
$\bar\gamma_1\sim\bar\gamma_2$ provided $\psi(\bar\gamma_1)$ and
$\psi(\bar\gamma_2)$ coincide except possibly for the $\kk$-valued
coefficients of the monomials attached to the domains of linearity. Given an
equivalence class $\xi\subset\bar\foB$
with respect to this equivalence relation,
we will show there is at most one broken line $\gamma_{\xi}$
for the unbarred data such that
\begin{equation}
\label{sumbareq}
\sum_{\bar\gamma\in\xi} \psi(\Mono(\bar\gamma))
=\Mono(\gamma_{\xi}),
\end{equation}
with there being no such broken line precisely if the above quantity is zero.
Furthermore, every broken line $\gamma$ for the unbarred data
with $\Mono(\gamma)\not\in I\cdot\kk[P_{\varphi_{\sigma}}]$
arises in this way.

Define $\gamma_{\xi}$ to be the broken line with underlying piecewise
linear map
given by any element of $\xi$, with the following attached monomials.
For any domain of linearity $L=[s,t]$ for $\gamma_{\xi}$, choose a maximal
subset $\xi_L\subset\xi$ of broken lines such that the attached
monomials for $\bar\gamma_1$ and $\bar\gamma_2$ on $(-\infty,t]$ do not
coincide for any $\bar\gamma_1,\bar\gamma_2
\in\xi_L$. Then define
\[
m_L(\gamma_{\xi})=\sum_{\bar\gamma\in\xi_L} m_L(\psi(\bar\gamma)).
\]
Assuming that the final monomial attached to $\gamma_{\xi}$
is not zero, one checks easily that
$\gamma_{\xi}$ is a broken line, now satisfying (3) of Definition
\ref{brokenlinedef}, and \eqref{sumbareq} is satisfied since for
$L$ the last domain of linearity of $\gamma_{\xi}$, one takes
$\xi_L=\xi$. Furthermore, it is easy to see that any broken line
for the unbarred data with the same underlying map and attached monomials
at most differing by their coefficients from $\gamma_{\xi}$ must in fact
coincide with $\gamma_{\xi}$. This shows the claim.

Since
$\bar\foB$ is finite,
there is some $k>0$ such that for any $\bar\gamma\in\bar\foB$,
$\Mono(\bar\gamma)\in \kk[\oP_{\ovarphi_{\tau}}]$ does not lie in
$\bar\fom^k\cdot \kk[\oP_{\ovarphi_{\tau}}]$. If we take
$\bar I=\psi^{-1}(I)$, then it is clear from \eqref{sumbareq} that
\begin{equation}
\label{psieq}
\psi(\overline{\Lift}_Q(q))=\Lift_Q(q),
\end{equation}
where $\overline{\Lift}_Q(q)$ is the lift defined with respect to
the ideal $\oI$ and the other barred data, and $\Lift_Q(q)$
is defined with respect to the unbarred data and the ideal $I$. Now
$\ofoD$ is consistent for $\bar\fom$, which implies (1) and (2) of
Definition \ref{consistentdef} hold for the ideal $\bar\fom^k+\oI$.
Since any monomial in $\oP\setminus\oI$ appearing in $\overline{\Lift}_Q(q)$
is in
$\oP\setminus (\bar\fom^k+\oI)$, we can use
\eqref{psieq} to deduce consistency of $\foD$ from consistency of
$\ofoD$.
\end{proof}

\emph{Step III. Reduction to the Gross-Siebert locus.}
As a consequence of Proposition \ref{toricmodelexists} and Corollary
\ref{reducetotoricmodelcor}, in order
to prove Theorem \ref{canonicalconsistency} (i.e., with $\foD=\foD^{\can}$),
we may assume we have a toric model $p: (Y,D) \to (\oY,\oD)$
with $\oD=\oD_1+\cdots+\oD_n$. Furthermore, by replacing
$(Y,D)$ with a deformation equivalent pair and using
Lemma \ref{definvlemma}, we can assume that
$p$ is the blowup at distinct points $x_{ij}$, $1\le j\le \ell_i$,
along $\oD_i$, with exceptional divisors $E_{ij}$. Assume $D_i$ is the
proper transform of $\oD_i$, corresponding to the ray $\rho_i\in\Sigma$.

By Proposition
\ref{barPtoPprop}, we can replace $P$ with a better suited choice of
monoid. We shall do this as follows in the case that
$\foD=\foD^{\can}$. As in Example \ref{goodPexample},
the nef cone $\ocK(Y) \subset A^1(Y,\RR)$ contains a strictly convex rational polyhedral
cone $\sigma$, so $\sigma^{\vee}\subset A_1(Y,\RR)$ is a
strictly convex rational polyhedral cone containing $\NE(Y)$.
The map $\eta:\NE(Y)\rightarrow P$ induces a map $\eta:A_1(Y,\RR)
\rightarrow P^{\gp}_\RR$. Since $P$ is toric, there is some
rational polyhedral cone $\sigma_P\subset P^{\gp}_\RR$
such that $P=\sigma_P\cap P^{\gp}$. In addition, let $H$ be an
ample divisor on $\oY$, so that $\NE(Y)\cap (p^*H)^{\perp}$
is a face of $\NE(Y)$, generated by the classes $[E_{ij}]$. Now take
\[
\sigma_{\oP}=\eta^{-1}(\sigma_P)\cap\sigma^{\vee}\cap \{q\in A_1(Y,\RR)\,|\,
p^*H\cdot q\ge 0\},
\]
and take
\[
\oP=\sigma_{\oP} \cap A_1(Y,\ZZ).
\]

As $\sigma_{\oP}$ is strictly convex,
$(\oP)^{\times}=\{0\}$, $\bar{\fom}=\oP\setminus \{0\}$, and if $\oI$ is
an $\bar{\fom}$-primary ideal, $\oP\setminus \oI$
is finite. Thus the hypotheses
of Theorem \ref{canonicalconsistency} trivially hold for $\bar{\eta}:\NE(Y)
\rightarrow \oP$.
By Proposition \ref{barPtoPprop}, we
can replace $P$ with $\oP$ to prove Theorem \ref{canonicalconsistency}.

The above discussion
shows that in order to complete a proof of consistency of $\foD^{\can}$,
(i.e., Theorem \ref{canonicalconsistency}),
we can operate under the following assumptions:

\begin{assumptions}
\label{GSassumptions}
\begin{itemize}
\item There is a toric model
\[
p:(Y,D)\rightarrow (\oY,\oD)
\]
which blows up distinct points $x_{ij}$ on $D_i$, with exceptional
divisors $E_{ij}$.
\item $\eta:\NE(Y)\rightarrow P$ is an inclusion, and $P^{\times}=\{0\}$.
Via Example \ref{standardphi}, this gives the function $\varphi$.
\item There is a face of $P$ whose intersection with $\NE(Y)$
is $\NE(Y)\cap (p^*H)^{\perp}$. Let $G$ be the prime monomial ideal
given by the complement of this face.
Note that $G\not=\fom$ unless $p$ is an isomorphism.
\item $J=\fom=P\setminus \{0\}$.
\item $\foD$ is a scattering diagram for the data $P, \varphi$ and $J$.
\end{itemize}
\end{assumptions}

\begin{definition}
\label{GSlocusdef}
The \emph{Gross--Siebert locus} is the open torus orbit $T^{\gs}$
of $\Spec\kk[P]/G$.
\end{definition}

We now want to work not with the maximal ideal $\fom$ but with the ideal
$G$, effectively extending the families $X_{I,\foD}^o$ with $\sqrt{I}=\fom$
to infinitesimal neighbourhoods
of the toric boundary stratum of $\Spec\kk[P]$ associated to $G$. We
will then find it easier to check
the explicit equalities of Definition \ref{consistentdef}
after restricting to these neighbourhoods of the Gross--Siebert
locus. To do so requires showing that the diagram $\foD$ we are working
with ($\foD^{\can}$ in this paper) is also a scattering diagram for the
data $P,\eta,G$. In the case of $\foD^{\can}$, this requires analyzing
elements of this scattering diagram supported on the $\rho_i$.

We first perform this analysis for $\foD^{\can}$;
we will then continue our proof assuming that the elements of
$\foD$ supported on the rays $\rho_i$ take the same form as the corresponding
elements of $\foD^{\can}$ modulo $G$.

For each ray $\rho_i$ in $\Sigma$,
we have a unique ray $(\rho_i,f_{\rho_i})\in\foD^{\can}$
with support $\rho_i$. The following describes $f_{\rho_i}\bmod G$.

\begin{lemma} Given Assumption \ref{GSassumptions} with $\foD=\foD^{\can}$,
viewing $f_{\rho_i}$ as an element of $\kk[P_{\varphi_{\rho_i}}]\otimes
R_I$ with $\sqrt{I}=\fom$,
we have
\label{grhoilemma}
\[
f_{\rho_i}=g_{\rho_i}\prod_{j=1}^{\ell_i}(1+ b_{ij} X_i^{-1})
\]
where $b_{ij} = z^{\eta([E_{ij}])}$ and $g_{\rho_i}\equiv 1
\mod G$. The $j^{th}$ term of the product is the contribution from
$\bA^1$-classes coming from multiple covers of the
$p$-exceptional divisor $E_{ij}$, and
$g_{\rho_i}$ is the product of contributions from all other $\bA^1$-classes.
\end{lemma}

\begin{proof}
Note that in defining $f_{\rho_i}$ using the definition of the canonical
scattering diagram, we take $\tilde Y=Y$. Now the only
terms that contribute to $f_{\rho_i}\bmod G$ will involve classes
$\beta\in \NE(Y)\subset A_1(Y)$ with $\eta(\beta)\not\in G$,
so in particular, such a $\beta$ must be a linear combination
$\sum_{j=1}^{\ell_i} c_j [E_{ij}]$, with $k_{\beta}=\sum c_j$. Furthermore,
if $f:C\rightarrow Y$ contributes to $N_{\beta}$,
$f(C)$ must be contained in $\bigcup_{i,j}
E_{ij}$. Indeed, if $f(C)$ has an irreducible component $C'$ not contained
in this set, then $\eta([C'])\in G$, so $\eta(f_*([C]))\in G$, as $G$ is
an ideal. But $\eta(f_*([C]))=\eta(\beta)$, which we have assumed is not
an element of $G$.

Since $f(C)$ is connected and intersects $D_i$, we now see that the
image of $f$ is $E_{ij}$ for some $j$, and in particular, $f$ is
a degree $k_{\beta}$ cover of $E_{ij}$. Then Theorem 6.1 of \cite{GPS09}
tells us that the contribution from $k_{\beta}$-fold multiple covers of
$E_{ij}$ is $(-1)^{k_{\beta}-1}/k_{\beta}^2$. From this we conclude that
\begin{align*}
f_{\rho_i}={} & \exp\left(h+\sum_{k=1}^{\infty}\sum_{j=1}^{\ell_i}
k\left((-1)^{k-1}\over k^2\right)
(b_{ij}X_i^{-1})^{k}\right)\\
= {} & \exp(h)\prod_{j=1}^{\ell_i} (1+b_{ij} X_i^{-1})
\end{align*}
where $h\equiv 0\mod G$. We take $g_{\rho_i}=\exp(h)$.
\end{proof}

\begin{corollary} $\foD^{\can}$ is a scattering diagram
for the data $(B,\Sigma)$, $P$, $\varphi$ and $G$.
\end{corollary}

\begin{proof}
Fixing an $I\subset P$
with $\sqrt{I}=G$, there exists a bound $n$ such that
$q\in P\setminus I$ implies $q\cdot p^* H<n$, where $H$
is a fixed ample divisor on $\oY$. Thus if $\beta$ is an $\AA^1$-class
with $\eta(\pi_*(\beta))\in P\setminus I$, there are only
a finite number of choices for $p_*\pi_*\beta$. We need to examine the
possible choices for $\pi_*\beta$.
Given a choice for $\alpha=p_*\pi_*\beta$, we have
$\pi_*\beta=p^*\alpha+\sum a_{ij} E_{ij}$ for some collection of
$a_{ij}\in\ZZ$. Clearly the $a_{ij}$ are bounded below by the requirement
that $(\pi_*\beta)\cdot D_i\ge 0$ for each $i$. On the other hand, if
$a_{ij}>0$ for some $i,j$, then $(\pi_*\beta)\cdot E_{ij}<0$, so if
$f\colon C\rightarrow Y$ is an $\AA^1$-curve with $f_*[C]=\pi_*\beta$ then
its reduced
image $C'$ must contain $E_{ij}$.
(Technically a relative stable map in
$\tilde Y^{\circ}$ is a map to an expanded degeneration of $\tilde Y^{\circ}$,
but we compose with the projection to $\tilde Y^{\circ}$ and then the
natural map $\tilde Y^{\circ}\rightarrow Y$.)
Write $C'=C''\cup E_{ij}$, with $C''$ a reduced divisor distinct from
$E_{ij}$. Suppose $C''$ is non-empty.
Necessarily either $C''\cap D$ is empty or $C''$ intersects
$D$ only at $E_{ij}\cap D$; otherwise $C'$ cannot be the image of a relative
stable map with one point of tangency with $D$. In either case there is
an integer $k$ such that $\shO_Y(C''+kE_{ij})|_D$ is the trivial sheaf.
However, by \cite{GHK12}, Proposition 4.1, for a general deformation
$(Y',D)$ of $(Y,D)$, the kernel of the restriction map $\Pic Y'
\rightarrow \Pic D$ is trivial. Thus by Lemma \ref{definvlemma}, $N_{\beta}=0$.
We conclude that there are only a finite number of choices of $\pi_*\beta$,
except when $\pi_*\beta$ is a multiple of some $E_{ij}$.
This shows condition (4) in Definition \ref{scatdiagdef} of scattering diagrams,
as well as condition (2).
Note that $\kappa_{\rho_i,\varphi} = [D_i] \in G$ for each $i$ so condition (3) is vacuous for $\dim\tau_{\fod}=1$.
If $\dim\tau_{\fod}=2$, any contributing $\AA^1$-class
$\beta$ satisfies $\pi_*\beta\in G$, so (3) holds.
\end{proof}

\begin{theorem}
\label{thetafunctions2}
We follow the above notation.
If $\foD^{\can}$ is consistent as a scattering diagram for
$(B,\Sigma)$, $P$, $\varphi$, and $G$, then
Theorem \ref{canonicalconsistency} is true.
\end{theorem}

\begin{proof} This just follows
from the series of reductions of Theorem \ref{canonicalconsistency} already
made and the observation that if $I'$ is an $\fom$-primary ideal,
then since $G\subset\fom$ one can find some $k$ such that $kG\subset I'$.
To show consistency holds for the ideal $I'$, we
use the assumed consistency to observe
consistency holds for the ideal $I=kG$, and this gives the desired result.
\end{proof}

\begin{remark}
\label{GSlocusrem}
Given a consistent scattering diagram $\foD$ for
$(B,\Sigma)$, $P$, $\varphi$, and $G$, and $\kappa_{\rho,\varphi}\in G$
for all rays $\rho\in\Sigma$,
Theorem \ref{univfamth}
shows that with $\sqrt{I}=G$,
\[
X_I:=\Spec \Gamma(X^o_{I,\foD},\shO_{X^o_{I,\foD}})
\]
is flat over $\Spec \kk[P]/I$, and $X_G=\VV_n\times \Spec \kk[P]/G$.

Let $T^{\gs} \subset \Spec \kk[P]/G$ be the Gross--Siebert locus, Definition
\ref{GSlocusdef}.
Note $T^{\gs}$ determines open subschemes of the thickenings
$\Spec \kk[P]/I$, which we will shall denote by $T^{\gs}_I$.

We can describe the subscheme $T^{\gs}_I$ of $\Spec\kk[P]/I$ as follows.
Let $E\subset P^{\gp}$ be the lattice generated by the face $P\setminus
G$. Then as a subset of $\Spec\kk[P]/G$, $T^{\gs}\cong \Spec
\kk[E]$. Furthermore, if we take the localization $P+E$ of
$P$ along the face $P\setminus G$, then $T^{\gs}_I$ as a subscheme of
$\Spec\kk[P]/I$ is $\Spec \kk[P+E]/(I+E)$.

Note that
$\fom_{P+E}=(P+E)\setminus E$, and $G=P\cap\fom_{P+E}$, so
we can write $\kk[E]=\kk[P+E]/\fom_{P+E}$.
\qed
\end{remark}

We can now view $\varphi$ as a multivalued
strictly $(P+E)$-convex function. Then we have the following obvious

\begin{lemma}
\label{PEreduction}
Suppose $\foD$
is a consistent scattering diagram for
the data $(B,\Sigma)$, $P+E$, $\varphi$, $\fom_{P+E}$, and
a scattering diagram for the data $(B,\Sigma)$, $P$, $\varphi$, and $G$.
Then $\foD$ is also consistent as a scattering diagram for the latter data.
In particular, by Theorem \ref{thetafunctions2},
Theorem \ref{canonicalconsistency}
holds if $\foD^{\can}$ is consistent as a scattering
diagram for $P+E$, $\fom_{P+E}$.
\end{lemma}

\begin{proof}
Since $P\cap \fom_{P+E}=G$, the equalities in Definition \ref{consistentdef}
can be tested for an ideal $I$ of $P$ with $\sqrt{I}=G$ by choosing
some ideal $I'\subseteq P+E$ with $\sqrt{I'}=\fom_{P+E}$ and $I'\cap P
\subset I$. Then the equalities of Definition \ref{consistentdef}
hold for the data $P+E$ and $I'$ by the assumed consistency, and hence
also for $P$ and $I$.
\end{proof}

Let $\oI\subset\fom_{P+E}$ be an ideal with $\sqrt{\oI}=\fom_{P+E}$.
Set $I=\oI\cap P$. Then $X^o_{I,\foD}$ is flat over
$\Spec\kk[P]/I$. Restricting $X^o_{I,\foD}$ to the open set $\Spec\kk[P+E]/\oI$
gives the flat family $X^o_{\oI,\foD}$.

We now replace $P$ by $P+E$ and $J$ by $\fom_{P+E}$ in what follows. We now
summarize our current situation with the following assumptions:

\begin{assumptions}
\label{GSassumptions2}
\begin{itemize}
\item There is a toric model
\[
p:(Y,D)\rightarrow (\oY,\oD)
\]
which blows up distinct points $x_{ij}$ on $D_i$, $1\le j\le\ell_i$,
with exceptional
divisors $E_{ij}$.
\item $\eta:\NE(Y)\rightarrow P$ is an inclusion.
Via Example \ref{standardphi}, this gives the function $\varphi$.
$E=P^{\times}=
P\cap (p^*H)^{\perp}$ is generated by the classes of exceptional curves
of $p$. Let $G=P\setminus E=\fom_P$.
\item $\foD$ is a scattering diagram for the data $P, \varphi$ and $G$.
Furthermore, for each ray $\rho_i\in\Sigma$, the unique outgoing ray
$(\rho_i,f_{\rho_i})\in\foD$ satisfies
\[
f_{\rho_i}=g_{\rho_i}\prod_{j=1}^{\ell_i} (1+b_{ij}X_i^{-1})
\]
with $g_{\rho_i}\equiv 1\mod G$ and $b_{ij}=z^{[E_{ij}]}$.
\end{itemize}
\end{assumptions}

We note we have shown that $\foD=\foD^{\can}$ achieves these assumptions.

\medskip

\emph{Step IV. Pushing the singularities to infinity.}
We work with Assumptions \ref{GSassumptions2}.
Consider the tropicalisation $(\oB,\oSigma)$ of
$(\oY,\oD)$. By Example \ref{basictoricexample},
$\oB$ in fact has no singularity at the origin, and is affine isomorphic
to $M_{\RR}=\RR^2$ (with $M=\ZZ^2$), while $\oSigma$
is precisely the fan for $\oY$.
In order to distinguish between constructions
on $(Y,D)$ and $(\oY,\oD)$, we decorate all existing notation with bars.
For example, if $\tau\in\Sigma$, denote the corresponding
cone of $\oSigma$ by $\bar\tau$. Let $\bar\varphi$ be the multivalued $P^{\gp}_{\bR}$-valued function on $\oB$ such that
\begin{equation}
\label{kappabardef}
\kappa_{\bar\rho,\bar\varphi}=p^*[\bar D_{\bar\rho}].
\end{equation}
Note that by Lemma \ref{varphitoric}, we can assume
$\bar\varphi$ is in fact a single-valued function on $M_{\RR}$.
This single-valuedness will be important to be able to apply
the method of Kontsevich and Soibelman, Theorem \ref{KSlemma}.

We now have sheaves $\oshP$ on $\oB_0$ and $\shP$ on $B_0$,
induced by the two functions $\bar\varphi$ and $\varphi$ respectively.

Note that since $\bar\varphi$ is single-valued and $\oB$ has no
singularities, $\oshP$ is the constant sheaf with fibre $P^{\gp}\oplus
M$.

There is a canonical piecewise linear map
\[
\nu:B\rightarrow \oB
\]
which restricts to an integral affine isomorphism $\nu|_{\sigma}:\sigma\rightarrow
\bar\sigma$, where $\sigma\in\Sigma_{\max}$ and $\bar\sigma\in
\oSigma_{\max}$ is the corresponding cell of $\oSigma$. Note this
map identifies $B(\ZZ)$ with $\oB(\ZZ)$.

For each maximal cone $\sigma\in\Sigma_{\max}$, the derivative $\nu_*$ of $\nu$
induces a canonical identification of $\Lambda_{B,\sigma}$ with
$\Lambda_{\oB,\osigma}$.
This then gives
an induced isomorphism of monoids:
\begin{equation} \label{pl}
\tilde\nu_{\sigma}:P_{\varphi_{\sigma}} \to P_{\bar
\varphi_{\bar\sigma}}
\end{equation}
given by
\[
\varphi_{\sigma}(m) + p \mapsto
\bar\varphi_{\bar\sigma}(\nu_*(m))
+p,
\]
for $p \in P$ and $m \in \Lambda_\sigma$. This identifies the $\kk[P]$-algebras
$\kk[P_{\varphi_{\sigma}}]$ and $\kk[P_{\bar\varphi_{\bar\sigma}}]$,
and the completions
$\widehat{\kk[P_{\varphi_{\sigma}}]}$ and
$\widehat{\kk[P_{\bar\varphi_{\bar\sigma}}]}$.

Because the map $\nu$ is only piecewise linear around rays
$\rho\in\Sigma$, there is only a piecewise linear identification
of $P_{\varphi_{\rho}}$ with $P_{\bar\varphi_{\bar\rho}}$ and hence
no identification of the corresponding rings. However, $\nu_*$
is still defined on the tangent space to $\rho$, and
there is an identification
\[
\tilde\nu_{\rho}:\{\varphi_{\rho}(m)+p\,|\,\hbox{$m$ is tangent to $\rho$,
$p\in P$}\}
\rightarrow
\{\bar\varphi_{\bar\rho}(m)+p\,|\,\hbox{$m$ is tangent to $\bar\rho$,
$p\in P$}\}
\]
given by
\[
\varphi_{\rho}(m)+p \mapsto \bar\varphi_{\bar\rho}(\nu_*(m)) + p.
\]

We now explain the Kontsevich-Soibelman lemma. This
has to do with scattering diagrams on the smooth affine surface
$M_{\bR} = \bR^2$ (such as $\oB = B_{(\oY,\oD)}$).
For this general discussion, we fix the data of a monoid $Q$ which comes
along with a map $r:Q\rightarrow M$. Let $\fom_Q=Q\setminus Q^{\times}$,
and let $\widehat{\kk[Q]}$ denote the completion of $\kk[Q]$ with
respect to the monomial ideal $\fom_Q$. (In our application we take $Q=P_{\bar\varphi}$ as defined in \eqref{Mumford cone}.)

We can then consider a variant of the notion of scattering diagram:

\begin{definition}
\label{scatdiagQdef}
We define a \emph{scattering diagram for the pair $Q$,
$r:Q\rightarrow M$}. This is a set
\[
\foD=\{(\fod,f_{\fod})\}
\]
where
\begin{itemize}
\item
$\fod\subset M_{\RR}$ is given by
\[
\fod=-\RR_{\ge 0} m_0
\]
if $\fod$ is an \emph{outgoing ray} and
\[
\fod=\RR_{\ge 0} m_0
\]
if $\fod$ is an \emph{incoming ray}, for some
$m_0\in M\setminus \{0\}$.
\item $f_{\fod}\in \widehat{\kk[Q]}$.
\item
$f_{\fod}\equiv 1\mod \fom_Q$.
\item $f_{\fod}=1+\sum_p c_pz^p$
for $c_p\in \kk$, $r(p)\not=0$ a positive multiple of $m_0$.
\item For any $k>0$, there are only a finite number of rays
$(\fod,f_{\fod})\in\foD$ with $f_{\fod}\not\equiv 1\mod \fom_Q^k$.
\end{itemize}
\end{definition}

\begin{definition} \label{pop}
Given a loop $\gamma$ in $M_{\RR}$ around the origin, we define
the \emph{path ordered product}
\[
\theta_{\gamma,\foD}:\widehat{\kk[Q]}\rightarrow \widehat{\kk[Q]}
\]
as follows.
For each $k>0$, let $\foD[k]\subset\foD$ be the subset of
rays $(\fod,f_{\fod})\in\foD$ with $f_{\fod}\not\equiv 1\mod \fom_Q^k$.
This set is finite. For $\fod\in\foD[k]$ with $\gamma(t_0)\in \fod$, define
\[
\theta_{\gamma,\fod}^k:\kk[Q]/\fom_Q^k\rightarrow\kk[Q]/\fom_Q^k
\]
by
\[
\theta^k_{\gamma,\fod}(z^q)=z^qf_{\fod}^{\langle n_{\fod},r(q)\rangle}
\]
for $n_{\fod}\in M^*$ primitive satisfying, with $m$ a non-zero
tangent vector of $\fod$,
\[
\langle n_{\fod},m\rangle=0,\quad \langle n_{\fod},\gamma'(t_0)\rangle<0.
\]
Then, if $\gamma$ crosses the rays $\fod_1,\ldots,\fod_n$ in order
with $\foD[k]=\{\fod_1,\ldots,\fod_n\}$, we can define
\[
\theta_{\gamma,\foD}^k=\theta_{\gamma,\fod_n}^k\circ\cdots\circ
\theta_{\gamma,\fod_1}^k.
\]
We then define $\theta_{\gamma,\foD}$ by taking the limit as $k\rightarrow
\infty$.
\end{definition}

The following is a slight generalisation of a result of
Kontsevich and Soibelman which appeared in \cite{KS06}.

\begin{theorem}
\label{KSlemma}
Let $\foD$ be a scattering diagram in the sense of
Definition \ref{scatdiagQdef}. Then there is another scattering
diagram $\Scatter(\foD)$ containing $\foD$ such that $\Scatter(\foD)
\setminus\foD$ consists only of outgoing rays and $\theta_{\gamma,\Scatter(\foD)}$
is the identity for $\gamma$ a loop around the origin.
\end{theorem}

For a proof of this theorem essentially as stated here, see \cite{GPS09},
Theorem 1.4. The result is unique if $\Scatter(\foD)\setminus\foD$
has at most one ray in each possible direction; we shall assume
$\Scatter(\foD)$ has been chosen to have this property. This can
always be done.

We apply this in the following situation.
We take $Q$ to be the monoid $P_{\bar\varphi}$
which yields the Mumford degeneration
associated to the data $(\oB,\oSigma)$, $\bar\varphi$ (recalling
$\oB=M_{\RR}$), defined by
\[
P_{\bar\varphi}=\{(m,\bar\varphi(m)+p)\,|\, m\in M,p\in P\}\subset M\times
P^{\gp}.
\]
This comes with a canonical map $r:P_{\bar\varphi}\rightarrow M$
by projection.

\begin{definition}
\label{nufoDdef}
Suppose we are in the situation of Assumptions \ref{GSassumptions2}.
We define a scattering diagram $\nu(\foD)$ on $\oB$
as follows. For every ray $(\fod,f_{\fod})
\in\foD$ not equal to $(\rho_i,f_{\rho_i})$ for some $i$,
$\nu(\foD)$ contains the
ray $(\nu(\fod),\tilde\nu_{\tau_{\fod}}(f_{\fod}))$, and for each
ray $(\rho_i,f_{\rho_i})$, $\nu(\foD)$ contains two rays,
$(\bar\rho_i,\tilde\nu_{\tau_{\fod}}(g_{\rho_i}))$ and
$(\bar\rho_i,\prod_{j=1}^{\ell_i}(1+b_{ij}^{-1}\bar
X_i))$.

We note that $\nu(\foD)$ may not actually be a scattering diagram
in the sense of Definition \ref{scatdiagQdef}, as it is possible
that $f_{\fod}\not\in \widehat{\kk[P_{\bar\varphi}]}$: if
$p\in P_{\varphi_{\tau}}$, then $\tilde\nu_{\tau}(p)\in P_{\ovarphi_{\tau}}$ but
need not lie in $P_{\ovarphi}$.
\end{definition}

In the case of $\foD=\foD^{\can}$, we can use the Kontsevich-Soibelman
lemma to describe $\nu(\foD^{\can})$. This will both show that
$\nu(\foD^{\can})$ is a scattering diagram in the sense of
Definition \ref{scatdiagQdef} and that it satisfies an important
additional property, namely
the condition that $\theta_{\gamma,\nu(\foD^{\can})}$ is equal to
the identity. This will allow us to prove consistency. Let
\begin{equation} \label{canonicalinputs}
\ofoD_0=\{(\bar\rho_i,\prod_{j=1}^{\ell_i}(1+b_{ij}^{-1}\oX_i))
\,|\, 1\le i\le n\}.
\end{equation}
Let $\fom_{\bar\varphi}=P_{\bar\varphi}\setminus P_{\bar\varphi}^{\times}$
as usual. Then by the strict convexity of $\bar\varphi$, $\oX_i
\in \fom_{\bar\varphi}$ so that $\ofoD_0$ is a scattering
diagram for the pair $P_{\bar\varphi}$, $r$ in the sense of Definition
\ref{scatdiagQdef}. Now define
\[
\ofoD := \Scatter(\ofoD_0)
\]
where we require $\ofoD \setminus \ofoD_0$ to have only one
outgoing ray in each direction (and no incoming rays).

The following will be Step V, which we defer until \S\ref{connectionwithGPS}.

\begin{theorem} \label{relatingcanandscatter}
$\ofoD=\nu(\foD^{\can})$. In particular, $\nu(\foD^{\can})$
is a scattering diagram in the sense of Definition \ref{scatdiagQdef}
and $\theta_{\gamma,\ofoD}
\equiv 1$ for a loop $\gamma$ around the origin.
\end{theorem}

\begin{example}
\label{M05stillmore}
Continuing with Example \ref{M05more}, note that the pair $(Y,D)$
can be obtained from the toric pair $(\oY,\oD)$ defined by the fan $\bar\Sigma$ with rays
generated by $(1,0)$, $(1,1)$, $(0,1)$, $(-1,0)$ and $(0,-1)$, corresponding
to $\oD_1,\ldots,\oD_5$, by blowing up one point on each of $\oD_4$ and $\oD_5$.
This description
determines $\ofoD_0$ and hence $\ofoD$. One can check
this description agrees with that given in Example \ref{M05more}
for $\foD^{\can}$, see e.g.\ \cite{GPS09}, Example 1.6 for a similar
computation.
\end{example}

Returning to the situation of Assumptions \ref{GSassumptions2}, suppose
in addition that
$\nu(\foD)$ is a scattering diagram in the sense of Definition
\ref{scatdiagQdef}. (For example, by Theorem \ref{relatingcanandscatter},
$\foD=\foD^{\can}$ satisfies these assumptions.)
For $I\subset P$ an ideal with $\sqrt{I}=J$, we now have
deformations $X^o_{I,\foD}$ and $\oX^o_{I,{\nu(\foD)}}$.
The latter scheme is glued from open sets
\[
\overline{U}_{\bar\rho,I}=\Spec \overline{R}_{\bar\rho,I}
\]
along open sets identified with $\Spec \overline{R}_{\bar\sigma,I}$.
Here we are decorating the rings coming
from the data on $\oB$ with bars as before, while we maintain
the notation $R_{\rho,I}$, etc., for those rings coming from the
data on $B$.

\begin{lemma}
\label{movingsingularitylemma}
Given Assumptions \ref{GSassumptions2}, assume also that
$\nu(\foD)$ is a scattering diagram
in the sense of Definition \ref{scatdiagQdef}. Then there are isomorphisms
\[
p_i:R_{\rho_i,I}\rightarrow \oR_{\bar\rho_i,I}
\]
and
\[
p_{i-1,i}:
R_{\sigma_{i-1,i},I}
\rightarrow
\oR_{\bar\sigma_{i-1,i},I}
\]
for all $i$ such that the diagrams
\[
\xymatrix@C=30pt
{
R_{\rho_i,I}
\ar[r]^{\psi_{\rho_i,-}}\ar[d]_{p_i}&
R_{\sigma_{i-1,i},I}
\ar[d]^{p_{i-1,i}}\\
\oR_{\bar\rho_i,I}
\ar[r]_{\psi_{\bar\rho_i,-}}&
\oR_{\bar\sigma_{i-1,i},I}
}\quad\quad
\xymatrix@C=30pt
{
R_{\rho_i,I}
\ar[r]^{\psi_{\rho_i,+}}\ar[d]_{p_i}&
R_{\sigma_{i,i+1},I}
\ar[d]^{p_{i,i+1}}\\
\oR_{\bar\rho_i,I}
\ar[r]_{\psi_{\bar\rho_i,+}}&
\oR_{\bar\sigma_{i,i+1},I}
}
\]
and
\[
\xymatrix@C=30pt
{
R_{\sigma_{i-1,i},I}
\ar[r]^{\theta_{\gamma,\foD}}\ar[d]_{p_{i-1,i}}&
R_{\sigma_{i-1,i},I}
\ar[d]^{p_{i-1,i}}\\
\oR_{\bar\sigma_{i-1,i},I}
\ar[r]_{\theta_{\bar\gamma,\nu(\foD)}}&
\oR_{\bar\sigma_{i-1,i},I}
}
\]
are commutative, where $\gamma$ is any path in $\sigma_{i-1,i}$
for which $\theta_{\gamma,\foD}$ is defined, and $\bar\gamma=
\nu\circ\gamma$.

Consequently, the maps $p_i$ and $p_{i-1,i}$ induce an isomorphism
\[
p:
X^o_{I,\foD}
\rightarrow
\oX^o_{I,\nu(\foD)}
\]
over $\Spec\kk[P]/I$.
\end{lemma}

\begin{proof}
Recall that
\begin{align}
R_{\rho,I}= {} &{R_I[X_{i-1},X_i^{\pm},X_{i+1}]
\over (X_{i-1}X_{i+1}-z^{\eta([D_i])}X_i^{-D_i^2}g_{\rho_i}
\prod_{j=1}^{\ell_i}(1+b_{ij}X_i^{-1}))},\\
\label{oRrhoeq}
\oR_{\rho,I}={} &
{R_I[\oX_{i-1},\oX_i^{\pm},\oX_{i+1}]
\over (\oX_{i-1}\oX_{i+1}-z^{\eta(p^*[\oD_i])}\oX_i^{-\oD_i^2}
\bar g_{\rho_i}\prod_{j=1}^{\ell_i}(1+b_{ij}^{-1}\oX_i))}.
\end{align}

We simply define $p_i$ to be the identity on $R_I$ and
$p_i(X_j)=\oX_j$. This makes
sense since $D_i^2=\bar D_i^2-\ell_i$ and $[D_i]=p^*[\bar D_i]-
\sum_{j=1}^{\ell_i} E_{ij}$, so that
\begin{align*}
p_i\left(
z^{\eta([D_i])}X_i^{-D_i^2}\prod_{j=1}^{\ell_i}(1+b_{ij}X_i^{-1})\right)
= {} &
z^{\eta(p^*[\bar D_i])}\oX_i^{-\bar D_i^2}\left (
\prod_{j=1}^{\ell_i} b_{ij}^{-1}\oX_i\right) \left(\prod_{j=1}^{\ell_i}
(1+b_{ij}\oX_i^{-1})\right)\\
= {} & z^{\eta(p^*[\bar D_i])}\oX_i^{-\bar D_i^2} \prod_{j=1}^{\ell_i}
(1+b_{ij}^{-1}\oX_i).
\end{align*}
The map $p_{i-1,i}$ is induced by $\tilde\nu_{\sigma_{i-1,i}}$ defined
in \eqref{pl}. It is then
straightforward to check the commutativity of the three diagrams.
\end{proof}

\begin{lemma} \label{brokenlinebijection}
Given Assumptions \ref{GSassumptions2},
suppose $\nu(\foD)$ is a scattering diagram
in the sense of Definition \ref{scatdiagQdef}.
For $Q\in\sigma_{i-1,i}$, we distinguish between
\[
\Lift_Q(q)\in R_{\sigma_{i-1,i},I}
\]
for the lift of $q\in B_0(\ZZ)$ and
\[
\Lift_{\nu(Q)}(\nu(q))\in \bar R_{\bar\sigma_{i-1,i},I}
\]
the lift of $\nu(q)$. Then
\begin{enumerate}
\item $p_{i-1,i}(\Lift_Q(q))=\Lift_{\nu(Q)}(\nu(q))$.
\item Under the natural identifications $(P_{\bar\varphi_{\tau}})^{\gp}
=(P_{\bar\varphi})^{\gp}$, for $\tau\in\oSigma\setminus\{0\}$,
$P_{\bar\varphi}\subset P_{\bar\varphi_{\tau}}$, and for any broken line
$\gamma$ for $q$, $\Mono(\gamma)\in\kk[P_{\bar\varphi}]$.
\item $\nu$ induces a bijection between broken lines:
If $\gamma:(-\infty,0]\rightarrow B_0$
is a broken line in $B_0$, then $\nu\circ\gamma$ is a broken line
in $\oB_0$, and conversely, if $\bar\gamma:(-\infty,0]\rightarrow\oshP$
is a broken line in $\oB_0$, then $\nu^{-1}\circ\bar\gamma$
is a broken line in $B_0$.
\end{enumerate}
\end{lemma}

\begin{proof}
(3) implies (1). For (3),
clearly it is enough to compare bending and attached monomials
of broken lines near a ray $\rho_i$.

Consider a broken line $\gamma$ in $B_0$ passing from $\sigma_{i-1,i}$ to
$\sigma_{i,i+1}$, and let $cz^q$ be the monomial attached to the
broken line before it crosses over $\rho_i$, so that $q\in
P_{\varphi_{\sigma_{i-1,i}}}$. Let $\theta_{\rho_i}$, $\bar\theta_{\bar\rho_i}$
be defined by
\begin{align*}
\theta_{\rho_i}(z^p) := {} & z^pf_{\rho_i}^{\langle n,r(p)\rangle}\\
\bar\theta_{\bar\rho_i}(z^p) := {} &
z^p\left(\bar g_{\rho_i}\prod_{j=1}^{\ell_i}(1+b_{ij}^{-1}\oX_i)
\right)^{\langle \bar n,\bar r(p)\rangle}
\end{align*}
where $(\bar\rho_i,\bar g_{\rho_i})\in\nu(\foD)$ is the outgoing
ray with support $\bar\rho_i$. Here $n, \bar n$ are primitive cotangent
vectors vanishing on tangent vectors to $\rho_i, \bar\rho_i$
and positive on $\sigma_{i-1,i}, \bar\sigma_{i-1,i}$ respectively.
Then we need to show that
\begin{equation}
\label{basicthingtoshow}
p_{i,i+1}(\theta_{\rho_i}(cz^q))=\bar\theta_{\bar\rho_i}(p_{i-1,i}(cz^q))
\end{equation}
to get the correspondence between broken lines.

Note that
\[
p_{i-1,i}(X_{i-1})=\bar X_{i-1},\quad p_{i,i-1}(X_i)=\bar X_i,
\quad p_{i,i+1}(X_i)=\bar X_i,
\]
but to compute $p_{i,i+1}(X_{i-1})$, we need to use the relation
(see Proposition \ref{ringfirstversion})
\[
X_{i-1}X_{i+1}=z^{\eta([D_i])}X_i^{-D_i^2}
\]
in $\kk[P_{\varphi_{\rho_i}}]$ to write
\[
X_{i-1}=z^{\eta([D_i])}X_i^{-D_i^2}X_{i+1}^{-1}.
\]
On the other hand, one has the relation
\[
\bar X_{i-1}\bar X_{i+1}=z^{\eta(p^*[\bar D_i])}\bar X_i^{-\bar D_i^2}
\]
in $\kk[P_{\bar\varphi_{\bar\rho_i}}]$, so
\begin{align*}
p_{i,i+1}(X_{i-1})= {} & z^{\eta([D_i]-p^*[\oD_i])}
\bar X_i^{- D_i^2+\bar D_i^2}\bar X_{i-1}\\
= {} & \bar X_i^{\ell_i}\bar X_{i-1}\prod_{j=1}^{\ell_i} b_{ij}^{-1}.
\end{align*}
Thus using Assumptions \ref{GSassumptions2} for the form of $f_{\rho_i}$, we
have
\begin{align*}
p_{i,i+1}(\theta_{\rho_i}(X_{i-1}))= {} &
p_{i,i+1}(X_{i-1}f_{\rho_i})\\
= {} & \bar X_i^{\ell_i}\bar X_{i-1} \left(\prod_{j=1}^{\ell_i}
b_{ij}^{-1}\right) \left(\prod_{j=1}^{\ell_i} (1+b_{ij}\bar X_{i}^{-1})\right)
\bar g_{\rho_i}\\
= {} & \bar X_{i-1}\bar g_{\rho_i}\prod_{j=1}^{\ell_i} (1+b_{ij}^{-1}
\bar X_i)\\
= {} & \bar\theta_{\bar\rho_i}(p_{i-1,i}(X_{i-1}))
\end{align*}
as desired. Also,
\[
p_{i,i+1}(\theta_{\rho_i}(X_i))=\bar X_i=\bar\theta_{\bar\rho_i}(p_{i-1,i}
(\bar X_i)).
\]
Thus \eqref{basicthingtoshow} holds. This shows (3).

For (2), the statement that $P_{\bar\varphi}\subset P_{\bar\varphi_{\tau}}$
is obvious. For $q\in\sigma\in\Sigma$, by definition
the monomial attached to the first domain of linearity of a broken
line for $q$ is $z^{\varphi_{\sigma}(q)}$, which
is identified under $\tilde\nu_{\sigma}$ with $z^{(\nu(q),\bar\varphi(\nu(q)))}
\in \kk[P_{\bar\varphi}]$. For any
$(\fod,f_{\fod})\in \nu(\foD)$, $f_{\fod}\in
\widehat{\kk[P_{\bar\varphi}]}$ by assumption,
and hence all monomials associated
to broken lines in $\oB_0$ lie in $\kk[P_{\bar\varphi}]$, hence (2).
\end{proof}

\begin{definition}
\label{newliftdef}
Let $\ofoD$ be a scattering diagram in the sense of Definition
\ref{scatdiagQdef}
for the pair $P, r:P\rightarrow M$ for some toric monoid $P$.
Let $I\subset P$ be an ideal with $\sqrt{I}=
\fom_{P}$. We define for $q\in \oB_0(\ZZ)$
and $Q\in\oB_0$,
\[
{\Lift}_Q(q)=\sum \Mono(\gamma)\in \kk[P]/I
\]
where the sum is over all broken lines $\gamma$ for $q$ with endpoint
$Q$ in $\oB_0$ with respect to the scattering diagram
$\ofoD$. One sees easily as in Lemma \ref{finitenesslemma}
that this is a finite sum.
\end{definition}

The last crucial result we need for consistency is the following result
of \cite{CPS}.

\begin{theorem}
\label{wallcrossingtheorem} With the assumptions of Definition
\ref{newliftdef}, suppose furthermore
that $\theta_{\gamma,\ofoD}\equiv 1$ for a loop $\gamma$ around
the origin.
Fix an ideal $I\subset P$ with
$\sqrt{I}=\fom_{P}$ and $q\in \oB_0(\ZZ)$.
If $Q,Q'\in M_{\RR}\setminus \Supp(\ofoD_I)$ are
general, and $\gamma$ is a path connecting $Q$ and $Q'$ for which
$\theta_{\gamma,\ofoD_I}$ is defined, then
\[
{\Lift}_{Q'}(q)=\theta_{\gamma,\ofoD_I}({\Lift}_Q(q))
\]
as elements of $\kk[P]/I$.
\end{theorem}

\begin{proof}
This is shown in \cite{CPS} in a rather more general setup. For a version
of the argument closer to the current setup, see the proof of Theorem 5.35
of \cite{G11}.
\end{proof}

\emph{Proof of Theorem \ref{canonicalconsistency}.}
By Lemmas \ref{grhoilemma} and
\ref{PEreduction}, we can assume we are in the situation of
Assumptions \ref{GSassumptions2} with $\foD=\foD^{\can}$.
In checking (1) of Definition \ref{consistentdef} for
$\foD^{\can}$ in this situation,
we want to check equalities
\[
\Lift_{Q'}(q)=\theta_{\gamma,\foD^{\can}}(\Lift_Q(q))
\]
for $Q,Q'\in\sigma_{i-1,i}$. By
Lemmas \ref{movingsingularitylemma}
and \ref{brokenlinebijection}, it
is sufficient to show that
\begin{equation}
\label{lifteq}
\Lift_{\nu(Q')}(\nu(q))=\theta_{\bar\gamma,\nu(\foD)}
(\Lift_{\nu(Q)}(\nu(q))).
\end{equation}
To check this equality we can compare coefficients
of monomials, and given any monomial $z^p$ appearing on the left-
or right-hand sides, we can apply Theorems \ref{relatingcanandscatter}
and \ref{wallcrossingtheorem},
where we take $P=P_{\bar\varphi}$, $I=\fom^k_{P_{\bar\varphi}}$
for sufficiently large $k$ so that
$p\not\in I$.
The hypothesis $\theta_{\gamma,\nu(\foD)}\equiv 1$
of Theorem \ref{wallcrossingtheorem} holds by Theorem
\ref{relatingcanandscatter}.

To show (2) of Definition \ref{consistentdef}, we can take $Q=Q_-$ and
$Q'=Q_+$ on opposite sides of a ray $\rho_i$.
If $\gamma$ is a short path joining $Q$ and $Q'$,
we still have \eqref{lifteq} after inverting $f_i$.
We have a map
\[
\psi:=(\psi_{\bar\rho_i,-},\psi_{\bar\rho_i,+}):
\oR_{\orho_i,I}\rightarrow
\oR_{\osigma_{i-1,i},I}\times \oR_{\osigma_{i,i+1},I}.
\]
If $f_i=\bar g_{\rho_i} \prod_{j=1}^{\ell_i}
(1+b_{ij}^{-1}\oX_i)$ then $\psi$ is given by
\begin{align*}
\oX_{i-1}\mapsto {} &  (z^{\ovarphi_{\orho_i}(v_{i-1})},
f_i z^{\ovarphi_{\orho_i}(v_{i-1})}),\\
\oX_i\mapsto {} & (z^{\ovarphi_{\orho_i}(v_{i})},
z^{\ovarphi_{\orho_i}(v_{i})}),\\
\oX_{i+1}\mapsto {} & (f_i
z^{\ovarphi_{\orho_i}(v_{i+1})},
z^{\ovarphi_{\orho_i}(v_{i+1})}).
\end{align*}
One checks easily that this map is injective. Furthermore,
the image is described as follows. Let $I_{\bar\rho_i}\subset
P_{\bar\varphi_{\bar\rho_i}}$ be the monoid ideal
\[
I_{\bar\rho_i}=
\{q \in P_{\orho_i}\,|\, \hbox{$q-\varphi_{\osigma_{i-1,i}}(r(q))
\in I$
or $q-\varphi_{\osigma_{i,i+1}}(r(q))
\in I$}\}.
\]
Then the image consists of those
elements $(g_-,g_+)$ such that every monomial of $g_-$ and $g_+$ has
exponent in $P_{\bar\varphi_{\bar\rho_i}}\subset
P_{\bar\varphi_{\osigma_{i-1,i}}}, P_{\bar\varphi_{\osigma_{i,i+1}}}$,
and the images $\bar g_{\pm}$ of $g_{\pm}$ in
$(\kk[P_{\orho_i}]/I_{\bar\rho_i})_{f_i}$ satisfy
$\theta_{\gamma,\nu(\foD)}(\bar g_-)=\bar g_+$, where this makes sense
as we have localized at $f_i$. (See e.g., the proof of Lemma 2.34 in
\cite{GS07} for a similar statement.)
Thus by \eqref{lifteq} and Lemma \ref{finitenesslemma}, (2),
there is an $\alpha\in \oR_{\orho_i,I}$ such that
\[
\psi_{\orho_i,-}(\alpha)=\Lift_{\nu(Q)}(\nu(q)),
\quad
\psi_{\orho_i,+}(\alpha)=\Lift_{\nu(Q')}(\nu(q)).
\]
Thus we may take $\Lift_{\rho_i}(q)=p_i^{-1}(\alpha)$, and by Lemma
\ref{brokenlinebijection}, $\psi_{\rho_i,\pm}(\Lift_{\rho_i}(q))
=\Lift_{Q_{\pm}}(q)$, giving consistency.
\qed

\subsection{Step V: The proof of Theorem \ref{relatingcanandscatter} and
the connection with \cite{GPS09}} \label{connectionwithGPS}
Here we derive Theorem \ref{relatingcanandscatter} from the
main result of \cite{GPS09}. We will need to review one form of
this result, which gives an enumerative interpretation for
the output of the Kontsevich-Soibelman lemma.

Fix $M\cong\ZZ^2$ as usual.
Suppose we are given positive integers $\ell_1,\ldots,\ell_n$ and
primitive vectors $m_1,\ldots,m_n\in M$. Let $\ell=\sum_{i=1}^n
\ell_i$ and $Q=M\oplus \NN^{\ell}$, with $r:Q\rightarrow M$ the projection.
Denote the variables in $\kk[Q]$ corresponding to the generators of
$\NN^{\ell}$ as $t_{ij}$, for $1\le i\le n$ and $1\le j\le \ell_i$.
Consider the scattering diagram for the data $r:Q\rightarrow M$
(in the sense of Definition \ref{scatdiagQdef})
\[
\foD=\{(\RR_{\ge 0} m_i,\prod_{j=1}^{\ell_i} (1+t_{ij}z^{m_i}))\,|\,
1\le i\le n\}.
\]
We wish to interpret $(\fod,f_{\fod})\in \Scatter(\foD)\setminus\foD$.
Choose a complete fan $\Sigma_{\fod}$ in $M_{\RR}$ which contains
the rays $\RR_{\ge 0}m_1,\ldots,\RR_{\ge 0}m_n$ as well as the ray $\fod$
(which may coincide with one of the other rays). Let $X_{\fod}$ be
the corresponding toric surface, and let $D_1,\ldots,D_n,D_{\out}$ be
the divisors corresponding to the above rays. Choose general points
$x_{i1},\ldots,x_{i\ell_i}\in D_i$, and let
\[
\nu:\tilde X_{\fod}\rightarrow X_{\fod}
\]
be the blow-up of all the points $\{x_{ij}\}$. Let $\tilde D_1,
\ldots,\tilde D_n, \tilde D_{\out}$ be the proper transforms
of the divisors $D_1,\ldots,D_n,D_{\out}$ and $E_{ij}$ the exceptional
curve over $x_{ij}$.

Now introduce the additional data of ${\bf P}=({\bf P}_1,\ldots,{\bf P}_n)$,
where ${\bf P}_i$ denotes a sequence $p_{i1},\ldots,p_{i\ell_i}$ of
$\ell_i$ non-negative numbers. We will use the notation ${\bf P}_i=
p_{i1}+\cdots+p_{i\ell_i}$ and call ${\bf P}_i$ an \emph{ordered partition}.
Define
\[
|{\bf P}_i|=\sum_{j=1}^{\ell_i} p_{ij}.
\]
We shall restrict attention to those ${\bf P}$ such that
\begin{equation}
\label{balancePeq}
-\sum_{i=1}^n |{\bf P}_i|m_i=k_{\bf P}m_{\fod}
\end{equation}
where $m_{\fod}\in M$ is a primitive generator of $\fod$ and $k_{\bf P}$
is a positive integer.

Given this data, consider the class $\beta\in A_1(X_{\fod},
\ZZ)$ specified by the requirement that, if $D$ is a toric divisor of
$X_{\fod}$ with $D\not\in \{D_1,\ldots,D_n,D_{\out}\}$, then $D\cdot\beta
=0$;
if $D_{\out}\not\in \{D_1,\ldots,D_n\}$,
\[
D_i\cdot \beta=|{\bf P}_i|,\quad\quad D_{\out}\cdot\beta=k_{{\bf P}};
\]
while if $D_{\out}=D_j$ for some $j$, then
\[
D_i\cdot\beta=\begin{cases} |{\bf P}_i|&i\not=j,\\
|{\bf P}_i|+k_{\bf P}&i=j.
\end{cases}
\]
That such a class exists follows easily from \eqref{balancePeq} and the
first part of Lemma \ref{H2toric} below.
It is also unique. We can then define
\[
\beta_{\bf P}=\nu^*(\beta)-\sum_{i=1}^p\sum_{j=1}^{\ell_i} p_{ij}
[E_{ij}]\in A_1(\tilde X_{\fod},\ZZ).
\]
We define $N_{\bf P}:=N_{\beta_{\bf P}}$ as in Definition \ref{basicrelinvs},
using $(\tilde Y,\tilde D)=(\tilde X_{\fod},\tilde D)$,
where $\tilde D$ is the proper transform of the toric boundary
of $X_{\fod}$, and using $C=\tilde D_{\out}$.
Then one of the main theorems of \cite{GPS09} (see \S 5.7 of that paper)
states

\begin{theorem}
\label{mainGPStheorem}
\begin{equation}
\label{1stGPSformula}
\log f_{\fod}=\sum_{\bf P}
k_{\bf P} N_{\bf P} t^{\bf P}z^{-k_{\bf P} m_{\fod}},
\end{equation}
where the sum is over all ${\bf P}$ satisfying \eqref{balancePeq}
and $t^{\bf P}$ denotes the monomial
$\prod_{ij} t_{ij}^{p_{ij}}$.
\end{theorem}

We can adapt this theorem for our purposes as follows. Fix a fan
$\Sigma$ in $M_{\RR}$ defining a complete non-singular toric surface $\oY$,
with $\oD=D_1+\cdots+D_n$ the toric boundary. Choose points
$x_{i1},\ldots,x_{i\ell_i}\in D_i$, and define a new surface
$Y$ as the blow-up $\nu:Y\rightarrow \oY$ at the points $\{x_{ij}\}$.
Let $E_{ij}$ be the exceptional curve over $x_{ij}$.

Let $P=\NE(\oY)$; because $\oY$ is toric, this is a finitely generated
monoid with $P^{\times}=\{0\}$. Let $\bar\varphi:M_{\RR}\rightarrow
P^{\gp}_\RR$ be the $\Sigma$-piecewise linear strictly
$P$-convex function given by Lemma \ref{varphitoric}.

We will need the following, an immediate corollary of Lemma \ref{varphitoric},
using the notation of that lemma applied to the fan $\Sigma$ for $\oY$:

\begin{lemma}
\label{H2toric}
If $\sum_{\rho} a_{\rho}t_{\rho}\in \ker s$, then
the corresponding element of $\ker s = A_1(\oY,\ZZ)$ is
\[
\sum_{\rho} a_{\rho}\bar\varphi(m_{\rho})\in P^{\gp}.
\]
\end{lemma}

Let $E\subset
A_1(Y,\ZZ)$ be the lattice spanned by the classes of the exceptional
curves of $\nu$, so that $A_1(Y,\ZZ)=\nu^*A_1(\oY,\ZZ)\oplus E$.
We then obtain a map
\[
\varphi=\nu^*\circ\bar\varphi:M_{\RR}\rightarrow \nu^*P^{\gp}\oplus E.
\]
Let
\[
Q=\{(m,p)\in M\oplus A_1(Y,\ZZ)\,|\,\hbox{$\exists p'\in \nu^*P\oplus E$
such that $p=p'+\varphi(m)$}\}.
\]
There is an obvious projection $r:Q\rightarrow M$,
and by strict convexity of $\bar\varphi$, $Q^{\times}=E$.

We consider the scattering diagram, $\ofoD_0$,
over $\widehat{\kk[Q]}$ given by
\[
\ofoD_0=\{(\RR_{\ge 0}m_i,\prod_{j=1}^{\ell_i} (1+z^{(m_i,
\varphi(m_i)-E_{ij})}))
\,|\,1\le i\le n\}.
\]

Then we have

\begin{theorem}
\label{GPScor}
Let $(\fod,f_{\fod})\in\Scatter({\ofoD_0})\setminus \ofoD_0$,
assuming that
there is at most one ray of $\Scatter(\ofoD_0)\setminus\ofoD_0$
in each
possible outgoing direction. (Note by definition of $\Scatter(\ofoD_0)$,
$(\fod,f_{\fod})$ cannot be incoming.)
Then, following the notation of Definition
\ref{basicrelinvs} and \ref{canscatdiagdef},
\begin{equation}
\label{2ndGPSformula}
\log f_{\fod}=\sum_{\beta} k_{\beta}N_{\beta}
z^{(-k_{\beta}m_{\fod},\pi_*(\beta)-\varphi(k_{\beta}m_{\fod}))}.
\end{equation}
Here $\pi:\tY\rightarrow Y$ is the toric blow-up of $Y$ determined by
$\fod$ and $C\subset\tY$ is the component of the boundary determined
by $\fod$. If $\fod$ is not one of the rays $\bR_{\geq 0} m_i$, then
we sum over all $\AA^1$-classes $\beta\in A_1(\tY,\ZZ)$
satisfying \eqref{betadef}, and if $\fod = \bR_{\geq 0} m_i$
we sum over all such classes {\bf except for} classes given
by multiple covers of one of the exceptional divisors $E_{ij}$.
\end{theorem}

\begin{proof} Let $Q'$ be the submonoid of $M\oplus\NN^{\ell}$ generated
by elements of the form $(m_i,d_{ij})$, where $d_{ij}$ is the $(i,j)$-th
generator of $\NN^{\ell}$. Note that $Q'$ itself is
freely generated by these elements. Thus we can define a map
\[
\alpha:Q'\rightarrow Q
\]
by $(m_i,d_{ij})\mapsto (m_i, \varphi(m_i)-E_{ij})$.
The scattering diagram
\[
\foD':=\{(\RR_{\ge 0}m_i,\prod_{j=1}^{\ell_i}(1+z^{(m_i,d_{ij})}))
\,|\, 1\le i\le n\}
\]
then has image under the map $\alpha$ (applying $\alpha$ to each
$f_{\fod}$) the scattering diagram $\ofoD_0$. Thus if we apply
$\alpha$ to each element of $\Scatter(\foD')$, we must get
$\Scatter(\ofoD_0)$, as $\theta_{\gamma,\Scatter(\foD')}$ being the
identity on $\widehat{\kk[Q']}$ implies that
$\theta_{\gamma,\alpha(\Scatter(\foD'))}$ is the identity on
$\widehat{\kk[Q]}$.

To obtain the result, we now note that the set of possible
$\AA^1$-classes in $\tY$ occurring in the expression \eqref{2ndGPSformula}
are precisely the classes $\{\beta_{\bf P}\}$ where ${\bf P}$ runs over
all partitions satisfying \eqref{balancePeq}. Now applying $\alpha$
to a term appearing in \eqref{1stGPSformula} of the form
\[
k_{\bf P}N_{\bf P}t^{\bf P} z^{-k_{\bf P}m_{\fod}}
=k_{\beta_{\bf P}}N_{\beta_{\bf P}}\left( \prod t_{ij}^{p_{ij}}\right)
z^{\sum_{i=1}^p
|{\bf P}_i| m_i},
\]
we get
\[
k_{\beta_{\bf P}}N_{\beta_{\bf P}}
z^{(-k_{\bf P}m_{\fod},\sum_{i=1}^p |{\bf P}_i|\varphi(m_i)-\sum_{i,j}
p_{ij} E_{ij})}.
\]
But by Lemma \ref{H2toric} and \eqref{balancePeq},
\[
\sum_{i=1}^p |{\bf P}_i|\varphi(m_i)-\sum_{i,j} p_{ij}E_{ij}
=\pi_*(\beta_{\bf P})-\varphi(k_{\bf P}m_{\fod}),
\]
hence the result.
\end{proof}

A direct comparison of the formula of the above theorem and the
formula in the definition of the canonical scattering diagram
then yields Theorem \ref{relatingcanandscatter}.

\section{Smoothness: Around the Gross--Siebert locus}
\label{smoothnesssection}
Next we prove that our deformation of $\bV_n$ is indeed a smoothing.
The main theorem of this section (Theorem \ref{injectivitym}) will show
this in the situation of Theorem \ref{maintheoremlocalcase} when $(Y,D)$
has a toric model. The full smoothness statement of
Theorem \ref{maintheoremlocalcase} will require some more work, which will
be carried out in \S\ref{positivecase}.

We prove smoothness by working
over the Gross--Siebert locus (Definition \ref{GSlocusdef}). Here our
deformation (when restricted to one-parameter subgroups
associated to $p^*A$, $A$ an ample divisor on $\oY$)
agrees with the construction of \cite{GS07}. This is important
here because the deformations of \cite{GS07} come with explicit charts
that cover all of $\bV_n$, from which it is clear that they give a
smoothing. So conceptually, the smoothing claim is clear. Because
we work with formal families the actual argument is a bit more delicate.
First we make rigorous the notion of a {\it smooth generic fibre}
for a formal family:

\begin{definition-lemma} \label{singdef} Let
$f: Z \to W$ be a flat finite type morphism of schemes of
relative dimension $d$. Then $\Sing(f) \subset Z$
is the closed embedding defined by the $d^{th}$ Fitting ideal
of $\Omega^1_{Z/W}$. $\Sing(f)$ is empty if and only if
$f$ is smooth.  Formation of $\Sing(f)$ commutes with
all base extensions of $W$.
\end{definition-lemma}

\begin{proof}
For the definition of the Fitting ideal, see e.g., \cite{E95}, 20.4.
The fact that it commutes with base-change follows from
the fact that $\Omega^1_{Z/W}$ commutes with base-change and
\cite{E95}, Cor.~20.5. That $\Sing(f)$ is empty if and only if
$f$ is smooth follows from \cite{E95}, Prop.~20.6 and the
definition of smoothness.
\end{proof}


Now for a formal family, {\it smoothness of the generic fibre} is
measured by the fact that $\Sing(f)$ does not surject
scheme-theoretically onto the base.
More precisely:

\begin{definition}
\label{genericsmoothness}
Let $S$ be a normal variety,
$V \subset S$ a connected closed subset, and $\mathfrak{S}$ the formal completion of $S$ along $V$.
Let $\mathfrak{f} \colon \mathfrak{X} \rightarrow \mathfrak{S}$ be an adic flat morphism of formal schemes of pure relative dimension
and $\mathfrak{Z} \subset \mathfrak{X}$ the scheme theoretic singular locus of $\mathfrak{f}$.
Then we say the \emph{generic fiber of $\mathfrak{f}$ is smooth} if the map $\cO_{\mathfrak{S}} \rightarrow \mathfrak{f}_*\cO_{\mathfrak{Z}}$ is not injective.
\end{definition}

For the statement of Proposition \ref{GSlocus}, we
fix our usual setting of a surface $(Y,D)$, and assume given
Assumptions \ref{GSassumptions2} and that $\nu(\foD)$ is a
scattering diagram in the sense of Definition \ref{scatdiagQdef}.
Suppose furthermore that
$\theta_{\gamma,\nu(\foD)}\equiv 1$ for a loop $\gamma$ around the
origin. Thus by Theorem \ref{wallcrossingtheorem}, $\foD$ is consistent.
These hypotheses on $\foD$ apply in particular when $\foD=\foD^{\can}$.

Let $T^{\gs}$ be the Gross--Siebert locus; we have
$T^{\gs}=\Spec \kk[P]/G$. Consistency of $\foD$ gives a flat family
\[
f_I:X_I\rightarrow\Spec\kk[P]/I
\]
over a thickening of $T^{\gs}$ whenever $\sqrt{I}=G$.

On the other hand, letting $\bar\Sigma$ be the fan for $\oY$
in $\oB=M_{\RR}$, we have the piecewise linear function
$\bar\varphi:\oB\rightarrow P^{\gp}_\RR$ with
$\kappa_{\bar\rho,\bar\varphi}=p^*[\oD_{\bar\rho}]$, as in \eqref{kappabardef}.
This now determines the Mumford family
\[
\bar f_{I}: \oX_{I} \to \Spec\kk[P]/I.
\]

Our goal is to compare these two families. Note that both $X_G
\rightarrow T^{\gs}$ and $\oX_G\rightarrow T^{\gs}$ are the trivial family
$\bV_n \times T^{\gs} \rightarrow T^{\gs}$. Thus either family contains
a canonical copy of $T^{\gs}$, i.e., $\{0\} \times T^{\gs}$,
where $0$ is the vertex of $\bV_n$.

\begin{proposition} \label{GSlocus} In the above situation,
fix an ideal $I$ with $\sqrt{I}=G$. There are open affine sets
$\oU_{I} \subset \oX_{I}$,
$U_{I} \subset X_{I}$, both sets containing
the canonical copy of $T^{\gs}$, and an isomorphism
\[
\mu_{I}: U_{I} \to \oU_{I}
\]
of families over $\Spec\kk[P]/I$.

Moreover, there is a non-zero monomial $y \in \kk[P]$
whose pullback to $X_{I}$ is in the stalk at any
point $x\in \{0\}\times T^{\gs}\subset \VV_n\times T^{\gs}$ of the ideal
of $\Sing_{f_{I}}$ for all $I$.
\end{proposition}

\begin{proof}
The generic fibre of the Mumford family over $\Spec \kk[P]$
is smooth: indeed the family is trivial over the
open torus orbit of $\Spec\kk[P]$,
with fibre an algebraic torus. It follows that there is a non-zero
monomial $y \in \kk[P]$ in the ideal of $\Sing_{\bar f}$ for the
global Mumford family
\[
\bar f: \Spec \kk[P_{\bar\varphi}] \to \Spec \kk[P].
\]
Of course its restriction then lies in the stalk at any point $x$ of the ideal
sheaf of $\Sing_{\bar f_{I}}$ for all $I$.
Thus once we establish the claimed isomorphisms, the final statement follows.

Recall from \S\ref{GSreductionsection} the construction
of $\ofoD:=\nu(\foD)$ and the scheme
$\oX^o_{I,\ofoD}$ from $\ofoD$. By Lemma
\ref{movingsingularitylemma},
$X^o_{I,\foD}\cong \oX^o_{I,\ofoD}$, so we in fact
have an isomorphism
\[
X_I\cong\Spec \Gamma(\oX^o_{I,\ofoD},\shO_{\oX^o_{I,\ofoD}})
=:\oX_{I,\ofoD}.
\]
So we can work with
$\oX_{I,\ofoD}$ instead of $X_{I}$. On the other hand,
the Mumford family $\Spec\kk[P_{\bar\varphi}]/I\kk[P_{\bar\varphi}]$
over $\Spec \kk[P]/I$ can be described similarly.
Using the empty scattering diagram instead of the scattering diagram
$\ofoD$, one has by Lemma \ref{opensubschemeGKZ}
\[
\oX_{I,\emptyset}=\Spec \Gamma(\oX^o_{I,\emptyset},\shO_{\oX^o_{I,\emptyset}}).
\]

Now define an ideal $I_0\subset P_{\bar\varphi}$ as follows.
For $\sigma\in\oSigma_{\max}$, let $\bar\varphi_{\sigma}$ denote
the linear extension of $\bar\varphi|_{\sigma}$. We set
\[
I_0:=\{(m,p)\in P_{\bar\varphi}\,|\, \hbox{$p-\bar\varphi_{\sigma}(m)
\in I$ for some $\sigma\in\oSigma_{\max}$}\}.
\]
Note that $\sqrt{I_0}=\fom_{P_{\bar\varphi}}$.
By assumption,
$\ofoD$ is a scattering diagram
for $P_{\bar\varphi}$, and hence there are only a finite number of
$(\fod,f_{\fod})\in\ofoD$ for which $f_{\fod}\not\equiv 1\mod I_0$.
Furthermore, modulo $I_0$, each $f_{\fod}$ is a polynomial.

Let $\ofoD_{I}$ be the scattering diagram obtained from $\ofoD$ by, for each
outgoing ray $(\fod,f_{\fod})$, truncating
each $f_{\fod}$ by throwing out all terms which lie in $I_0$. The
incoming rays remain unchanged. Thus
$\ofoD_{I}$ can be viewed as a finite scattering diagram. Let
\[
h:=\prod_{\fod\in\ofoD_{I}} f_{\fod}.
\]
This is an element of $\kk[P_{\bar\varphi}]$. Note that necessarily
$h\equiv 1\mod \fom_{P_{\bar\varphi}}$. Thus $h\not=0$
defines an open subset $\oU\subset \Spec \kk[P_{\bar\varphi}]/G
\kk[P_{\bar\varphi}]
=\oX_{G,\ofoD}
=\oX_{G,\emptyset}=\VV_n\times T^{\gs}$. Furthermore, $\oU$ contains the
canonical copy of $T^{\gs}$.
Since $\oX_{I,\ofoD}$ and $\oX_{I,\emptyset}$ both have underlying topological
space $\oX_{G}$, this defines open sets $\oU_{I,\ofoD}$ of $\oX_{I,\ofoD}$
and $\oU_{I,\emptyset}$ of $\oX_{I,\emptyset}$.
We shall show these two open subschemes are isomorphic.

First the following claim shows that $\oX_{I,\ofoD}\cong \oX_{I,\ofoD_I}$,
as for any $\tau\in\oSigma\setminus\{0\}$,
the automorphisms involved will have the same effect modulo $I_{\tau}$.
As a consequence, we can work with the scattering
diagram $\ofoD_I$.

\begin{claim}
Let $\tau\in\oSigma$, and suppose $(m,p)\in P_{\bar\varphi}$
satisfies $-m\in\tau$. Then $(m,p)\in I_0$ if and only if $(m,p)\in I_{\tau}$,
where
\[
I_{\tau}:=\{(m,p)\in P_{\bar\varphi_{\tau}}\,|\,
\hbox{$p-\bar\varphi_{\sigma}(m)\in I$ for some $\sigma\in\oSigma_{\max}$
with $\tau\subset\sigma$}\}.
\]
\end{claim}

\begin{proof}[Proof of claim]
Clearly $I_{\tau}\cap P_{\ovarphi}\subset I_0$, so one implication
is clear. Conversely, suppose that $(m,p)\in I_0$,
so that $p-\ovarphi_{\sigma}(m)\in I$ for some $\sigma\in\Sigma_{\max}$.
If $\tau\subset\sigma'\in\Sigma_{\max}$, let $\rho_1,\ldots,\rho_n$ be the
sequence of rays traversed in passing from $\sigma$ to $\sigma'$, chosen
so that all $\rho_1,\ldots,\rho_n$ lie in a half-plane bounded
by the line $\RR m$. Then
\[
\ovarphi_{\sigma'}(m)=\ovarphi_{\sigma}(m)+\sum_{i=1}^n \langle n_{\rho_i},
m\rangle \kappa_{\rho_i,\ovarphi},
\]
with $n_{\rho_i}$ primitive, vanishing on $\rho_i$, and positive
on $\rho_{i+1}$.
Note that since $-m\in\tau$, we must have $\langle n_{\rho_i},m\rangle\le
0$ for each $i$, and hence $p-\ovarphi_{\sigma'}(m)=p-\ovarphi_{\sigma}(m)
+p'$ for some $p'\in P$. Hence $(m,p)\in I_{\tau}$.
\end{proof}

To show that $\oU_{I,\ofoD}$ and $\oU_{I,\emptyset}$ are isomorphic,
let us describe these open subschemes explicitly away from the origin.
Recall that $\oX^o_{I,\foD}$ is obtained by gluing together schemes
which are spectra of rings $\oR_{\tau,I}$ for $\tau\in\oSigma$.
However in the case that $\dim\tau=1$, this
ring depends on the scattering diagram, so we write $\oR_{\tau,I,\foD}$
for $\foD=\ofoD$ or $\emptyset$.

If $\dim\tau=2$, then $\oR_{\tau,I,\foD}=
\kk[P_{\bar\varphi_{\tau}}]/I\kk[P_{\bar\varphi_{\tau}}]$. Since
$h\in \kk[P_{\bar\varphi}]\subset \kk[P_{\bar\varphi_{\tau}}]$, $h$
defines an element of $\oR_{\tau,I,\foD}$ in this case.

If $\dim\tau=1$, then
$\tau=\rho_i$ for some $i$, and we have
a surjection
\[
\oR_{\rho_i,I,\foD}\rightarrow \oR_{\rho_i,G,\foD} =
R_G[\oX_{i-1},\oX_i^{\pm 1},
\oX_{i+1}]/(\oX_{i-1}\oX_{i+1})\cong \kk[P_{\bar\varphi_{\rho_i}}]/G
\kk[P_{\bar\varphi_{\rho_i}}],
\]
so that $h\in \kk[P_{\bar\varphi}]\subset \kk[P_{\bar\varphi_{\rho_i}}]$ defines
an element of $\oR_{\rho_i,G,\foD}$. Choosing any lift of $h$ to
$\oR_{\rho_i,I,\foD}$, we note the localization $(\oR_{\rho_i,I,\foD})_h$
is independent of the
lift since the kernel of the above surjection is nilpotent.

We can then define regardless of $\dim\tau$,
\[
S_{\tau,I,\foD}:=(\oR_{\tau,I,\foD})_h.
\]

Note there is an isomorphism
\[
\psi_i:S_{\rho_i,I,\ofoD}  \rightarrow S_{\rho_i,I,\emptyset}
\]
given by
\[
\bar X_{i-1} \mapsto \bar X_{i-1}\left(\bar g_{\rho_i}\prod_{j=1}^{\ell_i}
(1+b_{ij}^{-1}\oX_i)\right),
\quad
\bar X_i\mapsto \bar X_i,\quad
\bar X_{i+1}\mapsto \bar X_{i+1}.
\]
This has an inverse because of the localization at $h$.

Given a path $\gamma$ in $M_{\RR}\setminus \{0\}$, note that by construction
of $h$, $\theta_{\gamma,\ofoD_{I}}$ makes sense as an automorphism
of the localization $\kk[P_{\bar\varphi}]_h$, since to define
the automorphism associated with crossing a ray $(\fod,f_{\fod})$, we only
need $f_{\fod}$ to be invertible. However since by construction
$h$ is divisible by $f_{\fod}$, $f_{\fod}$ is invertible. In particular,
$\theta_{\gamma,\ofoD_{I}}$ also makes sense as an automorphism
of $(\kk[P_{\bar\varphi_{\tau}}]/I\kk[P_{\bar\varphi_{\tau}}])_h
=S_{\tau,I, \emptyset}$ for any $\tau\in \oSigma\setminus \{0\}$. Thus using the
equality $S_{\sigma,I,\emptyset}=S_{\sigma,I,\ofoD}$ for $\dim
\sigma=2$
we see
that $\theta_{\gamma,\ofoD_{I}}$ also makes sense as an automorphism
of $S_{\sigma,I,\ofoD}$.

Choose an orientation on $M_{\RR}$, labelling the rays $\rho_1,\ldots,
\rho_n$ of $\oSigma$ in a counterclockwise order, with $\sigma_{i-1,i}$
as usual the maximal cone containing $\rho_{i-1}$ and $\rho_i$. For two distinct
points $p,q$ on the unit circle in $M_{\RR}$ not contained in
$\Supp(\ofoD_{I})$, let
$\gamma_{p,q}$ be a \emph{counterclockwise} path from $p$ to $q$, and
write $\theta_{p,q}$ for $\theta_{\gamma_{p,q},\ofoD_{I}}$
acting on any of the rings $S_{\tau,I,\emptyset}$.

For each $\rho_i$, let $p_{i,+}$ be a point on this unit circle
contained in the connected component of $\sigma_{i,i+1}\setminus
\Supp(\ofoD_{I})$ adjacent to $\rho_i$, and
$p_{i,-}$ a point in this unit circle contained in
the connected component of $\sigma_{i-1,i}\setminus
\Supp(\ofoD_{I})$ adjacent to $\rho_i$.

Choose a base-point $q$ on the unit circle not in
$\Supp(\ofoD_{I})$.

Recall in the construction of $\oX^o_{I,\foD}$, the open sets
$\Spec \oR_{\rho_{i+1},I,\foD}$ and $\Spec \oR_{\rho_i,I,\foD}$
are glued together along the common open set $\Spec
\oR_{\sigma_{i,i+1},I,\foD}$ using the
trivial automorphism or the  automorphism $\theta_{p_{i,+},p_{i+1,-}}$
in the cases $\foD=\emptyset$ or $\foD=\ofoD$ respectively. After
localizing at $h$, we have a commutative diagram
\[
\xymatrix@C=30pt
{
S_{\rho_i,I,\ofoD}\ar[rr]^{\theta_{p_{i,+},q}\circ\psi_i}
\ar[dr]_{\theta_{p_{i,+},p_{i+1,-}}\circ\psi_{\rho_i,+}}&&
S_{\rho_i,I,\emptyset}\ar[rd]^{\psi_{\rho_i,+}}&\\
&S_{\sigma_{i,i+1},
I,\ofoD}\ar[rr]^{\theta_{p_{i+1,-},q}}&&S_{\sigma_{i,i+1},I,\emptyset}\\
S_{\rho_{i+1},I,\ofoD}\ar[ur]^{\psi_{\rho_{i+1},-}}
\ar[rr]_{\theta_{p_{i+1,+},q}\circ\psi_{i+1}}&&S_{\rho_{i+1},I,\emptyset}
\ar[ur]_{\psi_{\rho_{i+1},-}}}
\]
Here the maps $\psi_{\rho,\pm}$ are the ones
defined in Proposition \ref{ringfirstversion} and \eqref{psirhoi}.
This shows that the isomorphisms $\theta_{p_{i,+}}\circ\psi_i$ between
$\Spec S_{\rho_i,I,\ofoD}$
and $\Spec S_{\rho_i,I,\emptyset}$
are compatible with the gluings, and hence give
an isomorphism between $\oU_{I,\foD}\setminus (\{0\}\times T^{\gs})$ and
$\oU_{I,\emptyset}
\setminus(\{0\}\times T^{\gs})$.

Now $\VV_n$ satisfies Serre's condition $S_2$.
Since $\oX_{I}$ and $X_{I}$ are flat deformations
of $\VV_n \times T^{\gs}$, by Lemma~\ref{relS2} the above isomorphism
extends across the codimension two set $\{0\}\times T^{\gs}$, giving the
desired isomorphism between $U_{I}$ and $\oU_{I}$.
\end{proof}

We now need to use the above observations along the Gross-Siebert locus
to obtain results about deformations away from the Gross-Siebert locus.
For the remainder of the section, we work with data $(Y,D)$, $\eta$, $P$,
but now as in Assumptions
\ref{GSassumptions}. Furthermore, we take $\foD=\foD^{\can}$.
Thus if we take $I$ an ideal with either
$\sqrt{I}=G$ or $\sqrt{I}=\fom$, we obtain a flat family
$X_I\rightarrow\Spec R_I=:S_I$, and in the former case, $X_I\rightarrow
S_I$ restricts to the open subscheme of $\Spec R_I$ whose underlying
open subset is $T^{\gs}$, giving the family over the thickening of the
Gross-Siebert locus.

With $J=\fom$ or $G$, let
$\mathfrak{f}_J \colon \mathfrak{X}_J \rightarrow \mathfrak{S}_J$
denote the formal deformation determined by the deformations
$X_{J^{N+1}} \rightarrow \Spec R_{J^{N+1}}$ for $N \ge 0$.
Thus $\mathfrak{S}_J= \Spf(\varprojlim \kk[P]/J^{N+1})$
is the formal spectrum of the $J$-adic completion of $\kk[P]$, $\mathfrak{X}_J$ is a formal scheme, and $\mathfrak{X}_J \rightarrow \mathfrak{S}_J$ is an adic flat morphism of formal schemes.
We refer to \cite{G60} for background on formal schemes.

Let $Z_I:=\Sing(f_I) \subset X_I$ denote the singular locus of $f_I \colon X_I \rightarrow S_I$. Thus $Z_I \subset X_I$ is a closed embedding of schemes.
Since the singular locus is compatible with base-change, the singular loci $Z_{J^n} \subset X_{J^n}$ determine a closed embedding $\mathfrak{Z}_J \subset \mathfrak{X}_J$ which we refer to as the singular locus of $\fof_J \colon \mathfrak{X}_J \rightarrow \mathfrak{S}_J$.

Again, with $J=\fom$ or $G$, we have a section
$s \colon S_J \rightarrow X_J=S_J\times\VV_n$
given by $s(t)=t\times \{0\}$ for $t\in S_J$.
We write $X_J^o := X_J \setminus s(S_J) \subset X_J$ and
$X_I^o \subset X_I$, $\mathfrak{X}_J^o \subset \mathfrak{X}_J$ for the induced open embeddings.

\begin{lemma} \label{reducetofinite}
In the above situation,
there exists $0 \neq g \in \kk[P]$ such that $\Supp(g \cdot \cO_{\mathfrak{Z}_J})$ is contained in $s(S_J)$.
In particular, $\fof_{J*}(g \cdot \cO_{\mathfrak{Z}_J})$ is a coherent sheaf on $\mathfrak{S}_J$.
\end{lemma}

\begin{proof}
We can write an explicit open covering $\{\foU_{i,J}\}$ of
$\mathfrak{X}_J^o$ in the two cases $J=\fom$ or $J=G$, as follows.
Write $a_i=z^{[D_{i}]}$ and $m_i=-D_i^2$.
In the case $J=\fom$,
\begin{equation}
\label{mathfrakU}
\mathfrak{U}_{i,J}=V(X_{i-1}X_{i+1}-a_iX_i^{m_i}) \subset
\bA^2_{X_{i-1},X_{i+1}} \times (\bG_m)_{X_i} \times \mathfrak{S}_J.
\end{equation}
In the case $J=G$,
\[
\mathfrak{U}_{i,J}=V(X_{i-1}X_{i+1}-a_iX_i^{m_i}\prod(1+b_{ij}X_i^{-1}))
\subset \bA^2_{X_{i-1},X_{i+1}} \times (\Gm)_{X_i} \times \mathfrak{S}_J,
\]
with $b_{ij}=z^{[E_{ij}]}$ as usual.

We now use the charts $\mathfrak{U}_{i,J}$ to compute the singular locus explicitly.
In the case $J=\fom$,
the singular locus $\mathfrak{Z}_{i,J}$ of $\mathfrak{U}_{i,J}/\mathfrak{S}_J$ is given by
\[
\mathfrak{Z}_{i,J}=V(X_{i-1},X_{i+1},a_i) \subset \mathfrak{U}_{i,J}.
\]
Hence if we define
$g=a_1 \cdots a_n$ then
$\Supp(g \cdot \cO_{\mathfrak{Z}_J})$ is contained in $s(S_J)$.

Similarly, if $J=G$, the structure sheaf of the singular locus of $\mathfrak{U}_{i,J}$ is annihilated by $g_i:= a_i\prod_{j \neq k}(b_{ij}-b_{ik})$. (Here $\prod_{j \neq k}(b_{ij}-b_{ik})$ is the discriminant of the polynomial $f(X_i):=\prod(X_i+b_{ij})$. It is a linear combination of $f$ and $f'$ with coefficients in $\kk[\{b_{ij}\}][X_i]$. See \cite{L02}, p.~200--204.) So we can take $g=g_1 \cdots g_n$.

The support of $g \cdot \cO_{\mathfrak{Z}_J}$ is a closed subset of $s(S_J)$, hence proper over $S_J$.
It follows that $\fof_{J*}(g \cdot \cO_{\mathfrak{Z}_J})$ is coherent by \cite{G61}, 3.4.2.
\end{proof}

\begin{theorem}\label{injectivitym}
Let $(Y,D), \eta, P$ satisfy Assumptions \ref{GSassumptions}, and take
$\foD=\foD^{\can}$. Then
the maps $\kk[P]\rightarrow {\fof_{\fom}}_*\cO_{\foZ_\fom}$ and
$\cO_{\mathfrak{S}_\fom} \rightarrow
{\mathfrak{f}_\fom}_*\cO_{\mathfrak{Z}_\fom}$ are not injective, so
the generic fibre of $\fof_\fom$ is smooth in the sense of Definition
\ref{genericsmoothness}.
This also implies that for $I=\fom^{N+1}$ and $N \gg 0$, the map
$\kk[P]/I \rightarrow f_{I*}\cO_{\Sing(f_I)}$ is not injective.
\end{theorem}

\begin{proof}
By Lemma~\ref{reducetofinite} there exists $0 \neq g \in \kk[P]$ such that $\Supp(g \cdot \cO_{\mathfrak{Z}_G}) \subset s(S_G)$.
Let $E$ be the subgroup of $P^{\gp}$ generated by $P\setminus G$, so that
$U=\Spec \kk[P+E]$ is an open subset of $\Spec \kk[P]$. Denote by $\foS_G^o$
the open subset of $\foS_G$ isomorphic to the completion of $U$
along the subscheme defined by $G+E$. This is the formal thickening
of the Gross-Siebert locus $T^{\gs}$.
By Proposition~\ref{GSlocus} there then exists $0 \neq h \in \kk[P+E]$
such that $\Supp(h \cdot \cO_{\mathfrak{Z}_G\cap\fof_G^{-1}(\foS_G^o)})$ is disjoint from $s(S_G\cap U)$.
By multiplying $h$ by a monomial whose exponent lies in $P\setminus G$, we
can assume that $h\in\kk[P]$.
Thus $gh \cdot \cO_{\mathfrak{Z}_G}$ has support in the closed subset
$s(S_G\setminus (S_G\cap U))$.
Since this sheaf is coherent, there exists a non-zero element
$k\in\fom\subset\kk[P]$  such that $ghk\cdot
\cO_{\foZ_G}=0$. Noting by construction that
$ghk\in \kk[P]$, we have $\kk[P]\rightarrow \Gamma(\foZ_{G},\shO_{\foZ_G})$
is not injective, hence the composition $\kk[P]\rightarrow
\Gamma(\foZ_G,\shO_{\foZ_G})\rightarrow \Gamma(\foZ_{\fom},\shO_{\foZ_{\fom}})$
is not injective. Since $\kk[P]\subset \Gamma(\foS_{\fom},\shO_{\foS_{\fom}})$,
$\shO_{\foS_{\fom}}\rightarrow \fof_{\fom*}\shO_{\foZ_{\fom}}$ is not
injective.
\end{proof}

\section{The relative torus}
\label{relativetorussection}
The flat deformations $X_{I,\foD^{\can}}\rightarrow \Spec \kk[P]/I$
produced by the canonical scattering diagram have a useful
special property: there is a natural torus action on the total
space $X_{I,\foD^{\can}}$ compatible with a torus action on
the base.  The meaning of this action will be clarified in
Part II, where we will prove that our family extends naturally, in the positive case,
to a universal family of Looijenga pairs $(Z,D)$ together with
a choice of isomorphism $D \stackrel{\sim}{\rightarrow} D_*$, where $D_*$ is a fixed $n$-cycle.
The torus action then corresponds to changing the choice of isomorphism.

Fixing the pair $(Y,D)$ as usual, $D=D_1+\cdots+D_n$,
let $\bA^D = \bA^n$ be the affine space with one coordinate
for each component $D_i$. Let $T^D$ be the diagonal
torus acting on $\AA^D$, i.e., the torus $T^D$ whose character group
\[
\chi(T^D) = \bZ^D
\]
is the free module with basis $e_{D_1},\ldots,e_{D_n}$.
\begin{definition} \label{weights}
We define a canonical map $w: A_1(Y) \to \chi(T^D)$ given by
$$
C \mapsto \sum (C \cdot D_i) e_{D_i}.
$$
\end{definition}
Suppose $P\subset A_1(Y)$ is a toric submonoid containing
$\NE(Y)$. We then get an action of $T^D$ on $\Spec \kk[P]$,
as well as on $\Spec \kk[P]/I$ for any monomial ideal $I$,
and hence also on $\Spf(\widehat{\kk[P]})$ for any completion
of $\kk[P]$ with respect to a monomial ideal.

We can also define a unique piecewise linear map
\[
w: B \to \chi(T^D)\otimes_{\ZZ}\RR
\]
with $w(0)=0$ and $w(v_i) = e_{D_i}$, for $v_i$ the primitive generator
of the ray $\rho_i$.

\begin{theorem}
\label{TDequivariance}
Let $I$ be an ideal
for which $X_{I,\foD^{\can}}\rightarrow\Spec\kk[P]/I$ is defined.
Then $T^D$ acts equivariantly on $X_{I,\foD^{\can}}\rightarrow\Spec\kk[P]/I$.
Furthermore, each theta function $\vartheta_q$, $q \in B(\bZ)$, is
an eigenfunction of this action, with character $w(q)$.
\end{theorem}

\begin{proof} It's enough to check this on the
open subset $X^o_{I,\foD^{\can}}\subset X_{I,\foD^{\can}}$.
We have a cover of $X^o_{I,\foD^{\can}}$ by open sets
the hypersurfaces
\[
U_{\rho_i,I}\subset\bA^2_{X_{i-1},X_{i+1}} \times (\Gm)_{X_i}\times
\Spec R_I
\]
given by the equation
\[
X_{i-1} X_{i+1} = z^{[D_i]} X_i^{- D_i^2} f_{\rho_i},
\]
where $f_{\rho_i}$ is the function attached to the ray
$\rho_i$ in $\foD^{\can}$. If we act on $X_j$ with weight
$w(v_j)$ and on $z^p$ with weight $w(p)$ (for $p\in P$), then
we note that for every $(\fod,f_{\fod})\in \foD^{\can}$,
every monomial in $f_{\fod}$ has weight zero by the explicit
description of $f_{\fod}$ in Definition \ref{canscatdiagdef}.
In particular, the equation defining $U_{\rho_i,I}$ is
clearly $T^D$-equivariant, and each of the monomials
is an eigenfunction.

Now $X^o_{I,\foD^{\can}}$ is obtained by gluing
$U_{\rho_i,\sigma_{i,i+1},I}\subset U_{\rho_i,I}$
with $U_{\rho_{i+1},\sigma_{i,i+1},I}\subset U_{\rho_{i+1},I}$,
using scattering automorphisms of $\foD^{\can}$,
and these open sets are naturally identified with
$(\Gm)^2_{X_i,X_{i+1}}\times \Spec R_I$.
The scattering automorphisms commute with the $T^D$ action,
by the fact that the scattering functions have weight zero.
Thus $T^D$ acts equivariantly on $X_{I,\foD^{\can}}\rightarrow
\Spec\kk[P]/I$.

Now we check our canonical global function $\vartheta_q$ is an eigenfunction,
with character $w(q)$. By construction, given a broken line $\gamma$,
the weights of monomials attached to adjacent domains of linearity
are the same, since the functions in the scattering diagram are
of weight zero. Thus the weight of $\Mono(\gamma)$ only depends on $q$.
This weight can be determined by fixing the base point $Q$ in a cone $\sigma$
which contains $q$, in which case the broken line for $q$ which doesn't
bend and is wholly contained in $\sigma$ yields the monomial
$z^{\varphi_{\sigma}(q)}$, which has weight $w(q)$.
Thus $\vartheta_q$ is an eigenfunction with weight $w(q)$.
\end{proof}

\section{Extending the family over boundary strata}
\label{positivecase}
Here we prove Theorem \ref{maintheoremlocalcase} and
Theorem \ref{extensiontheorem}. Let us review what we know so far.
For any pair $(Y,D)$, we know that $\foD^{\can}$ is consistent by
Theorem \ref{canonicalconsistency}. Thus, if the number $n$ of irreducible
components of $D$ satisfies $n\ge 3$, Theorem
\ref{univfamth} gives the construction of $f:X_I\rightarrow \Spec R_I$
of Theorem \ref{maintheoremlocalcase}. The algebra structure on $A_I$
has structure constants given by counts of broken lines as in Theorem
\ref{multrule}. The $T^D$ equivariance is
given by Theorem \ref{TDequivariance}. If furthermore $(Y,D)$ has a toric
model, then the smoothness statement follows from Theorem \ref{injectivitym}.

We will give a proof of Theorem \ref{extensiontheorem} and the remaining
cases of \ref{maintheoremlocalcase} by first proving Theorem
\ref{extensiontheorem} in the case that we know that we have the desired
algebra structure on $A_I$, and then bootstrap to the general case for
both theorems.

\subsection{Theorem \ref{extensiontheorem} in the case that $(Y,D)$ has
a toric model}
As usual, let $P$ be the toric monoid associated to a strictly convex
rational polyhedral cone $\sigma_P \subset A_1(Y)_{\bR}$ which
contains the Mori cone $\NE(Y)_{\bR}$. We have $\fom=P\setminus
\{0\}$.
For a monomial ideal $I \subset P$ we define
\[
A_I := \bigoplus_{q \in B(\bZ)} R_I \cdot \vartheta_q
\]
where $R_I = \kk[P]/I$.
We take throughout $\foD = \foD^{\can}$.

\begin{assumptions} \label{ringass}
For any
monomial ideal $I$ with $\sqrt{I} = \fom$, the
multiplication rule of Theorem~\ref{multrule} defines an
$R_I$-algebra structure on $A_I$, so that
$A_I \otimes_{R_I} R_{\fom} = H^0(\bV_n,\cO_{\VV_n})$.
\end{assumptions}

Note we have already shown that Assumptions~\ref{ringass} hold if
$n\ge 3$ by Theorems \ref{univfamth}, \ref{multrule} and
\ref{canonicalconsistency}.

Let $\Gamma \subset B(\bZ)$ be
a finite collection of integral points such that the corresponding
functions $\vartheta_q$ generate the $\kk$-algebra $H^0(\bV_n,\cO_{\VV_n})$.
(Then the $\vartheta_q$, $q \in \Gamma$ generate $A_I$ as an $R_I$-algebra if $\sqrt{I}=\fom$ and Assumptions~\ref{ringass} hold.)
Note for $n \geq 3$ we can take for $\Gamma$ the points $\{v_i\}$,
and for $n=1,2$, one can make a simple choice for $\Gamma$,
see \S\ref{mtsmalln}.

\begin{lemma} For any monomial ideal $J\subset P$,
$\bigcap_{k > 0} (J + \fom^k) = J$.
\end{lemma}

\begin{proof} The inclusion $\supset$ is obvious. For the other
direction, as the intersection is a monomial ideal, it's enough to
consider a monomial in the intersection. But notice that a monomial
is in $J +\fom^k$ iff it is either in $J$ or in $\fom^k$. The
result follows since $\bigcap \fom^k =0$.
\end{proof}

Assuming \ref{ringass},
let $\cA$ be the collection of monomial ideals $J \subset P$
with the following properties:
\begin{enumerate}
\item There is an $R_J$-algebra structure on
$A_J$ such that the canonical isomorphism of $R$-modules
$A_J \otimes_{R_J} R_{I+J} = A_{I+J}$ is an algebra isomorphism, for
all $\sqrt{I} = \fom$.
\item $\vartheta_q, q \in \Gamma$ generate $A_J$ as an $R_J$-algebra.
\end{enumerate}

By the lemma, the algebra structure in (1) is unique if it
exists. The algebra structure on all $A_I$ determines such
a structure on
$\widehat{A} := \invlim_{\sqrt{I} = \fom } A_I,$
$\widehat{A}_J := \invlim_{\sqrt{I} = \fom} A_{I + J}.$
Also, there are canonical inclusions
\begin{align*}
\widehat{A} &\subset \prod_{q \in B(\bZ)} \hR \cdot \vartheta_q \\
\widehat{A}_J &\subset \prod_{q \in B(\bZ)} \hR_J \cdot \vartheta_q
\end{align*}
where $\hR$ is the completion of $R$ at $\fom$ and $\hR_J=\invlim R/(I+J)$
the inverse limit over all ideals $I$ with $\sqrt{I}=\fom$.
Here the direct products are viewed purely as $\hR,\hR_J$ modules. We can
also view
\[
A_J := \bigoplus_{q \in B(\bZ)} R_J \cdot \vartheta_q \subset
\prod_{q \in B(\bZ)} \hR_J \cdot \vartheta_q.
\]
It is clear that $A_J \subset \widehat{A}_J$ (as submodules of the direct product).
Thus (1) holds if and only if the following holds:
\begin{itemize}
\item[($1'$)]
For each $p,q \in B(\bZ)$, at most finitely
many $z^{C}\vartheta_s$ with $C \not \in J$ appear in the product expansion
of Theorem \ref{multrule} for
$\vartheta_p \cdot \vartheta_q \in \widehat{A}_J$.
\end{itemize}

\begin{lemma} \label{intlemma} If $J\in\cA$ and $J\subset J'$, then
$J'\in\cA$. In addition, $\cA$ is closed under finite intersections.
\end{lemma}

\begin{proof} The first statement is clear.
Now assume $J_1,J_2 \in \cA$. It's clear that
($1'$) holds for $J_1 \cap J_2$, so $A_{J_1 \cap J_2}$ is an algebra.
Moreover we have an exact sequence of $\kk$-modules
\[
0 \to A_{J_1 \cap J_2} \to A_{J_1} \times A_{J_2} \to A_{J_1 + J_2} \to 0
\]
exhibiting $A_{J_1\cap J_2}$ as the fibre product $A_{J_1}\times_{A_{J_1
+J_2}} A_{J_2}=:A_1\times_B A_2=:A$. We now show this
fibre product is a finitely
generated $\kk$-algebra. Indeed, note that since the maps $A_1,A_2\rightarrow
B$ are surjective, so are the maps $A\rightarrow A_i$.
Let $\{u_i\}$ be a generating set for the ideal $\ker(A_2\rightarrow
B)$.
Since $A_i$ is Noetherian, one can find a finite such set.
Note that $\tilde u_i:=(0,u_i)\in A$. In addition, choose finite sets
$\{x_i\}$, $\{y_j\}$ generating $A_1$ and $A_2$ as $\kk$-algebras.
For each of these elements, choose a lift to $A$, giving a finite set of
lifts $\{\tilde u_i, \tilde x_i,\tilde y_i\}$, which we
claim generate $A$. Indeed, given $(x,y)\in A$, one can subtract
a polynomial in the $\tilde x_i$'s to obtain $(0,y')$. Necessarily
$y'\in\ker (A_2\rightarrow B)$, and hence we can write $y'=\sum f_i u_i$
with $f_i$ a polynomial in the $y_i$'s. Let $\tilde f_i$ be the
same polynomial in the $\tilde y_i$'s. Then $\sum \tilde f_i\tilde{u}_i=(0,y')$,
showing generation.

Thus $A_{J_1 \cap J_2}$ is also a finitely generated
$R_{J_1\cap J_2}$-algebra. Now the generation statement
follows from Lemma \ref{gen}, taking $R=R_{J_1\cap J_2}$, $S=A_{J_1\cap J_2}$,
$I=J_1/J_1\cap J_2$, $J=J_2/J_1\cap J_2$, $\Gamma=\{q_1,\ldots,q_m\}$, and the
map $R[T_1,\ldots,T_m]\rightarrow S$ given by $T_i\mapsto \vartheta_{q_i}$.
\end{proof}

\begin{lemma} \label{gen} Let $I,J \subset R$ be ideals in
a Noetherian ring, with
$I \cdot J =0$, and let $S$ be a finitely generated $R$-algebra,
and $R[T_1,\dots,T_m] \to S$ an $R$-algebra map which is surjective
modulo $I$ and $J$. Then the map is surjective.
\end{lemma}

\begin{proof} The associated map $\Spec S \to \AA^m\times\Spec R$
is proper, as can be easily checked using the valuative criterion for
properness. Indeed, any map $S\rightarrow K$ for a field $K$ factors
through either $S/IS$ or $S/JS$. Since this is a map of affine schemes,
$S$ is a finite $R[T_1,\dots,T_m]$-module. Now we
can apply Nakayama's lemma.
\end{proof}

\begin{proposition}
\label{minradideal}
There is a unique minimal
radical monomial ideal $I_{\min} \subset P$ such that (1) and (2) hold
for any monomial ideal $J$ with $I_{\min} \subset \sqrt{J}$.
\end{proposition}

\begin{proof}
Certainly any ideal $J$ with $\sqrt{J}=\fom$ lies in $\cA$.
Note that a radical monomial ideal is the complement of a union of
faces of $P$, so there are only a finite number of such ideals.
Suppose $I_1, I_2$ are two radical ideals such that
$J_i\in \cA$ for any $J_i$ with $I_i\subset\sqrt{J_i}$. Note that
any ideal $J$ with $I_1\cap I_2\subset \sqrt{J}$ can be written
as $J_1\cap J_2$, with $I_i\subset \sqrt{J_i}$. (Indeed, we can use
the primary decomposition of $J$. If
$J=\bigcap_k \fop_k$ is an intersection of primary ideals,
necessarily the prime ideal
$\sqrt{\fop_k}$ contains either $I_1$ or $I_2$ for each
$k$. Then let $J_1$ be the intersection of those $\fop_k$ whose
radical contains $I_1$ and
$J_2$ be the intersection of those $\fop_k$ whose radical contains $I_2$.)
Thus by Lemma \ref{intlemma}, $J\in \cA$. This shows the existence of
$I_{\min}$.
\end{proof}

\begin{proposition} \label{extensionprop}
Assume \ref{ringass}.
\begin{enumerate}
\item Suppose the intersection matrix $(D_i \cdot D_j)$ is not negative semi-definite. Then $I_{\min}= (0) \subset \kk[P]$.
\item Suppose $F \subset \sigma_P$ is a face such that $F$ does not contain the class of every component of $D$. Then $I_{\min} \subset P \setminus F$.
\end{enumerate}
\end{proposition}

\begin{proof}
We prove both cases simultaneously, writing $F:=P$ in case (1).
We claim there exists an effective divisor $W=\sum a_iD_i$
with support $D$ such that $W \cdot D_j > 0$ for all $D_j$ contained in $F$
and $a_i>0$ for all $i$.
For case (1), see Lemma~\ref{pospair}. In case (2), say $[D_1] \notin F$. Then we can take
$a_1 \gg a_2 \gg \cdots \gg a_n > 0$.

The algebra structure depends only on the deformation type
of $(Y,D)$.
By Proposition 4.1 of \cite{GHK12},
we may replace $(Y,D)$ by a deformation equivalent pair such that any irreducible
curve $C \subset Y$ intersects $D$.

Let $\oNE(Y)_{\bR} \subset A_1(Y,\bR)$ denote the closure of $\NE(Y)_{\bR}$.
Let $F':=\oNE(Y)_{\bR} \cap F$, a face of $\oNE(Y)_{\bR}$.
Define $\Delta=D-\epsilon W$, $0<\epsilon \ll 1$.
Then $(Y,\Delta)$ is KLT (Kawamata log terminal).

We claim $K_Y+\Delta \sim -\epsilon W$ is negative on $F' \setminus \{0\}$.
By construction $(K_Y+\Delta) \cdot D_j < 0$ for $[D_j] \in F'$ and $(K_Y+\Delta)\cdot C < 0$ for $C \not\subset D$.
Let $N$ be a nef divisor such that $F'=\oNE(Y)_{\bR} \cap N^{\perp}$.
Then $aN-(K_Y+\Delta)$ is nef and big for $a \gg 0$, and thus some multiple
of $N$ defines a birational morphism $g$ by the basepoint-free theorem \cite{KM98}, Theorem~3.3. Thus $F'$ is generated by exceptional curves of $g$.
We deduce that $(K_Y+\Delta)^{\perp} \cap F' =\{0\}$ and $(K_Y+\Delta)$ is negative on $F'\setminus \{0\}$ as claimed.

Now by the cone theorem \cite{KM98}, Theorem~3.7, $\oNE(Y)_{\bR}$ is rational polyhedral near $F'$ and there is a contraction $p \colon Y \rightarrow \oY$
such that $F'$ is generated by the classes of curves contracted by $p$.
It follows
that we can find $\NE(Y)_{\bR} \subset \sigma_{P'} \subset \sigma_P$ such that $F'$
is a face of $\sigma_{P'}$. Now the algebra structure for $P$ comes from
$P'$ by base extension, so (replacing $P$ by $P'$)
we can assume $F = F'$, and thus that $W$ is
positive on $F \setminus \{0\}$.


Now let $J$ be a monomial ideal with $\sqrt{J} = P \setminus F$.
Consider condition ($1'$). By the $T^D$-equivariance of Theorem
\ref{TDequivariance}, any
$z^C \vartheta_s$ that appears in $\vartheta_p \cdot \vartheta_q$
has the same weight for $T^D$.
Thus it is enough to show that the map
\[
w \colon B(\bZ)\times (P \setminus J) \to\chi(T^D), \quad (q,C) \mapsto w(q)+w(C)
\]
has finite fibres.
It is enough to consider fibres of $\sigma(\bZ)\times(P \setminus J)
\to \chi(T^D)$ for each $\sigma \in \Sigma_{\max}$. Note that
$\sigma(\bZ) \times P$ is the set of integral points of a rational
polyhedral cone, and $w$ is linear on this set.
Thus it is enough to check that
$\ker(w) \cap (\sigma(\bZ) \times F) = 0$. So suppose we have $q\in\sigma(\ZZ)$, $C\in F$ with
$w(q)+w(C)=0$. Say $\sigma = \sigma_{i,i+1}$. Then
$q = a v_i + b v_{i+1}$, for $a, b \in \bZ_{\geq 0}$.
We have
\[
w(q)+w(C) = a e_{D_i} + b e_{D_{i+1}} + \sum_j (C \cdot D_j)e_{D_j};
\]
thus if this is zero, we have $C \cdot D_j \leq 0$ for all $j$.
In particular,
$W \cdot C \leq 0$. Since $W$ is positive on $F \setminus \{0\}$,
$C = 0$. Now necessarily $a = b = q = 0$. This proves ($1'$).

For (2), let $A_J' \subset A_J$
be the subalgebra generated by the $\vartheta_{q}, q \in \Gamma$.
Fix a weight $w \in \chi(T^D)$. To show $A_J' = A_J$ it is enough to show
that the finite set
\[
\{z^C\vartheta_q \in A_J\,|\,
\hbox{$(q,C) \in B(\bZ) \times (P \setminus J)$
of weight $w$}
\}
\]
is contained in $A'_J$ (since the $z^C\vartheta_q$ give a $\kk$-basis
of $A_J$).
We argue by decreasing induction on $\ord_{\fom}(C)$ (see Definition
\ref{orderdef}).
Since the set of possible
$(q,C)$ is finite, there is an upper bound on the possible $\ord_{\fom}$'s.
So the claim is vacuously true for large $\ord_{\fom}$.
Consider $z^C \cdot \vartheta_p$, with $\ord_{\fom}(C)=h$.
Since the $\vartheta_q$ generate $A_J$ modulo $\fom$, we can find $a \in A'_J$
such that
\[
\vartheta_p = a + m
\]
with $m \in \fom \cdot A_J$.
Moreover, we can assume $a$, and thus $m$, is homogeneous for the $T^D$ action.
Now
\[
z^C \vartheta_p = z^C a + z^C m.
\]
Clearly $z^C m$ is a sum of terms $z^{D}\vartheta_{q}$
of weight $w$ and $\ord_\fom(D)>h$,
so $z^C m \in A_J'$ by induction. \end{proof}

\begin{remark}
Suppose $p \colon Y \rightarrow Y'$ is a contraction such that some component of $D$ is not contracted by $p$.
Let $F$ be the face of $\NE(Y)_{\bR}$ generated by classes of curves contracted by $p$.
Then $\NE(Y)_{\bR}$ is rational polyhedral near $F$. (This follows from the cone theorem, cf. the proof of Proposition~\ref{extensionprop}.)
In particular there exists a rational polyhedral cone $\sigma_P \subset A_1(Y,\bR)$ such that $\NE(Y)_{\bR} \subset \sigma_P$
and $\sigma_P$ coincides with $\NE(Y)_{\bR}$ near $F$.
\end{remark}

\begin{corollary} Theorem \ref{extensiontheorem} holds if $D$ has $n\ge 3$
irreducible components.
\end{corollary}

\begin{proof} Immediate from Proposition \ref{extensionprop}.
\end{proof}

\subsection{Proof of Theorems \ref{maintheoremlocalcase}
and \ref{extensiontheorem} in general}
\label{mtsmalln}

We now consider an arbitrary Looijenga pair $(Y,D)$, along with a toric
monoid $P$ with $\NE(Y)\subset P\subset A_1(Y,\ZZ)$.
Let $\tau:(Y',D') \to (Y,D)$ be a toric blowup such that $(Y',D')$ has a toric
model $p:(Y',D')\to (\oY,\oD)$.
We have the map $\tau_*:A_1(Y',\ZZ)\rightarrow A_1(Y,\ZZ)$.
We can find a strictly convex rational polyhedral cone
$\sigma_{P'}$ with
\[
\NE(Y')_{\bR} \subset \sigma_{P'} \subset A_1(Y',\bR)
\]
which has a face $F$ spanned by the $\tau$-exceptional curves,
and which surjects under $\tau_*$ onto $\sigma_P \subset A_1(Y,\bR)$. For any
monomial ideal $I \subset P$ with $\sqrt{I}=\fom$, let $I' \subset P'$
be the inverse image of $I$ under $\tau_*$.
Then $\sqrt{I'}$ is the prime monomial
ideal associated to the face $F$.  Since
the exceptional curves are a proper subset of $D'$
we have $\sqrt{I'}\in \cA(Y')$ by Proposition~\ref{extensionprop}.
Note that $\Spec\kk[P]/I$ is naturally a closed subscheme of
$\Spec\kk[P']/I'$, via the map induced by the surjection
$\tau_*:P'\rightarrow P$.
Now restrict the family $\cX_{I'} \to \Spec \kk[P']/I'$ to
$\Spec \kk[P]/I$. This gives an algebra structure on
\[
A_I := \bigoplus_{q \in B(\bZ)} (\kk[P]/I) \vartheta_q.
\]
We now verify Assumptions \ref{ringass}. First, we show that the
multiplication rule of this algebra is the one described in
Theorem~\ref{multrule}. The argument is just as in
the proof of Proposition~\ref{barPtoPprop}:
We have $B_{(Y',D')} = B_{(Y,D)}$ and
take $\psi:=\tau_*:P'\to P$. Note
$\psi(\foD^{\can}_{(Y',D')}) = \foD^{\can}_{(Y,D)}$
(i.e., the rays are the same,
and we apply $\psi$ to the decoration function). This does not
literally give a bijection on broken lines (because different exponents
in the decoration of a ray in $\foD^{\can}_{(Y',D')}$ could map to
the same exponent under $\psi$).
However, by Equation \eqref{sumbareq}, with $z$ a point close to $q$,
\begin{align*}
\sum_{\substack{(\gamma_1,\gamma_2) \\
                                  \Limits(\gamma_i) = (q_i,z) \\
                                  s(\gamma_1) + s(\gamma_2) = q}} c(\gamma_1)c(\gamma_2)
 &= \sum_{\substack{(\gamma_1,\gamma_2) \\
                 \Limits(\gamma_i) = (q_i,z) \\
                 s(\gamma_1) + s(\gamma_2) = q}}
\psi\left( \sum_{\gamma_1'\in\xi_{\gamma_1}} c(\gamma_1') \right)
\psi\left(\sum_{\gamma_2'\in \xi_{\gamma_2}} c(\gamma_2') \right)   \\
&= \psi\left(\sum_{\substack{(\gamma'_1,\gamma'_2) \\
                                  \Limits(\gamma'_i) = (q_i,z) \\
                                  s(\gamma_1') + s(\gamma_2') = q}}
                                 c(\gamma_1')c(\gamma_2') \right),
\end{align*}
where $\xi_{\gamma_i}$ denotes the set of all broken lines $\gamma_i'$
for $\foD^{\can}_{(Y',D')}$ such that $\psi\circ\gamma_i'=\gamma_i$ as
paths and the monomials attached to $\psi(\gamma_i')$ differ from
those attached to $\gamma_i$ only in the $\kk$-valued coefficients (see
the proof of Proposition \ref{barPtoPprop}). This
implies the claim.

Next we need to check that the fibre over
the zero stratum of $\Spec\kk[P]$ is $\bV_n$. In case $n\ge 3$, this
is straightforward from the multiplication rule. Indeed, modulo $\fom$,
every broken line contributing to the multiplication rule is a
straight line, and furthermore it cannot cross any ray of $\Sigma$.
>From this one sees that $A_{\fom}=R_{\fom}[\Sigma]$.

The cases $n=1$ and $2$ require special attention.
We will do the case of $n=1$,
as $n=2$ is similar (and simpler). We cut $B = B_{(Y,D)}$ along
the unique ray $\rho=\rho_1 \in \Sigma$, and consider the image under
a set of linear coordinates $\psi$ on $B\setminus\rho$. This identifies
$B\setminus\rho$ with a strictly convex
rational cone in $\bR^2$. Let $w,w'$ be the primitive generators of
the two boundary rays. Modulo $\fom$ the decoration on
every scattering ray is trivial, so every broken line is straight.
Moreover, no line can cross $\rho$ (or the attached monomial
becomes trivial modulo $\fom$ by the strict convexity of $\varphi$). Now it
follows for any $x \in B(\bR) \setminus \rho$ and any
$q \in (B \setminus \rho)(\bZ)$
there is a unique (straight) broken line with $\Limits = (q,x)$, while
there are exactly two (straight) broken lines with $\Limits = (v,x)$,
$v = v_1$ --- under $\psi$ these become two
distinct straight lines with directions $w,w'$. Performing a toric
blowup of $(Y,D)$ to get $n'=3$ can be accomplished by subdividing
the cone generated by $w$ and $w'$ along the rays generated by
$w+w'$ and $2w+w'$.  Then by Theorem \ref{extensiontheorem} in the
case $n=3$, we see that
$A_{\fom}$ is generated over $\kk$ by
\[
\vartheta_v= \vartheta_{w} = \vartheta_{w'}, \vartheta_{w + w'}, \vartheta_{2w + w'}
\]
where we abuse notation and use the same symbol for an integer
point in the convex cone generated by $w$ and $w'$,
and the corresponding point in $B(\bZ)$.
Now applying the multiplication rule of Theorem \ref{multrule} one checks easily
the equalities:
\begin{align*}
\vartheta_v \cdot \vartheta_{w + w'} &= \vartheta_{2w + w'} + \vartheta_{w + 2w'} \\
\vartheta_{2w + w'} \cdot \vartheta_{w + 2w'} & = \vartheta_{3w + 3w'} = \vartheta_{w + w'}^3.
\end{align*}
It follows that
$$
\vartheta_{2w + w'} \cdot \vartheta_v \cdot \vartheta_{w + w'} =
\vartheta_{2w + w'}^2 + \vartheta_{w + w'}^3
$$
and thus
$A_{\fom} = \kk[x,y,z]/(xyz - x^2 - z^3)$, which is isomorphic
to the ring of sections
\[
\bigoplus_{m \geq 0} H^0(C,\cO(m))
\]
for a line bundle $\cO(1)$ of degree one on an irreducible rational nodal
curve $C$ of arithmetic genus $1$. Thus $\Spec A_{\fom} = \bV_1$.

Combining this with Propositions \ref{minradideal} and
\ref{extensionprop}, this proves
Theorems \ref{maintheoremlocalcase}
and \ref{extensiontheorem} hold for all $(Y,D)$ except for the smoothness
statement of Theorem \ref{maintheoremlocalcase}.

To show smoothness, note that if $\fom'$ denotes the maximal monomial ideal
of $P'$, $\foX'_{\fom'}\rightarrow\foS_{\fom'}$ the formal deformation
provided by Theorem \ref{maintheoremlocalcase} for the pair $(Y',D')$
with the toric model,
we know that $\kk[P']\rightarrow H^0(\foZ_{\fom'},\shO_{\foZ_{\fom'}})$
is not injective by Theorem \ref{injectivitym}. Now choose (see the beginning
of the proof of Proposition \ref{extensionprop}) a divisor
$A=\sum_i a_iD'_i$ with $a_i\ge 0$ for all $i$ and $A$ relatively $\tau$-ample,
so that $A\cdot D'_j>0$ for any $D'_j$ contracted by $\tau$. This determines
a one-parameter subgroup $T^A\cong\Gm$
of $T^{D'}$ via the map $\chi(T^{D'})\rightarrow
\ZZ$ given by $e_{D_i'}\mapsto a_i$.

Let $J=P'\setminus F$, so that $[C]\in J$ if and only if $C$ is not
contracted by $\tau$. Thus if $[C]\in F$, $T^A$ acts on $z^{[C]}$ with
weight $C.A>0$, and for $q\in B(\ZZ)$, $T^A$ acts with non-negative weight
since $a_i\ge 0$ for all $i$. It then follows that the map
\[
H^0(\foZ_{J},\shO_{\foZ_J})\rightarrow H^0(\foZ_{\fom'},
\shO_{\foZ_{\fom'}})
\]
is injective because every component of $\foZ_{J}$ has a limit point
in $\foZ_{\fom'}$ under the $T^A$ action. So we conclude that
$\kk[P']\rightarrow H^0(\foZ_{J},\shO_{\foZ_J})$ is not injective.

Now $F^{\gp}$ is generated by the classes of the $D_i'$ contracted by
$\tau$. Let $T^F:=\Hom(F,\Gm)$. The composition $F^{\gp}\subset A_1(Y',\ZZ)
\rightarrow \chi(T^{D'})$ is a primitive embedding, because the intersection
matrix of
$F\subset \langle D_1',\ldots,D_r'\rangle$
is unimodular, where $D_1',\ldots,D_r'$ are the irreducible components
of the boundary of $Y'$. So the corresponding composition $T^{D'}\rightarrow
\Hom(A_1(Y',\ZZ),\Gm)\rightarrow T^F$ admits a splitting $T^F
\rightarrow T^{D'}$. By $T^F$-equivariance, the restriction of
the family $\foX'_{J}/\foS'_{J}$ to the open subscheme of
$\foS'_{J}$ defined by $T^F\subset S'_J$
is isomorphic to a direct
product of $\foX_{\fom}/\foS_{\fom}$ (coming from $(Y,D), P$)
with $T^F$. In particular, $\foX_{\fom}/\foS_{\fom}$ has smooth generic fibre.

\subsection{The case that $(Y,D)$ is positive}

\begin{lemma} \label{pospair} The following are equivalent for
a Looijenga pair $(Y,D)$:
\begin{itemize}
\item[(1.1)] There exist integers $a_1,\dots,a_n$ such that
$(\sum a_i D_i)^2 > 0$.
\item[(1.2)] There exist positive integers $b_1,\dots,b_n$ such
that $(\sum b_i D_i) \cdot D_j > 0$ for all $j$.
\item[(1.3)] $Y \setminus D$ is the minimal resolution of an
affine surface with (at worst) Du Val singularities.
\item[(1.4)] There exist $0 < c_i < 1$ such that
$-(K_Y + \sum c_i D_i)$ is nef and big.
\end{itemize}
If any of the above equivalent conditions hold, then so do
the following:
\begin{itemize}
\item[(2.1)] The Mori cone $\NE(Y)_{\bR}$ is rational polyhedral,
generated by finitely many classes of rational curves. Every
nef line bundle on $Y$ is semi-ample.
\item[(2.2)] The subgroup $G$ of $\Aut(\Pic(Y),\langle\cdot,\cdot\rangle)$
fixing the classes $[D_i]$ is finite.
\item[(2.3)] The union $R \subset Y$ of all curves
disjoint from $D$ is contractible.
\end{itemize}
\end{lemma}

\begin{definition} We say a Looijenga pair $(Y,D)$ is \emph{positive}
if it satisfies any of the equivalent conditions (1.1)-(1.4) of the above
lemma.
\end{definition}

\begin{proof}
We have
\[
K_Y + \sum c_i D_i = (K_Y+D) - \sum(1-c_i)D_i = - \sum (1-c_i)D_i
\]
so (1.2) implies (1.4), and (1.2) obviously implies (1.1).

If (1.1) holds then $(D^{\perp},\langle\cdot,\cdot\rangle)$,
where $D^{\perp}=\{H\in \Pic Y\,|\, H\cdot D_i=0 \quad \forall i\}$,
is negative definite, by the Hodge Index Theorem,
and this implies (2.2) and (2.3).

Suppose (1.4) holds.
By the basepoint-free theorem \cite{KM98}, 3.3,
the linear system $$|m(\sum b_i D_i)|=|-m(K_Y+\sum c_i D_i)|$$ defines a birational morphism for $m \in \bN$ sufficiently large, with exceptional locus the union $R$ of curves disjoint from $D$.
Adjunction shows $R$ is a contractible configuration of $(-2)$-curves, which gives (1.3).
(2.1) follows from the cone theorem \cite{KM98}, 3.7.

We show (1.1) implies (1.2).
By the Riemann--Roch theorem, if $W$ is a Weil
divisor (on any smooth surface) and $W^2 >0$ then either
$W$ or $-W$ is big (i.e., the rational map given by
$|nW|$ is birational for sufficiently large $n$). So,
possibly replacing the divisor by its negative,
we may assume $W = \sum a_i D_i$ is big. Write
$$
W' = \sum_{a_i > 0} a_i D_i = W + \sum_{-a_i > 0} (-a_i)D_i.
$$
Thus $W'$ is big, and replacing $W$ by $nW'$, we may
assume all $a_i \geq 0$ and $|W|$ defines a birational
(rational) map. Subtracting off the divisorial base-locus
(which does not affect the rational map) we may further
assume the base locus is at most zero dimensional. Now
$W = \sum b_i D_i$ is effective, nef and big, and supported on $D$.
We show we may assume that in addition $b_i>0$ and $W \cdot D_i > 0$ for each $i$.
If $W \cdot D_i > 0$, then we may assume $b_i>0$ (by adding $\epsilon D_i$ to $W$ if necessary).
Now consider the set $S \subset \{1,\ldots,n\}$ of components $D_i$ of $D$ such that $W \cdot D_i=0$.
By connectedness of $D$ we find $b_i > 0$ for each $i \in S$. Thus $\Supp(W)=D$.
By the Hodge index theorem the intersection matrix $(D_i \cdot D_j)_{i,j \in S}$ is negative definite.
Hence there exists a linear combination $E=\sum_{i \in S} \alpha_i D_i$, with $\alpha_i \in \bZ$ for each $i \in S$,
such that $E \cdot D_i<0$ for each $i \in S$.
Now replacing $W$ by $W-\epsilon E$, we obtain $W \cdot D_i> 0$ for each $i=1,\ldots,n$.

Finally we show (1.3) implies (1.1). Since $U=Y\setminus D$ is
the resolution of an affine variety $U'$ with du Val singularities,
we have $U'=Y'\setminus D'$ where $Y'$ is a normal projective surface and $D'$ is a Weil divisor such that $D'$ is the support of an ample divisor $A$.  Let $\pi \colon \tilde{Y} \rightarrow Y'$ be a resolution of singularities such that
$\pi|_{\pi^{-1}(U')}:\pi^{-1}(U')\rightarrow U'$ is the
resolution $U\rightarrow U'$.
Furthermore, we can assume the inclusion $U \subset Y$ extends to a birational
morphism $f \colon \tilde{Y} \rightarrow Y$.
Let $\tilde{D}$ be the inverse image of $D'$ under $\pi$, so $\tilde{Y} \setminus \tilde{D} = U$.
The divisor $\pi^*A$ has support $\tilde{D}$. So we can write $\pi^*A=f^*(\sum a_iD_i)+\sum \mu_jE_j$ where the $E_j$ are the $f$-exceptional curves and $a_i, \mu_j \in \bZ$. Then $(\sum a_i D_i)^2 \ge (\pi^*A)^2 = A^2 > 0$.
\end{proof}

\begin{corollary} \label{polymultpositive} Let $(Y,D)$ be a positive Looijenga
pair. Let $P = \NE(Y)$. The multiplication rule Theorem \ref{multrule}
applied with $\foD=\foD^{\can}$
determines a finitely generated $T^D$-equivariant
$R = \kk[P]$-algebra structure on
the free $R$-module
\[
A = \bigoplus_{q \in B(\bZ)} R \cdot \vartheta_q.
\]
Furthermore,
$\Spec A \to \Spec R$ is a flat affine family of Gorenstein semi-log canonical
(SLC) surfaces
with central fibre $\bV_n$, and smooth generic fibre.
Any collection of $\vartheta_q$ whose restrictions generate
$A/\fom = H^0(\bV_n,\cO_{\VV_n})$ generate $A$ as an $R$-algebra.
In particular the $\vartheta_{v_i}$ generate for $n \geq 3$.
\end{corollary}

\begin{proof} Everything but the singularity statement follows
from Theorem~\ref{extensiontheorem}.  The Gorenstein
SLC locus in the base is open, and $T^D$-equivariant. Taking a big and
nef divisor
$H=\sum a_i D_i$ with $a_i>0$ and $H\cdot D_i>0$ for all $i$, we obtain
a one-parameter subgroup of $T^D$ given by the map $\chi(T^D)\rightarrow
\ZZ$, $e_{D_i}\mapsto a_i$. By definition of the weights, the weights
of $z^{[C]}$ for $C\in \NE(Y)$ and $\vartheta_p$ for $p\in B(\ZZ)$ are
all non-negative. Thus the corresponding torus $T_H\cong\Gm$ gives
a contracting action. In particular, since $\VV_n$
is Gorenstein and SLC, all fibres are Gorenstein and SLC. Moreover,
the map
\[
H^0(Z,\shO_Z)\rightarrow H^0(\foZ_{\fom},\shO_{\foZ_{\fom}})
\]
is injective, where $Z$ is the singular locus of the family $X\rightarrow
S$. But since $R\rightarrow H^0(\shO_{\foZ_{\fom}},\shO_{\foZ_{\fom}})$
is not injective, as shown in the proof of Theorem \ref{maintheoremlocalcase},
we deduce that the map $R\rightarrow H^0(Z,\shO_Z)$ is not injective.
Letting $f$ be in the kernel of the map, the fibres over $\Spec R
\setminus V(f)$ are smooth.
\end{proof}

In Part II we will prove that when $D$ is positive, our mirror family
admits a canonical fibrewise $T^D$-equivariant compactification
$\cX \subset (\cZ,\cD)$. The restriction $(\cZ,\cD) \to T_Y:=\Pic (Y)
\otimes \Gm$ comes
with a trivialization $\cD \stackrel{\sim}{\rightarrow} D_* \times T_Y$.
We will show that
$(\cZ,\cD)$ is the universal family of Looijenga pairs $(Z,D_Z)$ deformation equivalent to $(Y,D)$ together with
a choice of isomorphism $D_Z \stackrel{\sim}{\rightarrow} D_*$.
Now for any positive pair $(Z,D_Z)$ together with a choice of isomorphism $\phi \colon D_Z \stackrel{\sim}{\rightarrow} D_*$,
our construction equips the complement $U = Z \setminus D_Z$ with canonical theta functions $\vartheta_q$, $q \in B_{(Z,D)}(\bZ)$.
We will give a characterisation in terms of the intrinsic geometry of $(Z,D_Z)$.
Changing the choice of isomorphism $\phi$ changes $\vartheta_q$ by a character of $T^D=\Aut^0(D_*)$, the identity component of $\Aut(D_*)$.
Here we illustrate with two examples:

\begin{example} \label{thetaexample1}
Consider first the case $(Y,D)$ a $5$-cycle of $(-1)$-curves on the
(unique) degree $5$ del Pezzo surface,
Example  \ref{M05more}. In this case
$T^D = T_Y=\Pic(Y)\otimes_{\ZZ}\Gm$,
and thus by the $T^D$-equivariance, all fibres of the
restriction $\cX \to T_Y$ are isomorphic. We consider the fibre over
the identity $e \in T_Y$, thus specializing the equations of
Example \ref{M05more} by setting all $z^{D_i} = 1$.
It's well known that these equations define an embedding of
the original $U=Y \setminus D$ into $\bA^5$ --- if we take the
closure in $\bP^5$ (for the standard compactification $\bA^5 \subset \bP^5$)
one checks easily we obtain $Y$ with $D$ the hyperplane section at
infinity.

Now it is easy to compute the zeroes and poles:
\[
(\vartheta_{v_i}) = E_i + D_i - D_{i+2} - D_{i-2}
\]
(indices mod $5$).
In particular $\{\vartheta_{v_i} =0\} = E_i \cap U \subset U$, which characterizes
$\vartheta_{v_i}$ up to scaling.
\end{example}

\begin{example}\label{thetaexample2}
Now let $(Y,D = D_1 + D_2 + D_3)$ be (the deformation type of) a cubic
surface together with a triangle of lines. Let $\cX \subset \Spec(\kk[\NE(Y)]) \times \bA^3$ be the canonical embedding given by $\vartheta_i:=\vartheta_{v_i}$, $i=1,2,3$.
In this case, as we shall see in Part II, the scattering diagram is particularly
beautiful, with every ray $\fod$ of rational slope occuring, with precisely six
curves on the cubic surface contributing to $f_{\fod}$.
We will also show in Part II that the mirror is given by the equation
\[
\vartheta_1\vartheta_2\vartheta_3= \sum_i z^{D_i} \vartheta_i^2 +
\sum_i \left(\sum_j z^{E_{ij}}\right)z^{D_i} \vartheta_i +
\sum_{\pi} z^{\pi^*H} - 4z^{D_1+D_2+D_3}.
\]
Here the $E_{ij}$ are the interior $(-1)$-curves meeting $D_i$, and the sum over
$\pi$ is the sum over all possible toric models $\pi \colon Y \rightarrow \oY$
of $(Y,D)$ to a pair $(\oY,\oD)$  isomorphic to $\bP^2$ with its toric boundary.
(Such $\pi$ are permuted simply transitively by the Weyl group $W(D_4)$ by
\cite{L81}, Prop.~4.5, p.~283.)  The same family,
in the same canonical coordinates, was discovered by Oblomkov \cite{Ob04}. As we learned
from Dolgachev, after a change of variables (in $\bA^3$), and restricting
to $T_Y$ (the locus over which the fibers have at worst Du Val singularities) this is identified with
the universal family of affine cubic surfaces (the complement to a triangle
of lines on projective cubic surface) constructed by Cayley in \cite{C1869}.
The universal
family of cubic surfaces with triangle is obtained as the closure in
$\bA^3 \subset \bP^3$.
In particular,
as in the first example, our mirror family compactifies naturally to the universal
family of Looijenga pairs deformation equivalent to the original $(Y,D)$. There
is again a geometric characterisation of $\vartheta_i$ (up to scaling):
The linear system $|-K_Y - D_i| = |D_j + D_k|$ (here $\{i,j,k\} = \{1,2,3\}$) is a basepoint free pencil.
It defines a ruling $\pi: Y \to \bP^1$ which restricts to a double cover $D_i \rightarrow \bP^1$.
Let $\{a,b\} \subset \bP^1$ be the branch points of $\pi|_{D_i}$.
Let $p = \pi(D_j + D_k) \in \bP^1$.
There is a unique point $q \in \bP^1 \setminus \{ a, b, p\}$
fixed by the unique involution of $\bP^1$ interchanging
$a$ and $b$ and fixing $p$. Let $Q =\pi^*(q) \in |D_j + D_k|$ be the
corresponding divisor. The curve $Q \subset \bP^3$ is a smooth conic.
In Part II we will show
\begin{proposition} $(\vartheta_i) = Q - D_j - D_k$. \end{proposition}
\end{example}

\section{Looijenga's conjecture}
\label{cusp}

In this section we apply the main construction of this paper to give a proof of
Looijenga's conjecture on smoothability of cusp singularities,
Theorem \ref{loocor2}.
The simple conceptual idea is explained in the introduction.
Here we give the rather involved details.

\subsection{Duality of cusp singularities}
\label{dualitysec}

We review the notion of dual cusp singularities.
By definition,
a cusp is a normal surface singularity for which the exceptional
locus of the minimal resolution is a cycle of rational curves. The self-intersections of these exceptional curves determine the analytic
type of the singularity, see \cite{L73}. Cusps have a quotient construction
due to Hirzebruch \cite{Hi73} which we explain here. See also \cite{L81},
III, \S2 for this point of view.

Let $M=\bZ^2$, and let
$T \in \SL(M)$ be a hyperbolic matrix, i.e., $T$ has a real eigenvalue
$\lambda >1$. Then $T$ determines a pair of {\it dual} cusps as follows:
Let $w_1,w_2 \in M_{\bR}$ be eigenvectors with eigenvalues
$\lambda_1=1/\lambda$, $\lambda_2 = \lambda$,
chosen so that $w_1 \wedge w_2 > 0$ (in the standard counter-clockwise
orientation of $\bR^2$). Let $\oC,\oC'$ be the strictly convex cones spanned
by $w_1,w_2$ and $w_2,-w_1$, and let $C,C'$ be their interiors, either of
which is preserved by $T$. Let $U_C,U_{C'}$ be the corresponding tube
domains, i.e.,
\[
U_C := \{z \in M_{\bC}\,|\,
\Im(z) \in C\}/M \subset M_{\bC}/M = M \otimes \bG_m.
\]
$T$ acts freely and properly discontinuously on $U_C, U_{C'}$. Write
$Y_C, Y_{C'}$ for the holomorphic hulls of
$U_{C}/\Gamma,U_{C'}/\Gamma$, where $\Gamma$ is the group generated by $T$.
These each have one additional point, $p\in Y_C$, $p'\in Y_{C'}$,
and $(Y_C,p)$, $(Y_{C'},p')$ are normal
surface germs of cusps.

\begin{definition}
$(Y_C,p)$ and $(Y_{C'},p')$ are \emph{dual cusp singularities}.
\end{definition}

All cusps (and their duals) arise this way.

\begin{remark}
\label{dualremark}
If $M$ is identified with its dual
by choosing an isomorphism $\bigwedge^2 M\cong\ZZ$, the cone $\oC'$
is identified
with the dual cone of $\oC$. In this way  $C'/\Gamma$ and
$C/\Gamma$
are dual integral affine manifolds, which suggests that
the duality between the corresponding cusps is a form of mirror symmetry.
\end{remark}

To resolve the cusp singularities of $Y_C$, say, one considers the convex
hull $\Xi$ of integral points in $C$, and let $v_i$, $i\in
\ZZ$, be the integral points in $\partial\Xi$, listed so that
$v_{i-1}$ and $v_{i+1}$ are the integral points adjacent to $v_i$ on
$\partial\Xi$. Let $\tilde\Sigma$ be the infinite fan with two-dimensional
cones generated by $v_{i-1}, v_i$ for $i\in\ZZ$. Note that $T$ acts on
$\Xi$, and thus acts by translation $T(v_i)=v_{i+n}$ for some integer $n$,
which we can take to be positive by reversing the ordering of the $v_i$ if
necessary. Then $X_{\tilde\Sigma}$ is a toric variety with an infinite
chain of $\PP^1$'s, and $T$ acts on $X_{\tilde\Sigma}$. There is a tubular
neighbourhood $N$ of this infinite chain of $\PP^1$'s on which the group
generated by $T$ acts properly discontinuously, see \cite{AMRT75}, p.\ 48.
Then $N/\Gamma$ is a minimal resolution of singularities of a neighbourhood
of the singularity of $Y_C$. Note the exceptional divisors $D_i$, $i$ taken
modulo $n$, with $D_i$ corresponding to the ray of $\tilde\Sigma$ generated
by $v_i$, satisfy \eqref{negdefrel}. Thus if there is a Looijenga pair $(Y,D')$
with $D'=D_1'+\cdots+D_n'$ and $(D_i')^2=D_i^2$ for each $i$, the corresponding
affine manifold with singularities is precisely $B=|\tilde\Sigma|/\Gamma$
by Example \ref{negdefexample}, and $B_0=C/\Gamma$.

In fact, the dual cusp singularity can be described directly from
the cone $C$ and the polyhedron $\Xi$:

\begin{lemma}\label{dualcusp}
Let $T, \oC, \Xi$ and the $v_i$ be as above, giving a cusp singularity
$p\in Y_C$.
Let $Z$ be the toric variety (only locally of finite type) associated to $\Xi$. Let $E \subset Z$ be the toric boundary of $Z$, an infinite chain of smooth rational curves corresponding to the boundary of $\Xi$. Then there exists a tubular neighbourhood $E \subset N \subset Z$ such that the $\Gamma$ action on $\Xi$ induces a properly discontinuous $\Gamma$ action on $N$. Let $F \subset \tilde{X}$ denote the quotient of $E \subset N$ by $\Gamma$. So $F$ is a cycle of smooth rational curves. Then $F \subset \tilde{X}$ can be contracted to a singularity
$p' \in X$, which is a copy of the dual cusp $p' \in Y_{C'}$.
Moreover, $\tilde{X}$ is obtained from the minimal resolution of $p' \in X$
by contracting all the $(-2)$-curves.
\end{lemma}

\begin{proof}
Let $\tilde\Sigma'$ be the normal fan for the polytope $\Xi$ and
$\oC''$ the closure of its support.
We observe that $\oC''$ coincides with the dual of $\oC$,
together with the induced
$\Gamma$-action. By Remark \ref{dualremark},
it follows that $\tilde{X}$ is a partial resolution of a
copy of the dual cusp.

The surface $\tilde{X}$ has Du Val singularities of type $A$. Indeed, $v_i$ is a vertex of $\Xi$ iff $m_i: = -D_i^2 > 2$. The corresponding point of $Z$ is smooth if $m_i=3$ and a singularity of type $A_{m_i-3}$ if $m_i>3$ (by direct calculation using $v_{i-1}+v_{i+1}=m_iv_i$). Also $K_{\tilde{X}}$ is relatively ample over $X$ by Lemma~\ref{Kampletoric} below. Indeed, the vectors $u_i :=v_{i}-
v_{i-1}$ are the primitive integral vectors in the direction of the edges of $\Xi$, and
\[
u_{i+1}-u_{i}=v_{i+1}+v_{i-1}-2v_i=(m_i-2)v_i,
\]
so the lines $v_i + \bR \cdot(u_{i+1}-u_{i}) = \bR \cdot v_i$ all
meet at the origin.
We deduce that $\tilde{X}$ is obtained from the minimal resolution of $X$ by contracting all $(-2)$-curves as claimed.
\end{proof}

\begin{lemma}\label{Kampletoric}
Let $P \subset \bR^2$ be a rational convex polygon and $X$ the associated toric surface.
Fix an orientation of the boundary of $P$ and let $e_0,e_1,e_2$ denote
oriented consecutive edges of the boundary of $P$.
Let $C \subset X$ denote the component of the toric boundary associated to the
bounded edge $e_1 \subset P$.
Let $v_{0} = e_0 \cap e_1$ and $v_{1}=e_1 \cap e_2$ be the vertices of $e_1$.
Let $u_0,u_1,u_2 \in \bZ^2$ denote the primitive integral vectors in the direction of $e_0,e_1,e_2$.
Then $K_X \cdot C > 0$ if and only if the lines $v_{0}+\bR(u_1-u_0)$ and $v_{1}+\bR(u_2-u_1)$ meet on the opposite side of $e_1$ to $P$.
\end{lemma}

\begin{proof}
Write $M=\bZ^2 \subset \bR^2$ for the lattice of characters of the torus of $X$.
Choose an orientation $\bigwedge^2 M \simeq \bZ$ and use it to identify $M$ with its dual lattice $N$.
Let $U\subset X$ denote the union of the two toric affine open subsets
corresponding to the vertices $v_0$ and $v_1$ of $P$. Then $U$ is
a toric open neighbourhood of $C\subset X$.
Then, under this identification and up to a sign, the fan $\Sigma$ of $U$ in $N_{\bR}$ consists of the two cones $\langle u_0,u_1 \rangle_{\bR_{\ge 0}}$, $\langle u_1,u_2 \rangle_{\bR_{\ge 0}}$ and their faces.
The condition on the lines $v_{0}+\bR(u_1-u_0)$ and $v_{1}+\bR(u_2-u_1)$ in the statement is equivalent to the condition that the primitive generator $u_1$ of the central ray of the fan $\Sigma$
lies on the same side of the affine line spanned by the primitive generators $u_0$ and $u_2$ of the two outer rays of $\Sigma$ as the origin $0 \in N$.
Now by \cite{R83}, 4.3, this condition is equivalent to $K_X \cdot C > 0$.
\end{proof}

Because this quotient construction is analytic
we will have to deal with convergence issues to show that our
construction extends to this analytic situation.

\subsection{Cusp family}\label{cusp1}

In this subsection, we fix the following.
Let $(Y,D)$ be a rational surface with anti-canonical cycle, now
over the field $\kk=\CC$. We obtain $(B,\Sigma)$, with $\Sigma$ having
one-dimensional cones $\rho_i$ and two-dimensional cones $\sigma_{i,i+1}$
as usual.

We assume that the intersection matrix $(D_i \cdot D_j)_{1 \le i,j \le n}$
is negative definite.
Let $f \colon Y \rightarrow Y'$ be the contraction (in the analytic category)
of $D \subset Y$ to a cusp singularity $q \in Y'$.
We assume that $f$ is the minimal resolution of $Y'$, that is,
$D_i^2 \le -2$ for all $i$. We further assume that $n \ge 3$ to avoid additional
technical issues of the flavour dealt with in \S\ref{mtsmalln}.
The case of Looijenga's conjecture with $n\le 2$ is in fact trivial, see
the proof of Theorem \ref{looijengasconjecture}.

Let $L$ be a nef divisor on $Y$ such that
\[
\NE(Y)_{\bR_{\ge 0}}
\cap L^{\perp}=\langle D_1,\ldots,D_n\rangle_{\bR_{\ge 0}}.
\]
Here the subscript $\RR_{\ge 0}$ denotes the real cone in $A_1(Y,\RR)$
generated by the given elements or set.
(Indeed, if $Y'$ is projective we can take
$L=h^*A$ for $A$ an ample divisor on $Y'$.
In general, let $A$ be an ample divisor on $Y$. There exist unique
$a_i \in \bQ$ such that $L:=A + \sum a_iD_i$ is orthogonal to $D_j$
for each $j$.
By negative definiteness of $\langle D_1,\ldots,D_n \rangle$, we have $a_i > 0$ for each $i$.
It follows that $L$ is nef.)

Let $\sigma_P \subset A_1(Y,\bR)$ be a rational polyhedral cone containing
$\NE(Y)$, $P = \sigma_P \cap A_1(Y,\bZ)$ the associated toric monoid.
We assume
that $\sigma_P$ is strictly convex and $\sigma_P \cap L^{\perp}$ is a face of
$\sigma_P$. Let
$\fom=P \setminus \{0\}$ and
$J = P \setminus P \cap L^{\perp} \subset P$, the radical monomial
ideal associated to
the face $\sigma_P \cap L^{\perp}$ of $\sigma_P$.
We will write $S=\Spec \kk[P]$, and for any monomial ideal $I$, we write
$S_I=\Spec\kk[P]/I$.

We take the multivalued piecewise linear function $\varphi$ as usual
to have bending parameter $\kappa_{\rho,\varphi}=[D_{\rho}]\in P$.
We wish to build a deformation of the $n$-vertex over $S_I$ with $\sqrt{I}=J$.
However, this is already a problem over $S_J$ because none of the
$\kappa_{\rho,\varphi}$ lie in $J$. Thus we can't apply the results of
\S\ref{modifiedmumfordsection} directly as the standard open sets
$U_{\rho,J}$ will not glue compatibly because of issues involving triple
intersections. To deal with this, we need to shrink these open sets. This
procedure is carried out as follows.

\begin{theorem}\label{cuspfamily}
Fix $R>1$.
There exists an analytic open neighbourhood $S'_J$ of $0 \in S_J$ and an analytic flat family $f_J \colon X_J \rightarrow S'_J$ together with a section $s \colon S'_J \rightarrow X_J$ satisfying the following properties:
\begin{enumerate}
\item The general fibre $X_{J,t}$ of $f_J$ is a Stein analytic surface with a unique singularity $s(t) \in X_{J,t}$ isomorphic to the dual cusp to $q \in Y'$.
\item For each ray $\rho_i \in \Sigma$ there is an open analytic subset $V_{\rho_i,J} \subset X_J$ and open analytic embeddings
\[
V_{\rho_i,J} \subset \{(X_{i-1},X_i,X_{i+1})\in U_{\rho_i,J}\;|\;|X_{i-1}|<
R|X_i|, \, |X_{i+1}|<R|X_i|\}
\subset U_{\rho_i,J}
\]
where
\[
U_{\rho_i,J} :=V(X_{i-1}X_{i+1}-z^{[D_{\rho_i}]}X_i^{-D_{\rho_i}^2})
\subset \bA^2_{X_{i-1},X_{i+1}} \times (\Gm)_{X_i} \times S_J
\]
such that
\begin{enumerate}
\item $X^o_J := X_J \setminus s(S'_J) = \bigcup_{\rho \in \Sigma} V_{\rho,J}$.
\item $V_{\rho,J} \cap V_{\rho',J} = \emptyset$ unless $\rho=\rho'$ or $\rho$ and $\rho'$ are the edges of a maximal cone $\sigma \in \Sigma$.
\end{enumerate}
\item The restriction of $X_J/S'_J$ to $S_{J+\mathfrak{m}^{N+1}}$ is identified with an analytic neighbourhood of the vertex in the restriction of the family $X_{\mathfrak{m}^{N+1}}/S_{\mathfrak{m}^{N+1}}$ given by
Theorem~\ref{univfamth}, (1) with $\foD=\foD^{\can}$, for each $N \ge 0$.
\end{enumerate}
\end{theorem}

\begin{proof}
(1) We use the notation of Example \ref{negdefexample}, so that
the pair $(Y,D)$ determines
an infinite fan $\tilde{\Sigma}$ in $M_{\RR}$ with the primitive generators of the rays
being the $v_i$ for $i\in\ZZ$. We also have $T\in \SL(M)$ acting
on the fan $\tilde{\Sigma}$. We have $B=|\tilde{\Sigma}|/\Gamma$, where $\Gamma$
is the group generated by $T$.

Now as in \S \ref{dualitysec},
let $\Xi \subset M_{\bR}$ be the convex hull of the points $v_i \in M$.
Thus $\Xi$ is an infinite convex polytope.
Let $\tilde{\Sigma}'$ be the subdivision of $\Xi$ induced by $\tilde{\Sigma}$.
In what follows, we will build a Mumford degeneration $Z/S_J$ with
special fibre $Z_0$ the stable toric variety associated to $\tilde{\Sigma}'$.
In other words, $Z_0$ will be the union of the toric surfaces associated
to the maximal polytopes in $\tilde{\Sigma}'$.

$S_J$ is the affine toric variety associated to the face
$\sigma_{\bdy}:=\sigma_P \cap L^{\perp}$
of $\sigma_P$, and $P_{\bdy}:=\sigma_{\bdy} \cap A_1(Y,\bZ)$ contains the classes of the
components of the boundary of $Y$.
We define a piecewise linear convex function
$\tvarphi \colon |\tilde{\Sigma}'| \rightarrow P_{\bdy}^{\gp}\otimes_{\ZZ}{\RR}$
by restriction of a piecewise linear convex function $\tvarphi$ on
$|\tilde\Sigma|$. This function is the single-valued representative for
$\varphi$ on the universal cover of $B_0$, and as such, is
defined up to an integral linear function
by specifying its bending parameter
\[
\kappa_{\tvarphi,\rho_i}=[D_{i\bmod n}]\in P_{\bdy}
\]
if $\rho_i=\RR_{\ge 0} v_i$.

Then $\tvarphi$ determines a Mumford degeneration: this is a
slight generalization of \S\ref{Mumfordsection}. One defines
\[
\tilde\Xi:=\{(m,r)\,|\, m\in\Xi, r\in \tvarphi(m)+\sigma_{\bdy}\}
\subset
M_{\RR}\oplus (P^{\gp}_{\bdy}\otimes_{\ZZ}\RR).
\]
Let $C(\tilde\Xi)$ be the closure of
\[
\{(sm,sr,s)\,|\,(m,r)\in\tilde\Xi,
s\ge 0\}\subset M_{\RR}\oplus(P_{\bdy}^{\gp}\otimes_{\ZZ}\RR)\oplus\RR.
\]
Then $\CC[C(\tilde\Xi)\cap
(M\oplus P^{\gp}_{\bdy} \oplus\ZZ)]$ has a natural grading given by the last
coordinate, and the degree zero part of this ring is easily seen to contain
$\CC[P_{\bdy}]$. Thus we obtain the Mumford family determined by $\tvarphi$
as
\[
Z:=\Proj \CC[C(\tilde\Xi)\cap (M\oplus P^{\gp}_{\bdy} \oplus\ZZ)]
\rightarrow \Spec \CC[P_{\bdy}]=S_J.
\]

One sees easily that the fibre over $0\in S_J$ of
$Z \rightarrow S_J$ has infinitely
many components indexed by the $2$-cells of the subdivision
$\tilde{\Sigma}'$ of $\Xi$,
each of which is a copy of the blowup of $\bA^2$ at the origin.
The general fibre is a toric surface (only locally of finite type) containing an
infinite chain of smooth rational curves,
which specializes to the union of the exceptional curves of the
blowups in $Z_0$.
By construction $\Gamma$ acts on $Z$ over $S_J$ (because $\tvarphi$ is
$\Gamma$-invariant modulo integral affine functions).
Let $E \subset Z/S_J$ be the family of curves described above
(the relative toric boundary).
The group $\Gamma$ acts properly discontinuously on a tubular neighbourhood $N$ of $E \subset Z$ (cf. \cite{AMRT75}, p.~48).
Let $p \colon (F \subset \tilde{X}) \rightarrow S_J$ denote the quotient of $(E \subset N) \rightarrow S_J$ by $\Gamma$.

The divisor $F \subset \tilde{X}$ is Cartier and the dual of its normal bundle is relatively ample over
a neighbourhood of $0 \in S_J$. Indeed, the special fibre $\tilde{X}_0$ is a union of $n$ irreducible components each isomorphic to a tubular neighbourhood of the exceptional locus in the blowup of $\bA^2$,
and $F_0 \subset \tilde{X}_0$ is the cycle of $n$ smooth
rational curves formed by the exceptional curves of the blowups.
Hence the normal bundle of $F_0$ in $\tilde{X_0}$ has degree
$-1$ on each component of $F_0$.
Moreover, we have $R^1p_*(\cN^{\vee}_{F/\tilde{X}})^{\otimes k} = 0$ for each $k > 0$.
Indeed, by cohomology and base change it suffices to show that $H^1((\cN^{\vee}_{F_0/\tilde{X}_0})^{\otimes k})=0$, and this follows from Serre duality. Now by a relative version of Grauert's contractibility criterion \cite{F75},
Thm.~2, taking global sections of the structure sheaf defines a contraction
$p \colon \tilde{X} \rightarrow X_J/S_J$ to a family of Stein
analytic spaces with exceptional locus $F$.
The general fibre $X_{J,t}$ of $X_J/S_J$ is the dual cusp by Lemma~\ref{dualcusp}. The section $s:S_J\rightarrow X_J$ takes $x\in S_J$ to the cusp of
$X_{J,x}$.

We now show that $X_J/S_J$ is flat and the special fibre is the neighbourhood of the $n$-vertex obtained by contracting $F_0 \subset \tilde{X}_0$.
The key point is that $R^1p_*\cO_{\tilde{X}}$ is a locally free $\cO_{S_J}$-module, cf. \cite{W76}, Theorem~1.4(b).
Indeed, we have
\[
R^1p_*\cO_{\tilde{X}}(-F)=0
\]
by cohomology and base change, the theorem on formal functions, and the vanishing $H^1((\cN_{F_0/\tilde{X}_0}^{\vee})^{\otimes k})=0$ for $k > 0$ used above. So, pushing forward the exact sequence
\[
0 \rightarrow \cO_{\tilde{X}}(-F) \rightarrow \cO_{\tilde{X}} \rightarrow \cO_{F} \rightarrow 0
\]
we obtain
\[
R^1p_*\cO_{\tilde{X}}=R^1p_*\cO_{F} \simeq \cO_{S_J}.
\]
Recall that $S_J$ is a toric variety, so in particular Cohen-Macaulay.
Let $t_1,\ldots,t_r$ be a regular sequence at $0 \in S_J$ of length $\dim S_J$ and write
\[
S^i_J=V(t_1,\cdots,t_i) \subset S_J,
\]
$\tilde{X}^i=\tilde{X}|_{S^i_J}$, and let $X^i_J/S^i_J$ be the family
over $S^i_J$ defined by $\cO_{X^i_J}=p_*\cO_{\tilde{X}^i}$.
Arguing as above we find that $R^1p_*\cO_{\tilde{X}^i} \simeq \cO_{S^i_J}$.
Pushing forward the exact sequence
\[
0 \rightarrow \cO_{\tilde{X}^i}
\mapright{\cdot t_{i+1}}
\cO_{\tilde{X}^i} \mapright{}
\cO_{\tilde{X}^{i+1}} \mapright{} 0
\]
we deduce that the natural map
\begin{equation}\label{slice}
\cO_{X^i_J}/ t_{i+1} \cO_{X^i_J} \rightarrow \cO_{X^{i+1}_J}
\end{equation}
is an isomorphism. Hence by the local criterion of flatness \cite{Matsumura89},
Ex.~22.3, p.~178, it suffices to show that
$X^r_J/S_J^r$ is flat with special fibre the $n$-vertex. But $S_J^r$
is the spectrum of an Artinian $\bC$-algebra, so this follows from
\cite{W76}, Theorem~1.4(b).

(2) Write $Z^o=Z \setminus E$, and $m_i=-D_{i \bmod n}^2$, $a_i=z^{[D_{i \bmod n}]}$ for $i \in \bZ$.
We have an open covering
\[
Z^o = \bigcup_{i \in \bZ} U_{i,J},
\]
where
\[
U_{i,J}=V(x_{i-1}x_{i+1}-a_ix_i^{m_i}) \subset \bA^2_{x_{i-1},x_{i+1}}
\times (\Gm)_{x_i} \times S_J.
\]
Similarly, we have an open covering
\[
Z = \bigcup_{i \in \bZ} \oU_{i,J},
\]
where
\begin{align*}
\oU_{i,J}= {} & V(x'_{i-1}x'_{i+1}-a_ix_i^{m_i-2}) \subset
\AA^3_{x'_{i-1},x_i,x'_{i+1}} \times S_J,\\
E \cap \oU_{i,J} = {} & V(x_i) \subset \oU_{i,J},
\end{align*}
and $\oU_{i,J} \setminus E = U_{i,J}$ via $x_{i-1}=x_ix'_{i-1}$, $x_{i+1}=x_ix'_{i+1}$.
(Note that $m_i = -D_i^2 \ge 2$ by assumption.)

Recall that the infinite cyclic group $\Gamma$ acts on $Z/S_J$, there is a $\Gamma$-invariant tubular neighbourhood $N \subset Z$ of $E \subset Z$ on which the action is properly discontinuous, and $F \subset \tilde{X}$ is obtained as the quotient of $E \subset N$ by $\Gamma$.
In terms of the open covering above the action is given by $U_{i,J} \rightarrow U_{i+n,J}$, $x_j \mapsto x_{j+n}$.
Note that the map $U_{i,J} \cap N \rightarrow \tilde{X}$ is \emph{not} an open embedding (because, for example, $U_{i,J}$ contains the general fibre of $Z^o/S_J$). Fix $R>1$. We define $W_{i} \subset U_{i,J} \cap N$ by
\[
W_{i} = \{(x_{i-1},x_i,x_{i+1}) \in U_{i,J}\cap N\;|\; |x_{i-1}|<R|x_i|, \, |x_{i+1}| < R|x_i|\}
\]
and similarly define $\oW_i \subset \oU_{i,J} \cap N$ by
\[
\oW_{i} = \{(x'_{i-1},x_i,x'_{i+1})\in \oU_{i,J}\cap N\;|\;
|x'_{i-1}|<R, |x'_{i+1}|<R\}.
\]
Then $\oW_i \setminus E = W_i$.

The $\oW_i$ cover the special fibre $E_0$ of $E/S_J$ (using $R>1$). The open set $\bigcup \oW_i \subset N$ is $\Gamma$-invariant, the quotient $(N,E) \rightarrow (\tilde{X},F)$ by $\Gamma$ is a covering map,  and $p \colon \tilde{X} \rightarrow X_J$ is proper with exceptional locus $F$. Hence we may assume (passing to an analytic neighbourhood $S_J'$ of $0 \in S_J$ and $s(0) \in X_J$) that $N=\bigcup \oW_i$.

By Lemma~\ref{disjointestimate} below there exists $\delta>0$ such that
\[
\oW_i \cap \{(x_{i-1},x_i,x_{i+1})\,|\,|x_i|< \delta\} \subset \{|x_j|<1
\quad\forall j\}
\]
for each $i$.
We replace $\oW_i$ by $\oW_i \cap \{|x_i| < \delta\}$, and modify $W_i$
similarly.
Then as above we may assume that $N=\bigcup \oW_i$, and $N \subset \{|x_i|<1 \quad\forall i\}$.
We claim that $W_i \cap W_j = \emptyset$ for all $j > i+1$.
It suffices to work on the general fibre of $Z^o/S_J'$, which is an algebraic torus. The coordinate functions $x_i$ are characters of this torus (up to a multiplicative constant). By construction we have $|x_i|<1$ for each $i$ on $N$. Hence, shrinking the base $S_J'$, we may assume that
$|a_ix_i^{m_i-2}| < 1/R^2$ for each $i$.
The relation
$$x_{i-1}x_{i+1}=a_ix_i^{m_i}$$
gives the inequality
$$|x_{i+1}/x_i| < (1/R^2)|x_i/x_{i-1}|.$$
Combining such inequalities we obtain
\begin{equation}\label{disjointinequality}
|x_{j+1}/x_j| < (1/R^2)^{j-i}|x_{i+1}/x_i| \quad \mbox{ for } j>i.
\end{equation}
Now $|x_{i+1}/x_i| < R$ on $W_i$ and $|x_{j-1}/x_j| < R$ on $W_j$,
so $|x_{j}/x_{j-1}| > (1/R^2)|x_{i+1}/x_i|$ on $W_i \cap W_j$. For $j>i+1$ this contradicts the inequality (\ref{disjointinequality}), hence $W_i \cap W_j = \emptyset$ as claimed.

It follows that $W_i$ embeds in $X_J$, using the fact that we have assumed
$n\ge 3$. Let $V_i$ denote its image, with indices now understood modulo $n$. Thus $X^o_J=\bigcup V_i$ and the inverse image of $V_i$ is the (disjoint) union of $W_j$ such that $j \equiv i \bmod n$.
We have $V_i \cap V_j = \emptyset$ for $j \neq i-1,i,i+1$ by our claim above.
So, writing $V_{\rho,J}:=V_i$ for $\rho$ the ray of $\Sigma$ corresponding to $D_i$, the condition (2)(b) is satisfied.

(3) We have the open covering $X^o_J= \bigcup V_i$ and open embeddings $V_i \subset U_{i,J}$, and
an open covering $X^o_{\mathfrak{m}^{N+1}}=\bigcup U_{i,\mathfrak{m}^{N+1}}$. The restrictions of $U_{i,J}/S_J$ and $U_{i,\mathfrak{m}^{N+1}}/S_{\mathfrak{m}^{N+1}}$ to $S_{J+\mathfrak{m}^{N+1}}$ are identified, and the gluing maps coincide. It follows that the restriction of $X_J/S_J$ is identified with a neighbourhood of the vertex in the restriction of $X_{\mathfrak{m}^{N+1}}/S_{\mathfrak{m}^{N+1}}$ using Lemma~\ref{relS2}.
\end{proof}

\begin{lem}\label{disjointestimate}
We use the notation of the proof of Theorem~\ref{cuspfamily}(2).
Let $x_j$ be the coordinate functions on $Z^o=\bigcup U_{i,J}$. There exists $\delta> 0$ such that on each open set $U_{i,J}$ if $|x_i|<\delta$, $|x_{i-1}|<R|x_i|$, $|x_{i+1}|<R|x_i|$, and $|z^p|<1$ for all $p \in P_{\bdy}$ then $|x_j|<1$ for all $j$.
\end{lem}
\begin{proof}
The points $v_{i-1}$, $v_i$, and $v_{i+1}$ are consecutive integral points on the boundary of the infinite convex polytope $\Xi$ with asymptotic directions $w_1,w_2$.
It follows that $w_1=\alpha_{i1}v_i+\beta_{i1}(v_{i-1}-v_i)$ and $w_2=\alpha_{i2}v_i+\beta_{i2}(v_{i+1}-v_i)$ for some $\alpha_{i1},\beta_{i1},\alpha_{i2},\beta_{i2} \in \bR_{>0}$. Note that the ratios $\beta_{i1}/\alpha_{i1},\beta_{i2}/\alpha_{i2}$ only depend on $i$ modulo $n$ (because $w_1,w_2$ are eigenvectors of $T$ and $T(v_i)=v_{i+n}$).
Let $\mu$ be the maximum of the ratios
$\beta_{i1}/\alpha_{i1},\beta_{i2}/\alpha_{i2}$ for $i=1,\ldots,n$.
Let $\delta=R^{-\mu}$.
If $j>i$ then $v_j = \alpha v_i + \beta (v_{i+1}-v_i)$ with $\beta/\alpha < \beta_{i2}/\alpha_{i2}$. The coordinate function $x_j$ can be written as $z^px_i^{\alpha}(x_{i+1}/x_i)^{\beta}$ on $V_i$, where $p \in P_{\bdy}$.
Thus $|x_j|< \delta^{\alpha}R^{\beta} < 1$ for $|x_i| < \delta$ and $|z^p|<1$. The same is true for $j<i$ by symmetry.
\end{proof}

\subsection{Thickening of the cusp family}\label{cusp2}

We continue to work with the setup at the beginning of \S\ref{cusp1}, with
$(Y,D)$, $L$, $\sigma_P$, $P$, $\fom$ and $J$ as given there.

\begin{theorem}\label{thickenedcuspfamily}
Let $p_J \colon X_J \rightarrow S_J'$ be the analytic family of Theorem~\ref{cuspfamily}. Possibly after replacing $S_J'$ by a smaller neighbourhood of $0 \in S_J$ and $X_J$ by a smaller neighbourhood of $s(S_J') \subset X_J$,
independent of the choice of $I$ below, the following holds.
Let $I \subset P$ be a monomial ideal such that $\sqrt{I}=J$ and let $S'_I \subset S_I$ denote the induced thickening of $S'_J \subset S_J$.
There is an infinitesimal deformation $f_I \colon X_I \rightarrow S'_I$ of $f_J \colon X_J \rightarrow S'_J$ such that for each $N>0$ the restriction to
$\Spec\kk[P]/(I+\mathfrak{m}^{N+1})$ is identified with an analytic neighbourhood of the vertex in the restriction of the family $X_{\mathfrak{m}^{N+1}}/S_{\mathfrak{m}^{N+1}}$
given by Theorem~\ref{univfamth}, (1) applied with $\foD=\foD^{\can}$.
\end{theorem}

\begin{proof}
As usual, $\foD^{\can}$ is the canonical scattering diagram
on $B$ associated to the pair $(Y,D)$. Note that the hypotheses
(I) and (II) of Theorem \ref{canonicalconsistency} are satisfied for
the ideal $J$ (but not (III)). This is because we can take $L=A+\sum a_iD_i$
with $A$ ample on $Y$ and $a_i>0$, and any $\AA^1$-class $\beta$
intersects the $D_i$ non-negatively. Then for any $k$ there are only a
finite number of $\AA^1$-classes $\beta$ such that $L\cdot \beta<k$.
In addition, there are no $\AA^1$-classes $\beta$ with $\beta\cdot L=0$.

Let $\foD$ be the scattering diagram obtained by reducing $\foD^{\can}$ modulo $I$ as follows.
For each ray $\fod$ of $\foD^{\can}$, we truncate the attached function $f_{\fod}$ by removing monomial terms lying in $I \cdot \kk[P_{\varphi_{\tau_{\fod}}}]$, and we discard the ray if the truncated function equals $1$.
Because of (II), $\foD$ has only finitely many rays $\fod$ and the attached
functions $f_{\fod}$ are finite sums of monomials. Because of (I),
$\foD$ is empty if $I=J$.

We use the scattering diagram $\foD$ to define a complex analytic space $X^{o}_{I}/S_I'$ as follows.
Recall from \eqref{localequations} the description of the schemes $U_{\rho_i,I}/S_I$ for $\rho_i \in \Sigma$:
\[
U_{\rho_i,I}=V(X_{i-1}X_{i+1}-z^{[D_{\rho}]}X_i^{-D_{\rho}^2}f_{\rho})
\subset \bA^2_{X_{i-1},X_{i+1}} \times (\Gm)_{X_i} \times S_I.
\]
We have open subsets $U_{\rho_i,\sigma_{i-1,i},I}, U_{\rho_i,\sigma_{i,i+1},I}
\subset U_{\rho_i,I}$ defined
by $X_{i-1}\not=0$ and $X_{i+1}\not=0$ respectively.
We have canonical identifications
\begin{equation}\label{identifyoverlaps}
U_{\rho_i,\sigma_{i,i+1},I}=U_{\sigma_{i,i+1},\sigma_{i,i+1},I} =
U_{\rho_{i+1},\sigma_{i,i+1},I}
\end{equation}
where
\[
U_{\sigma_{i,i+1},\sigma_{i,i+1}}=(\Gm)_{X_i,X_{i+1}}^2 \times S_I.
\]
Recall that we have an open covering $X^o_J= \bigcup_{\rho \in \Sigma} V_{\rho,J}$ and open analytic embeddings $V_{\rho,J} \subset U_{\rho,J}$.
Write
\[
V_{\sigma_{i,i+1},\sigma_{i,i+1},J} := (V_{\rho_i,J} \cap
U_{\rho_i,\sigma_{i,i+1},J}) \cap (V_{\rho_{i+1},J} \cap U_{\rho_{i+1},
\sigma_{i,i+1},J}) \subset U_{\sigma_{i,i+1},\sigma_{i,i+1},J}
\]
where we use the identification (\ref{identifyoverlaps}).
Let $V_{\rho_i,\sigma_{i,i+1},J} \subset V_{\rho_i,J}$, $V_{\rho_{i+1},
\sigma_{i,i+1},J} \subset V_{\rho_{i+1},J}$
denote the open subsets corresponding to $V_{\sigma_{i,i+1},\sigma_{i,i+1},J}$
under \eqref{identifyoverlaps}.
Let $V_{\rho_i,I}$, $V_{\rho_i,\sigma_{i,i+1},I}$, etc., be the infinitesimal
thickenings of these open sets determined by the thickenings
$U_{\rho_i,I}$ of $U_{\rho_i,J}$.
Let $\theta_{\gamma,\foD} \colon U_{\rho_{i+1},\sigma_{i,i+1},I} \rightarrow
U_{\rho_i,\sigma_{i,i+1},I}$
be the gluing isomorphism defined as in \S\ref{scatdiagsection}. Note that as
the canonical scattering diagram $\foD^{\can}$ is trivial modulo $J$,
$\theta_{\gamma,\foD}$ restricts to the identification (\ref{identifyoverlaps}) modulo $J$, and thus restricts to an isomorphism
\[
V_{\rho_{i+1},\sigma_{i,i+1},I} \rightarrow V_{\rho_i,\sigma_{i,i+1},I}.
\]
Gluing the $V_{\rho,I}$ via these isomorphisms we obtain an infinitesimal deformation $X^{o}_{I}/S'_{I}$ of $X^{o}_J/S'_J$. Note that there are no triple overlaps of the $V_{\rho,J}$ by Theorem~\ref{cuspfamily}(2)(b), hence no compatibility condition for the gluing automorphisms.
It is clear from the construction that the families $X^o_I/S'_I$ and $X_{\mathfrak{m}^{N+1}}/S_{\mathfrak{m}^{N+1}}$ are compatible.

We define sections $\vartheta_q \in \Gamma(X^o_I,\cO_{X^o_I})$
for $q \in B(\ZZ)$, compatible with the sections of Theorem~\ref{univfamth},
(2).
We proceed as in the algebraic case: we first define a local section $\Lift_Q(q)$ for each choice of basepoint $Q \in B_0 \setminus \Supp(\foD)$ on a corresponding open patch of $X^o_I$ using the broken lines construction.
The new difficulty here is that the functions $\Lift_Q(q)$ are not algebraic,
even over the unthickened locus $S_J'$. Indeed, by definition $\Lift_Q(q) = \sum \Mono(\gamma)$ is a formal sum of monomials corresponding to broken lines $\gamma$ for $q$ with endpoint $Q$. Note that with our current choice of ideal
$J$, this sum is always infinite, as is already evident in Example
\ref{firstbrokenexample}.  So we must prove convergence. This is done in
\S\ref{convergence}, see
Propositions~\ref{lift1} and \ref{lift2}.

Once this convergence is proved,
we observe that these patch to give well-defined global sections.
This follows from the consistency of
$\foD^{\can}$ and compatibility of $X^o_I/S'_I$ with
$X^o_{\fom^{N+1}}/S_{\mathfrak{m}^{N+1}}$ for $N \ge 0$.

We define an infinitesimal thickening $X_I/S'_I$ of $X_J/S'_J$ by $\cO_{X_I}=i_*\cO_{X^o_I}$ where $i \colon X^o_J \subset X_J$ is the inclusion.
Then $X_I/S'_I$ is flat by Lemma~\ref{flatnesslemma}
and the existence of the lifts $\vartheta_q$.
\end{proof}

\subsubsection{Convergence of Lifts}\label{convergence}

Let $\oC \subset M_{\bR}$ be the closure of the support $|\tilde{\Sigma}|$ of the fan $\tilde{\Sigma}$, a closed convex cone.
Let $w_1, w_2$ be generators of $\oC$.
Then $w_1,w_2$ are eigenvectors of $T$ with eigenvalues $\lambda^{-1}, \lambda$ for some $\lambda \in \bR$. We may assume that $\lambda > 1$.
Let $\pi \colon  \tilde{B}_0 \rightarrow B_0$ denote the universal cover of $B_0$.
So $\tilde{B}_0$ is identified with $C=\Int(\oC)$,
with deck transformations given by the action of $\Gamma=\langle T \rangle$ on
$C$.
Let $\tcP$, $\tilde{\varphi}$, and $\tilde B_0(\bZ)$ denote the pullbacks of
$\cP$, $\varphi$, and $B_0(\ZZ)$. Let $\tilde{\foD}$ denote the scattering diagram on $\tilde{B}_0$ induced by $\foD$.
We fix a trivialization  of $\tcP$ as the constant
sheaf with fibre $P^{\gp}\oplus M$.

The behaviour of the broken lines $\gamma$ is best studied by passing to the universal cover $\tilde{B}_0$ of $B_0$.
Let $Q \in B_0$, $q \in B_0(\ZZ)$, and choose lifts
$\tQ \in \tilde{B}_0$, $\tq \in \tilde B_0(\ZZ)$.
Then a broken line $\gamma$ on $B_0$ for $q$ with endpoint $Q$ lifts uniquely to a broken line $\tilde{\gamma}$ on $\tilde{B}_0$ for $T^N(\tq)$ with endpoint $\tQ$, for some $N \in \bZ$, and the attached monomials are identified via $\tcP=\pi^*\cP$.
Note that $T^N(\tq)$ approaches $\bR_{\ge 0}\cdot w_2$ as $N \rightarrow \infty$ and $\bR_{\ge 0}\cdot w_1$ as $N \rightarrow -\infty$.

If $\gamma$ is a broken line for a point $\tq \in \tilde B_0(\ZZ)$
with endpoint $\tQ \in \tB_0$ then $\gamma \colon (-\infty,0] \rightarrow
\tB_0$ is a piecewise linear path in $\tB_0=C$ with initial direction
$-\tq$, ending at $\tQ$, and crossing all the rays of $\tilde\foD$
between $\bR_{\ge 0} \tq$
and $\bR_{\ge 0}\tQ$ in order. Let $t_1,\ldots, t_{l} \in (-\infty,0)$ denote
the points where $\gamma$ is not affine linear.
Each point $\gamma(t_i)$ lies on a ray $\fod_i$ of $\tilde{\foD}$ and the change
$\gamma'(t_i+\epsilon)-\gamma'(t_i-\epsilon)$ in the direction of $\gamma$ as
it crosses $\fod_i$ is an integral multiple of the primitive generator of
$\fod_i$. Moreover this multiple is positive because each ray of the
canonical scattering diagram is an outgoing ray in the terminology of
Definition~\ref{scatdiagdef}. So the path $\gamma$ is ``convex when
viewed from the origin''.

It is convenient for the convergence calculation to decompose the
monomials for broken lines as follows.
Let $\gamma \colon (-\infty,0] \rightarrow \tilde B_0$ be a broken line,
$t \in (-\infty,0]$ a point such that $\gamma$ is affine linear near $t$ and
$\gamma(t)$ lies in the interior of a maximal cone $\sigma$ of
$\tilde{\Sigma}$, and $cz^q$ the monomial attached to the domain
of linearity of $\gamma$ containing $t$.
Here $c \in \kk$ and $q \in P_{\tilde{\varphi}_{\sigma}} \subset P^{\gp}
\oplus M$.
We write $cz^q=az^{\tilde{\varphi}_{\sigma}(m)}$, where $m=r(q) \in M$
and $a=cz^{q-\tilde{\varphi}_{\sigma}(m)}$ is a monomial in $\kk[P]$.
We also use the same decomposition for the monomials occuring in the
scattering functions $f_{\fod}$ for $\fod \in \tilde{\foD}$. Let  $\fod$
be a ray in $\tilde{\foD}$ with primitive integral generator $m \in M$
and $\tau=\tau_{\fod}$ the smallest cone of $\tilde{\Sigma}$ containing
$\fod$. Then $f_{\fod}-1$ is a sum of monomials $cz^q$ where $c \in \kk$
and $q \in P_{\tilde{\varphi}_{\tau}}$, $0 \neq -r(q) \in \fod$. We write
$cz^q=az^{\tilde{\varphi}_{\tau}(m)}$ where $m=r(q)$ and $a=cz^{q-\tilde{\varphi}_{\tau}(m)}$ is a monomial in $\kk[P]$.

The scattering diagram $\foD$ on $B$ is finite (because we have reduced modulo $I$), and $\tilde{\foD}$ is its inverse image under $\pi \colon \tilde{B}_0 \rightarrow B_0$. Thus $\tilde{\foD}$ has only finitely many $\Gamma$-orbits of rays. Moreover, the $\Gamma$-action on the scattering functions $f_{\fod}$ is induced by the given action on $M_{\bR}$ and the trivial action on $\kk[P]$ as follows: writing
$f_{\fod}= 1 + \sum a_m z^{\tilde{\varphi}_{\tau}(m)}$ as above, $f_{T(\fod)}=1+\sum a_m z^{\tilde{\varphi}_{T(\tau)}(T(m))}$.

\begin{lemma}\label{brokenlinebounds}
Let $\tQ \in \tB_0 \setminus \Supp(\tilde{\foD})$ be a point contained in the interior of a maximal cone $\sigma \in \tilde{\Sigma}$.
We consider broken lines $\gamma$ on $\tilde{B}_0$ for $T^N(\tq)$, some $N \in \bZ$, with endpoint $\tQ$, such that $\Mono(\gamma) \notin I \cdot \kk[P_{\tilde{\varphi}_{\sigma}}]$.
Let $k \ge 0$ be such that $(k+1)J \subset I$.
\begin{enumerate}
\item The number of bends of $\gamma$ is at most $k$.
\item The number of broken lines for $T^N(\tq)$ with endpoint $\tQ$ is $O(|N|^k)$.
\item For $\gamma$ a broken line for $T^N(\tq)$ with endpoint $\tQ$ write $\Mono(\gamma)=a_{\gamma}z^{\tilde{\varphi}_{\sigma}(m_{\gamma})}$ where $a_{\gamma}$ is a monomial in $\kk[P]$. Assume that $|z^p| < \epsilon < 1$ for all $p \in P$.
Then
\[
|a_{\gamma}| = O(\prod_{\rho \in \Sigma\atop \dim\rho=1}
|z^{[D_{\rho}]}|^{|N|}).
\]
\end{enumerate}
\end{lemma}
\begin{proof}
(1) At a bend $t_i \in (-\infty,0]$  of $\gamma$
the attached monomial $c_{i}z^{q_{i}}$ is replaced by the monomial
$c_{i+1}z^{q_{i+1}}=cz^q \cdot c_iz^{q_{i}}$ where $cz^q$ is a term in a positive power of the scattering function $f_{\fod}$ associated to the ray $\fod$ containing $\gamma(t_i)$.
In particular $cz^q \in J \cdot \kk[P_{\tilde{\varphi}_{\tau_{\fod}}}]$.
Since $\Mono(\gamma) \notin I \cdot \kk[P_{\tilde{\varphi}_{\sigma}}]$ and $(k+1)J \subset I$ it follows that there are at most $k$ bends.

(2) Such a broken line crosses $O(|N|)$ scattering rays. If $\gamma$ is a broken line for $T^N(\tq)$ then the initial attached monomial is specified, equal to $z^{\tvarphi(T^N(\tq))}$.
At a scattering ray $\fod$, let $u$ denote the primitive generator of $\fod$, $f=f_{\fod}$ the attached function, and let $cz^q$ be the monomial attached to the incoming segment of the broken line. Then the possible continuations of the broken line past $\fod$ correspond to the monomial terms in $f^d$, where $d=|r(q) \wedge u|$ is the index of the sublattice of $M$ generated by $r(q)$ and $u$. Note that since
$f \equiv 1 \mod J \cdot \kk[P_{\tilde{\varphi}_{\tau_{\fod}}}]$ the number of monomial terms in $f^d$ not lying in $I \cdot \kk[P_{\tilde{\varphi}_{\tau_{\fod}}}]$ is bounded independent of $d$.
Further, since there are a finite number of $\Gamma$-orbits of scattering rays,
and the $\Gamma$-action preserves monomials, there is a bound on the number of monomial terms independent of the ray $\fod$.
Thus for a broken line $\gamma$ for $T^N(\tilde q)$ there are
$\begin{pmatrix} O(|N|)\\ k\end{pmatrix}=O(|N|^k)$ choices of how it may bend
by (1). So the total number of broken lines is $O(|N|^k)$.

(3) By symmetry we may assume that $N \ge 0$.
Let $\fod \in \tilde{\foD}$ be a scattering ray, $f=f_{\fod}$ the attached function, and $\gamma$ a broken line that crosses $\fod$.
Suppose first that $\fod$ is contained in the interior of a maximal cone $\sigma$ of $\tilde{\Sigma}$.
Let $az^{\tilde{\varphi}_{\sigma}(m)}$ be the monomial attached to the incoming segment of $\gamma$ near $\fod$.
Let $u$ be the primitive generator of $\fod$.
Then the outgoing monomial $a'z^{\tilde{\varphi}_{\sigma}(m')}$ is obtained from the incoming monomial by multiplication by a monomial term in $f^d$, where $d=|m \wedge u|$.
Write $f=1+f_1+\cdots + f_r$, a sum of monomials.
Since $f \equiv 1 \mod J \cdot \kk[P_{\varphi_{\sigma}}]$ we have
\begin{equation} \label{multinomial}
f^d \equiv \sum_{i_1+\cdots+i_r \le k} {d \choose i_1,\ldots,i_r} f_1^{i_1} \cdots f_r^{i_r} \mod I \cdot \kk[P_{\varphi_{\sigma}}].
\end{equation}
The multinomial coefficient
\[
{d \choose i_1,\ldots,i_r} := \frac{d!}{i_1!\cdots i_r! (d-i_1-\cdots-i_r)!}
\]
is bounded by $d^k$.
The direction of the scattering ray $\fod$
is $u=T^s(\beta)$ where $0 \le s \le N$ and $\beta \in M$ is
chosen from a finite set.

The vector $m \in M$ is of the form
\[
m=T^N(\tq)-\sum_{i=1}^l T^{s_i}\alpha_i
\]
where $l \le k$, $0 \le s_i \le N$ for each $i$, and the $\alpha_i \in M$ are chosen from a finite set.
Indeed, as in the proof of (2), for the monomial terms $cz^q$ occurring in the powers $f^d$ of the function $f=f_{\fod}$ attached to a given scattering ray $\fod$, only finitely many exponents $q \in P_{\tilde{\varphi}_{\tau_{\fod}}}$ occur (working modulo $I \cdot \kk[P_{\tilde{\varphi}_{\tau_{\fod}}}]$). So there are only finitely many possible changes of exponent $q$ for the attached monomial $cz^q$ of a broken line at a scattering ray modulo the action of $\Gamma$.

Now identify $M=\bZ^2$ and let $\| \cdot \|$ denote the standard norm on $M_{\bR}=\bR^2$. Then
$$d =|m \wedge u| \le \| m \| \cdot \| u \| =O(\lambda^{2N}).$$
So, the coefficient $a' \in \kk[P]$ of the outgoing monomial is given by $a'=c\cdot z^p \cdot a$ where $c \in \CC$, $p \in P$, and $|c|=O(\lambda^{2kN})$.
Thus $|a'| = O(\lambda^{2kN}) \cdot |a|$ for $|z^p| < 1$.

Next, let $\rho \in \tilde{\Sigma}$ be a ray and $\sigma_+$, $\sigma_-$ the
maximal cones containing $\rho$. Suppose $\gamma$ is a broken line that
crosses $\rho$, travelling from $\sigma_-$ to $\sigma_+$. Let
$a_-z^{\tilde{\varphi}_{\sigma_-}(m)}$ be the monomial attached to the
incoming segment of $\gamma$ near $\rho$ and
$a_+z^{\tilde{\varphi}_{\sigma_+}(m')}$ the monomial attached to
the outgoing segment.
By the definition of $\tilde{\varphi}$,
\[
z^{\tilde{\varphi}_{\sigma_-}(m)}=(z^{[D_{\rho}]})^{-\langle
n_{\rho},m\rangle}z^{\tilde{\varphi}_{\sigma_+}(m)},
\]
where $n_{\rho}\in N$ is primitive, annihilates $\rho$, and is
positive on $\sigma_+$.
Write $d:=-\langle n_{\rho},m\rangle$;
note $d=|u \wedge m| > 0$ where $u \in M$ is the primitive generator of $\rho$.
If $\gamma$ does not bend at $\rho$ then $a_+=(z^{[D_{\rho}]})^d \cdot a_-$.
In general $a_+=(z^{[D_{\rho}]})^d \cdot a_-'$ where
$a_-'z^{\tilde{\varphi}_{\sigma_-}(m')}$ is obtained from
$a_-z^{\tilde{\varphi}_{\sigma_-}(m)}$ as above (by applying the scattering
automorphism associated to $\rho$ and selecting a monomial term).

We need to show the exponent $d=|u \wedge m|>0$ of $z^{[D_{\rho}]}$ in the previous paragraph is large for some lift $\tilde\rho$ of any given ray
$\rho\in\Sigma$.
This will allow us to absorb the $O(\lambda^{2N})$ factors coming from bends
of $\gamma$ and obtain the estimate (3). Let $\gamma$ be a broken line for $T^N(\tq)$. Then
$$m_{\gamma}=T^N(\tq)-\sum_{i=1}^l T^{s_i}\alpha_i$$
as above, where $0 \le l \le k$, $N \ge s_1 \ge \cdots \ge s_l \ge 0$, and the $\alpha_i$ lie in a finite set.
Write $s_0=N$ and $s_{l+1}=0$. Choose $j$ such that $s_j-s_{j+1} \ge N/(k+1)$.
Now consider the exponent $d=|u \wedge m|$ for $m=T^N(\tq)-\sum_{i=1}^j
T^{s_i}\alpha_i$ given by the monomial attached to the segment of the broken
line between bends $j$ and $j+1$. Let $\tilde\rho\in\tilde\Sigma$ be the lift
of $\rho\in\Sigma$ between bends $j$ and $j+1$
which is closest to bend $j+1$ (such a lift exists if $N$ is sufficiently
large). Let $u=T^{s_{j+1}}\beta$ be the
primitive generator of $\tilde\rho$.
Then $|u \wedge m|=|\beta \wedge T^{s_j-s_{j+1}}m'|$ where $m'=T^{-s_j}(m)$.
Writing $m'=\mu_1'w_1+\mu_2'w_2$, we see that $|\mu_1'|$ is bounded since
$w_1$ has eigenvalue $\lambda^{-1}<1$. Also, $\mu_2'$ is bounded away from
zero by Lemma~\ref{discrete}.
Now
\[
u \wedge m = \beta \wedge T^{s_j-s_{j+1}}m' = \mu_1'(\beta \wedge w_1)\lambda^{-(s_j-s_{j+1})}  + \mu_2'(\beta \wedge w_2)\lambda^{s_j-s_{j+1}},
\]
where $s_j-s_{j+1}>N/(k+1)$, so
\[
|u \wedge m| > c \cdot \lambda^{N/(k+1)}
\]
for some constant $c>0$.

Combining our results now gives, when $|z^{[D_{\rho}]}|<1$ for
all $p\in P_{\bdy}$, the estimate
\[
a_{\gamma}= O((\lambda^{2kN})^k \cdot \prod_{\rho \in \Sigma
\atop\dim\rho=1} |z^{[D_{\rho}]}|^{c \cdot \lambda^{N/(k+1)}}).
\]
where the first factor bounds the contribution associated to bends of $\gamma$ and the second factor bounds the contribution associated to rays $\rho$ of the fan crossed by $\gamma$, as described in the preceding two paragraphs.
This implies the estimate (3) in the statement, using $|z^{[D_{\rho}]}|<1$.
Indeed, the above expression is of the form
$$\lambda^{a N} \cdot x^{c \cdot \lambda^{b N}}$$
where $a,b,c > 0$ and $\lambda >1$ are constants, and $x=\prod |z^{[D_{\rho}]}|$.
This is bounded by $Cx^N$ for $0 \le x < \epsilon < 1$, for some constant $C$ (depending on $\epsilon$).
(To see this, we may assume $x \neq 0$, take logarithms, and establish an inequality
\begin{equation} \label{logarithm}
aN \log \lambda + c\lambda^{bN}\log x \le \log C + N \log x
\end{equation}
We have $-\log x > -\log \epsilon > 0$. Rearranging (\ref{logarithm}),
we require that, for some choice of $C$,
$$aN \log \lambda  \le \log C + (-\log x) (c \lambda^{bN}-N)$$
for all $N$. This holds because
$$aN \log \lambda \le (-\log x)(c \lambda^{bN}-N)$$
for $N$ sufficiently large.)
\end{proof}

\begin{lemma}\label{discrete}
Let $A \subset \bR$ be a finite set and $\lambda \in \bR$, $\lambda > 1$.
For $k \in \bN$ let $S_k \subset \bR$ be the set of real numbers $s$ of the form
$s=\sum_{i=1}^l c_i\lambda^{n_i}$
where $l \le k$ and $c_i \in A$, $n_i \in \bZ_{\ge 0}$ for each $i$.
Then $S_k$ is discrete for each $k$.
\end{lemma}
\begin{proof}
Proof by induction on $k$.
We have $S_0=\{0\}$. Suppose $S_k$ is discrete. We have
$S_{k+1}=\bigcup_{n \ge 0} \lambda^n(S_k+A)$. Since $\lambda >1$ we deduce that $S_{k+1}$ discrete.
\end{proof}

For Propositions \ref{lift1} and \ref{lift2} below, the assertions hold after possibly replacing $X_J$ by a smaller neighbourhood of $s(0) \in X_J$ (independent of $I$, $Q$ and $q$).

\begin{proposition} \label{lift1}
Let $Q \in  B_0 \setminus \Supp(\foD)$ be a point contained in the interior of a maximal cone $\sigma$ of $\Sigma$.
Each term of the formal sum $\Lift_Q(q)=\sum \Mono(\gamma)$ is an analytic function on $V_{\sigma,\sigma,I}$ and the sum defines an analytic function on $V_{\sigma,\sigma,I}$.
\end{proposition}
\begin{proof}
Recall that $V_{\sigma,\sigma,I}$ is an infinitesimal thickening of the reduced complex analytic space $V_{\sigma,\sigma,J}$.
Write
\[
V_{\sigma,\sigma} := \{(X_1,X_2)\,|\,|X_1|<R|X_2|,|X_2|<R|X_1|\} \subset
(\Gm)_{X_1,X_2}^2 \times S.
\]
Then $V_{\sigma,\sigma}$ is a reduced complex analytic space containing $V_{\sigma,\sigma,I}$ as a locally closed subspace.
We show that the sum $\Lift_Q(q)$ converges (uniformly on compact sets) to an analytic function on a neighbourhood of $V_{\sigma,\sigma,I}$ in $V_{\sigma,\sigma}$.

Let $\tQ \in \tB_0$ be a lift of $Q$ and $\tilde{\sigma}$ the lift of $\sigma$ containing $\tQ$.
Let $u_1,u_2$ be the primitive generators of $\tilde{\sigma}$ (a basis of $M$) such that the orientation of $u_1,u_2$ agrees with that of $w_1,w_2$.
Let $X_i=z^{\tilde{\varphi}_{\tilde{\sigma}}(u_i)}$, $i=1,2$, be the associated coordinate functions on $V_{\sigma,\sigma,I}$, so that
\[
V_{\sigma,\sigma,I} \subset \{ (X_1,X_2)\,|\, |X_1|<R|X_2|, |X_2|<R|X_1| \}
\subset (\Gm)_{X_1,X_2}^2 \times S'_I.
\]
For $m \in M$, writing $m=\alpha_1u_1+\alpha_2u_2$, we have $z^{\tilde{\varphi}_{\tilde{\sigma}}(m)}=X_1^{\alpha_1}X_2^{\alpha_2}$.

As already noted, broken lines $\gamma$ on $B_0$ for $q$ with endpoint $Q$ lift uniquely to broken lines on $\tB_0$ for $T^N(\tq)$ with endpoint $\tQ$, for some $N \in \bZ$, and the attached monomials are identified.
Write $\Mono(\gamma)=a_{\gamma}z^{\tilde{\varphi}_{\tilde{\sigma}}(m_{\gamma})}$ and $m_{\gamma}=\alpha_1u_1+\alpha_2u_2$.
Clearly $\Mono(\gamma)=a_{\gamma}X_1^{\alpha_1}X_2^{\alpha_2} \in \kk[P][X_1^{\pm 1},X_2^{\pm 1}]$ is an analytic function on $V_{\sigma,\sigma}$.
Also write $m_{\gamma}=\mu_1w_1+\mu_2w_2$.
By Lemma~\ref{endofbrokenline}, (1), $\mu_1$ and $\mu_2$ are bounded below
(using the symmetric statement interchanging $w_1$ and $w_2$ if $T^N(\tilde q)
\in \langle \tilde Q,w_1\rangle_{\RR_{\ge 0}}$).
The points $u_1$ and $u_2$ are adjacent integral points on the boundary of the infinite convex polytope $\Xi$ with asymptotic directions $w_1,w_2$.
It follows that  $w_1=\beta_1 u_2 + \gamma_1 (u_1-u_2)$ and $w_2=\beta_2 u_1 + \gamma_2 (u_2-u_1)$, for some $\beta_1,\beta_2,\gamma_1,\gamma_2 > 0$.
Hence
$$|z^{\tilde{\varphi}_{\tilde{\sigma}}(m_{\gamma})}|=|X_1^{\alpha_1}X_2^{\alpha_2}|=(|X_2|^{\beta_1}|X_1/X_2|^{\gamma_1})^{\mu_1}(|X_1|^{\beta_2}|X_2/X_1|^{\gamma_2})^{\mu_2}.$$
Now $|X_1/X_2| < R$, $|X_2/X_1|<R$ on $V_{\sigma,\sigma}$.
Thus, as we have chosen $\delta$ in Lemma \ref{disjointestimate} so that
$0< \delta < \min (R^{-\gamma_1/\beta_1},R^{-\gamma_2/\beta_2})$, if
$\mu_1,\mu_2$ are both positive, $|z^{\tilde{\varphi}_{\tilde{\sigma}}(m_{\gamma})}|$ is bounded for $|X_1|,|X_2|<\delta$. On the other hand,
suppose $\mu_1$, say, is negative. Then if $|X_2|>\delta'>0$, we have
\[
(|X_2|^{\beta_1}|X_1/X_2|^{\gamma_1})^{\mu_1} < (\delta')^{\beta_1\mu_1}
|X_2/X_1|^{-\mu_1\gamma_1}< (\delta')^{\beta_1\mu_1} R^{-\mu_1\gamma_1}.
\]
Since $\beta_1$ and $\gamma_1$ are fixed and $\mu_1$ is bounded below,
the above quantity is bounded. Similarly, if $\mu_2$ is negative,
$(|X_1|^{\beta_2}|X_2/X_1|^{\gamma_2})^{\mu_2}$ is bounded provided
$|X_1|>\delta'>0$. Thus in any
event, $|z^{\tilde{\varphi}_{\tilde{\sigma}}(m_{\gamma})}|$ is bounded
for $0<\delta'<|X_1|, |X_2|<\delta$.
(We will only obtain uniform convergence of the series $\Lift_Q(q)$ on compact
subsets of $V_{\sigma,\sigma}$.)
By Lemma~\ref{brokenlinebounds}, (3),
if $|z^p|<\epsilon<1$ for all $p \in P$ we have
$|a_{\gamma}|=O(\epsilon^{|N|})$. By Lemma~\ref{brokenlinebounds}, (2),
the number of broken lines for $T^N(\tq)$ is $O(|N|^k)$.
Combining, we deduce that $\Lift_q(Q) = \sum \Mono(\gamma)$ is convergent on the open analytic subset $V'_{\sigma,\sigma}$ of $V_{\sigma,\sigma}$ defined by
$|X_1|,|X_2| < \delta$ and $|z^p|< 1$ for all $p \in P$, for some $\delta > 0$ (independent of $I$ and $q$). After replacing $X_J$ by an analytic neighbourhood of the vertex $s(0)\in X_J$, we may assume that $V_{\sigma,\sigma,I} \subset V'_{\sigma,\sigma}$.
\end{proof}

\begin{proposition}\label{lift2}
Let $Q \in B_0 \setminus \Supp(\foD)$ be a point contained in the interior
of a maximal cone $\sigma$ of $\Sigma$ and let $\rho$ be an edge of $\sigma$.
Consider the formal sum $\Lift_Q(q)=\sum \Mono(\gamma)$. For $Q$ sufficiently close to $\rho$ each term of the sum is an analytic function on $V_{\rho,I}$ and the sum defines an analytic function on $V_{\rho,I}$.
\end{proposition}

\begin{proof}
Write $\rho=\rho_i$, without loss of generality assume
$\sigma=\sigma_{i,i+1}$, so that $\sigma_{i-1,i} \in \Sigma$
is the other maximal cone containing $\rho_i$.
Let $\tQ,\tilde{\rho_i},\tilde{\sigma}_{i,i+1},\tilde{\sigma}_{i-1,i}$
be compatible lifts to $\tB_0$.
Let $u_i,u_{i-1},u_{i+1}$ be the primitive generators of $\tilde{\rho_i}$
and the remaining edges of $\tilde{\sigma}_{i-1,i}$ and
$\tilde{\sigma}_{i,i+1}$,
and write $X_i,X_{i-1},X_{i+1}$ for the corresponding coordinates on
$V_{\rho,I}$.
So
\begin{align*}
V_{\rho,I} \subset \{(X_{i-1},X_i,X_{i+1})\,|\, |X_{i-1}|<R|X_i|, \,
|X_{i+1}|<R|X_i|\} \subset {} &
V(X_{i-1}X_{i+1}-z^{[D_{\rho}]}X_i^{-D_{\rho}^2}f_{\rho})\\
\subset {} & \bA^2_{X_{i-1},X_{i+1}} \times (\Gm)_{X_i} \times S'_I.
\end{align*}
Define
\begin{align*}
V_{\rho}= \{(X_{i-1},X_i,X_{i+1})\,|\,|X_{i-1}|<R|X_i|, \, |X_{i+1}|<R|X_i|\}
\subset {} & V(X_{i-1}X_{i+1}-z^{[D_{\rho}]}X_i^{-D_{\rho}^2}f_{\rho})\\
\subset {} & \bA^2_{X_{i-1},X_{i+1}} \times (\Gm)_{X_i} \times S.
\end{align*}
We assume that the orientation of $u_{i-1},u_{i+1}$ is the same as that of
$w_1,w_2$.

We first consider broken lines $\gamma$ lying in the cone generated by
$u_i$ and $w_2$. Write $\Mono(\gamma)=a_{\gamma}
z^{\tilde{\varphi}_{\tilde{\sigma}_{i,i+1}}(m_{\gamma})}$,
and $m_{\gamma}=\alpha u_i+\alpha_+u_{i+1}=\mu_1w_1+\mu_2w_2$.
By Lemma~\ref{endofbrokenline},(1), $|\mu_1|$ is bounded, and $\mu_2 > 0$ for all but finitely many $\gamma$. By Lemma~\ref{endofbrokenline},(2), $\alpha_+ \ge 0$, so $\Mono(\gamma)=a_{\gamma}X_i^{\alpha}X_{i+1}^{\alpha_+} \in
\kk[P][X_i^{\pm 1},X_{i-1},X_{i+1}]$
is analytic on $V_{\rho}$ for each $\gamma$.
In particular we may assume in what follows (discarding finitely many terms $\Mono(\gamma)$) that $\mu_2>0$.
Writing $w_1= -\beta_1 u_{i+1} + \gamma_1 u_i$ and $w_2=\beta_2 u_i +
\gamma_2(u_{i+1}-u_i)$,
we have $\beta_1,\beta_2,\gamma_1,\gamma_2>0$ and
\[
|z^{\tilde{\varphi}_{\tilde{\sigma}}(m_{\gamma})}|=
(|X_{i+1}|^{-\beta_1}|X_i|^{\gamma_1})^{\mu_1}(|X_i|^{\beta_2}
|X_{i+1}/X_i|^{\gamma_2})^{\mu_2}.
\]
Recall that $|X_{i+1}/X_i|, |X_{i-1}/X_i| <R$ on $V_{\rho}$.
Note then that for $0 < \delta < R^{-\gamma_2/\beta_2}$,
the second factor on the right is bounded for $|X_i|<\delta$ as we are taking
$\mu_2>0$. If
$\mu_1<0$, then we have
\[
(|X_{i+1}|^{-\beta_1}|X_i|^{\gamma_1})^{\mu_1}
< R^{-\beta_1\mu_1} |X_i|^{(-\beta_1+\gamma_1)\mu_1}
\]
which is bounded for $\delta'< |X_i| <\delta$ for any small $\delta'>0$.
Finally, if $\mu_1>0$, we use
the equation for $V_{\rho}$, which gives
\[
|X_{i+1}|^{-1}=|X_{i-1}| \cdot |X_i|^{D_{\rho}^2}
\cdot |f_{\rho}|^{-1}\cdot |z^{-[D_{\rho}]}|.
\]
The function $f_{\rho}$ on $V_{\rho}$ restricts to the constant function $1$ over $S_J$.
Hence we may impose the condition $|f_{\rho}|>\delta'$ for any small
$\delta'>0$. Note that
\[
|X_{i+1}|^{-\beta_1\mu_1}=(|X_{i-1}||X_i|^{D_{\rho}^2}
|f_{\rho}|^{-1}|z^{-[D_{\rho}]}|)^{\beta_1\mu_1}<R^{\beta_1\mu_1}
|X_i|^{\beta_1\mu_1(1+D_{\rho}^2)}|f_{\rho}|^{-\beta_1\mu_1}
|z^{[D_{\rho}]}|^{-\beta_1\mu_1}.
\]
Thus we see that in any event,
if $\delta'<|X_i|<\delta$, $|f_{\rho}|>\delta'$ and $|z^p|<\epsilon<1$ for
all $p\in P$, then
\[
|z^{\tilde{\varphi}_{\tilde{\sigma}}(m_{\gamma})}| \cdot |z^{[D_{\rho}]}|^c
\]
is bounded, where $c = \beta_1 \cdot \sup(\{\mu_1\},0)$ is a constant.
Now by Lemma~\ref{brokenlinebounds}, (3), again using $|z^p|<\epsilon<1$
for all $p \in P$, then
\[
|\Mono(\gamma)| = |a_{\gamma}z^{\tilde{\varphi}_{\tilde{\sigma}}(m_{\gamma})}|
=O(\epsilon^{|N|}).
\]
Recall that the number of broken lines for $T^N(\tilde{q})$ is $O(|N|^k)$ (Lemma~\ref{brokenlinebounds}, (2)).
We deduce that the sum $\sum \Mono(\gamma)$ over broken lines $\gamma$ lying in $\langle u_i,w_2 \rangle_{\bR_{\ge 0}}$ is uniformly convergent on compact sets for $|X_i|<\delta$, $f_{\rho} \neq 0$, and $|z^p|<1$ for all $p \in P$, where $\delta > 0$ is independent of $I$ and $q$.

Symmetrically, if we choose a basepoint  $Q'\in \sigma_{i-1,i}$ sufficiently
close to $\rho$, we can use the same argument for broken lines lying in
$\langle u_i,w_1 \rangle_{\bR_{\ge 0}}$ ending at $\tQ'$.
Such a broken line will have
$\Mono(\gamma)=a_{\gamma}X_{i-1}^{\alpha_-} X_i^{\alpha}$ with $\alpha_-\ge 0$.
The same argument as above shows that the sum
$\sum \Mono(\gamma)$ over all such
broken lines $\gamma$ is uniformly convergent on compact sets for
$|X_i|<\delta$, $f_{\rho}\not=0$, and $|z^p|<1$ for all $p\in P$,
where $\delta>0$ is again independent of $I$ and $q$. However, the statement
we are trying to prove involves the lift at $Q$, not $Q'$. For this purpose,
we will omit terms coming from broken lines with $\alpha_-=0$ from the
above sum, as such terms will also arise in the above analysis at $\tQ$.

To deal with this issue, note that given
such a broken line $\gamma$ ending at $\tQ'$ with $\alpha_->0$, we get a finite
number of broken lines
$\gamma_j$, $1\le j \le s$, ending at a point $\tQ''\in\tilde\sigma_{i,i+1}$
sufficiently close to $\trho_i$ as follows. Extend $\gamma$
until it reaches $\trho_i$ (which can be done since $\alpha_->0$).
Then applying the automorphism associated to
crossing $\trho_i$ to $\Mono(\gamma)$ gives $a_{\gamma}X_{i-1}^{\alpha_-}
X_i^{\alpha}f_{\rho}^{\alpha_-}$, which we write as $\sum_{j=1}^s
a_{\gamma_j}X_{i-1}^{\alpha_{-}}X_i^{\alpha_j}$. The $j^{th}$ monomial
in the sum gives a new broken line $\gamma_j$ bending at $\trho_i$ (unless
$\alpha_j=\alpha$, in which case we just extend the final line segment
of $\gamma$), with new attached monomial $a_{\gamma_j}X_{i-1}^{\alpha_-}
X_i^{\alpha_j}=a_{\gamma_j}z^{\tvarphi_{\tsigma_{i,i+1}}(m_{\gamma_j})}$.
As long as $\tQ''$ lies in the cone generated by $\rho_i$ and
$-m_{\gamma_j}$,
by extending or shortening the last line segment of
$\gamma_j$ and applying a homothety, one can obtain a broken line $\gamma_j$
ending at $\tilde Q''$. Thus for $\tQ''$ sufficiently close to $\trho_i$,
we obtain broken lines $\gamma_1,\ldots,\gamma_s$ ending at $\tQ''$.
However, the choice of $\tQ''$ may depend on $\gamma$.
To see there is a choice of $\tQ''$ which works for all $\gamma$,
note that for all but a finite number of $\gamma$, $m_{\gamma}$ lies
in the half-space $\RR_{\ge 0}\cdot w_1 + \RR \cdot w_2$ by
Lemma \ref{endofbrokenline}, (1). Further, since $\gamma_j$ bends
at $\trho_i$, $m_{\gamma}\in \ZZ \cdot u_i +\ZZ_{<0} \cdot u_{i+1}$. Thus
$m_{\gamma}=\beta u_i +\beta_+ u_{i+1}$ for $\beta>0$, $\beta_+ < 0$,
and $m_{\gamma_j}=(\beta - l) u_i + \beta_+ u_{i+1}$ for some $l>0$
with a bound only depending on $f_{\rho_i}\mod I$.
>From this it follows that there cannot be a sequence of
$\gamma$ and $\gamma_j$ constructed from $\gamma$ as above
so that the cones generated by $u_i$ and $-m_{\gamma_j}$
get smaller and smaller.

Thus we see that taking $\tQ$ sufficiently
close to $\trho_i$, the broken lines ending at $\tQ'$ contained in
$\langle w_1,\tQ'\rangle_{\RR_{\ge 0}}$ with $\alpha_->0$
give in the above fashion broken lines ending at $\tQ$ contained in
$\langle w_1, \tQ\rangle_{\RR_{\ge 0}}$, and every such broken line
ending at $\tQ$ clearly arises in this way. Furthermore, there are no
broken lines contained in $\langle w_1, \tQ\rangle_{\RR_{\ge 0}}$
with $\alpha_-=0$, and thus all broken lines ending at $\tQ$ have
been accounted for.

Now consider again a broken line $\gamma$ contained in $\langle w_1,
\tQ'\rangle_{\RR_{\ge 0}}$ with $\alpha_->0$.
By construction, $\sum_{j=1}^s
\Mono(\gamma_j)=\Mono(\gamma) f_{\rho}^{\alpha_-}$. We wish to understand
the contribution of $\sum_{j=1}^s \Mono(\gamma_j)$ to $\Lift_Q(q)$, and
to do so, we
write $\Mono(\gamma_j)$ in terms of $X_i, X_{i+1}$ using the relation
$X_{i-1}X_{i+1}=z^{[D_{\rho}]}X_i^{-D_{\rho}^2}$ in
$\kk[P_{\tilde\varphi_{\tilde\rho}}]$ (see Proposition \ref{ringfirstversion}).
So we can write
\[
\sum_j\Mono(\gamma_j)=
a_{\gamma}
z^{\alpha_-[D_{\rho}]}X_{i+1}^{-\alpha_-}X_i^{-\alpha_-D_{\rho}^2+\alpha}
f_{\rho}^{\alpha_-}.
\]
This defines a (possibly rational) function on $V_{\rho}$.
On the other hand, $\Mono(\gamma)=a_{\gamma}X_{i-1}^{\alpha_-}X_i^{\alpha}$
defines a holomorphic function on $V_{\rho}$, and using the equation
\begin{equation}
\label{basicrelagain}
X_{i-1}X_{i+1}=z^{[D_{\rho}]}X_i^{-D_{\rho}^2}f_{\rho}
\end{equation}
which is satisfied
on $V_{\rho}$, we have
\[
\Mono(\gamma)=a_{\gamma}
z^{\alpha_-[D_{\rho}]}X_{i+1}^{-\alpha_-}X_i^{-\alpha_-D_{\rho}^2+\alpha}
f_{\rho}^{\alpha_-}
\]
as a function on $V_{\rho}$. Thus we see that $\Mono(\gamma)$ and
$\sum_j\Mono(\gamma_j)$ coincide as functions on $V_{\rho}$, in the above
interpretation. (Essentially we are just using the fact that
the relation \eqref{basicrelagain} encodes the automorphism associated
to crossing $\rho$).
Thus the fact that
$\sum_{\gamma}\Mono(\gamma)$ defines an analytic function on $V_{\rho,I}$
for broken lines $\gamma$ ending at $\tQ'$ contained in $\langle u_i,w_1
\rangle_{\RR_{\ge 0}}$ implies that the sum $\sum_{\gamma}\Mono(\gamma)$
over all broken lines $\gamma$ ending at $\tilde Q$ and contained in
$\langle \tQ, w_1\rangle_{\RR_{\ge0}}$ is analytic. This implies
$\Lift_Q(q)$ is analytic on $V_{\rho,I}$.
\end{proof}

\begin{lemma} \label{endofbrokenline}
Let $\tQ \in \tB_0 \setminus \Supp(\tilde{\foD})$ be a point contained in the interior of a maximal cone $\tilde{\sigma} \in \tilde{\Sigma}$.
Consider broken lines $\gamma$ on $\tB_0$ for $T^N(\tq)$ with endpoint $\tQ$, for all $N \in \bZ$ such that $T^N(\tq) \in \langle \tQ,w_2
\rangle_{\bR_{\ge 0}}$.
Write $\Mono(\gamma)=a_{\gamma}z^{\tilde{\varphi}_{\tilde{\sigma}}(m_{\gamma})}$, and $m_{\gamma}=\mu_1w_1+\mu_2w_2$.
\begin{enumerate}
\item $|\mu_1|$ is bounded, and $\mu_2$ is positive for all but finitely many $\gamma$.
In particular, $\mu_1,\mu_2$ are bounded below.
\item Let $u_1,u_2$ be generators of $\tilde{\sigma}$ with the same orientation as $w_1,w_2$.
Then for $\tilde{Q}$ sufficiently close to $\tilde{\rho}:=\bR_{\ge 0}\cdot u_1$, $m_{\gamma}$ lies in the half space $\bR \cdot u_1+\bR_{\ge 0}\cdot u_2$ for each $\gamma$.
\end{enumerate}
\end{lemma}
\begin{proof}
(1) We may assume the lift $\tilde q$ of $q$ is chosen so that
$T^N(\tilde q)\in \langle \tilde Q,w_2\rangle_{\RR_{\ge 0}}$ if and only if
$N\ge 0$.
Note that the rays spanned by $w_1,w_2$ are irrational so $\mu_1,\mu_2 \neq 0$.
Suppose for a contradiction that there is an infinite sequence of broken lines $\gamma$ such that
$m_{\gamma}=\mu_1w_1+\mu_2w_2$ with $\mu_2 < 0$.
Each broken line has at most $k$ bends and there are a finite
number of $\Gamma$-orbits of possible changes $\alpha \in M$ of the
derivative of $\gamma$ at a bend;
see Lemma~\ref{brokenlinebounds} and its proof.
So, passing to a subsequence, we may assume that the bends
(in order of increasing $t \in (-\infty,0]$) of
each $\gamma$ are of types
$T^{s_1}\alpha_1,\ldots,T^{s_l}\alpha_l$ for some fixed
$\alpha_1,\ldots,\alpha_l \in M$, $l\le k$, and
$N \ge s_1 \ge s_2 \ge \cdots \ge s_l \ge 0$ (depending on $\gamma$).
Say $N-s_{i}$ is bounded for $i \le l'$ and unbounded otherwise.
Passing to a subsequence, we may assume that $N-s_i$ is constant for $i \le l'$.
Let $\gamma'$ be the broken line obtained by truncating $\gamma$ after the first $l'$ bends, and moving by a homothety
so that (extending its final line segment) $\gamma'$ ends at $\tQ$.
Then $m_{\gamma'}=T^N(m)$ for some fixed $m \in M$. Furthermore,
$m_{\gamma}$ is obtained from $m_{\gamma'}$ by adding a positive linear
combination of $w_1$ and $w_2$. But since $\mu_2<0$ and $w_1, w_2$ are
eigenvectors of $T$ with eigenvalues $\lambda^{-1}, \lambda$,
it follows that we must have $m=\nu_1w_1+\nu_2w_2$ with $\nu_2<0$.
But then for sufficiently large $N$, $T^N(m)$ does not lie in the
half-space $\bR \cdot \tQ + \bR_{\ge 0} \cdot w_2$.
This is a contradiction because $m_{\gamma'}$ always lies in this half-space.

To see that $|\mu_1|$ is bounded, recall that $m_{\gamma}=T^N(\tq)-\sum_{i=1}^l T^{s_i}\alpha_i$ where $0 \le l \le k$, $0 \le s_i \le N$, and the $\alpha_i$ are selected from a finite set.  Now since $Tw_1=\lambda^{-1}w_1$ it follows that $|\mu_1|$ is bounded.

(2) Let $\mathfrak{u}$ be the connected component of
$\tilde B_0\setminus\Supp_I(\tilde{\foD})$ contained in $\tilde{\sigma}$ and
containing $\tilde{\rho}$ in its closure.
Let $\tQ' \in \mathfrak{u}$ be a point such that $\tQ' \in \langle \tilde{\rho},\tQ \rangle_{\bR_{\ge 0}}$.
Then if $T^N(\tq) \in \langle \tQ,w_2 \rangle_{\bR_{\ge 0}}$ and
$\gamma'$ is a broken line for $T^N(\tq)$ with endpoint $\tQ'$, we obtain a broken line $\gamma$ for $T^N(\tq)$ with endpoint $\tQ$ and $m_{\gamma}=m_{\gamma'}$ as follows. First apply a homothety to obtain a broken line passing through $\tQ$, then truncate at $\tQ$.
This gives an injective map between the set of broken lines for $T^N(\tilde q)
\in \langle \tilde Q, w_2\rangle_{\RR_{\ge 0}}$ ending at $\tQ'$
with the set of such broken lines ending at $\tQ$.

Now suppose $\gamma$ is a broken line for $T^N(\tq)$ with endpoint $\tQ$ and $m_{\gamma}$ not lying in the half-space $\bR \cdot u_1 + \bR_{\ge 0} \cdot u_2$. Since $m_{\gamma}$ lies in the half-space $\bR \cdot \tQ + \bR_{\ge 0} \cdot w_2$ we find $m_{\gamma}$ lies in the cone generated by $-\tQ, -u_1$.
In particular, $m_{\gamma}$ does not lie in the half-space $\bR \cdot w_1 + \bR_{\ge 0} w_2$, so by (1) there are only finitely many such $\gamma$.
Now by the above construction it follows that for $\tQ$ sufficiently close to $\tilde{\rho}$ there are none.
\end{proof}

\subsection{Smoothness}\label{smoothness}

We will now complete the proof of:

\begin{theorem}(Looijenga's conjecture)\label{looijengasconjecture}
Suppose that $\kk=\bC$ and the intersection matrix $(D_i \cdot D_j)_{1 \le i,j \le n}$ is negative definite, so that $D \subset Y$ can be contracted to a cusp singularity $q \in Y'$.
Then the dual cusp to $q \in Y'$ is smoothable.
\end{theorem}

We continue to work with the setup at the beginning of \S\ref{cusp1}, with
$L$, $\sigma_P$, $P$, $\fom$ and $J$ as given there.
By Theorem \ref{thickenedcuspfamily}, if $I$ is a monomial ideal with
$\sqrt{I}=J$, we obtain a deformation $X_I\rightarrow S_I'$ of
$X_J\rightarrow S_J'$. For the remainder of the section, we shall write
$S_I$ for $S_I'$.
Let $\mathfrak{f}_J \colon \mathfrak{X}_J \rightarrow \mathfrak{S}_J$
denote the formal deformation determined by the deformations
$X_{J^{N+1}} \rightarrow S_{J^{N+1}}$ for $N \ge 0$.
Thus, $\mathfrak{S}_J$ is the formal complex analytic space obtained as the completion of $S$ along $S_J$, $\mathfrak{X}_J$ is a formal complex analytic space, and $\mathfrak{X}_J \rightarrow \mathfrak{S}_J$ is an adic flat morphism.
We similarly have the family $\mathfrak{X}_{\fom}\rightarrow
\mathfrak{S}_{\fom}$ of formal schemes already studied in
\S\ref{smoothnesssection}.
We refer to \cite{G60} and \cite{B78} for background on
formal schemes and formal complex analytic spaces.

We have a section $s \colon S_J \rightarrow X_J$ such that, for $t \in S_J$
general, the point $s(t) \in X_{J,t}$ on the fibre is the
cusp. We write $X_J^o := X_J \setminus s(S_J) \subset X_J$ and $X_I^o \subset X_I$, $\mathfrak{X}_J^o \subset \mathfrak{X}_J$ for the induced open embeddings.

Let $Z_I:=\Sing(f_I) \subset X_I$ denote the singular locus of $f_I \colon X_I \rightarrow S_I$, see Definition-Lemma \ref{singdef}.
Thus $Z_I \subset X_I$ is a closed embedding of schemes or complex analytic spaces.
Since the singular locus is compatible with base-change, the singular loci $Z_{J^n} \subset X_{J^n}$ determine a closed embedding $\mathfrak{Z}_J \subset \mathfrak{X}_J$ which we refer to as the singular locus of $\fof_J \colon \mathfrak{X}_J \rightarrow \mathfrak{S}_J$.

\begin{lemma} \label{reducetofinite2}
In the above situation,
there exists $0 \neq g \in \kk[P]$ such that $\Supp(g \cdot \cO_{\mathfrak{Z}_J})$ is contained in $s(S_J)$.
In particular, $\fof_{J*}(g \cdot \cO_{\mathfrak{Z}_J})$ is a coherent sheaf on $\mathfrak{S}_J$.
\end{lemma}

\begin{proof}
The proof is essentially the same as that of Lemma \ref{reducetofinite}.
Let $\mathfrak{U}_{i,J}$ be defined as in \eqref{mathfrakU}.
Then $\mathfrak{X}_J^o$ is a union of open subspaces $
\mathfrak{V}_{i,J}$, $i=1,\ldots,n$, such that $\mathfrak{V}_{i,J}$ is an
analytic open subspace of $\mathfrak{U}_{i,J}$ for each $i$. We then
take $g=a_1\cdots a_n$ as in the proof of Lemma \ref{reducetofinite},
so that $\Supp(g\cdot \shO_{\mathfrak{Z}_J})$ is contained in $s(S_J)$.
So the support of $g \cdot \cO_{\mathfrak{Z}_J}$ is a closed subset of
$s(S_J)$, hence proper over $S_J$.
It follows that $\fof_{J*}(g \cdot \cO_{\mathfrak{Z}_J})$ is coherent by
\cite{B78}, 3.1.
\end{proof}

Let $u(J)$ denote the natural map
\[
u(J) \colon \cO_{\mathfrak{S}_J} \rightarrow \fof_{J*}(\cO_{\mathfrak{Z}_J}),
\]
and $u(\fom)$ similarly the natural map
\[
u(\fom) \colon \cO_{\mathfrak{S}_{\fom}} \rightarrow \fof_{\fom*}
(\cO_{\mathfrak{Z}_{\fom}}).
\]

\begin{lemma}\label{injectivityJiffm}
$u(J)$ is injective if and only if $u(\mathfrak{m})$ is injective.
\end{lemma}

\begin{proof}
Let $0 \neq g \in \kk[P]$ be the element given by Lemma~\ref{reducetofinite2}.
Let $\cK_J$ be the kernel of $u(J)$ and $\cK'_J$ the kernel of $g \cdot u(J)$.
Thus $\cK_J,\cK'_J$ are ideal sheaves in $\cO_{\mathfrak{S}_J}$ and $g \cdot \cK'_J \subset \cK_J \subset \cK'_J$.
The local rings of $\mathfrak{S}_J$ are domains by Lemma~\ref{integral}, so $\cK_J=0$ if and only if $\cK'_J=0$.
The sheaf $\cK'_J$ is coherent because the image of $g \cdot u(J)$ is contained in the coherent subsheaf
$\fof_{J*}(g \cdot \cO_{\mathfrak{Z}_J}) \subset \fof_{J*}\cO_{\mathfrak{Z}_J}$.

We claim that the natural map
$$\cK'_J \otimes_{\cO_{\mathfrak{S}_J}} \cO_{\mathfrak{S}_{\mathfrak{m}}} \rightarrow \cK'_{\mathfrak{m}}$$
is an isomorphism.
Let $z \in \mathfrak{S}_{\mathfrak{m}}$ be the unique point, coinciding with the
zero-dimensional torus orbit of $S$,
and let $\hat{\cO}_{S,z}$ denote the completion of $\cO_{S,z}$ at its maximal ideal.
Note that $\hat{\cO}_{S,z}$ coincides with $\cO_{\mathfrak{S}_{\mathfrak{m}},z}$
and the completion of $\cO_{\mathfrak{S}_J,z}$ at its maximal ideal.
It suffices to show that the map
$$
\cK'_{J,z} \otimes_{\cO_{\mathfrak{S}_J,z}} \hat{\cO}_{S,z}
\rightarrow \cK'_{\mathfrak{m},z}
$$
is an isomorphism.
We have an exact sequence of coherent sheaves
$$0 \rightarrow \cK'_J \rightarrow \cO_{\mathfrak{S}_J} \rightarrow \fof_{J*}(g \cdot \cO_{\mathfrak{Z}_J})$$
and so an exact sequence of $\hat{\cO}_{S,z}$-modules
$$0 \rightarrow \cK'_{J,z} \otimes \hat{\cO}_{S,z} \rightarrow \hat{\cO}_{S,z}
\rightarrow \fof_{J*}(g \cdot \cO_{\mathfrak{Z}_J})_z \otimes \hat{\cO}_{S,z}.$$
Now
$$\fof_{J*}(g \cdot \cO_{\mathfrak{Z}_J})_z \otimes \hat{\cO}_{S,z}
=(g \cdot \cO_{\mathfrak{Z}_J})_{s(z)} \otimes \hat{\cO}_{S,z}
=\widehat{(g \cdot \cO_{\mathfrak{Z}_J})}_{s(z)}
= g \cdot \hat{\cO}_{\mathfrak{Z}_J,s(z)}$$
where the hats denote completion with respect to the maximal ideal of $\cO_{\mathfrak{S}_J,z}$.
Thus $\cK'_{J,z} \otimes \hat{\cO}_{S,z}$ is the kernel of the map
$$\hat{\cO}_{S,z} \rightarrow g \cdot \hat{\cO}_{\mathfrak{Z}_J,s(z)}.$$
By the base-change property for the singular locus, this map coincides with the corresponding map for $\mathfrak{m}$. This proves the claim.

The support of the ideal sheaf $\cK'_J$ is either empty or $S_J$ (because the local rings of $\mathfrak{S}_J$ are domains and $S_J$ is connected).
So $\cK'_J=0$ if and only if $\cK'_{\mathfrak{m}}=0$ by the claim.
\end{proof}

\begin{lemma}\label{integral}
The local rings of $\mathfrak{S}_J$ are integral domains.
\end{lemma}
\begin{proof}
The completion of the local ring of $\mathfrak{S}_J$ at a point $z \in S_J$ is identified with the completion of the local ring of the toric variety $S$
at $z$. By Serre's criterion for normality, the completion of a normal Noetherian ring at a maximal ideal is a local normal Noetherian ring;
in particular, it is a domain. Since $\cO_{\mathfrak{S}_{J},z}$ is a local Noetherian ring, it is contained in its completion. We deduce that $\cO_{\mathfrak{S}_{J},z}$ is a domain.
\end{proof}

\begin{proof}[Proof of Theorem~\ref{looijengasconjecture}]
Let $f \colon Y \rightarrow Y'$ be the contraction of $D \subset Y$.
We may assume $f$ is the minimal resolution of $Y'$.
We may further assume $n \ge 3$.
Indeed, the embedding dimension of the dual cusp equals $\max(n,3)$ by \cite{N80}, Corollary~7.8, p.~232 and \cite{KM98}, Theorem~4.57, p.~143, see also
\cite{L81}, pg.\ 307.
So in particular for $n \le 3$ the dual cusp is a hypersurface and thus smoothable.
Let $L$ be a nef divisor on $Y$ such that $\NE(Y)_{\bR_{\ge 0}} \cap L^{\perp}=\langle D_1,\ldots,D_n \rangle_{\bR_{\ge 0}}$.
Let $\sigma_P \subset A_1(Y,\bR)$ be a strictly rational polyhedral cone containing $\NE(Y)$ such that $\sigma_P \cap L^{\perp}$ is a face of $\sigma_P$.
Let $P=\sigma_P \cap A_1(Y,\bZ)$ and $J=P \setminus P \cap L^{\perp}$.
By Lemma~\ref{injectivityJiffm} and Theorem~\ref{injectivitym}, $u(J)$ is not injective.

Let $x \in S_J$ be a point lying in the interior of the toric variety $S_J$ and $h \in \cO_{\mathfrak{S}_J,x}$ a nonzero element of the kernel of $u(J)$ near $x$. By Lemma~\ref{arc} there is a morphism
$$v \colon \Spec\bC[t]/(t^{N+1}) \rightarrow \mathfrak{S}_J$$
taking the unique point of the domain to $x$
and $0 \neq v^*(h) \in \bC[t]/(t^{N+1})$.
Let $Y/\Spec(\bC[t]/(t^{N+1}))$ be the pullback of $\mathfrak{X}_J/\mathfrak{S}_J$ by $v$ and $Z \subset Y$ its singular locus. Then $Y/\Spec(\bC[t]/(t^{N+1}))$ is a deformation of the dual cusp singularity. Furthermore, $\cO_Z$ is
annihilated by $v^*(h)$, and the ideal generated by $v^*(h)$ must contain
$t^N$, so $\cO_Z$ is annihilated by $t^N$. By \cite{A76}, Theorem~5.1, there is an algebraic finite type deformation $Y'/\Spec\bC\lfor t\rfor$
whose restriction to $\Spec(\bC[t]/(t^{N+1}))$ is locally analytically isomorphic to $Y/\Spec(\bC[t]/(t^{N+1}))$. Let $Z' \subset Y'$ denote the singular locus of $Y'/\Spec\bC\lfor t\rfor$.
Then $\cO_{Z'}$ is a finite $\bC\lfor t\rfor$-module
because the fibre $Y'_0$ has an isolated singularity (using \cite{Matsumura89},
Theorem~8.4, p.~58).
Now $\cO_Z=\cO_{Z'}/t^{N+1}\cO_{Z'}$ and $t^N\cO_Z = 0$, so $t^N\cO_{Z'}=t^{N+1}\cO_{Z'}$ and thus $t^N\cO_{Z'}=0$ by Nakayama's lemma. Hence the general fibre of $Y'/\Spec\bC\lfor t\rfor$ is smooth, and $Y'/\Spec\bC\lfor t\rfor$ is a smoothing of the dual cusp.
\end{proof}

\begin{lemma} \label{arc}
Let $A$ be the completion of a finitely generated normal Cohen-Macaulay
$\bC$-algebra at a maximal ideal.
Let $0 \neq a \in A$. Then there exists $N \ge 0$ and a $\bC$-algebra map
$f: A \to \bC[t]/(t^{N+1})$ such that $f(a) \neq 0$.
\end{lemma}
\begin{proof}
Extend $a$ to a regular sequence $a,t_1,\ldots,t_r$ of length $\dim A$.
Then the normalization of $A/(t_1,\ldots,t_r)$ is a finite direct sum of
copies of $\bC\lfor t\rfor $. Now the result is clear.
\end{proof}


\begin{thebibliography}{99}



\bibitem[AC11]{AC11} D.~Abramovich, Q.~Chen, Stable logarithmic
maps to Deligne-Faltings pairs II, arXiv:1102.4531.

\bibitem[A02]{A02} V.~Alexeev,
Complete moduli in the presence of semiabelian group action,  Ann. of Math.~(2)~155 (2002),
no.~3, 611--708.


\bibitem[A98]{A98} K.~Altmann, $P$-resolutions of cyclic quotients from the toric viewpoint, in Singularities (Oberwolfach, 1996), 241--250,
Progr. Math.~162, Birkh\"auser, 1998.

\bibitem[A76]{A76} M.~Artin, Lectures on deformations of singularities, Lectures on Mathematics and Physics~54, Tata Inst. Fund. Res., 1976.

\bibitem[A07]{A07} D.~Auroux, Mirror symmetry and T-duality in the complement of an anticanonical divisor, J. G\"okova Geom. Topol.~1 (2007), 51--91.



\bibitem[AMRT75]{AMRT75} A.~Ash, D.~Mumford, M.~Rapoport, Y.~Tai,
Smooth compactification of locally symmetric varieties, Math. Sci. Press, 1975.



\bibitem[B78]{B78} J.~Bingener, \"Uber formale komplexe R\"aume,
Manuscripta Math.~24 (1978), no.~3, 253--293.



\bibitem[C1869]{C1869} A. ~Cayley, A Memoir on Cubic Surfaces,
Philosophical Transactions of the Royal Society of London~159 (1869), 231--326.



\bibitem[CPS]{CPS} M.~Carl, M.~Pumperla, and B.~Siebert,
A tropical view of Landau-Ginzburg models, available at
http://www.math.uni-hamburg.de/home/siebert/preprints/LGtrop.pdf

\bibitem[CO06]{CO06} C.-H.~Cho, Y.-G.~Oh, Floer cohomology and disc instantons
of Lagrangian torus fibers in Fano toric manifolds, Asian J. Math. {\bf 10},
(2006), 773-814.


\bibitem[E95]{E95} D.~Eisenbud, \emph{Commutative algebra with a view
toward Algebraic Geometry,} Graduate Texts in Mathematics~150,
Springer-Verlag, New York, 1995.



\bibitem[FG09]{FG09} V.~Fock and A.~Goncharov, Cluster ensembles, quantization
and the dilogarithm, Ann. Sci.\'Ec. Norm. Sup\'er. (4) 42 (2009), no. 6,
865--930.






\bibitem[FM83]{FM83} R.~Friedman and R.~Miranda,
Smoothing cusp singularities of small length,
Math. Ann.~263 (1983), no.~2, 185--212.

\bibitem[FP84]{FP84} R.~Friedman and H.~Pinkham,
Smoothings of cusp singularities via triangle singularities. With an appendix by H.~Pinkham,
Compositio Math.~53 (1984), no.~3, 303--324.


\bibitem[F75]{F75} A.~Fujiki, On the blowing down of analytic spaces,
Publ. Res. Inst. Math. Sci.~10 (1974/75), 473--507.





\bibitem[Giv]{Giv} A.~Givental, Homological geometry. I. Projective
hypersurfaces. Selecta Math.~1 (1995), 325--345.



\bibitem[G09]{G09} M.~Gross,  Mirror symmetry for $\bP^2$ and tropical
geometry, Adv.\ Math., {\bf 224} (2010), 169--245.

\bibitem[G11]{G11} M.~Gross, Tropical geometry and mirror symmetry, CBMS Regional Conf. Ser. in Math.~114, A.M.S., 2011.

\bibitem[GHK12]{GHK12} M.~Gross, P.~Hacking, S.~Keel, Moduli of surfaces with
anti-canonical cycle, to appear in Comp.\ Math., preprint, 2012.

\bibitem[GHK13]{GHK13} M.~Gross, P.~Hacking, S.~Keel, Birational geometry of
cluster algebras, to appear in Algebraic Geometry, preprint, 2013.

\bibitem[GHKK]{GHKK} M.~Gross, P.~Hacking, S.~Keel, M.~Kontsevich,
Canonical bases for cluster algebras, preprint, 2014.

\bibitem[GHKII]{GHKII}
M.~Gross, P.~Hacking, S.~Keel, Mirror symmetry for log Calabi-Yau surfaces II,
in preparation.

\bibitem[GHKS]{GHKS}
M.~Gross, P.~Hacking, S.~Keel, B.~Siebert, Theta functions on varieties
with effective anticanonical class, in preparation.

\bibitem[K3]{K3}
M.~Gross, P.~Hacking, S.~Keel, B.~Siebert, Theta functions for K3
surfaces, in preparation.



\bibitem[GPS09]{GPS09} M.~Gross, R.~Pandharipande, B.~Siebert, The
tropical vertex, Duke Math.\ J.\ {\bf 153}, (2010), 197--362.

\bibitem[GS06]{GS06} M.~Gross, B.~Siebert, Mirror symmetry via logarithmic
degeneration data, I. J. Differential Geom., {\bf 72}, (2006) 169--338.

\bibitem[GS07]{GS07} M.~Gross, B.~Siebert, From real affine geometry to
complex geometry, Annals of Mathematics, {\bf 174}, (2011), 1301-1428.

\bibitem[GS08]{GS08} M.~Gross, B.~Siebert, An invitation to toric
degenerations,
Surveys in differential geometry. Volume XVI. Geometry of special holonomy and
related topics,  43--78, Surv.\ Differ.\ Geom., {\bf 16},
Int.\ Press, Somerville, MA, 2011.

\bibitem[GS11]{GS11} M.~Gross, B.~Siebert, Logarithmic Gromov-Witten invariants,
J.\ Amer.\ Math.\ Soc.\ {\bf 26}  (2013), 451--510.

\bibitem[GSTheta]{GSTheta} M.~Gross, B.~Siebert, Theta functions
and mirror symmetry, arXiv:1204.1991 [math.AG].

\bibitem[G60]{G60} A.~Grothendieck, \'El\'ements de g\'eom\'etrie alg\'ebrique~I.
Le langage des sch\'emas, Inst. Hautes \'Etudes Sci. Publ. Math. No.~4.

\bibitem[G61]{G61} A.~Grothendieck, \'El\'ements de g\'eom\'etrie alg\'ebrique~III,
\'Etude cohomologique des faisceaux coh\'erents~I,
Inst. Hautes \'Etudes Sci. Publ. Math. No.~11 (1961).

\bibitem[H04]{H04} P. Hacking,
Compact moduli of plane curves, Duke Math. J.~124 (2004), no.~2, 213--257.

\bibitem[Hi73]{Hi73} F.~Hirzebruch, Hilbert modular surfaces, Enseignement
Math., {\bf 19}, (1973), 183-281.


\bibitem[H77]{H77} R.~Hartshorne, Algebraic geometry, Grad. Texts in
Math.~52, Springer, 1977.

\bibitem[IP03]{IP03} E.-N.~Ionel, T.~Parker, Relative Gromov-Witten invariants,
Ann. of Math. {\bf 157} (2003), 45--96.


\bibitem[KM98]{KM98} J.~Koll\'ar, S.~Mori, Birational geometry
of algebraic varieties, Cambridge Tracts in Math.~134. C.U.P., 1998.




\bibitem[KS06]{KS06} M.\ Kontsevich, Y.\ Soibelman:
        \emph{Affine structures and non-Archimedean analytic spaces}, in:
        \textsl{The unity of mathematics} (P.~Etingof, V.~Retakh,
        I.M.~Singer, eds.),  321--385, Progr.\ Math.~244,
        Birkh\"auser~2006.

\bibitem[L02]{L02} S.~Lang, Algebra, 3rd ed., Grad. Texts in Math.~211,
Springer, 2002.


\bibitem[L73]{L73} H.~Laufer, Taut two-dimensional singularities,
Math. Ann.~205 (1973), 131--164.

\bibitem[Li00]{Li00} J.~Li, Stable morphisms to singular schemes and relative
stable morphisms, J. Diff. Geom. {\bf 57}, (2000), 509--578.

\bibitem[Li02]{Li02} J.~Li, A degeneration formula of GW-invariants, J. Diff. Geom.,
{\bf 60}, (2002) 199-293.

\bibitem[LT98]{LT98} J.~Li, G.~Tian, Virtual moduli cycles and Gromov-Witten
invariants of algebraic varieties, J.\ of the AMS, {\bf 11}, (1998), 119--174.

\bibitem[LR01]{LR01} A.-M.~Li, Y.~Ruan, Symplectic surgery and Gromov-Witten
invariants of Calabi-Yau 3-folds, I, Invent. Math. {\bf 145} (2001),
151--218.

\bibitem[L76]{L76} E.~Looijenga, Root systems and elliptic curves, Invent. Math. ~38 (1976),
{17--32}.

\bibitem[L81]{L81} E. Looijenga,
Rational surfaces with an anticanonical cycle.  Ann. of Math.~(2)~114
(1981), no. 2, 267--322.

\bibitem[M82]{M82}J. ~M\'erindol,  Les singularit\'es simples elliptiques, leurs d\'eformations,
les surfaces de del Pezzo et les transformations quadratiques.
Ann. Sci. \'Ecole Norm. Sup. (4) 15 (1982), no. 1, 17--44


\bibitem[Ma89]{Matsumura89} H.~Matsumura, Commutative ring theory, C.U.P., 1989.


\bibitem[Mum]{Mum} D.~Mumford, An analytic construction of degenerating abelian
varieties over complete rings, Compositio Math. {\bf 24}, (1972), 239-272.



\bibitem[N80]{N80} I.~Nakamura, Inoue-Hirzebruch surfaces and a duality of
hyperbolic unimodular singularities,
Math. Ann.~252 (1980), no.~3, 221--235.



\bibitem[Ob04]{Ob04} A.~Oblomkov, Double affine Hecke algebras
of rank $1$ and affine cubic surfaces, Int. Math. Res. Not.~2004, no. 18,
877--912.




\bibitem[P74]{P74} H.~Pinkham, Deformations of algebraic varieties with $\bG_m$ action, Ast\'erisque~20, S.M.F., 1974.

\bibitem[R83]{R83} M.~Reid, Decomposition of toric morphisms, in Arithmetic and geometry, Vol.~II, 395--418,
Progr. Math.~36, Birkh\"auser, 1983.





%



\bibitem[SYZ96]{SYZ96} A.~Strominger, S.-T.~Yau, and E.~Zaslow, Mirror Symmetry
is $T$-duality, Nucl.\ Phys.\ {\bf B479}, (1996) 243--259.

\bibitem[Ty99]{Ty99} A.~Tyurin, Geometric quantization and mirror symmetry,
arXiv:math/9902027.




\bibitem[W76]{W76} J.~Wahl, Equisingular deformations of normal surface
singularities~I, Ann. of Math.~(2)~104 (1976), no.~2, 325--356.



\end{thebibliography}
\end{document}